\documentclass[a4paper, 11pt, headings = small, abstract]{scrartcl}
\usepackage[english]{babel}
\usepackage{geometry}

\usepackage{amsmath}
\usepackage{amsfonts}
\usepackage{amssymb}
\usepackage{amsthm}
\usepackage{mathtools}
\usepackage{paralist}
\usepackage{comment}
\usepackage{booktabs}
\usepackage[normalem]{ulem}
\usepackage{csquotes}
\usepackage[
backend=biber,
style=alphabetic,
url=false,
giveninits=true,
isbn=false,
]{biblatex}
\DeclareNameAlias{author}{family-given}
\usepackage[dvipsnames]{xcolor}
\usepackage[colorlinks=true,linkcolor=blue,citecolor=blue,pdfborder={0 0 0}]{hyperref}
\usepackage{orcidlink}

\usepackage{dsfont}
\usepackage{bm}

\usepackage{graphicx}
\usepackage{enumerate}

\usepackage{tikz}
\usetikzlibrary{positioning}
\usetikzlibrary{backgrounds}

\newtheorem{theorem}{Theorem}[section]
\newtheorem{proposition}[theorem]{Proposition}
\newtheorem{lemma}[theorem]{Lemma}
\newtheorem{corollary}[theorem]{Corollary}

\theoremstyle{definition}
\newtheorem{definition}[theorem]{Definition}
\newtheorem{condition}[theorem]{Condition}

\theoremstyle{remark}

\newtheoremstyle{boldremark}  % Name of the style
  {3pt}                       % Space above
  {3pt}                       % Space below
  {}                          % Body font (normal)
  {}                          % Indent amount
  {\bfseries}                 % Theorem head font (bold)
  {.}                         % Punctuation after theorem head
  { }                         % Space after theorem head
  {}                          % Theorem head spec (empty means ‘normal’)

\theoremstyle{boldremark}
\newtheorem{remark}[theorem]{Remark}
\newtheorem{example}[theorem]{Example}

\allowdisplaybreaks[4]

\numberwithin{equation}{section}
\numberwithin{figure}{section}
\numberwithin{table}{section}

\newcommand{\R}{\mathbb{R}}

\newcommand{\Z}{\mathbb{Z}}
\newcommand{\N}{\mathbb{N}}
\newcommand{\eps}{\varepsilon}

\newcommand{\Gb}{\mathbb {G}}

\newcommand{\Pb}{\mathbb P}

\newcommand{\Nc}{\mathcal{N}}

\newcommand{\weak}{\rightsquigarrow}

\newcommand{\e}{\mathrm{e}}

\newcommand{\Prob}{{\Pr}}    %probability
\renewcommand{\P}[1]{\Prob\left(#1\right)}
\newcommand{\caseq}{=}

\newcommand{\Exp}{\operatorname{\mathds E}}
\newcommand{\Var}{\operatorname{Var}}
\newcommand{\Cov}{\operatorname{Cov}}

\newcommand{\argmax}{\operatornamewithlimits{\arg\max}}
\newcommand{\diff}{\,\mathrm{d}}
\newcommand{\flo}[1]{{\lfloor #1 \rfloor}} 
\newcommand{\floor}[1]{{\lfloor#1\rfloor}}

\newcommand{\dbl}{{\ensuremath {\operatorname{db} }}}
\newcommand{\sbl}{{\ensuremath {\operatorname{sb} }}}
\newcommand{\abl}{{\ensuremath {\operatorname{ab} }}}
\newcommand{\mbl}{{\ensuremath {\operatorname{mb} }}}

\newcommand{\indi}{{\perp \hspace{-.185cm} \perp}}
\newcommand{\pd}{\mathrm{pd}}
\newcommand{\TT}{{{\ensuremath {\operatorname{TopTwo}}}}}

\newcommand{\RL}{\mathrm{RL}}

\usepackage{pgfplots}
\pgfplotsset{compat=1.18} 
\newcommand{\G}{\mathbb{G}}
\newcommand{\z}{\boldsymbol{z}}
\newcommand{\zz}{\boldsymbol{Z}}
\newcommand{\W}{\boldsymbol{W}}

\newcommand{\y}{\boldsymbol{y}}

\newcommand{\pderiv}[2]{\partial_{#2} #1}

\newcommand{\Frechet}{\ensuremath{\text{Fr}\acute{\text{e}}\text{chet}}}
\DeclareMathOperator{\Pareto}{Pareto}
\DeclareMathOperator{\Bernoulli}{Bernoulli}

\DeclareMathOperator{\indic}{{\bm 1}}

\newcommand{\AR}{\mathrm{AR}}
\newcommand{\ARMAX}{\mathrm{ARMAX}}

\begin{document}

\title{\fontsize{16}{19} Extreme Value Analysis based on Blockwise Top-Two Order Statistics}

\author{
Axel B\"ucher\thanks{Ruhr-Universität Bochum, Fakultät für Mathematik. Email: \href{mailto:axel.buecher@rub.de}{axel.buecher@rub.de}} \orcidlink{0000-0002-1947-1617}
\and
Erik Haufs\thanks{Ruhr-Universität Bochum, Fakultät für Mathematik. Email: \href{mailto:erik.haufs@rub.de}{erik.haufs@rub.de}} \orcidlink{0009-0008-8194-7445}
}

\date{\today}

\maketitle

\begin{abstract} 
Extreme value analysis for time series is often based on the block maxima method, in particular for environmental applications. In the classical univariate case, the latter is based on fitting an extreme-value distribution to the sample of (annual) block maxima. 
Mathematically, the target parameters of the extreme-value distribution also show up in limit results for other high order statistics, which suggests estimation based on blockwise large order statistics. It is shown that a naive approach based on maximizing an independence log-likelihood yields an estimator that is inconsistent in general. A consistent, bias-corrected estimator is proposed, and is analyzed theoretically and in finite-sample simulation studies. The new estimator is shown to be more efficient than traditional counterparts, for instance for estimating large return levels or return periods.
\end{abstract}

\color{black}

\smallskip


\tableofcontents

\section{Introduction}

Extreme value statistics is concerned with analyzing extreme events such as heavy rainfall, floods, or stock market crashes, based on observed time series data \cite{Bei04}. In the univariate, stationary case, common target parameters include the 100-year return level (the threshold expected to be exceeded once every 100 years) and the return period of an extreme event of interest (the expected time until an event of the same or greater magnitude occurs). Efficient methods to assess these quantities involve using the sample of yearly maxima: on the one hand, this sample can be treated approximately as an independent and identically distributed (iid) sample, and on the other hand, the marginal distribution can be well-approximated by the three-parameter generalized extreme-value (GEV) distribution \cite{Lea83}. Consequently, parametric estimates of the GEV parameters can be easily converted into promising estimates for return periods or return levels; see, e.g., Section 3.3.3 in \cite{Col01}.

The previous approach is known as the block maxima method, and recent years have witnessed a growing interest in understanding the underlying mathematical principles. Historically, statistical methods were studied under the simplifying assumption that the block maxima sample is a genuine iid sample from the GEV distribution \cite{Pre80, Hos85}, thereby ignoring that both the independence and the GEV assumption are only met asymptotically for the block size tending to infinity. Deeper theoretical insights may be gained by treating the block size as a parameter sequence that is allowed to increase with the sample size. Under such an assumption, typical estimators like the maximum likelihood estimator or the probability weighted moment estimator are still consistent and asymptotically normal, see \cite{Dombry2015, Ferreira2015, Dom19} and \cite{Bücher2014, Bücher2018-disjoint} for the serially independent and dependent case, respectively. Moreover, it has been found that estimators based on block maxima may be made more efficient by considering sliding rather than disjoint block maxima, both in the univariate \cite{Bücher2018-sliding, BucZan23} and in the multivariate case \cite{Zou21, BucSta24}, or, in the iid case, by even considering all block maxima \cite{OorZho20}.

The current paper is motivated by yet another approach that allows for improving upon the classical disjoint block maxima method. Specifically, the three GEV parameters $(\gamma, \mu, \sigma) \in \R^2 \times (0,\infty)$ not only show up in the asymptotic distribution of the block maximum, but also in that of the $m$ largest order statistics ($m$-LOS), where $m\in \{2,3,\dots\}$ is fixed \cite{Welsch_1972, Hsing1988}. For instance, in the case where the underlying observations are independent and identically distributed, the  joint limiting distribution of the affinely standardized $m$-LOS, often referred to as the $\mathrm{GEV}_m(\gamma,\mu,\sigma)$ distribution, has the Lebesgue density
\begin{align} \label{eq:m-gev}
    f_m(x_1,\dots,x_m~|~\gamma,\mu,\sigma)
    &= \nonumber
    \sigma^{-m}\exp\Big\{-\Big(1+\gamma\frac{x_m-\mu}{\sigma}\Big)^{-1/\gamma} 
    \\ & \hspace{1cm}
    -\frac{\gamma+1}{\gamma}\sum_{j=1}^m \log\Big(1+\gamma\frac{x_j-\mu}{\sigma}\Big)\Big\} \bm 1(x_1 >  \dots > x_m).
\end{align}
Fitting this distribution to the sample of blockwise $m$-LOS has been proposed by \cite{Wei78, Smi86, Taw88}; see also Section~3.5 in \cite{Col01}. Since then, only limited theoretical advances on $m$-LOS have been made; e.g. \cite{Wang1995,bader2017automated} proposed a method on how to (automatically) select $m$ in finite sample settings. 
Despite the lack of further theoretical developments, the method has been readily applied in hydrology \cite{Ram02, Aarnes2012, osti_194289, GuedesSoares2004}, climatology \cite{Nemukula2018, An2007, Busababodhin2021} and other fields \cite{Said2011, Chenavier2019, Silva2020}. The fact that underlying observations are required to be independent and identically distributed is often ignored or approached based on some ad-hoc initial filtering step.

To the best of our knowledge, the approach outlined in the previous paragraph has neither been studied mathematically, nor has it been rigorously extended to the context of time series data, where the limiting distribution of the $m$-LOS is in general different from the one in \eqref{eq:m-gev}; see, for instance, Theorem~\ref{thm:welsch} below.
This is a clear gap in the literature, given that the approach is typically applied to environmental data, which is rarely serially independent. 
Even the consistency of the maximum likelihood method applied to the model in \eqref{eq:m-gev} is unclear then, as it relies on imposing a likelihood that is demonstrably incorrect for serially dependent data. It will be one of our main results that it is, in fact, inconsistent.

To illustrate the mathematical principles, we focus below on the univariate, heavy-tailed time series case, which allows to work with the two-parametric Fréchet distribution rather than the three-parametric GEV distribution. For simplicity, we restrict attention to the two largest order statistics in each block only (i.e., $m=2$), subsequently referred to as the `top-two' (TT) approach.
Our main results are as follows: first, we show that TT estimation based on maximizing the likelihood derived under independence (e.g., using the density from \eqref{eq:m-gev}, but with $m=2$ and Fréchet rather than GEV margins) is inconsistent in general, both for disjoint and sliding blocks. Next, we propose bias-corrected versions of the TT estimators and show that they are consistent under mild conditions. 
We further demonstrate that the TT sliding blocks version exhibits a smaller asymptotic variance than both the disjoint blocks version and the block maxima-based estimators, regardless of serial dependence for shape estimation, and for scale estimation when serial dependence is not too strong. Regarding bias, the TT estimators require a different condition than the max-only estimators, and depending on the data-generating process, may exhibit smaller or larger bias.

In an extensive simulation study, we show that the TT estimators typically outperform both their max-only counterparts as well as the all block maxima estimator from \cite{OorZho20}, both for shape estimation as well as for return level estimation (in some time series models, the all block maxima estimator was found to be superior for shape estimation at the `optimal' block size at which the respective MSE curve is minimized). A small case study illustrates the usefulness of the results.

The remaining parts of this paper are organized as follows: some mathematical preliminaries on limit results for large order statistics are provided in Section~\ref{sec:preliminaries}. The limit results give rise to a pseudo maximum likelihood estimator, which is studied mathematically in Section~\ref{sec:estimation-general} for general observation schemes. The theory is then specialized to the case of block maxima extracted from a stationary time series in Section~\ref{sec:estimation-blockmaxima}, and further to an underlying iid series in Section~\ref{sec:estimation-blockmaxima-iid}. The main results of the Monte Carlo simulation study are presented in Section~\ref{sec:simulation}, and the case study is given in Section~\ref{sec:case-study}. A conclusion is provided in Section~\ref{sec:conclusion}.
All proofs are deferred to Sections~\ref{sec:proofs-estimation-general}-\ref{sec:proofs-estimation-blockmaxima-iid}. Finally, some additional results on the Fr\'echet-Welsch-distribution are collected in \ref{sec:frechet-welsch-auxiliary}, some covariance formulas are collected in Section~\ref{sec:asymptotic-covariances} and some additional simulation results are presented in Section~\ref{sec:sim-additional}. Throughout, the arrow $\rightsquigarrow$ denotes weak convergence.

\section{Mathematical Preliminaries on the Two Largest Order Statistics}
\label{sec:preliminaries}

For a real-valued stationary time series $(\xi_t)_{t\in\N}$ and block size $r \in \N$, define
\begin{align*}
   M_{r} := \xi_{(1),[1:r]}, \qquad 
   S_{r} := \xi_{(2),[1:r]},
\end{align*}
where $\xi_{(1),[1:r]} \ge \dots \ge \xi_{(r),[1:r]} $ denotes the order statistic (sorted in decreasing order) calculated from the observations $\xi_i$ with $i\in[1\!:\!r]:=\{1, \dots, r\}$.
Throughout, we assume the following heavy-tailed max-domain of attraction condition:  there exists a sequence $(\sigma_r)_r \subset (0,\infty)$ and a positive parameter $\alpha$ such that
\begin{align} \label{eq:doa2}
\lim_{r\to\infty} \Prob(M_r/\sigma_r \le x)  = \exp(-x^{-\alpha}), \quad x>0.
\end{align}
The following theorem characterizes the class of possible limit distribution of the random vector $(M_r/\sigma_r, S_r/\sigma_r)$ under the additional assumption of strong mixing \cite{Dou94}. Let 
\begin{align} 
\label{eq:C-rho-class}
    \mathcal C 
    &= \nonumber
    \big\{\rho:[0,1]\to[0,1] \text{ concave and  nonincreasing}
    \\
    &\hspace{6cm}
    \text{with } 0 \leq \rho(\eta)\leq 1-\eta \text{ for all } \eta \in [0,1]\big\}.
\end{align}

\begin{theorem}[\cite{Welsch_1972}]\label{thm:welsch}
Let $(\xi_t)_{t\in\N}$ be a stationary strong-mixing time series. If there exist sequences of constants $(a_r)_{r\in\N}\subset(0,\infty),\ (b_r)_{r\in\N}\subset\R$, such that 
\begin{align}
\label{eq:ms-conv}
    \lim_{r \to \infty}\Prob\big(M_r\leq a_rx+b_r, S_r\leq a_ry+b_r\big)= H(x,y), \qquad (x,y) \in \R^2,
\end{align}
for some bivariate limit distribution $H$ whose first marginal distribution is non-degenerate, then the first marginal cdf of $H$ is the cdf $G$ of an extreme-value distribution and there exists 
$\rho \in \mathcal C$
such that
\begin{align}\label{eq:H}
    H(x,y)=\begin{cases*}
    G(x), & $y\geq x$, \\
    G(y)\Big\{ 1-\rho\big(\eta_{G}(x,y)\big)\log G(y)\Big\}, & $y<x$,
    \end{cases*}
\end{align}
where $\eta_{G}(x,y) \in [0,1]$ is defined as \begin{align*}
     \eta_{G}(x,y)
     := \begin{cases}
         \tfrac{\log G(x)}{\log G(y)} & \text{if } G(y) \in (0,1), \\
         0 & \text{if } G(y) \in\{0,1\}. 
     \end{cases}
\end{align*}
and where we use the convention $0 \cdot \infty = 0$.
If, additionally, $(\xi_t)_t$ is an i.i.d.\ sequence, we have $\rho(\eta)=\rho_\indi(\eta):=1-\eta$.
\end{theorem}

Conversely, as shown by \cite{Mori_1976}, for any $\rho \in \mathcal C$, there exists a strictly stationary, strong-mixing time series such that \eqref{eq:ms-conv} is met; see also Example~\ref{ex:mori} below.

As a consequence of Theorem~\ref{thm:welsch}, if $(\xi_t)_{t\in\N}$ is strongly mixing and satisfies \eqref{eq:doa2} and if the random vector $(M_r/\sigma_r, S_r/\sigma_r)$ converges weakly,
then the limit distribution has the joint cdf $H_{\rho, \alpha, 1}$, where, for $\rho$ as in the above theorem and $\alpha, \sigma>0$, 
\begin{align}\label{eq:H_Frech}
    H_{\rho,\alpha,\sigma}(x,y)=\begin{cases*}
    \exp\Big(-\big(\frac{x}{\sigma}\big)^{-\alpha}\Big), & $y\geq x>0$, \\
    \exp\Big(-(\frac{y}{\sigma})^{-\alpha}\Big)\Big\{ 1+\rho\big(\eta_{\alpha}(x,y)\big)\big(\frac{y}{\sigma}\big)^{-\alpha}\Big\}, & $x>y>0$,
    \end{cases*}
\end{align}
and where
$\eta_{\alpha}(x,y)=(y/x)^\alpha$.
We refer to the associated distribution as the Fr\'echet-Welsch-distri\-bution; notation $\mathcal W=\mathcal W(\rho, \alpha, \sigma)$. Note that the weak limit result $(M_r/\sigma_r, S_r/\sigma_r) \weak \mathcal W(\rho, \alpha, 1)$ implies the approximate distributional equality $(M_r, S_r) \approx_d \mathcal W(\rho, \alpha, \sigma_r)$ for sufficiently large block size $r$, which will be the basis for the statistical methods proposed in later sections.

We collect some important properties of the Fr\'echet-Welsch-distribution.

\begin{remark}[The Fr\'echet-Welsch-distribution]\label{rem:welsch} ~ \newline
\noindent
[a] Marginal distributions. The first marginal distribution of $H_{\rho,\alpha,\sigma}$ is the Fr\'echet($\alpha, \sigma$)-distribution, that is, its cdf is given by
\begin{align} \label{eq:w_margCDF1}
    H^{(1)}_{\rho,\alpha,\sigma}(x):=\exp\Big(-\Big(\frac{x}{\sigma}\Big)^{-\alpha}\Big).
\end{align}
The second marginal distribution depends on $\rho$ only through $\rho_0 := \rho(0)$; its cdf is given by
\begin{align} \label{eq:w_margCDF2}
    H^{(2)}_{\rho,\alpha,\sigma}(y):=\exp\Big(-\Big(\frac{y}{\sigma}\Big)^{-\alpha}\Big)\Big(1+\rho_0\Big(\frac{y}{\sigma}\Big)^{-\alpha}\Big).
\end{align}
Note that both margins are absolutely continuous with respect to the Lebesgue measure
with respective densities given by
\begin{align}
    p^{(1)}_{\rho, \alpha,\sigma}(x)
    &:=
    \frac{\partial}{\partial x}H^{(1)}_{\rho, \alpha,\sigma}(x)=\alpha\sigma^\alpha x^{-\alpha-1}\exp\Big(-\Big(\frac{x}{\sigma}\Big)^{-\alpha}\Big) ,\notag\\
    p^{(2)}_{\rho, \alpha,\sigma}(y)
    &:=
    \frac{\partial}{\partial y} H^{(2)}_{\rho, \alpha,\sigma}(y)= \alpha \sigma^\alpha y^{-\alpha-1}\exp\Big(-\Big(\frac{y}{\sigma}\Big)^{-\alpha}\Big)\Big[1-\rho_0+\rho_0 \Big(\frac{y}{\sigma}\Big)^{-\alpha}\Big]\label{eq:general_dens}
\end{align}

\smallskip
\noindent
[b] Absolute continuity. In general, the Fr\'echet-Welsch-distribution does not have a Lebesgue density. A sufficient condition is provided in Lemma~\ref{lem:density-new3} below: if $\rho$ is twice differentiable on $[0,1]$ at all but finitely many points, then $\mathcal W(\rho, \alpha, \sigma)$ has a Lebesgue density if and only if  $\int_0^1 \rho'(z) + z \rho''(z) \diff z = -1$.

\smallskip
\noindent
[c]
Moments. Additional results concerning certain moments are given in Section~\ref{sec:frechet-welsch-auxiliary}.
\end{remark}

Parts of the remaining paper will explicitly work with the special case where $\rho(\eta) = \rho_{\indi}(\eta) = 1-\eta$, which arises for underlying iid observations. We call the associated distribution \textit{IID Fr\'echet-Welsch}.

\begin{definition}[The IID Fr\'echet-Welsch-distribution]\label{def:iidfrechetwelsch}
    The \textit{IID Fr\'echet-Welsch-dstribution}, notationally $\mathcal{W}_\indi = \mathcal{W}_\indi(\alpha, \sigma) := \mathcal W(\rho_\indi, \alpha, \sigma)$, is defined by its  cdf 
\begin{align}\label{eq:H_iid}
    H_{\alpha,\sigma}(x,y):=H_{\rho_{\indi},\alpha,\sigma}(x,y)=\begin{cases*}
     \exp\Big(-\big(\frac{x}{\sigma}\big)^{-\alpha}\Big), & $y\geq x$ \\
    \exp\Big(-\big(\frac{y}{\sigma}\big)^{-\alpha}\Big)\Big\{ 1 + \big(\frac{y}{\sigma}\big)^{-\alpha} -\big(\frac{x}{\sigma}\big)^{-\alpha}\Big\}, & $y<x$.
    \end{cases*}
\end{align}
\end{definition}
The iid Fr\'echet-Welsch-distribution is absolutely continuous with respect to the Le\-besgue measure with density
\begin{align}\label{eq:sw_density}
    p(x,y):=p_{\alpha,\sigma}(x,y) 
    := 
    \alpha^2\sigma^{2\alpha}(xy)^{-\alpha-1}\exp\Big(-\Big(\frac{y}{\sigma}\Big)^{-\alpha}\Big)\indic(x>y),
\end{align}
which is a version of the density in \eqref{eq:m-gev} with $m=2$, but with Fréchet rather than GEV margins.
Note that existence of the Lebesgue density offers the possibility of (pseudo) maximum likelihood inference.

\begin{example}[Stationary time series and models for $\rho$]\label{ex:models}
As mentioned right after Theorem~\ref{thm:welsch}, any 
$\rho \in \mathcal C$ may appear in the limit \eqref{eq:ms-conv}, for some suitable strongly mixing series (Example 1 in \cite{Mori_1976}). We briefly discuss some special cases.

    \smallskip
    \noindent
    [a] Linear functions. The function $\rho(\eta) = c(1-\eta)$ with $c\in[0,1]$ has been discussed in \cite{Novak98}, including some specific examples and sufficient (and partly necessary) conditions. For $c=1$, this reduces to $\rho=\rho_\indi$ as discussed in Definition~\ref{def:iidfrechetwelsch}. For $c<1$, 
    we have $c=-\rho'(1)\ne -1$, whence the associated Welsch-distribution does not have a Lebesgue density by Lemma~\ref{lem:density-new3}. For $c=0$, we obtain the function that is constantly equal to zero, which we denote by $\rho_\pd$ as it yields perfect monotone dependence. Remarkably, $\rho_\pd$ may arise for non-trivial time series models for instance, for $\xi_t=\max(Z_t, Z_{t-1})$ with $Z_t$ iid standard Fr\'echet \cite[Example 1]{Welsch_1972}.

    \smallskip
    \noindent
    [b] Power functions. The function $\rho(\eta) = c^{-1}(1-\eta^c)$ with $c\in(1,\infty)$ satisfies $\rho'(1)=-1$; the associated Welsch-distribution hence has a Lebesgue density. The construction in Example 1 in \cite{Mori_1976} simplifies: letting $(Z_t)_{t\in\N}$ and $(\zeta_t)_{t\in\N}$ be independent iid sequences with distribution $Z_t \sim \Pareto(\alpha)$ and $\zeta_t\sim \Pareto((c-1)\alpha)$ and defining  $\xi_t=\max\{Z_{t-1},\zeta_t^{-1}Z_t\}$,  
    we obtain that \eqref{eq:ms-conv} is met with $H=H_{\rho,\alpha,1}$, $a_r=r^{1/\alpha}$ and $b_r=0$.

    \smallskip
    \noindent
    [c] A class of kink functions. For $c \in [0,1)$, consider the function $\rho(\eta)=\min\{c,1-\eta\}$. Since $\int_0^1 \rho'(\eta)+z \rho''(\eta)\diff z=-c \ne -1$, the associated Welsch-distribution does not have a Lebesgue density. One can show that this $\rho$-function appears in the classical ARMAX(1)-model defined by the recursion $\xi_t=\max\{(1-c) \xi_{t-1},cZ_t\}$ with $(Z_t)_t$ iid standard Fr\'echet, or in the AR(1)-model defined by the recursion $\xi_t=(1-c)\xi_{t-1} + Z_t$ with $Z_t$ iid standard Cauchy. We will consider versions of these models in the simulation study.
\end{example}

A final observation: for any $\rho \in \mathcal C$ from \eqref{eq:C-rho-class}, the properties of $\rho$ imply that $c(1-\eta) \le \rho(\eta) \le \min\{c,1-\eta\}$ for all $\eta \in [0,1]$, where $c:= \rho_0:= \rho(0)$. In other words, the linear and kink functions from Example~\ref{ex:models} provide lower and upper bounds on $\rho$ that only depend on~$\rho_0$.

    \begin{figure}
        \centering
        \includegraphics[width=0.5\linewidth]{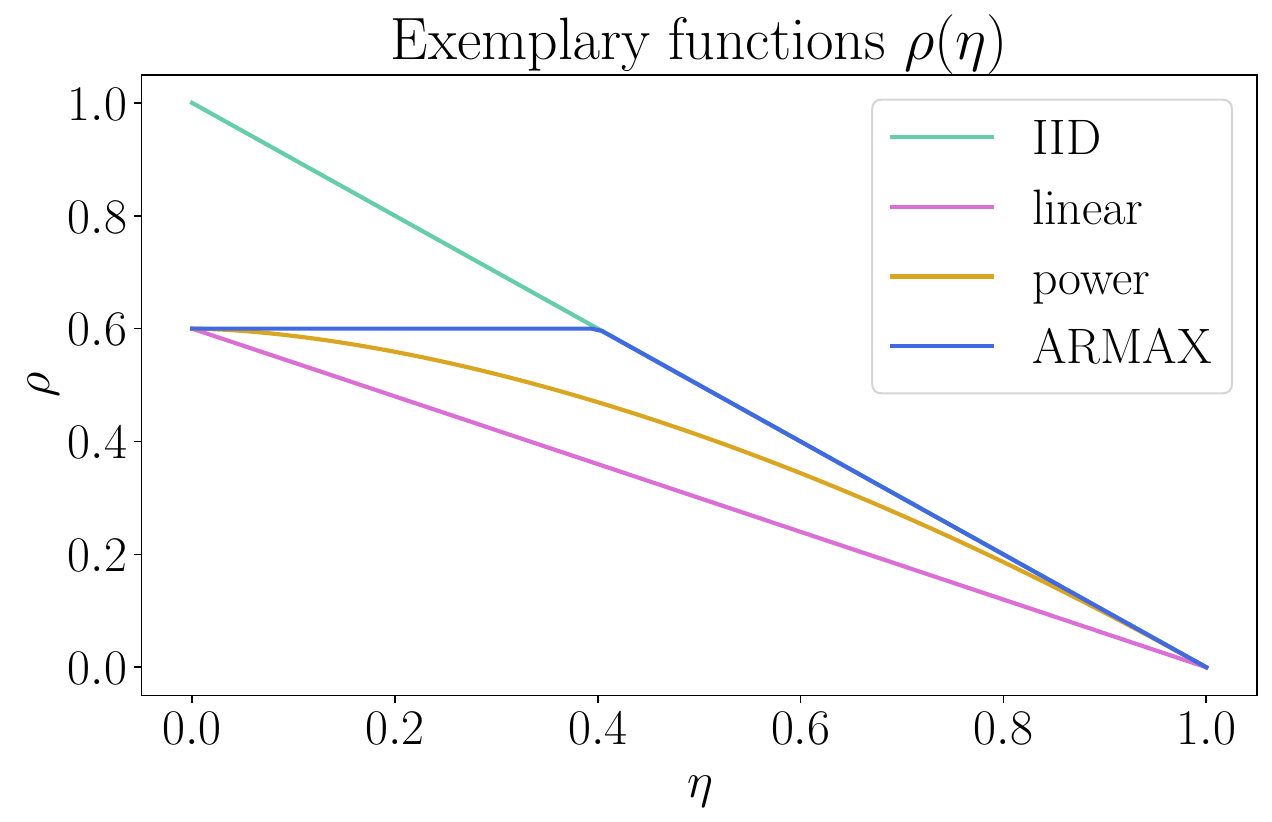}\vspace{-1em}
        \caption{Different $\rho$ functions. The examples `linear', `power' and `ARMAX' correspond to Example \ref{ex:models} [a] ($c=0.6$), [b] ($c=0.4$) and [c] ($c=0.6$), respectively.}
        \label{fig:enter-label}
    \end{figure}

\section{Pseudo Maximum-Likelihood Estimation for the Fr\'echet-Welsch distribution}
\label{sec:estimation-general}

Suppose we are given a sample $\bm z =( (x_1,y_1),\dots,(x_k,y_k) )$ of $k \ge 2$ bivariate vectors such that $0<y_i \le x_i$ 
for all~$i$; for the moment, no assumption is made on the data-generating process. As motivated in the introduction, adapting the proposal in Section 3.5 of \cite{Col01} to the case of Fréchet marginals, we are interested in fitting the iid Fr\'echet-Welsch distribution $\mathcal {W}_\indi(\alpha, \sigma)$ to $\bm z$.

In view of the fact that $\mathcal{W}_\indi$ has a Lebesgue density, see \eqref{eq:sw_density}, 
we may rely on standard maximum likelihood estimation, with the respective $\mathcal{W}_\indi$-
log-likelihood given by
\begin{align}
\label{eq:sw_logl}
    \ell(\alpha,\sigma|\z)&=2k\log\alpha +2k\alpha\log\sigma-\sum_{i=1}^k \big\{ (\alpha+1)\log (x_iy_i)+\sigma^\alpha y_i^{-\alpha} \big\},
\end{align}
see \eqref{eq:sw_density}. Define $\theta=(\alpha, \sigma)$, let $\Theta=(0,\infty)^2$ and let
\begin{align*}
    M_{-\alpha}(\y):=\Big(\frac{1}{k}\sum_{i=1}^k y_i^{-\alpha}\Big)^{-1/\alpha}
\end{align*}
denote the power mean function with exponent $-\alpha$.

\begin{lemma}[Existence and uniqueness]\label{lemm:ex_uni}
If the pairs $(x_i,y_i)$ are not all equal, then there exists a unique maximizer 
\begin{align} \label{eq:mle}
\hat\theta(\z)=\big(\hat\alpha(\z),\hat\sigma(\z)\big)=\argmax_{\theta\in\Theta} \ell(\alpha,\sigma|\z).
\end{align}
More precisely, $\hat\alpha(\z)$ is the unique root of the function
\begin{align}
    \alpha \mapsto \Psi_k(\alpha|\z):=&\ 2\alpha^{-1}+2\cdot M_{-\alpha}^\alpha(\y)\cdot\frac{1}{k}\sum_{i=1}^k y_i^{-\alpha}\log y_i-\frac{1}{k}\sum_{i=1}^k\log (x_iy_i)\label{eq:def_psik}
\end{align}
and we have $\hat\sigma(\z)= 2^{1/\hat\alpha(\bm z)}M_{-\hat\alpha(\bm z)}(\y)$.
\end{lemma}

Note that $\Psi_k(\alpha | c \bm z) = \Psi_k(\alpha | \bm z)$ for any $c>0$, which implies that $\hat \alpha(c \bm z) = \hat \alpha (\bm z)$ and $\hat \sigma(c \bm z)=c \hat \sigma(\bm z)$.

\subsection{On the (lack of) consistency of the ML Estimator}

In the remaining parts of this section, we suppose to be given, for each positive integer $n$, a random array of observations
\begin{align}\label{eq:Z}
    \zz_n= \begin{pmatrix}
    \begin{pmatrix}
        X_{n,1}\\Y_{n,1}
    \end{pmatrix},&\cdots,&
    \begin{pmatrix}
        X_{n,k_n}\\Y_{n,k_n}
    \end{pmatrix}
    \end{pmatrix}
\end{align}
taking values in $(0,\infty)^{2\times k_n}$, where $k_n\geq 2$ is a positive integer sequence such that $k_n\to\infty$ as $n\to\infty$. It is instructive (but not necessary) to think of $Z_{n,i} = (X_{n,i}, Y_{n,i})$ as the largest two order statistics in a block of subsequent observations taken from an underlying stationary time series $(\xi_t)_t$ for which Theorem~\ref{thm:welsch} applies. As such, the random variables $(X_{n,i}, Y_{n,i})$ will be assumed to (approximately) follow the Fr\'echet-Welsch distribution $\mathcal{W}(\rho,\alpha_0,\sigma_n)$ for some 
$\rho \in \mathcal C$,
some $\alpha_0>0$ and some sequence of scale parameters $\sigma_n>0$; the assumption will be made precise in Condition~\ref{cond:1} below.

We are interested in estimating the parameters $(\alpha_0, \sigma_n) \in (0,\infty)^2$, treating $\rho$ as a nuisance parameter. Since the general Fréchet-Welsch family lacks a $\sigma$-finite dominating measure, it seems reasonable to apply the (pseudo) MLE $\hat \theta(\bm Z_n)$ from~\eqref{eq:mle} instead; note that a Pseudo MLE based on an incorrect likelihood may or may not be consistent in general \cite[Page 4]{White1982}. In fact, this approach is implicitly taken when applying the traditional top-two method to time series data.

We start by studying the first-order asymptotic behavior under minimal assumptions on the data-generating process. In view of the fact that the estimator is based on specific empirical moments (see Lemma~\ref{lemm:ex_uni}), it seems natural to assume that these moments converge to the respective moments of the Fréchet-Welsch distribution; this becomes our first Condition~\ref{cond:1} below.

The condition implies that $\Psi_{k_n}(\alpha|\bm Z_n)$ from \eqref{eq:def_psik} has the weak limit

\begin{align} \label{eq:psi}
\Psi_\infty^{(\rho, \alpha_0)}(\alpha)
    =
    \frac{2}{\alpha}+2\frac{\int_0^\infty y^{-\alpha}\log y\diff H^{(2)}(y)}{\int_0^\infty y^{-\alpha} \diff H^{(2)}(y)}-\int_0^\infty \log y\diff H^{(2)}(y)-\int_0^\infty \log x\diff H^{(1)}(x)
\end{align}
for $n\to\infty$, where $H^{(1)}=H^{(1)}_{\rho,\alpha_0,1}$ and $H^{(2)}=H^{(2)}_{\rho,\alpha_0,1}$ are the marginal cdfs of the $\mathcal{W}(\rho,\alpha_0,1)$-distribution from \eqref{eq:w_margCDF1} and \eqref{eq:w_margCDF2}, respectively. We start by stating some properties of this tentative limit.
Recall the gamma function $\Gamma(x)=\int_0^\infty t^{x-1}e^{-t}\diff t$ and the Euler-Mascheroni constant $\gamma\approx 0.5772$.

\begin{lemma}\label{lem:digamma2}
    For each fixed $\rho \in \mathcal C$
    and $\alpha_0\in(0,\infty)$, we have $\Psi_\infty^{(\rho, \alpha_0)}(\alpha) = (2/\alpha_0) \cdot \Pi_{\rho_0}( \alpha/\alpha_0)$, where $\rho_0:=\rho(0)$ and
    \begin{align} \label{eq:Pi_rho}
    \Pi_{\rho_0}(y):=\frac{1}{y}-\frac{\Upsilon_{\rho_0}'(y)}{\Upsilon_{\rho_0}(y)}+\frac{\rho_0}{2}-\gamma \qquad (y > 0),
\end{align}
    with 
    \begin{align} \label{eq:upsilon_NEW}
    \Upsilon_{\rho_0}(x):=\rho_0\Gamma(x+2)+(1-\rho_0)\Gamma(x+1).
    \end{align}
    Moreover, for each $\rho_0\in[0,1]$, the function $y\mapsto\Pi_{\rho_0}(y)$ is a continuous decreasing bijection from $(0,\infty)$ to $\R$ with $\Pi_{\rho_0}(1) \le 0$, which allows to define 
\begin{align} \label{eq:varpi}
    \varpi_{\rho_0} 
    := 
    \mathrm{UniqueZero} (y \mapsto \Pi_{\rho_0}(y)) \in (0,1];
\end{align} 
    see Figure~\ref{fig:varpi} for the graph of $\rho_0 \mapsto \varpi_{\rho_0}$.
We have $\varpi_{\rho_0} = 1$ if and only if $\rho \in\{\rho_\indi, \rho_\pd\}$ with $\rho_\pd \equiv 0$ introduced in Example~\ref{ex:models} [a]. Additionally, the map $\rho_0 \mapsto \varpi_{\rho_0}$ is Lipschitz continuous on $[0,1]$ and continuously differentiable on $(0,1)$ with a bounded derivative.
\end{lemma}

\begin{figure}
    \centering
    \includegraphics[height=4cm]{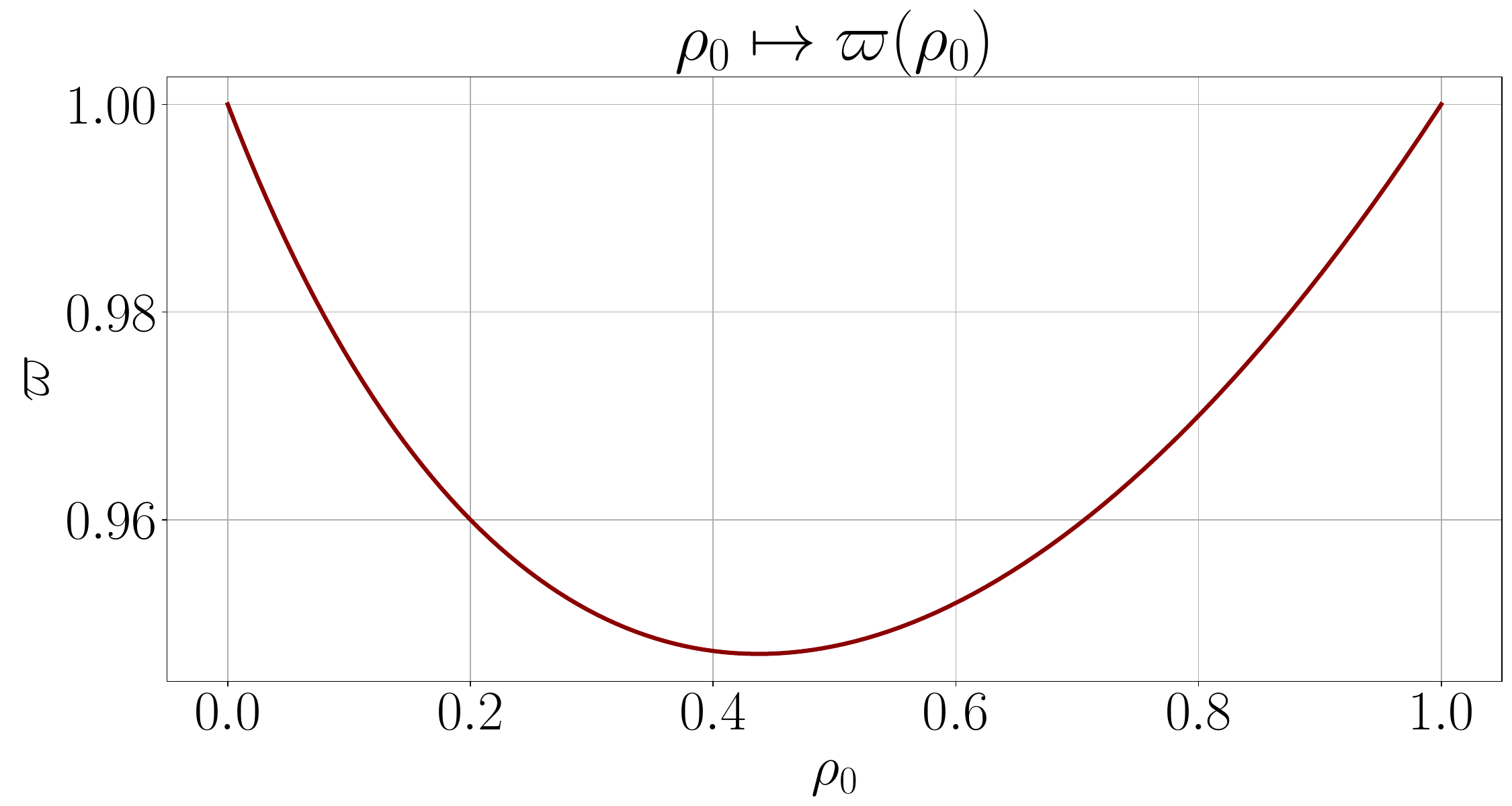}
    \includegraphics[height=4cm]{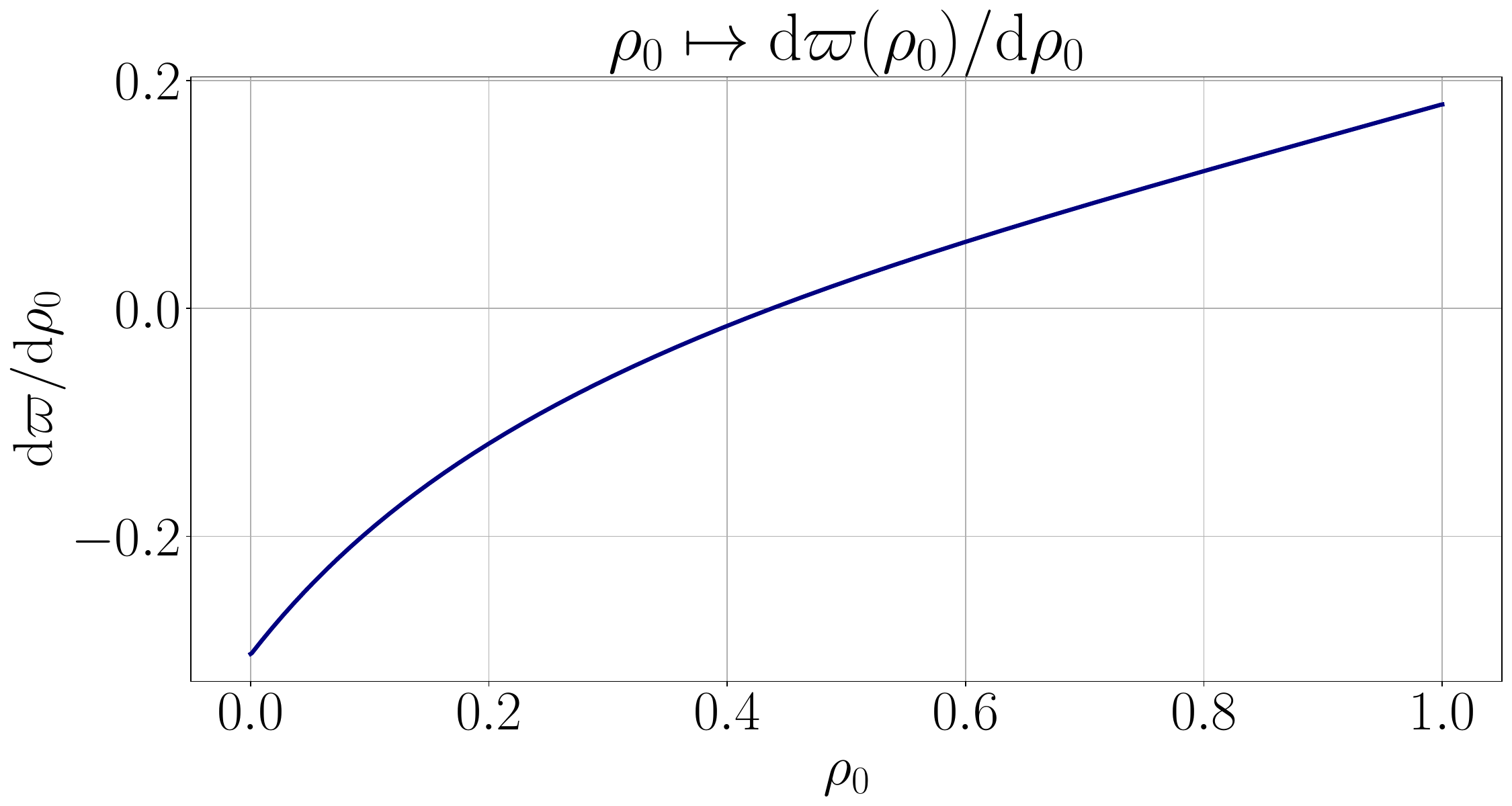}
    \caption{Left: graph of $\rho_0\mapsto \varpi_{\rho_0}$. Right: graph of its derivative. }
    \label{fig:varpi}
\end{figure}

As a consequence of Lemma \ref{lem:digamma2}, $\alpha \mapsto \Psi_\infty^{(\rho, \alpha_0)}(\alpha)$ has the unique root
\begin{align} \label{eq:alpha1_NEW}
    \alpha_1 := \alpha_1(\alpha_0, \rho) := \varpi_{\rho_0}\cdot\alpha_0,
\end{align}
with $\alpha_1 = \alpha_0$ if and only if $\rho \in\{\rho_\indi, \rho_\pd\}$.
It will turn out that the ML estimator for $\alpha_0$ converges to $\alpha_1$ in probability; it is hence inconsistent unless $\rho \in\{\rho_\indi, \rho_\pd\}$.

We now make the required convergence of empirical moments more precise. For $0<\alpha_-<\alpha_+<\infty$, consider the class of functions from $(0,\infty)^2$ into $\R$ defined as
\begin{multline}
    \mathcal{F}_1(\alpha_-,\alpha_+)
    :=
    \{(x,y) \mapsto \log x\}
    \cup
    \{(x,y) \mapsto \log y\}
    \cup
    \{(x,y) \mapsto y^{-\alpha}:\alpha_-<\alpha<\alpha_+\} \\
    \cup
    \{(x,y) \mapsto y^{-\alpha}\log y:\alpha_-<\alpha<\alpha_+\}. \label{eq:F1}
\end{multline}

\begin{condition}\label{cond:1}
There exists $\rho \in \mathcal C$,
$\alpha_0>0$ and a sequence $(\sigma_n)_n \subset (0,\infty)$ such that
\begin{align} \label{eq:moments_convergence}
    \frac{1}{k_n}\sum_{i=1}^{k_n} f\Big(\frac{X_{n,i}}{\sigma_n}, \frac{Y_{n,i}}{\sigma_n}\Big)\rightsquigarrow \int_{(0,\infty)^2} f(x,y)\diff H_{\rho, \alpha_0,1}(x,y), \qquad n\to\infty,
\end{align}
for all $f\in\mathcal F_1(\alpha_-,\alpha_+)$, where $\alpha_-, \alpha_+$ are some constants such that $0<\alpha_-<\alpha_1<\alpha_+<\infty$, with $\alpha_1=\alpha_1(\rho, \alpha_0)$ from \eqref{eq:alpha1_NEW}.
\end{condition}

Note that Condition~\ref{cond:1} can be regarded as a mathematical quantification of the above heuristics that `$(X_{n,i}, Y_{n,i})$ is (approximately) Fr\'echet-Welsch distributed'.

In the subsequent sections we show that it naturally follows from a  domain of attraction condition and integrability assumptions in case the $(X_{n,i}, Y_{n,i})$ correspond to blockwise top two order statistics extracted from a stationary time series. Finally, note that we require convergence of the empirical moments involving $y \mapsto y^{-\alpha}$ and $y \mapsto y^{-\alpha} \log y$ in a neighborhood of the root $\alpha_1$ of $\Psi_\infty^{(\rho, \alpha_0)}$; an assumption that is natural when studying the asymptotic behavior of estimators that arise as a root of an estimation equation \cite{vdVaart}.

On the event where not all $Z_{n,i}$ are equal, Lemma~\ref{lemm:ex_uni} shows that the MLE $\hat \theta_n := (\hat\alpha_n,\hat\sigma_n) := \hat \theta (\bm Z_n)$ from \eqref{eq:mle} exists and is unique. For definiteness, we define $\hat\alpha_n=\infty$ and $\hat\sigma_n=Y_{n,1}$ on the event $\{Z_{n,1}=\cdots=Z_{n,k_n}\}$.

\begin{theorem}[(Lack of) consistency]
\label{thm:consist}
Let $\bm Z_n$ be a triangular array of random variables as in \eqref{eq:Z} with $k_n\to\infty$ that satisfies Condition \ref{cond:1} and 
\begin{align}\label{eq:neglect}
    \lim_{n\to\infty} \Prob\big(Z_{n,1}=\cdots=Z_{n,k_n}\big)=0.
\end{align}
Recall $\Upsilon_{\rho_0}$ from \eqref{eq:upsilon_NEW}, $\varpi_{\rho_0}$ from \eqref{eq:varpi} and $\alpha_1$ from \eqref{eq:alpha1_NEW} and define
\begin{align}
\label{eq:s1}
s_1 = s_1(\rho,\alpha_0) = 
\Big( \frac{2}{\Upsilon_{\rho_0}(\varpi_{\rho_0})}\Big)^{1/{\alpha_1}}
\end{align}
Then, as $n\to\infty$, $
    (\hat\alpha_n,\hat\sigma_n/\sigma_n)\rightsquigarrow (\alpha_1,s_1)$.
Moreover, the limit $(\alpha_1, s_1)$ is equal to $(\alpha_0,1)$
if and only if $\rho=\rho_\indi$ in Condition \ref{cond:1}. If $\rho=\rho_\pd$, we have $(\alpha_1, s_1)=(\alpha_0, 2^{1/\alpha_0})$.  

\end{theorem}

\begin{remark}[An alternative pseudo-maximum likelihood estimator]
\label{rem:alternative_pseudoll}
The obtained inconsistency of $(\hat \alpha_n, \hat \sigma_n)$ is a nuisance which we will correct in Section~\ref{sec:bias-corr} by estimating~$\rho_0$. As an alternative to what we propose there, it also seems natural to fit a more flexible parametric class of Fr\'echet-Welsch distributions. A likelihood-based approach would be feasible in case each distribution in the class has a Lebesgue density. This is for instance the case for the one-parametric power function class in Example~\ref{ex:models}[b], that is, for $\rho(\eta)=c^{-1}(1-\eta^c)$ for some $c>0$. The respective density of the associated Fr\'echet-Welsch distribution is then given by
\[
p_{c,\alpha, \sigma}(x,y) =
\alpha^2 \sigma^{\alpha}\exp\Big( - \Big(\frac{y}\sigma\Big)^{-\alpha} \Big) x^{- c \alpha-1} y^{ - (1-c)\alpha - 1} \Big\{ c-1 + \Big(\frac{y}{\sigma}\Big)^{-\alpha} \Big\} \bm 1(x>y>0);
\]
note that $c=1$ results in the $\mathcal{W}_\indi$-density from \eqref{eq:sw_density}. The properties of the respective pseudo maximum likelihood estimator were investigated in a small simulation study using the models described in Section~\ref{sec:simulation}. It was found that the estimator did not perform better than the bias-corrected version of $(\hat \alpha_n, \hat \sigma_n)$ proposed in Section~\ref{sec:bias-corr}. We are therefore not pursuing this any further.
\end{remark}

\begin{remark}[Maximizing a wrong likelihood -- an odd idea?]\label{rem:normal_example}
    So far we have maximized a likelihood that does not correspond to the true data-generating process (iid \Frechet-Welsch vs.\ general \Frechet-Welsch). Therefore, the lack of consistency established in Theorem~\ref{thm:consist} may not be surprising. However, it is not uncommon that maximizing a misspecified likelihood still yields consistent parameter estimates, see, for instance,  \cite[Page 4]{White1982}. To illustrate this point, consider the following example.

Suppose $n$ iid observations are drawn from the bivariate normal distribution $X_i\sim \Nc_2(\mu,\Sigma)$ with covariance entries $\Sigma_{11}=\sigma_1^2,~\Sigma_{12}=\varrho \sigma_1\sigma_2$ and $\Sigma_{22}=\sigma_2^2$
and consider estimation of the parameter $\theta=(\mu_1,\mu_2,\sigma_1,\sigma_2)$, with $\varrho$ known. If  $L_\varrho(\theta)$ denote the likelihood of $\theta$ for fixed $\varrho$, then both $\hat\theta=\arg\max_\theta L_\varrho(\theta)$ and
$\tilde\theta=\arg\max_\theta \tilde L_0(\theta)$ are consistent for $\theta$, although it can be shown that $\hat\theta$ is more efficient than $\tilde\theta$ if $\varrho \ne 0$.
\end{remark}

\subsection{Asymptotic Distribution of the ML Estimator}
We formulate conditions under which $(\hat\alpha_n, \hat \sigma_n/\sigma_n)$, after proper affine standardization, converges weakly to a normal distribution. For $0<\alpha_{-} <\alpha_{+} <\infty$ define
\begin{align}
\label{eq:F2}  \mathcal{F}_2(\alpha_-,\alpha_+)
:=
\mathcal{F}_1(\alpha_-,\alpha_+)
&\cup\{(x,y)\mapsto y^{-\alpha}(\log y)^2:\alpha_-<\alpha<\alpha_+\}
\nonumber\\&\cup
\{(x,y)\mapsto (\log y)^2\},
\end{align}
with $\mathcal F_1(\alpha_{-} , \alpha_{+})$ from \eqref{eq:F1}.

\begin{condition}\label{cond:2}
There exists 
$\rho \in \mathcal C$,
$\alpha_0>0$ and a sequence $(\sigma_n)_n \subset (0,\infty)$ such that
\eqref{eq:moments_convergence} from Condition \ref{cond:1} holds for all $f \in \mathcal{F}_2(\alpha_{-} , \alpha_{+})$, where $\alpha_-, \alpha_+$ are some constants such that $0<\alpha_-<\alpha_1<\alpha_+<\infty$, with $\alpha_1=\alpha_1(\rho, \alpha_0)$ from \eqref{eq:alpha1_NEW}. Moreover, there exists a sequence $0<v_n \to \infty$ and a random vector $\W=(W_1,W_2,W_3,W_4)^\top$ such that 
\begin{align} \label{eq:w}
    \W_n := \big(\G_nf_1,\G_nf_2,\G_nf_3,\G_nf_4\big)^\top\rightsquigarrow \W, \qquad  n\to\infty,
\end{align}
where
\begin{align}\label{eq:fct_H}
    (f_1,f_2,f_3,f_4)= \big((x,y)\mapsto y^{-\alpha_1}\log y,(x,y)\mapsto y^{-\alpha_1},(x,y)\mapsto \log y, (x,y) \mapsto \log x\big)
\end{align}
and where
\begin{align}\label{eq:cond2}
    \G_nf = v_n\bigg\{ \frac{1}{k_n}\sum_{i=1}^{k_n}f\Big(\frac{X_{n,i}}{\sigma_n}, \frac{Y_{n,i}}{\sigma_n}\Big) - \int_{(0,\infty)^2} f(x,y)\diff H_{\rho,\alpha_0,1}(x,y)\bigg\}.
\end{align}
\end{condition}

In view of the above discussion of Condition \ref{cond:1} and standard results for second-order asymptotics of M- and Z-estimators \cite{vdVaart}, the convergence in \eqref{eq:fct_H} is a natural condition; see also \cite{Bücher2018-disjoint}. The extension of the function class from $\mathcal{F}_1$ to $\mathcal{F}_2$ arises from the fact that second-order asymptotics also require convergence of empirical moments that show up in the gradient of $\alpha \mapsto \Psi_{k_n}(\alpha | \bm Z_n)$.

\begin{theorem}[Asymptotic Distribution]\label{thm:asymptotic}
Let $\zz_n$ be a triangular array of random variables as in \eqref{eq:Z} with $k_n\to\infty$ that satisfies  \eqref{eq:neglect} and Condition \ref{cond:2}. Then, with $\alpha_1$ from \eqref{eq:alpha1_NEW} and $s_1$ from \eqref{eq:s1},
as $n\to\infty$, 
\begin{align}\label{eq:thmasymp}
        v_n\begin{pmatrix}
            \hat \alpha_n-\alpha_1\\
            \hat \sigma_n/\sigma_n-s_1
        \end{pmatrix}
        =M_{\rho_0}(\alpha_0)
        \W_n
        +o_{\mathbb P}(1)\rightsquigarrow M_{\rho_0}(\alpha_0)\W,
\end{align}
where $\W_n$ and $\W$ are as in Condition \ref{cond:2} and where $M_{\rho_0}(\alpha_0)\in\R^{2\times 4}$ is a matrix explicitly given in the proof, see Equations \eqref{eq:mrho1} and \eqref{eq:mrho2}. If $\rho_0 =1$ (that is, $\rho=\rho_\indi$ and $\alpha_1=\alpha_0$), we have 
\begin{align}
\label{eq:malpha}
         M_1(\alpha_0)=\frac{6}{2\pi^2-3}\begin{pmatrix}
             \alpha_0^2 & \frac{3-2\gamma}{2}\alpha_0 & -\alpha_0^2 & -\alpha_0^2 \\
             \frac{2\gamma-3}{2} &\frac{3-2\pi^2-3(3-2\gamma)^2}{12\alpha_0} & \frac{3-2\gamma}{2}& \frac{3-2\gamma}{2}
         \end{pmatrix}.
\end{align}
\end{theorem}

\subsection{A consistent bias-corrected estimator} \label{sec:bias-corr}

Recall that the limit of $(\hat \alpha_n, \hat \sigma_n/\sigma_n)$ in Theorem~\ref{thm:consist} depends on $\rho$ only via $\rho_0$.
Hence, if we had an estimator $\hat\rho_{0,n}$ of $\rho_0$ taking values in $[0,1]$, 
we could define a plug-in bias-corrected estimator $(\widetilde \alpha_n, \widetilde \sigma_n)$ for $(\alpha_0,\sigma_n)$ 
by
\begin{align} \label{eq:bias-corrected-estimator}
\widetilde\alpha_n:=\hat\alpha_n/\hat\varpi_n,
\qquad
\widetilde \sigma_n =\hat \sigma_n \Big( \frac{\Upsilon_{\hat \rho_{0,n}}(\hat \varpi_n)}2 \Big)^{1/\hat \alpha_n},
\end{align}
where $\hat\varpi_n=\varpi_{\hat\rho_{0,n}}$ denotes the unique root of $y \mapsto \Pi_{\hat\rho_{0,n}}(y)$; see Lemma~\ref{lem:digamma2}.
Note that $(\widetilde \alpha_n, \widetilde \sigma_n)$ is a function of $(\hat \alpha_n, \hat \sigma_n, \hat \rho_{0,n})$ only.
A specific example how to estimate $\rho_0$ will be given in Section \ref{sec:bias-corr-block-maxima} below. For the next result we require $\hat\rho_{0,n}$ to be consistent for $\rho_0$.

\begin{theorem}[Consistency of the bias-corrected estimator]
\label{thm:bias-corrected-consist}
Suppose that the conditions of Theorem~\ref{thm:consist} are met, and that
$\hat \rho_{0,n} \rightsquigarrow \rho_0$ as $n\to \infty$. Then, 
\begin{align*}
    (\widetilde\alpha_n,\widetilde\sigma_n/\sigma_n)\rightsquigarrow (\alpha_0,1), \qquad n \to \infty.
\end{align*}
\end{theorem}

\begin{proof}
This is an immediate consequence of Theorem~\ref{thm:consist},  the assumption on $\hat \rho_{0,n}$ and the continuous mapping theorem, observing that both $\rho_0 \mapsto \varpi_{\rho_0}$ and $(\rho_0, \alpha) \mapsto \{ \Upsilon_{\rho_0}(\varpi_{\rho_0})/2\}^{1/\alpha}$ are continuous.
\end{proof}

Asymptotic normality of the bias-corrected estimator may be deduced from joint asymptotic normality of $(\hat \alpha_n, \hat \sigma_n, \hat \rho_{0,n})$ via the functional delta method. For simplicity, we restrict attention to the case where $\hat \rho_{0,n} = \rho_0 + o_{\mathbb P}(v_n^{-1})$ with $v_n$ from Condition~\ref{cond:2}. In that case, under the conditions of Theorem~\ref{thm:asymptotic}, $\hat \rho_{0,n}$ converges at a faster rate than $(\hat \alpha_n, \hat \sigma_n/\sigma_n)$.

\begin{theorem}[Asymptotic distribution of the bias-corrected estimator]
\label{thm:bias-corrected-asymp}
Suppose that the conditions of Theorem~\ref{thm:asymptotic} are met, and that
$\hat \rho_{0,n} = \rho_0 + o_{\mathbb P}(v_n^{-1})$ as $n\to \infty$.  Then, as $n\to\infty$, 
\begin{align}\label{eq:bias-corrected-asymp}
        v_n\begin{pmatrix}
            \widetilde \alpha_n-\alpha_0\\
            \widetilde \sigma_n/\sigma_n-1
        \end{pmatrix}
        =
        M_{\rho_0}^{\mathrm{bc}}(\alpha_0)
        \W_n
        +o_{\mathbb P }(1)\rightsquigarrow M_{\rho_0}^{\mathrm{bc}}(\alpha_0)\W,
\end{align}
where, recalling $M_{\rho_0}(\alpha_0)$ from Theorem~\ref{thm:asymptotic} and $s_1$ from \eqref{eq:s1},
\begin{align} \label{eq:malpha-bc}
M_{\rho_0}^{\mathrm{bc}}(\alpha_0)
=
        \begin{pmatrix}
            1/\varpi_{\rho_0} & 0 \\
            (\alpha_1)^{-1}\log(s_1)  & 1/s_1
        \end{pmatrix}
        M_{\rho_0}(\alpha_0) \in \R^{2 \times 4}.
\end{align}
If $\rho_0=1$ (i.e., $\rho=\rho_\indi$), we have $M_{1}^{\mathrm{bc}}(\alpha_0)=M_1(\alpha_0)$ as in \eqref{eq:malpha}. 
\end{theorem}

\section{Top-Two Order Statistics Extracted from a Stationary Time Series}
\label{sec:estimation-blockmaxima}

Throughout this section, we suppose to observe a finite stretch of observations $\xi_1, \dots, \xi_n$ taken from a time series that satisfies the following condition inspired by Theorem~\ref{thm:welsch}.

\begin{condition}[Domain of attraction]
\label{cond:doa}
    The time series $(\xi_t)_{t\in\Z}$ is strictly stationary with a continuous marginal cdf $F$. Moreover, there exists a function $\rho \in \mathcal C$,
    a positive number $\alpha_0$, and a sequence $(\sigma_r)_{r\in\N}$ of positive numbers with $\sigma_r\to\infty$ for $r\to\infty$ such that
    \begin{align}\label{eq:doa}
        \begin{pmatrix}
            M_r/\sigma_r\\
            S_r/\sigma_r
        \end{pmatrix}
        \rightsquigarrow \mathcal{W}(\rho,\alpha_0, 1), \qquad r\to\infty.
    \end{align}
    Finally, the sequence $(\sigma_r)_{r\in\N}$ is regularly varying with index $1/\alpha_0$.
\end{condition}

Note that the condition is a natural extension of Condition 2.1 in \cite{Bücher2018-sliding} to the largest two observed values within a block of size $r$; see also Condition 3.1 in \cite{Bücher2018-disjoint}.
As in those papers, we are interested in estimating the unknown parameters $\alpha_0$ and $\sigma_r$, for some large block size parameter $r\in \{1, \dots, n\}$, based on the observed stretch of observations.

\subsection{Disjoint blocks}
\label{subsec:djb}

We start by discussing estimators that are based on the largest two order statistics calculated within successive disjoint blocks of size $r$.
For that purpose, let $k=\lfloor n/r \rfloor$ denote the number of such blocks that fit into the sampling period $\{1, \dots, n\}$. For integer $i\in \{1, \dots, k\}$, let
\begin{align}
    M_{r,i}:=\xi_{(1) , I_i}, 
    \qquad
    S_{r,i}:=\xi_{(2) , I_i}
    \label{eq:MS}
\end{align}
denote the two largest observations in the $i$th disjoint block of observations; here, $I_i = \{(i-1)r+1, \dots, ir\}$.
In view of Condition~\ref{cond:doa}, each vector $(M_{r,i}, S_{r,i})$ approximately follows the $\mathcal W(\rho, \alpha_0, \sigma_r)$-distribution, for sufficiently large block size $r$. This suggests to use the estimator $\hat \theta$ from \eqref{eq:mle}, applied to the sample $((M_{r,1}, S_{r,1}), \dots, (M_{r,k}, S_{r,k}))$.
It is the main goal of this section to show (in)consistency and asymptotic normality of $\hat \theta$ in an appropriate asymptotic framework. The framework, as well as the conditions are largely inspired by Section 3 in \cite{Bücher2018-disjoint}.

Formally, for the approximation $(M_{r,i}, S_{r,i}) \approx_d \mathcal W(\rho, \alpha_0, \sigma_r)$ to be accurate in the limit, we require the block size to increase to infinity, that is, $r=r_n\to\infty$ for $n\to\infty$. Moreover, consistency can only be achieved when the information increases, that is, when the number of blocks, $k_n= \lfloor n/r_n\rfloor$, goes to infinity as well. Finally, for technical reasons, the theory will developed for the estimator 
\begin{align}
\label{eq:thetan-bm}
\hat \theta_n^{(\dbl)} := (\hat \alpha_n^{(\dbl)}, \hat \sigma_{n}^{(\dbl)}) := \hat \theta\big( (M_{r_n,1} \vee c, S_{r_n,1} \vee c), \dots, (M_{r_n,k} \vee c, S_{r_n,k} \vee c)\big)
\end{align}
with $\hat \theta$ from \eqref{eq:mle},
where $c$ denotes some arbitrary small positive truncation constant. The truncation by $c$ guarantees that all observations are positive, as required for the likelihood in \eqref{eq:sw_logl} to be well-defined. Further note that Condition~\ref{cond:doa} implies that
\begin{align*}
    \Prob(M_{r_n,i} \le c,S_{r_n,i}\le c)
    \leq
    \Prob(M_{r_n,i}\le c)
    =
    \Prob(M_{r_n,i}/\sigma_{r_n} \le c/\sigma_{r_n})\to 0, \qquad n\to\infty,
\end{align*}
for any $c>0$, which shows that $(M_{r_n,i} \vee c, S_{r_n,i} \vee c)= (M_{r_n,i}, S_{r_n,i})$ with probability converging to one. Still, the smallest $S_{r_n,i}$ may be smaller than $c$, which we will prevent from happening with the following condition. As shown in Lemma~\ref{lem:no_ties}, the condition, together with the max-domain of attraction condition, will also imply the no-tie condition in Lemma~\ref{lemm:ex_uni}.

\begin{condition}[All second largest order statistics diverge]\label{cond:all_diverge} 
For every $c\in (0,\infty)$, we have
\begin{align*}
    \lim_{n\to\infty} \Prob\big(\min\{S_{r_n,1},\dots,S_{r_n,k_n}\}\leq c\big)=0.
\end{align*}
\end{condition}

The condition can often be shown using the union bound, $\Prob\big(\min\{S_{r_n,1},\dots,S_{r_n,k_n}\}\leq c\big) \le k_n \Prob(S_{r_n,1} \le c)$, suitable bounds on the cdf of $S_{r_n,1}$ and a condition relating $k_n$ and $r_n$; see, for instance, Example~\ref{ex:mori}.

Next, the serial dependence within the time series will be controlled using Rosenblatt's alpha-mixing coefficients, which need to decay sufficiently fast. For a positive integer $\ell$, put
\begin{align*}
    \alpha(\ell) = \sup\Big\{\big|\Prob(A\cap B)-\Prob(A)\Prob(B)\big|:A\in\sigma(\xi_t:t\leq 0),B\in\sigma(\xi_t:t\geq \ell)\Big\},
\end{align*}
where $\sigma(\cdot)$ denotes the $\sigma$-field generated by its argument.

\begin{condition}[$\alpha$-mixing rate]\label{cond:alphamixing}
We have $\lim_{\ell\to\infty}\alpha(\ell)=0$. Moreover, there exists $\omega>0$ such that 
\begin{align}\label{eq:alphamixing}
  \lim_{n\to\infty} (n/r_n)^{1+\omega}\alpha(r_n)=0.
\end{align}
Finally, there exists a sequence $(\ell_n)_n$ of integers such that $\ell_n\to\infty, \ell_n=o(r_n)$, $(n/r_n) \alpha(\ell_n)=o(1)$ and $(r_n/\ell_n)\alpha(\ell_n)=o(1)$.
\end{condition}

Note that Condition \ref{cond:alphamixing} can be interpreted as requiring the block sizes $r_n$ to be sufficiently large. 
The condition is not quite restrictive, and allows for long-range dependence in the sense that alpha-mixing coefficients may be non-summable. For instance, if $\alpha(\ell) = O(\ell^{-\beta})$ for $\ell\to\infty$ and some $\beta>0$, a simple calculation shows that \eqref{eq:alphamixing} is met  for any sequence $r_n$ that is of larger order than $n^{(1+\eps)/(1+\beta)}$ for some $\eps\in(0,\beta)$. Moreover, if we then choose $\ell_n=\lceil r_n^{1-\delta} \rceil$ for some $0<\delta <\min(\eps/\beta, \beta/(1+\beta))$, all four conditions on $\ell_n$ from Condition \ref{cond:alphamixing} can be shown to hold.

Within the proofs, we need the convergence of certain expectations involving $M_{r}$ or $S_{r}$ from \eqref{eq:doa}. That convergence is a consequence of uniform integrability, which in turn follows from the following condition on negative power moments of $S_r$ in the left tail and on logarithmic moments of $S_r$ in the right tail. 

\begin{condition}[Integrability]\label{cond:mom}
There exists some $\nu>1/\omega$ with $\omega$ from Condition \ref{cond:alphamixing}, such that 

\begin{align} 
\label{eq:intfinite}
\limsup_{r\to\infty} \Exp\!\big[h_{\nu}\big((M_r\vee 1)/\sigma_r\big)\big]<\infty,
\qquad
\limsup_{r\to\infty} \Exp\!\big[h_{\nu,\alpha_1}\big((S_r\vee 1)/\sigma_r\big)\big]&<\infty,
\end{align}
where 
$h_\nu(x)= \big(\log x\indic(x>\mathrm e)\big)^{2+\nu}$ and 
$h_{\nu,\alpha_1}(x)=\big(x^{-\alpha_1}\indic(x\leq \mathrm e)\big)^{2+\nu}$
with $\alpha_1=\alpha_1(\rho, \alpha_0)$ as in \eqref{eq:alpha1_NEW}.
\end{condition}

Note that the condition provides control on the right tail of $M_r$ and on the left tail of $S_r$. In view of $S_r \le M_r$, we then have control on both tails of both $M_r$ and $S_r$. 
We refer to \cite{Bücher2018-disjoint} for further discussions. 
Finally, we impose the following bias condition.

\begin{condition}[Bias]\label{cond:bias}
There exists $c_0>0$ such that, for every function $f=f_j$ from \eqref{eq:fct_H} with $j\in\{1, 2, 3,4 \}$ and with $\alpha_1=\alpha_1(\rho, \alpha_0)$ as in \eqref{eq:alpha1_NEW}, the limit $B(f) := \lim_{n\to\infty} B_n(f)$ exists, where
\begin{align*}
     B_n(f)
     =
    \sqrt {n/r_n}\bigg(\!\!\Exp\!\big[f\big( (M_{r_n}\vee c_0) / \sigma_{r_n}, (S_{r_n}\vee c_0)/\sigma_{r_n}\big)\big]-\int_{(0,\infty)^2} f(x,y)\diff H_{\rho,\alpha_0,1}(x,y)\!\bigg).
\end{align*}
\end{condition}

\begin{remark}
Akin to the mixing condition in Condition~\ref{cond:alphamixing}, the bias condition can be regarded as a high-level condition on the block size $r_n$. Indeed, as argued below, if Conditions~\ref{cond:doa} and \ref{cond:mom} are met, the bias condition is always met with $B(f) = 0$ if we choose $r_n$ sufficiently large. Non-trivial limits can be obtained in specific examples, see Example~\ref{ex:mori} for a time series model and Section~\ref{sec:estimation-blockmaxima-iid} for the iid case. 

We now prove the above claim that $B(f)=0$ is always possible by choosing $r_n$ sufficiently large. Write $H_r$ for the joint cdf of $((M_r\vee 1)/\sigma_r,(S_r \vee 1)/\sigma_r)$. By Conditions \ref{cond:doa} and \ref{cond:mom}, we have 
\begin{align*}
    \delta_r:= \max_{j=1,\dots ,4} \Big|\int_{(0,1)^2} f_j\diff (H_r-H_{\rho,\alpha_0,1})\Big| =o(1), \qquad (r\to\infty).
\end{align*}
For $n\in\N$ let $r_n = \min\{ r \in \N_{\ge \sqrt n} \mid m(r) \ge n\}$, where $m(r) := r / \delta_r$. Then $|B_n(f_j)| \le \sqrt{n/r_n} \delta_{r_n}=\sqrt{nr_n} (\delta_{r_n}/r_n)=\sqrt{nr_n} / m(r_n) \le \sqrt{r_n/n}$, and it remains to show that $r_n \in [n], r_n \to \infty$ and $r_n=o(n)$ for $n\to\infty$. First, $r_n\in[n]$ is met for all sufficiently large $n$, namely at least for those $n$ for which $\delta_n \le 1$. Second, $r_n \to \infty$ is a consequence of the fact that $r_n \ge \sqrt n$. Finally, to see that $r_n=o(n)$, introduce, for $r\in\N$, the nondecreasing function $L(r):= \inf_{s\geq r} (1/{\delta_s})$, which satisfies $L(r) \to \infty$ for $r \to \infty$. By definition, we have $m(r)\ge rL(r)$ for all $r \in \N$, which implies $(r_n-1)L(r_n-1) \le m(r_n-1) < n$, where the second inequality follows by definition of $r_n$. Rearranging the inequality yields $r_n/n < 1/L(r_n-1) + 1/n$ which converges to zero for $n \to \infty$.
We conclude $r_n=o(n)$, as asserted.
\end{remark}

Subsequently, we fix an arbitrary $c>0$  and let $\G_n^{(\dbl)} = \G_n$ denote the empirical process from \eqref{eq:cond2} with $v_n=\sqrt {n/r_n}$, $\sigma_n = \sigma_{r_n}$  
and with
\begin{align}\label{eq:trunc_toptwo}
        Z_{n,i} = (X_{n,i}, Y_{n,i}) = 
        (M_{r_n,i}\vee c,
        S_{r_n,i}\vee c)
    , \qquad i \in \{1,\dots,k_n\}.
\end{align}
We then have the following result.

\begin{theorem}\label{thm:blocks}
Suppose that Conditions \ref{cond:doa}, \ref{cond:all_diverge}, \ref{cond:alphamixing}, \ref{cond:mom} and \ref{cond:bias} are satisfied. Then, for any $c>0$, with probability tending to one, the estimator $\hat \theta_n^{(\dbl)}$ from \eqref{eq:thetan-bm} is well-defined and unique, and we have, as $n\to\infty$,
\begin{align}
    \sqrt{n/r_n}\begin{pmatrix}
        \hat\alpha_n^{(\dbl)}-\alpha_1\\
        \hat\sigma_n^{(\dbl)}/\sigma_{r_n}-s_1
    \end{pmatrix}
    &= 
   M_{\rho_0}(\alpha_0)
    \W_n^{(\dbl)}+o_{\mathbb P}(1)\notag
    \\&\rightsquigarrow M_{\rho_0}(\alpha_0)  \Nc_4(\bm B, \Sigma^{(\dbl)}_{\rho, \alpha_0})
\end{align}
with  $\alpha_1$ from \eqref{eq:alpha1_NEW} and $s_1$ from \eqref{eq:s1}. Here, $M_{\rho_0}(\alpha_0) \in \R^{2 \times 4}$ is as in Theorem~\ref{thm:asymptotic}, 
\begin{align*}
\bm W_n^{(\dbl)} = (\G_n^{(\dbl)} f_1, \G_n^{(\dbl)} f_2, \G_n^{(\dbl)} f_3, \G_n^{(\dbl)} f_4)^\top, \quad
\bm B = (B(f_1), B(f_2), B(f_3), B(f_4))^\top,
\end{align*}
with $f_j$ from \eqref{eq:fct_H}, and $\Sigma^{(\dbl)}_{\rho, \alpha_0} = (\sigma_{ij}^{(\dbl)})_{i,j=1}^4$ has entries 
\[
\sigma_{ij}^{(\dbl)}= \Cov_{(X,Y) \sim \mathcal W(\rho, \alpha_0,1)}(f_i(X,Y),f_j(X,Y)).
\]
Explicit formulas for $\Sigma^{(\dbl)}_{\rho, \alpha_0}$ are provided in Lemma~\ref{lem:cov_disjoint_general}; remarkably, the matrix depends on~$\rho$ only via $\rho_0=\rho(0)$ and $\rho_1 = \int_0^1z^{-1}[\rho_0 - \rho(z)]\diff z \ge 0$.

\end{theorem}

A careful look at the proof shows that regular variation of $(\sigma_r)_r$ from Condition~\ref{cond:doa} is only needed to deduce that $\sigma_{m_r}/\sigma_r \to 1$ for a certain integer sequence $(m_r)_{r\in\N}$ such that $m_r/r\to 1$ as $r\to\infty$.

\subsection{Sliding Blocks}
\label{subsec:slb}

Inspired by the results in \cite{Bücher2018-sliding}, we next consider a sliding blocks version of the estimators from the previous subsection.
For integers $s$ and $t$ with $1 \le s < t \le n$, define
\begin{align} \label{eq:sliding_ms}
    M_{s:t} := \xi_{(1), \{s, \dots, t\}}
    \qquad
    S_{s:t}:= \xi_{(2), \{s, \dots, t\}}
\end{align}
as the two largest order statistics among the observations $\xi_i$ with $i\in \{s, \dots, t\}$.
Note that the disjoint blocks versions from \eqref{eq:MS} can be written as $(M_{r,i}, S_{r,i}) = (M_{(i-1)r+1:ir}, S_{(i-1)r+1:ir})$ for $i\in\{1, \dots, \flo{n/r}\}$.
In view of Condition~\ref{cond:doa}, each vector $(M_{s:s+r-1}, S_{s:s+r-1})$ constructed from a block of successive observations of size $r$, with $s\in\{1, \dots, n-r+1\}$, approximately follows the $\mathcal W(\rho, \alpha_0, \sigma_r)$-distribution, for sufficiently large block size $r$. 
Following the argumentation in the previous section, this motivates the estimator
\begin{align}
\label{eq:thetan-bm-sliding}
\hat \theta_n^{(\sbl)} 
:= 
(\hat \alpha_n^{(\sbl)}, \hat \sigma_{n}^{(\sbl)}) 
:= 
\hat \theta\big( (M_{1:r} \vee c, S_{1:r} \vee c), \dots, (M_{n-r+1:n} \vee c, S_{n-r+1:n} \vee c) \big)
\end{align}
with $\hat \theta$ from \eqref{eq:mle}, where $c$ denotes a positive truncation constant and where we require $r=r_n\to \infty$ with $r_n=o(n)$ as $n\to\infty$. 
As in the previous section, we need to guarantee that the no-tie condition in Lemma~\ref{lemm:ex_uni} is satisfied with probability converging to one, and that the truncation by $c$ does not matter asymptotically. The next condition, which is a slight adaptation of Condition \ref{cond:all_diverge}, is sufficient; see also Condition 2.2 in \cite{Bücher2018-sliding} for a similar assumption.

\begin{condition}[All second largest order statistics of size $\flo{r_n/2}$ diverge]
\label{cond:all_diverge2} 
For every $c\in (0,\infty)$, the event that all second largest order statistics calculated from disjoint blocks of size $\tilde r_n=\flo{r_n/2}$ are larger than $c$ converges to one; i.e., 
\begin{align*}
    \lim_{n\to\infty} \Prob\Big(\min\big\{S_{1:\tilde r_n},\dots,S_{(\tilde k_n-1)\tilde r_n+1:\tilde r_n \tilde k_n}\big\}\leq c\Big)=0
\end{align*}
where $\tilde k_n = \flo{n/\tilde r_n}$ denotes the number of disjoint blocks of size $\tilde r_n$ that fit into the sampling period $\{1, \dots, n\}$.
\end{condition}

Subsequently, let $\G_n^{(\sbl)}=\G_n$ denote the empirical process from \eqref{eq:cond2} with $k_n = n-r+1$, $v_n=\sqrt {n/r_n}, \sigma_n = \sigma_{r_n}$ and with
\begin{align}
\label{eq:trunc_toptwo2}
        Z_{n,i} = (X_{n,i}, Y_{n,i}) = 
        (M_{i:i+r_n-1}\vee c,
        S_{i:i+r_n-1}\vee c)
    , \qquad i \in \{1,\dots,k_n\}.
\end{align}

\begin{theorem}\label{thm:sl_asy}
Suppose that Conditions \ref{cond:doa}, \ref{cond:alphamixing}, \ref{cond:mom}, \ref{cond:bias} and \ref{cond:all_diverge2} are met. 
Then, for any $c>0$ and with probability tending to one, the estimator $\hat \theta_n^{(\sbl)}$ from \eqref{eq:thetan-bm-sliding} is well-defined and unique and we have, as $n\to\infty$,
\begin{align*}
    \sqrt{n/r_n}\begin{pmatrix}
        \hat\alpha_n^{(\sbl)}-\alpha_1\\
        \hat\sigma_n^{(\sbl)}/\sigma_{r_n}- s_1
    \end{pmatrix}
    &= 
    M_{\rho_0}(\alpha_0)
    \W_n^{(\sbl)}+o_{\mathbb P}(1)\notag
    \\&\rightsquigarrow M_{\rho_0}(\alpha_0)  \Nc_4(\bm B, \Sigma^{(\sbl)}_{\rho, \alpha_0})
\end{align*}
with  $\alpha_1$ from \eqref{eq:alpha1_NEW} and $s_1$ from \eqref{eq:s1}. Here, $M_{\rho_0}(\alpha_0) \in \R^{2 \times 4}$ is as in Theorem~\ref{thm:asymptotic}, 
\begin{align*}
\bm W_n^{(\sbl)} = (\G_n^{(\sbl)} f_1, \G_n^{(\sbl)} f_2, \G_n^{(\sbl)} f_3, \G_n^{(\sbl)} f_4)^\top, \quad
\bm B = (B(f_1), B(f_2), B(f_3), B(f_4))^\top,
\end{align*}
with $f_j$ from \eqref{eq:fct_H}, and $\Sigma^{(\sbl)}_{\rho, \alpha_0} = (\sigma_{ij}^{(\sbl)})_{i,j=1}^4$ has entries 
\begin{align*}
    \sigma_{ij}^{(\sbl)} = 2 \int_0^1 \Cov
    \big(f_i(X,Y), f_j(\tilde X, \tilde Y)\big) \diff\zeta,
\end{align*}
where $(X, Y, \tilde X, \tilde Y)$ is a random vector whose bivariate cdfs needed for evaluating the covariance are given by $K_{\rho,\alpha_0,\zeta}$ from \eqref{eq:sl_distr}.
If $\rho=\rho_\indi$, we have $\alpha_1=\alpha_0, s_1=1$, $\sigma_{ij}^{(\dbl)}=2 s_{ij}(\alpha_0)$ with $s_{ij}(\alpha)$ from Lemma~\ref{lem:cov_sl}, and $M_{\rho_0}(\alpha_0)=M_1(\alpha_0)$ is explicitly given in \eqref{eq:malpha}.
\end{theorem}

\subsection{Bias-corrected estimation}
\label{sec:bias-corr-block-maxima}

The inconsistency of the disjoint and sliding blocks MLE can be resolved by the bias-correction approach from Section~\ref{sec:bias-corr}. For that purpose, we need an estimator for $\rho_0$ that converges sufficiently quickly to $\rho_0$. Note that, under suitable regularity conditions, we have $\rho_0 = \pi(1)$, where $\pi=(\pi(m))_{m\in\N}$ denotes the cluster size distribution of the time series $(\xi_t)_{t\in\Z}$; see \cite{Bei04}, Section 10, or \cite{Hsing1988}, Theorem 3.3.

Estimators for $\pi$ can be found in \cite{Hsi91, Fer03, Rob09b, Rob09, Jennessen2022}. Throughout the simulation study, we choose to work with the disjoint blocks estimator from Formula (2.6) in \cite{Jennessen2022}: for a block size $r'=r_n'\to\infty$ (typically smaller than $r=r_n$ used in the previous sections), the estimator is defined as
\begin{align} \label{eq:pihat}
\hat \pi_n(1) = \frac4{k'(k'-1)} \sum_{i \ne j} \bm 1\Big\{ \sum_{s \in I_j} \bm 1 \big[ \xi_s > \max(\xi_t: t \in I_i) \big] = 1 \Big\},
\end{align}
where $k'=\lfloor n/r' \rfloor$, where the summation is over all indexes $i,j \in \{1, \dots, k'\}$ with $i\ne j$ and where $I_i=\{ (i-1)r'+1, \dots, ir'\}$ denotes the $i$th disjoint block of indexes of size $r'$. Under suitable regularity conditions, $\sqrt{k'}(\hat \pi_n(1) - \pi(1))$ is asymptotically normal for $n\to\infty$, see Theorem 4.1 in \cite{Jennessen2022}. 
As a consequence, if we choose $r'=r'_n$ such that $r'_n=o(r_n)$ for $n\to\infty$ with $r_n$ as in Sections~\ref{subsec:djb} and \ref{subsec:slb}, we have $\sqrt{k}(\hat \pi_n(1)  - \pi(1)) = o_{\mathbb P}(1)$. 
The same is then true for the $[0,1]$-valued estimator $\hat \rho_{0,n}:= \min(\hat \pi_n(1),1)$,  that is, $\sqrt{k}(\hat \rho_{0,n}  - \rho_0) = o_{\mathbb P}(1)$, as required for an application of the results in Section~\ref{sec:bias-corr}. 
Hence, defining $\hat\varpi_n = \varpi_{\hat \rho_{0,n}}$ and 
\begin{align} \label{eq:bias-correction-bm}
\widetilde\alpha_n^{(\mbl)} := \hat\alpha_n^{(\mbl)} / \hat\varpi_n,
\qquad
\widetilde\sigma_n^{(\dbl)} =\hat\sigma_n^{(\mbl)} \Big\{ \frac{\Upsilon_{\hat \rho_{0,n}}(\hat \varpi_n)}2 \Big\}^{1/\hat \alpha_n^{(\mbl)}}
\end{align}
for $\mbl \in \{ \dbl, \sbl\}$, we obtain the following result.

\begin{corollary} \label{cor:blocks-bias-correction}
Suppose $\hat \rho_{0,n}=\rho_0+o_{\mathbb P}(k_n^{-1/2})$. 
Then, under the notations and conditions of Theorem~\ref{thm:blocks} (for $\mbl=\dbl$) or Theorem~\ref{thm:sl_asy} (for $\mbl=\sbl$), we have
\begin{align} \label{eq:blocks-bias-correction}
    \sqrt{n/r_n}\begin{pmatrix}
        \widetilde\alpha_n^{(\mbl)}-\alpha_0\\
        \widetilde\sigma_n^{(\mbl)}/\sigma_{r_n}-1
    \end{pmatrix}
    &= 
   M_{\rho_0}^{\mathrm{bc}}(\alpha_0)
    \W_n^{(\mbl)}+o_{\mathbb P}(1) \rightsquigarrow 
    M_{\rho_0}^{\mathrm{bc}}(\alpha_0)  \Nc_4(\bm B, \Sigma^{(\mbl)}_{\rho, \alpha_0}),
\end{align}
with $M_{\rho_0}^{\mathrm{bc}}(\alpha_0)$ as defined in \eqref{eq:malpha-bc}.
\end{corollary}

\begin{proof}
The result follows from an application of Theorem \ref{thm:bias-corrected-asymp}. The required conditions of Theorem~\ref{thm:asymptotic} are established in the proofs of Theorem~\ref{thm:blocks} (for $\mbl=\dbl$) and \ref{thm:sl_asy} (for $\mbl=\sbl$).
\end{proof}

It is important to stress again that, for $\rho=\rho_\indi$, the limit distribution in \eqref{eq:blocks-bias-correction} is the same as for $(\hat\alpha_n^{(\mbl)},  \hat\sigma_n^{(\mbl)})$. Hence, in the case where the original estimator was already consistent, there is no price to be paid for additionally estimating $\rho_0=1$.

\begin{remark}[On the asymptotic variance]
\label{rem:comparing_covs_general}
The asymptotic distribution in \eqref{eq:blocks-bias-correction} can be rewritten as
\[
M_{\rho_0}^{\mathrm{bc}}(\alpha_0) \mathcal{N}_4\big(\boldsymbol B,\Sigma^{(\mbl)}_{\rho, \alpha_0}\big)
=
\mathcal{N}_2\big(\boldsymbol B_{\TT},\Sigma^{(\mbl)}_{\TT}(\alpha_0,\rho)\big),
\]
where
\begin{align} 
\label{eq:sigma-disjoint-general}
\Sigma^{(\mbl)}_{\TT}(\alpha_0,\rho) = M_{\rho_0}^{\mathrm{bc}}(\alpha_0)\Sigma^{(\mbl)}_{\rho, \alpha_0} M_{\rho_0}^{\mathrm{bc}}(\alpha_0)^\top \in \R^{2 \times 2}.
\end{align}
Explicit values of $\Sigma^{(\mbl)}_{\rho, \alpha_0}$ are derived in Lemma \ref{lem:cov_disjoint_general} and Lemma \ref{lem:cov_sl}, which also allow for explicit evaluation of the matrix product in the previous display. 
Notably, $\Sigma^{(\dbl)}_{\rho, \alpha_0}$ depends on $\rho$ only via $\rho_0$ and  $\rho_1 = \int_0^1z^{-1}[\rho_0 - \rho(z)]\diff z \ge 0$, while a more complicated dependence arises for $\Sigma^{(\sbl)}_{\rho, \alpha_0}$. The diagonal elements, i.e., the asymptotic variances of the shape and scale estimators, are depicted in Figure \ref{fig:stationary_cov} as a function of $\rho_0$ (and for $\alpha_0=1$), for the three parametric classes provided in Example~\ref{ex:models}. As a benchmark, we also add horizontal lines that correspond to the asymptotic variances of the plain disjoint and sliding block maxima MLE from \cite{Bücher2018-disjoint} and \cite{Bücher2018-sliding}, respectively, which are are explicitly stated in \eqref{eq:max_dbm_sbm_matrix} below.
\begin{figure}[t]
    \centering
    \includegraphics[width=.9\linewidth]{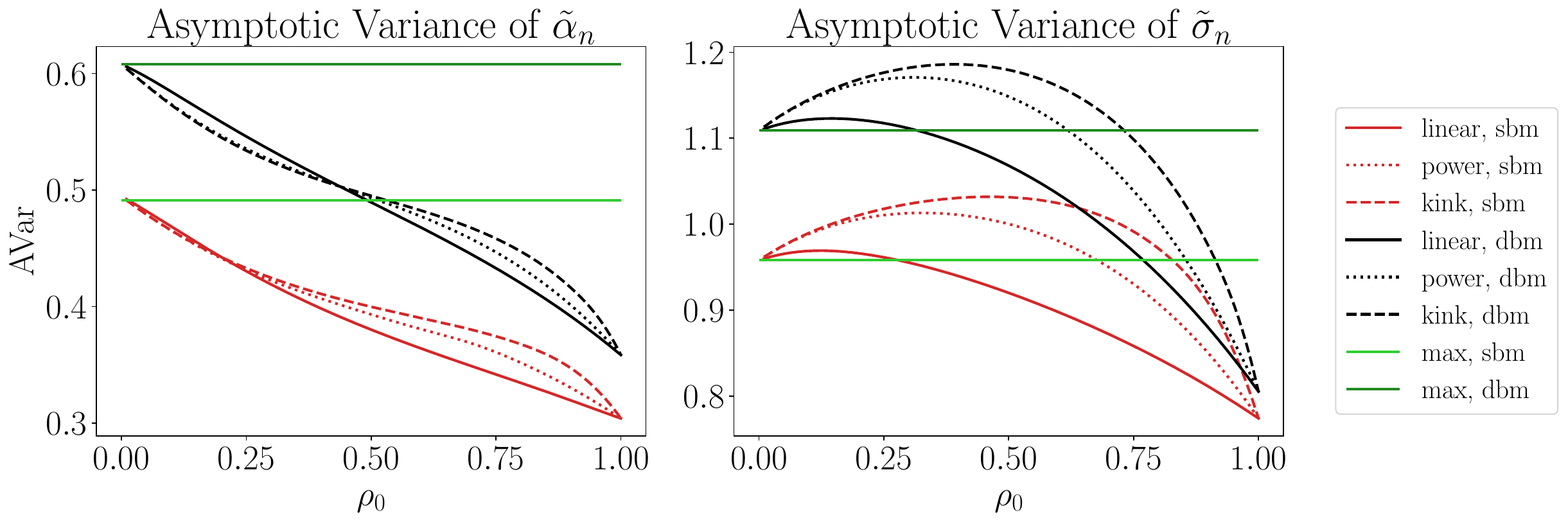}
    \caption{Standardized asymptotic variance of shape (left) and scale (right) estimators, that is, the diagonal entries of the asymptotic covariance matrices
    $\Sigma^{(\mbl)}_{\TT}(1, \rho)$ and $\Sigma^{(\mbl)}_\mathrm{max}(1)$ from \eqref{eq:sigma-disjoint-general} and \eqref{eq:max_dbm_sbm_matrix}, respectively, 
    at $\alpha_0=1$ and as a function of $\rho_0$.
    The examples ``linear", ``power" and ``kink" correspond to Example \ref{ex:models} [a], [b] and [c], respectively. For the disjoint blocks version, the respective curves for an arbitrary $\rho \in \mathcal C$ lie between the `linear" and ``kink" curves. 
    } 
    
    \label{fig:stationary_cov}
\end{figure}

We find that the top-two shape estimators exhibit a smaller asymptotic variance than their block maxima counterparts, uniformly over all considered $\rho$-functions. In fact, for the disjoint blocks version, the bounds derived in Lemma \ref{lem:cov_disjoint_general} show that the depicted curves correspond to `best and worst cases', that is, all possible variance curves (over $\rho \in  \mathcal C$) lie between the curves corresponding to the linear and the kink model.

The findings are more complicated for the scale estimator: it is only for values of $\rho_0$ in a neighborhood around 1 (i.e., close to independence) that the top-two scale estimators exhibit a smaller variance than their block maxima counterparts. The specific neighborhood depends on the model: it is quite large for the linear model (approximately $[0.31,1]$ for disjoint and $[0.28,1]$ for sliding) and quite small for the kink model (approximately $[0.73,1]$ for disjoint and $[0.83,1]$ for sliding).

Together, these findings indicate that the top-two estimator should be used for estimating the shape $\alpha$, while the block maxima MLE may be preferable for estimating $\sigma$ in situations exhibiting moderately strong serial dependence. Thus, when interested in target parameters depending on both the shape and the scale (such as return levels), it may be beneficial to mix both estimators; we refer to Section \ref{sec:simulation} and Equation \eqref{eq:botw} below for details. 

\end{remark}

We next provide an explicit example where all conditions of Theorem~\ref{thm:blocks} are met. In particular, we provide explicit formulas for the bias terms in Corollary~\ref{cor:blocks-bias-correction}, which allows for a theoretical comparison with max-only estimators in terms of their asymptotic bias and MSE.

\begin{example}[A version of {\cite[Example 1]{Mori_1976}}]
\label{ex:mori}
Let $\rho \in \mathcal C$ be arbitrary. By concavity, $\rho'$ exists and is continuous everywhere except at countably many points, see Theorem 25.3 in \cite{Roc97}. Let $F(x) = 0$ for $x<0$, $F(x)=1$ for $x \ge 1$, and $F$ the right-continuous extension of $-\rho'$ on $[0,1)$; this defines a probability distribution $P_\rho$ with support $[0,1]$. For $\rho(\eta)=\rho_\indi(\eta)=1-\eta$, we have $P_\rho=\delta_0$, and for $\rho(\eta) \equiv 0$, we have $P_\rho = \delta_1$.

Let $(Z_t)_t$ be iid standard Pareto, and let $\zeta_t\sim P_\rho$ be iid and independent of $(Z_t)_t$. Define
\[
\xi_t = \max(Z_{t-1}, \zeta_t Z_t)^{1/\alpha}, \qquad t \in \Z.
\]
Apparently, $(\xi_t)_t$ is strictly stationary and 1-dependent. If $r=r_n\in[n]$ is such that $r_n\to\infty, r_n=o(n)$ and such that 
$\lambda_1:=\lim_{n\to\infty} \sqrt{n/r_n^3} =\lim_{n\to\infty} \sqrt{k_n} / r_n \in [0,\infty)$
exists; see also \eqref{eq:lambda1} in Theorem~\ref{thm:IID} below; 
then Conditions \ref{cond:doa}, \ref{cond:all_diverge}, \ref{cond:alphamixing}, \ref{cond:mom} and \ref{cond:bias} are met with $\alpha_0=\alpha$ and $\sigma_r=r^{1/\alpha}$, with the bias $B(f_j)$ from Condition~\ref{cond:bias} explicitly given in \eqref{eq:bias-mori} below. Remarkably, the bias depends on $\rho$ only via $\rho_0$, and  if $\lambda_1=0$, we have $B(f_j)=0$.

The bias for the top-two shape estimators (note that it is the same for the disjoint and sliding blocks version), that is
$\boldsymbol B_{\TT}(\lambda_1) 
    = 
    \lambda_1 M_{\rho_0}^{\mathrm{bc}}(\alpha_0)
    \boldsymbol B'$
with $M_{\rho_0}^{\mathrm{bc}}(\alpha_0)$ as defined in \eqref{eq:malpha-bc} and $\bm B' = (B'(f_j))_{j=1, \dots, 4}$ as defined just before \eqref{eq:bias-mori},
is depicted in Figure~\ref{fig:example_bias}, for the case where $\lambda_1=1$ and $\alpha_0=1$ and as a function of $\rho_0$. As a benchmark, we also added a respective curve for the block maxima estimators, whose asymptotic bias is 
\begin{align} \label{eq:block-maxima-mle-bias-mori}
\boldsymbol B_{\max}(\lambda_1) = 
    \lambda_1\frac{6}{\pi^2} 
    \begin{pmatrix}
        \alpha_0\\ \{ \pi ^2 (2-\rho_0)+6 \gamma  (5-2 \rho_0)-6\}/(6 \alpha_0)
    \end{pmatrix} 
\end{align}
as shown in Section~\ref{subsec:proof-of-mori-example}. We observe that the bias of the top-two and the max-only approaches are of comparable magnitude, with some slight advantages for the former.

\begin{figure}[t]
    \centering
    \includegraphics[width=0.7\linewidth]{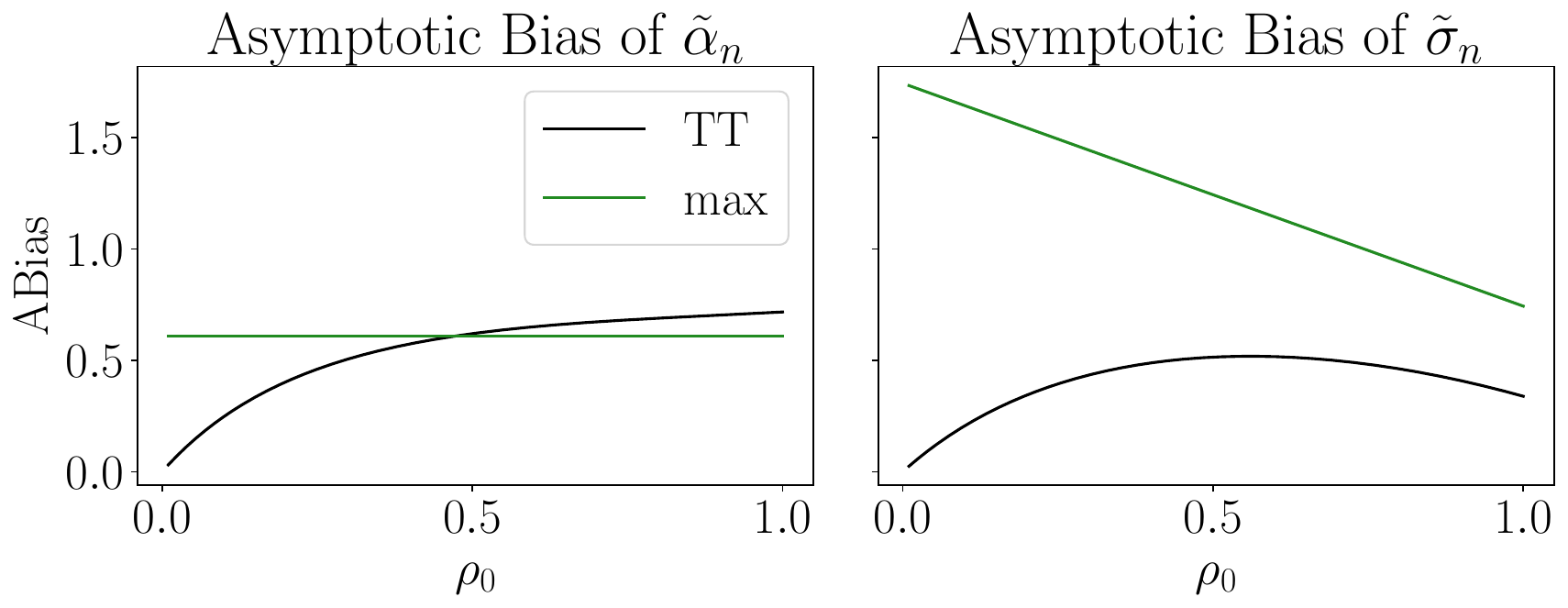}
    \caption{Standardized asymptotic bias of shape (left) and scale (right) estimators as a function of $\rho_0$, for $\alpha_0=1$ and $\lambda_1=1$. More precisely, the depicted values correspond to the mean of the asymptotic distributions of $\sqrt k_n(\widetilde \alpha_n - \alpha_0)$ and $\sqrt k_n(\widetilde \sigma_n/\sigma_n - 1)$, respectively, under the assumption that $\sqrt{k_n}/r_n = \lambda_1+o(1)$ for $n \to \infty$.}
    \label{fig:example_bias}
\end{figure}

Together with the derivations in Remark~\ref{rem:comparing_covs_general}, the different methods may be compared in terms of asymptotic expansions of their mean squared error at finite block size $r_n$, formally defined as
\[
\mathrm{AMSE}(\widetilde \alpha_\TT^{(\mbl)}) 
=
\frac{r_n}{n}\big(\Sigma^{(\mbl)}_{\TT}(\alpha_0,\rho)\big)_{11}
    +(\boldsymbol B_{\TT}(1/r_n))_1^2,
\]
and likewise for the scale and block maxima estimators.
This is partly illustrated in Figure~\ref{fig:mori-mse} for the case of fixed sample size $n=1000$ and for $\rho(\eta)=c\cdot(1-\eta)$, $c=\rho_0 \in\{0.2, 0.5, 0.9\}$. We observe the typical bias-variance tradeoff, with the top-two methods outperforming the max-only methods for most block sizes. Similar results were obtained for the scale estimation, see Section~\ref{subsec:bias-variance-expansions-mori}.

\begin{figure}[t]
    \centering
    \includegraphics[width=0.98\linewidth]{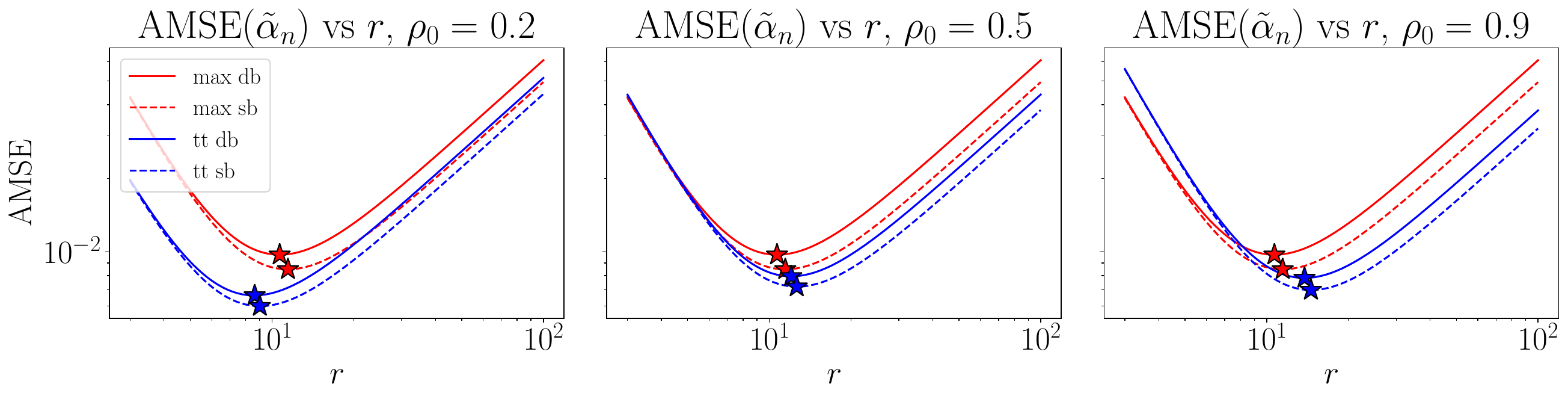}
    \caption{
    Asymptotic MSE of $\widetilde \alpha_n^{(\mbl)}$ as a function of the block size $r$, for fixed $\alpha_0=1$, $n = 1000$ and three choices of $\rho(\eta)=c\cdot(1-\eta)$, $c\in\{0.2, 0.5, 0.9\}$.}
    \label{fig:mori-mse}
\end{figure}

\end{example}

\section{Top-Two Order Statistics Extracted from an iid Sample}
\label{sec:estimation-blockmaxima-iid}

In this section, we specialize the results from the previous section to the case where $\xi_1,\xi_2,\dots$ are iid random variables with common distribution function $F$. In this setting, fitting extreme-value distribution based on block maxima has also been considered in \cite{Dombry2015,Ferreira2015,OorZho20}.

Because of the serial dependence, the conditions from the previous section can be simplified considerably. For instance, weak convergence of the two largest order statistics as required in Condition~\ref{cond:doa} is already a consequence of weak convergence of the largest order statistic only \cite[Theorem 3.5]{Col01}. In addition, the mean vector of the asymptotic normal distributions in Section~\ref{sec:estimation-blockmaxima} can be made explicit provided a standard second order condition on the weak convergence of affinely standardized maxima is met.

More specifically, recall that $F$ is in the maximum domain of attraction of the $\Frechet$ distribution family with shape parameter $\alpha_0\in(0,\infty)$ if there exists a positive scalar sequence $(a_r)_{r\in\N}$ such that, for every $x\in (0,\infty)$,
\begin{align}
    \lim_{r\to\infty} F^r(a_rx)= \exp\big(-x^{-\alpha_0}\big), \label{eq:iid_doa}
\end{align}
which corresponds to weak convergence of the first marginal distribution in \eqref{eq:doa}.
Note that \eqref{eq:iid_doa} is equivalent to regular variation of $-\log F$ at infinity with index $-\alpha_0$: we have $F(x)<\infty$ for all $x\in\R$ and 
\begin{align}
\label{eq:regvar}
    \lim_{u\to\infty} \frac{- \log F(u x)}{- \log F(u)} = x^{-\alpha_0}
\end{align}
for all $x\in(0,\infty)$ \cite{Gnedenko1943}. Moreover, for \eqref{eq:iid_doa} to be satisfied, the sequence $(a_r)_{r\in\N}$ may be chosen as any sequence satisfying
\begin{align} \label{eq:aniid}
    \lim_{r\to\infty} -r\log F(a_r)=1,
\end{align}
and it is necessarily regularly varying of index $1/\alpha_0$.

For the results to follow, the only condition needed is a second-order refinement of the convergence in \eqref{eq:regvar}, see \cite[Section 3.6]{Bingham1987} for details on second-order regular variation.
For $\tau \in\R$, define $h_\tau: (0, \infty)\to\R$ by
\begin{align*}
h_\tau(x)=\int_1^x y^{\tau-1}\diff y=
\begin{dcases*}
    \frac{x^\tau-1}{\tau}, & if $\tau\neq 0$, \\
    \log x, & if $\tau=0$.
    \end{dcases*}
\end{align*}

\begin{condition}[Second-Order Condition]\label{cond:sorv}
There exists $\alpha_0\in(0,\infty)$, $\tau\in (-\infty,0]$ and a real function $A:(0,\infty)\to\R$ of constant, non-zero sign such that $\lim_{u\to\infty}A(u)=0$ and such that, for all $x\in (0,\infty)$,
\begin{align}
    \lim_{u\to\infty} \frac{1}{A(u)}\Big(\frac{-\log F(ux)}{-\log F(u)}-x^{-\alpha_0}\Big) = x^{-\alpha_0} h_\tau(x).\label{eq:secorderrv}
\end{align}
\end{condition}

The function $A$ can be regarded as capturing the speed of convergence in \eqref{eq:regvar}. The form of the limit function in \eqref{eq:secorderrv} arises naturally, as explained in \cite[Remark 4.3]{Bücher2018-disjoint}.

Note that the estimator $\hat \sigma^{(\mbl)}_n$ may be  considered as an estimator for each $a_{r_n}$ for which $(a_r)_{r\in\N}$ satisfies \eqref{eq:aniid}. The mean of the asymptotic distribution of $\sqrt{k_n}(\hat \sigma^{(\mbl)}_n / a_{r_n}-1)$ will turn out to depend on the specific choice of $(a_r)_{r \in \N}$. The most canonical choice is the sequence $(a_r)_{r\in \N}$ defined by $-r\log F(a_r)=1$; in fact, for the max-only estimators, \cite{Bücher2018-disjoint} only provide results for that choice. For more general sequences, the effect on the asymptotic distribution will be captured below by assuming existence of the limit in \eqref{eq:lambda3}.

\begin{theorem} \label{thm:IID}
Let $\xi_1,\xi_2,\dots$ be independent random variables with continuous distribution function $F$ satisfying Condition \ref{cond:sorv}. Let $(a_r)_{r\in\N}$ be a sequence satisfying \eqref{eq:aniid}, let the block sizes $(r_n)_{n\in\N}$ be such that $r_n\to\infty$ and $k_n=\lfloor n/r_n\rfloor\to\infty$ as $n\to\infty$ and assume that the following three limits exist:
\begin{align} \label{eq:lambda1}
    \lambda_1 &:= \lim_{n\to\infty} \frac{\sqrt{k_n}}{r_n} \in[0,\infty),
    \\ 
    \lambda_2 &:=\label{eq:lambda2}
    \lim_{n\to\infty} \sqrt{k_n}A(a_{r_n})\in\R,
    \\\lambda_3 &:=\label{eq:lambda3}
    \lim_{n\to\infty} \sqrt{k_n}\big(-r_n\log F(a_{r_n})-1\big)\in\R.
\end{align}

Then, for any $c>0$ and 
with probability tending to one, the estimators $\hat \theta_n^{(\dbl)}$ from \eqref{eq:thetan-bm} and $\hat \theta_n^{(\sbl)}$ from \eqref{eq:thetan-bm-sliding} are well-defined and unique, and we have, as $n \to \infty$,
\begin{align}
\label{eq:iid-asymp}
    &\sqrt{n/r_n} \begin{pmatrix}
    \hat\alpha_n^{(\mbl)}-\alpha_0\\
    \hat\sigma_n^{(\mbl)}/a_{r_n}-1
    \end{pmatrix}
    \rightsquigarrow  M_1(\alpha_0) \mathcal{N}_4\big( B(\alpha_0,\tau),\Sigma^{(\mbl)}_{\rho_{\indi}, \alpha_0}\big), 
\end{align}
with $M_1(\alpha_0)$ from \eqref{eq:malpha}, with $\Sigma^{(\dbl)}_{\rho_{\indi}, \alpha_0}$ having entries $\sigma_{ij}^{(\dbl)}$ from Lemma~\ref{lem:cov_disjoint_general}, with $\Sigma^{(\sbl)}_{\rho_{\indi}, \alpha_0}$ having entries $\sigma_{ij}^{(\sbl)} = 2s_{ij}(\alpha_0)$ from
Lemma~\ref{lem:cov_sl},
and with
\[   
B(\alpha_0,\tau) = 
    \frac{\lambda_1}{\alpha_0}
     \Lambda_1(\alpha_0)
    +
    \frac{\lambda_2}{\alpha_0^2} \Lambda_2(\alpha_0, \bar \tau)
    + 
    \frac{\lambda_3}{\alpha_0} \Lambda_3(\alpha_0),
\]
where, for $\bar \tau := |\tau|/\alpha_0$,
\begin{align}
\label{eq:bias-iid}
    \Lambda_1(\alpha_0) &:= \nonumber 
    \begin{pmatrix}
        \gamma-5/2\\
        \alpha_0 \\
        -1/2\\0
    \end{pmatrix}, \quad
    \Lambda_2(\alpha_0, \bar \tau) := \frac1{\bar \tau}
    \begin{pmatrix}
                5-2\gamma-\Gamma(3+\bar{\tau})-\Gamma'(3+\bar{\tau})\\
                \alpha_0\big\{\Gamma(3+\bar{\tau})-2\big\} \\
                 1-\Gamma(2+\bar{\tau}) \\
                1-\Gamma(1+\bar{\tau})
            \end{pmatrix}, \quad \\
    \Lambda_3(\alpha_0) &:=
    \begin{pmatrix}
        5-2\gamma\\
        -2\alpha_0 \\
        1\\ 1
    \end{pmatrix}
\end{align}
for $\bar \tau >0$ and $\Lambda_2(\alpha_0, 0)$ defined by continuity.
Moreover, if $\hat \rho_{0,n}=1+o_{\mathbb P}(k_n^{-1/2})$, the results from \eqref{eq:iid-asymp} also hold if $(\hat\alpha_n^{(\mbl)},\hat\sigma_n^{(\mbl)})$ is replaced by the bias-corrected estimators $(\widetilde\alpha_n^{(\mbl)},\widetilde\sigma_n^{(\mbl)})$ from \eqref{eq:bias-correction-bm}.
\end{theorem}

\begin{remark}[On the asymptotic bias]
\label{rem:iid-bias}
The asymptotic distribution in \eqref{eq:iid-asymp} crucially depends on the three limit relations in \eqref{eq:lambda1}--\eqref{eq:lambda3}. Remarkably, only the condition in \eqref{eq:lambda2} was required in \cite{Bücher2018-disjoint} and \cite{Bücher2018-sliding} to derive bias formulas for the plain disjoint and sliding block maxima estimators. This discrepancy can partly be explained by an error in their statement that was discovered when working on the above theorem: during their proof of Theorem 4.2, \cite{Bücher2018-disjoint} impose the condition that $-r_n\log F(a_{r_n})=1$ (middle of page 1457), which immediately implies that $\lambda_3=0$ and restricts the claimed generality of their results. More precisely, if $\lambda_3 \ne 0$, different bias formulas arise in their theorem that are explicitly given in Lemma~\ref{lem:correction-bias} in the supplement for completeness. As such, it is only the first convergence in \eqref{eq:lambda1} that is
inherent to the top-two estimator: it results from a Taylor expansion of the logarithm that is needed within the proofs when dealing with empirical means of the second largest order statistics. If $A(a_r) =o(1/r)$, the second condition with $\lambda_2 \ne 0$ implies the first convergence with $\lambda_1=0$. For $A(a_r)$ of the exact order $1/r$, $\lambda_2 \ne 0$ will typically be equivalent to $\lambda_1 \ne 0$. If $A(a_r)$ is of faster order than $1/r$, then the first convergence with $\lambda_1>0$ will imply the second with $\lambda_2=0$. The phenomenon is illustrated in more detail in Section~\ref{subsec:bias-variance-expansions}.

Finally, note that the first row of the bivariate bias vector $M_1(\alpha_0)B(\alpha_0, \tau)$ does not depend on $\lambda_3$; indeed, 
$M_1(\alpha_0) \big(5-2\gamma, -2\alpha_0, 1, 1)^\top=(0,1)^\top$. This is not surprising in view of the fact that $\hat \alpha_n^{(\mbl)}$ is scale-invariant, which means that we can restrict attention to the case $-r \log F(a_r)=1$ (i.e., $\lambda_3=0$) for deriving its asymptotic distribution.
\end{remark}

\begin{remark}[On the asymptotic variance]\label{rem:comparing_covs}
Recall the asymptotic covariance matrices $\Sigma^{(\mbl)}_{\TT}(\alpha_0,\rho)$ from \eqref{eq:sigma-disjoint-general}. For $\rho=\rho_\indi$, we obtain that $\Sigma^{(\mbl)}_{\TT}(\alpha_0):=\Sigma^{(\mbl)}_{\TT}(\alpha_0,\rho_\indi)$ simplifies to
\[
\Sigma^{(\dbl)}_{\TT}(\alpha_0)
\approx
\begin{pmatrix}
        0.358\alpha_0^2 & -0.331 \\
        -0.331 & 0.805/\alpha_0^2
\end{pmatrix},
\qquad
\Sigma^{(\sbl)}_{\TT}(\alpha_0)
\approx
\begin{pmatrix}
    0.304\alpha_0^2 & -0.338 \\ -0.338 & 0.774\alpha_0^2
\end{pmatrix}
\]
These matrices shall be compared with the asymptotic covariance matrices for the disjoint and sliding block maxima MLE \cite{Bücher2018-disjoint, Bücher2018-sliding}, respectively,
which are given by
\begin{align}\label{eq:max_dbm_sbm_matrix}
    \Sigma^{(\dbl)}_\mathrm{max}(\alpha_0)
    \approx
    \begin{pmatrix}
            0.608\alpha_0^2& -0.257 \\
            -0.257 & 1.109/\alpha_0^2
        \end{pmatrix},
        \qquad
    \Sigma^{(\sbl)}_\mathrm{max}(\alpha_0)
    \approx 
    \begin{pmatrix}
            0.495\alpha_0^2& -0.324 \\
            -0.324 & 0.958/\alpha_0^2
        \end{pmatrix}.
\end{align}
Further, the asymptotic covariance matrix of the all block maxima estimator from \cite{OorZho20} is given by
\[
\Sigma^{(\mathrm{ab})}_\mathrm{max}(\alpha_0)
\approx
\begin{pmatrix}
        0.3927\alpha_0^2& -0.3767 \\
        -0.3767 & 0.7483/\alpha_0^2
    \end{pmatrix}.
\]
Comparing the five matrices, we observe that
\[
\Sigma^{(\sbl)}_\TT(\alpha_0)
<_\mathbb L 
\begin{dcases}
    \Sigma^{(\mathrm{ab})}_\mathrm{max}(\alpha_0)\\
    \Sigma^{(\dbl)}_\TT(\alpha_0)
\end{dcases}\Bigg\}
<_\mathbb L 
\Sigma^{(\sbl)}_\mathrm{max}(\alpha_0)
<_\mathbb L 
\Sigma^{(\dbl)}_\mathrm{max}(\alpha_0),
\]
where $<_\mathbb L$ denotes the Loewner-ordering between symmetric matrices. Note that $\Sigma^{(\mathrm{ab})}_\mathrm{max}$ and $\Sigma^{(\dbl)}_\TT$ cannot be ordered: the former exhibits a larger asymptotic variance for estimating the shape and a smaller for estimating the scale.
Remarkably, the asymptotic variance of the top-two sliding shape estimator is about 22\% smaller than the respective variance of the all block-maxima estimator,  and even about 50\% smaller than that of the classical disjoint block maxima MLE.
\end{remark}

\begin{remark}[On the Asymptotic MSE] \label{rem:iid-mse}
Having explicit formulas both for the bias and the variance, we may compare the estimators in terms of their asymptotic MSE. For the sake of brevity, we limit our discussion to the estimation of the shape parameter. In that case, as explained in Remark~\ref{rem:iid-bias}, we may and will assume that $a_{r_n}$ satisfies  $-r_n \log F(a_{r_n})=1$, which implies that $\lambda_3 =0$. The asymptotic expansion for the MSE of $\hat \alpha_{\TT}^{(\mbl)}$ at finite block size $r_n$ is hence given by
\begin{align*}
    \mathrm{AMSE}(\hat\alpha_\TT^{(\mbl)}) 
    &= \frac{r_n}{n}\big(\Sigma^{(\mbl)}_{\TT}(\alpha_0,\rho_\indi)\big)_{11}
    +
    \mathrm{ABias}^2(\hat\alpha_\TT^{(\mbl)})
\end{align*}
where, using the notation from (\ref*{eq:bias-iid}),
\begin{align*}
    \mathrm{ABias}(\hat\alpha_\TT^{(\mbl)})
    &=
    \begin{pmatrix}
    1 &
    0
    \end{pmatrix}
    M_1(\alpha_0) \Big( 
    \frac{1}{r_n\alpha_0}
    \Lambda_1(\alpha_0)
+\frac{A(a_{r_n})}{\alpha_0^2} \Lambda_2(\alpha_0, \bar \tau) \Big).
\end{align*}
In  view of Lemma~\ref{lem:correction-bias}, similar formulas can be derived for $\hat \alpha_{\max}^{(\mbl)}$. The all block maxima estimator from \cite{OorZho20} is excluded from the subsequent discussion, as its asymptotic bias has not been derived explicitly in that paper.

In view of Condition~\ref{cond:sorv} and standard results on regular variation, the function $r \mapsto A(a_r)$ is regularly varying with index $-\bar{\tau}$, where $\bar \tau = |\tau|/\alpha_0$. Subsequently, we assume that it is of the form $A(r)=c \cdot \alpha_0\cdot r^{-\bar\tau}$ for some $c \ne 0$; an assumption that for instance applies if $\xi_t$ is Pareto($\alpha_0$)-distributed, with $c=-1/2$ and $\bar{\tau} = 1$ (see Section~\ref{subsec:bias-variance-expansions-iid}). Under this assumption, $\mathrm{AMSE}(\hat\alpha_n^{(\mbl)})$ is a function of $n, r_n, \alpha_0, c$ and $\bar \tau$, and we study its dependence on each of these parameters in Figure~\ref{fig:iid-mse-all-in-1}.

We start by discussing the top row of Figure~\ref{fig:iid-mse-all-in-1}, where we study $\mathrm{AMSE}(\hat\alpha_n^{(\mbl)})$ as a function of the block size $r$, keeping the other parameters fixed. More specifically, we fix $\alpha_0 = 5$ (a common tail index in environmental extremes), $c=-1$, $\bar \tau = 1$ and consider three sample sizes, $n \in \{10^3, 10^4, 10^5\}$. We observe that the maxima-only estimators outperform the top-two-estimators for small block sizes, and vice versa for large block sizes. The minimal values (over $r$) are obtained for the sliding top-two-Estimator, with its minimal AMSE being about 75\% of the minimal AMSE of the classical disjoint block maxima estimator.

The dependence of the minimal values over $r$ as a function of $c$, $\bar \tau$ and $\alpha_0$ is depicted in in the middle row of Figure~\ref{fig:iid-mse-all-in-1}, where we fix $n=1000$ and vary one of the parameters in each of the three plots, keeping the others fixed at $c=-1$, $\bar \tau = 1$ and $\alpha_0=5$. We observe that it is only for small absolute values of $c$ that the max-only estimators outperform the top-two-estimators. 

Finally, in in the bottom row of Figure~\ref{fig:iid-mse-all-in-1}, we study $\mathrm{AMSE}(\hat\alpha_n^{(\mbl)})$ as a function of the number of blocks $k$, keeping $\alpha_0 = 5$, $c=-1$ and $\bar \tau = 1$ and $r=30$ (left), $r=90$ (middle) and $r=365$ (right) fixed. Note that these choices of $r$ correspond to commonly used 
block sizes in environmental extremes (a month, a season, or a year of daily data). For block size $r=365$, the top-two estimators uniformly outperform the block maxima estimators over the considered range of $k \in \{10, \dots, 10\,000\}$. For $r=90$ and $r=30$, the top-two estimators are better for $k$ up to about 2000 and 200, respectively. Note that record lengths of observational data in environmental extremes are typically small; most often smaller than $k=100$ years or seasons.

\begin{figure}[!thp]
    \centering
    \includegraphics[width=0.98\linewidth]{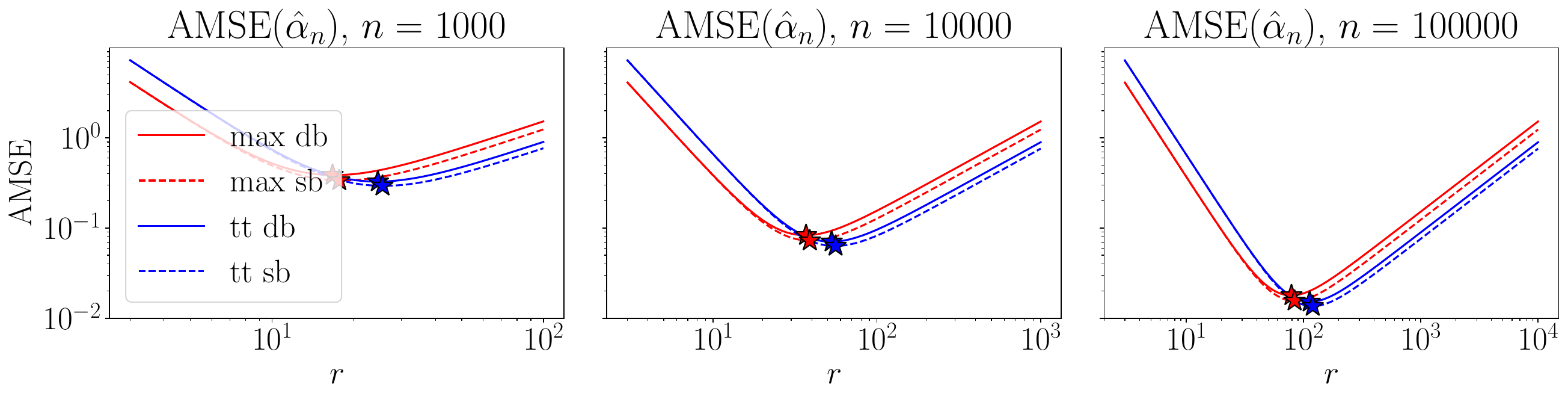}
    \includegraphics[width=0.98\linewidth]{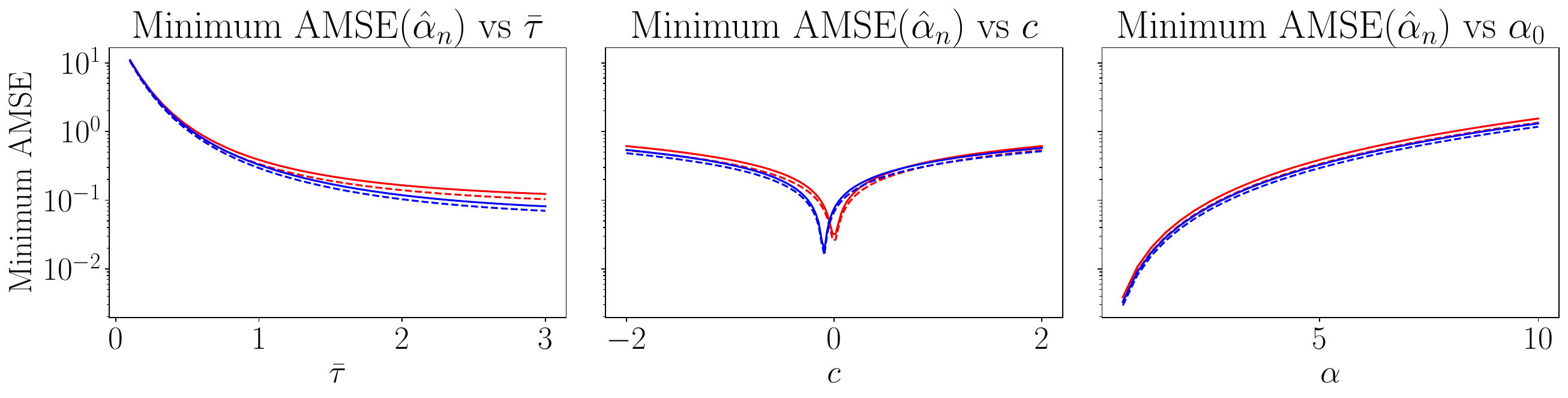}
    \includegraphics[width=0.98\linewidth]{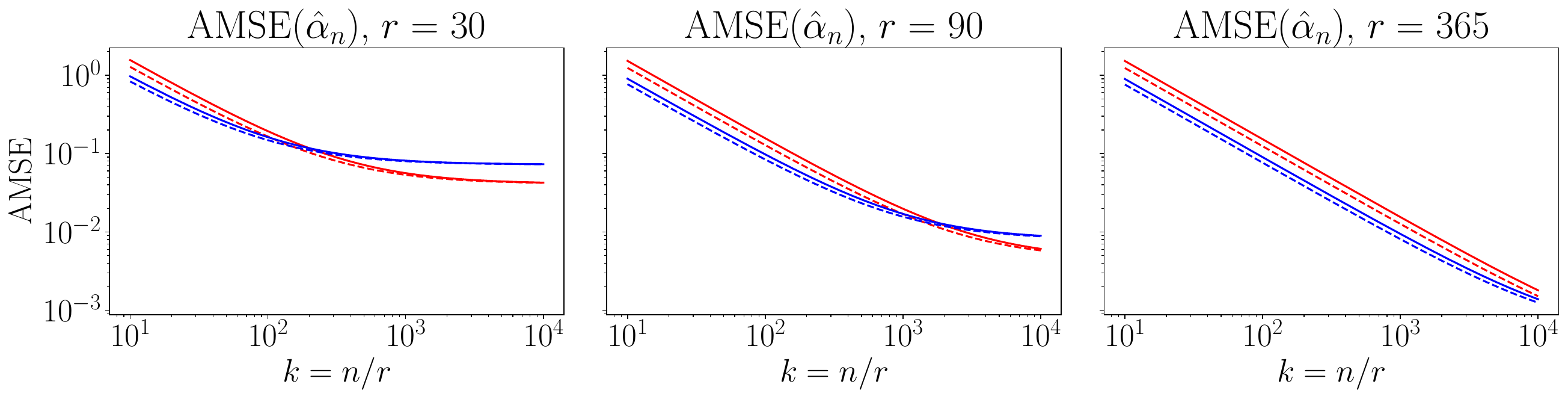}
    \caption{Asymptotic expansions in the IID case.
    Top row: $\mathrm{AMSE}(\hat\alpha_n^{(\mbl)})$ as a function  of the block size $r$, for fixed $\alpha_0 = 5$, $c=-1$, $\bar \tau =1$, and three sample sizes. \\
    Middle row: $\min_{r} \mathrm{AMSE}(\hat\alpha_n^{(\mbl)})$  for fixed $n=1000$ and as a function of $\bar \tau$ (left), $c$ (middle) and $\alpha_0$ (right), keeping the other parameters fixed at 
    $\alpha_0 = 5$, $c=-1$, $\bar \tau =1$ where applicable. \\
    Bottom row: $\mathrm{AMSE}(\hat\alpha_n^{(\mbl)})$ as a function  of the number of blocks $k$, for fixed $\alpha_0 = 5$, $c=-1$, $\bar \tau =1$, and three block sizes. }
    \label{fig:iid-mse-all-in-1}
\end{figure}

\end{remark}

\section{Monte Carlo Simulation Study}\label{sec:simulation}

A large scale Monte Carlo simulation study was performed to investigate the finite-sample properties of the proposed estimators, with a particular focus on a comparison to recent and traditional competitors from the literature. The results are partly summarized in this section, while a more comprehensive overview is provided in Section~\ref{sec:sim-additional}. All empirical performance measures are based on $N=1\,000$ simulation runs. Implementations are publicly available in \cite{haufs_xtremes_24}.

We concentrate on five different estimators: the disjoint and sliding blocks version of the bias-corrected top-two estimator $\hat \theta_{\TT}^{(\mbl)} := (\widetilde \alpha_n^{(\mbl)}, \widetilde \sigma_n^{(\mbl)})$ from \eqref{eq:bias-correction-bm} with $\mbl\in\{\dbl,\sbl\}$ (results on the uncorrected estimators can be found in Section~\ref{sec:sim-additional}), the disjoint and sliding blocks maxima estimator from \cite{Bücher2018-disjoint} and \cite{Bücher2018-sliding}, denoted by $\hat \theta_{\max}^{(\mbl)}$ with $\mbl \in \{\dbl,\sbl\}$, and a rescaled oracle version of the ABM estimator from \cite{OorZho20}, denoted by $\hat \theta_{\max}^{(\abl)}$. More specifically, rescaling the ABM estimator has been suggested by \cite{OorZho20} (``a proper transformation involving the extremal index is needed") and is motivated by the fact that, in contrast to the other estimators, the plain ABM estimator, say $\tilde \theta_{\max}^{(\abl)}$, should rather be regarded as an estimator for $(\alpha, \tilde{\sigma}_{r_n})$ than for $(\alpha, \sigma_{r_n})$, where $\tilde{\sigma}_{r}$ denotes the normalizing sequence associated with $r$ maxima of i.i.d.\ variables from $F$. As explained in \cite[Section~10.2.3]{Bei04}, the two scaling sequences are related via $\sigma_{r} = \tilde{\sigma}_r \, \theta^{1/\alpha}$, where $\theta$ is the extremal index of $(\xi_t)_t$. The oracle version of the ABM estimator is hence defined as $\hat \theta_{\max}^{(\mbl)} = f_\theta(\tilde \theta_{\max}^{(\abl)})$, where $f_\theta(\alpha, \sigma) = (\alpha, \sigma \theta^{1/\alpha})$. Since \cite{OorZho20} do not suggest any specific method for estimating the extremal index, we choose to work with the unknown true value as a benchmark; see below for details.

Throughout, we consider three different time series models: 
\begin{compactenum}[(1)]
    \item The iid-$\Pareto$-model: $(\xi_t)_t$ is an iid sequence from the generalized Pareto distribution with cdf 
    $F_\alpha(x) = ( 1 - x^{-\alpha}) \bm 1(x\ge1)$, where $\alpha>0$. 
    Condition~\ref{cond:doa} is met with $\rho=\rho_\indi$ and $\alpha_0=\alpha$. 
    \item The $\ARMAX$-$\Pareto$-model: for $\beta \in (0,1]$, let $\tilde \xi_t$ be a stationary solution of the recursion $\tilde \xi_t=\max(\beta \tilde \xi_{t-1}, (1-\beta) Z_t)$, where $(Z_t)_t$ is iid standard Fr\'echet, and let $\xi_t = F_\alpha^{-1}(-1/\log \tilde \xi_t)$. It can be shown that $\xi_t$ has cdf $F_\alpha$, and that Condition~\ref{cond:doa} is met with $\rho(\eta)=\min(1-\beta,1-\eta)$ and $\alpha_0=\alpha$; see also Example \ref{ex:models}~[c].
    \item The $\AR$-$\Pareto$-model: for $\beta\in(0,1]$, let $\tilde \xi_t$ be a stationary solution of the recursion $\tilde \xi_t= \beta \tilde \xi_{t-1}+ Z_t$, where $(Z_t)_t$ is iid standard Cauchy distributed, and let $\xi_t = F_\alpha^{-1}(F_{\tilde \xi_t}(\tilde \xi_t))$. It can be shown that $\xi_t$ has cdf $F_\alpha$, and that Condition~\ref{cond:doa} is met with $\rho(\eta)=\min(1-\beta,1-\eta)$ and $\alpha_0=\alpha$; see also Example \ref{ex:models}~[c].
\end{compactenum}
The parameter $\beta$ controlling the temporal dependence is chosen from the set $\{0.2,0.5,0.8\}$, while $\alpha$ is fixed to $\alpha=1$. Note that the extremal index is $\theta=1$  for Model (1) and $\theta = 1 - \beta$ for both Model (2) and (3), see \cite{BucZan23}.
In this section, we only report results for the iid model and the AR model with $\beta=0.5$; the remaining results can be found in Section~\ref{sec:sim-additional}, where we also present some results for the model from Example~\ref{ex:mori}.

We consider two target parameters: the tail index $\alpha_0$ itself, and the $(T,r)$-return level; a central object of interest in environmental extremes. Formally, the latter is defined, for a given block size $r$ and parameter $T\in\N$ of interest, as
\begin{align*}
    \RL(T,r):=F_r^{\leftarrow}(1-1/T)=\inf\{x\in\R:F_r(x)\geq 1-1/T\},
\end{align*}
where $F_r(x):=\Prob(M_r\leq x)$. As the true value of the return level is not known explicitly for the $\AR$-$\Pareto$-model, we approximate it by an initial Monte Carlo simulation based on a sample of $10^6$ simulated block maxima.

Under Condition \ref{cond:doa} and in view of \eqref{eq:w_margCDF1}, $F_r(x)$ may be approximated by $H_{\alpha_0,\sigma_r}(x):=\exp(-(x/\sigma_r)^{-\alpha_0})$, the cdf of the $\Frechet$ distribution with shape parameter $\alpha_0$ and scale $\sigma_r$. Since the quantile function of the $\Frechet$ family is  $H^{\leftarrow}_{\alpha,\sigma}(p)=\sigma(-\log p)^{-1/\alpha}$, a reasonable plug-in estimator for $\RL(T,r)$ is given by
\begin{align*}
    \widehat\RL_{\mathrm{method}}^{(\mbl)}(T,r)
    :=
    \widehat\RL(T,r) \big( \hat \theta_{\mathrm{method}}^{(\mbl)} \big)
    :=
    \hat\sigma_{\mathrm{method}}^{(\mbl)} b_T^{-1/\hat\alpha_{\mathrm{method}}^{(\mbl)}}, 
\end{align*}
where $b_T=-\log(1-1/T)$, $\mbl \in\{\dbl,\sbl,\abl\}$ and $\mathrm{method} \in \{\max, \TT\}$. Consistency and asymptotic normality of the estimator follows straightforwardly from the delta-method; we refer to Section 3 in \cite{Bücher2018-sliding} for details. For reasons that become clear later, we also consider a mixed $\max$-$\TT$-estimator
\begin{align}\label{eq:botw}
\widehat\RL_{\mathrm{botw}}(T,r)
=
\widehat\RL(T,r) \big( \hat \alpha_{\TT}^{(\sbl)}, \hat \sigma_{\max}^{(\sbl)}\big),
\end{align}
where the index $\mathrm{botw}$ stands for `best of two worlds'. In this section, we only report results for $T=100$; respective results for $T\in\{50,200\}$ can be found in Section~\ref{sec:sim-additional}.

\subsection{Fixed block size}
\label{sec:sim-fixed-r}

In many applications of the block maxima method, block sizes are chosen according to standard fixed time periods, often for convenience or due to established protocols such as those used in extreme event attribution studies by the World Weather Attribution initiative \cite{philip2020protocol}.
Typical choices are $r=365$ for yearly maxima of daily data, or $r=90$ for summer-season maxima. Although such fixed choices are not necessarily optimal from a statistical perspective (particularly in terms of the bias–variance trade-off), we believe that studying this setting remains important for assessing the performance of block maxima methods in commonly encountered practical scenarios.

In the current section, we fix $r=100$; additional results for $r\in\{50,200\}$ can be found in the supplement. The estimators' performance is measured by the mean-squared error; a more detailed decomposition into the squared bias and the variance does not provide any additional insights as the bias turns out to be of much smaller order than the variance. Regarding the block size parameter needed for the estimation of $\rho_0$ in the bias correction from Section~\ref{sec:bias-corr-block-maxima}; see in particular \eqref{eq:pihat}, we chose to fix $r'=50$.

We start by considering the estimation of the shape parameter. The respective simulation results are summarized in Figure~\ref{fig:fixed_bs_shape_and_rl} (left half), and provide the following insights: first, the sliding blocks top-two estimator is the best estimator in all scenarios under consideration. Second, each of the sliding blocks versions consistently outperforms its disjoint blocks counterpart. Third, the top-two estimators are consistently better than their max-only counterparts.  
Finally, the all block maxima method ranks third for the iid case, but is by far the worst estimator in the serially dependent case. All these findings are consistent with the theoretical results; this connection is further illustrated in Section~\ref{subsec:bias-variance-expansions}.

\begin{figure}[t]
    \centering
    \includegraphics[width=.45\textwidth]{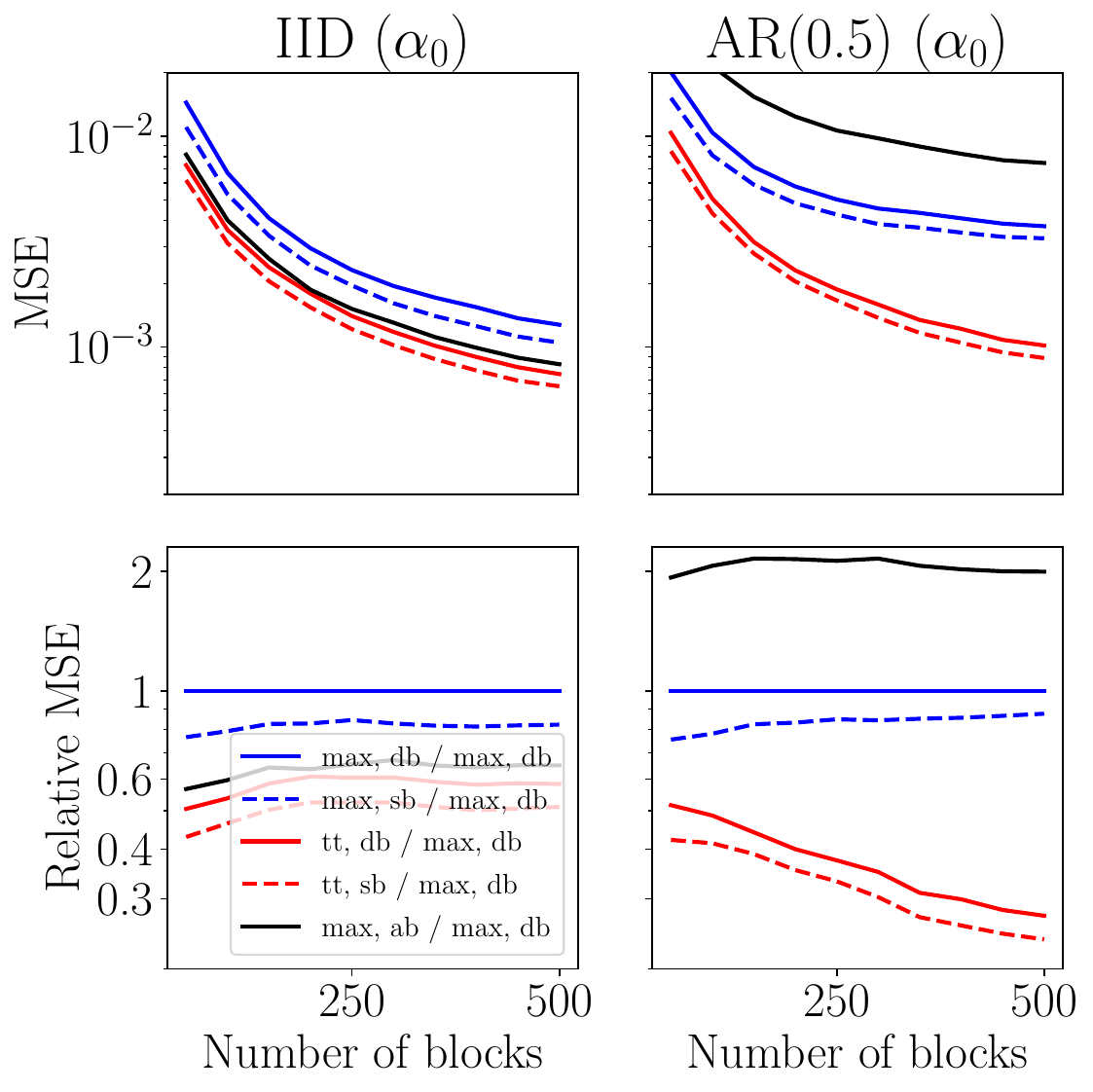}
    \includegraphics[width=.45\textwidth]{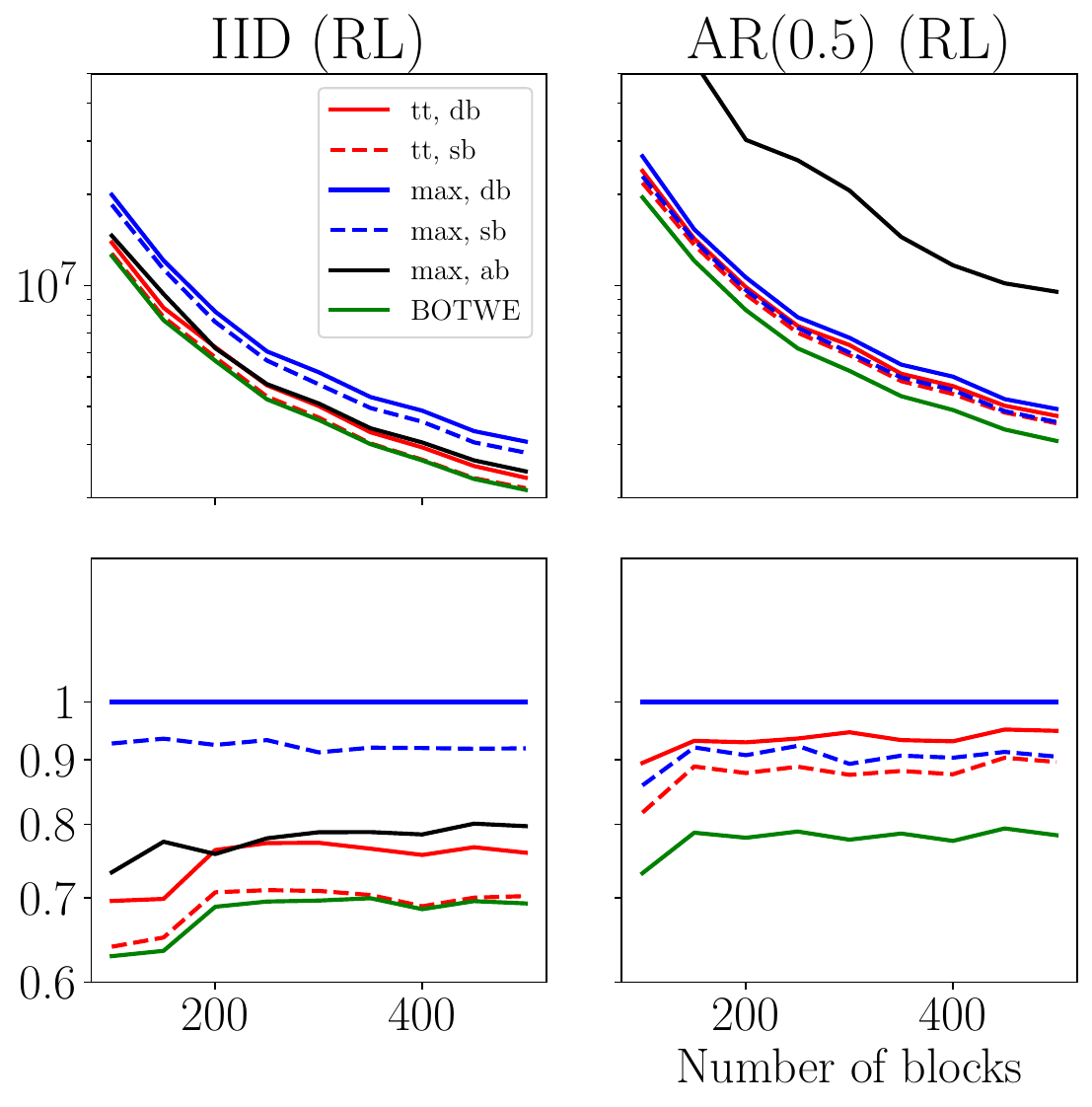}
    \caption{Estimation of the shape parameter $\alpha_0$ (left half) and of $\RL(100,100)$ (right half) for fixed block size
$r=100$. Top row: mean squared error. Bottom row: relative mean squared error with respect to the disjoint block
maxima estimator, $\mathrm{MSE}(\,\hat\cdot\,)/\mathrm{MSE}(\,\hat \cdot\,^{(\dbl)}_{\max})$.}
    \label{fig:fixed_bs_shape_and_rl}
\end{figure}

We next consider the estimation of the (100,100)-return level, with the respective simulation results summarized in Figure~\ref{fig:fixed_bs_shape_and_rl} (right half). The findings for the ABM estimator are similar as for shape estimation: it performs slightly worse than the top-two estimators in the iid case, but by far worst in the time series case, despite the explicit use of the true extremal index.

Interestingly and in contrast to the shape estimation, the top-two estimators do not clearly outperform the sliding max-only estimator in the serially dependent case. In view of their better performance for shape estimation, this must be due to a worse performance for scale estimation, which can in fact be explained by the theoretical findings in Remark~\ref{rem:comparing_covs_general}. This observation motivates the botw-estimator from \eqref{eq:botw}, where we use the top-two approach for shape estimation and the max-only approach for scale estimation. Perhaps unsurprisingly, the botw-estimator outperforms all other estimators in most scenarios (unless the serial dependence is very strong; see Section~\ref{sec:sim-additional}).

\subsection{Fixed total sample size}
\label{subsec:simulation-fixed-sample-size}

In this section, we consider the setting where the sample size $n$ is fixed and the target parameter does not depend on the block size $r$. In this case, $r$ can be viewed as a tuning parameter, to be selected so as to balance estimation bias and variance. To study this choice in a finite-sample context, we focus on the estimation of the shape parameter $\alpha_0$.

For simplicity, we restrict attention to $n=10\,0000$, and consider block sizes $r$ ranging from $r=5$ to $r=100$. 

The results are summarized in Figure~\ref{fig:fixed_n}. We again observe that the sliding blocks versions outperform their disjoint blocks counterparts, in particular for larger block sizes. The max-only estimators are mostly better than their top-two counterparts for smaller block sizes, and vice versa for larger block sizes. 
No estimator is universally best for all block sizes.
The minimum of the respective curves tends to be attained at smaller values of $r$ for the max-only estimators than for the top-two estimators. The overall minimal value is attained by the sliding top-two estimator (iid case) or by the all block maxima estimator (AR case). This latter observation is not universal, however, as illustrated by additional simulation results in Section~\ref{sec:sim-additional} where the ABM estimator is inferior in a different time series model. In particular, the apparent superiority of the ABM estimator in Figure~\ref{fig:fixed_n} can be explained by the fact that its bias crosses zero at relatively small block sizes, thereby compensating for its larger variance; see Figure~\ref{fig:fixed_n_decomp} for further details.

\begin{figure}[t]
    \centering
    \includegraphics[width = 0.85 \textwidth]{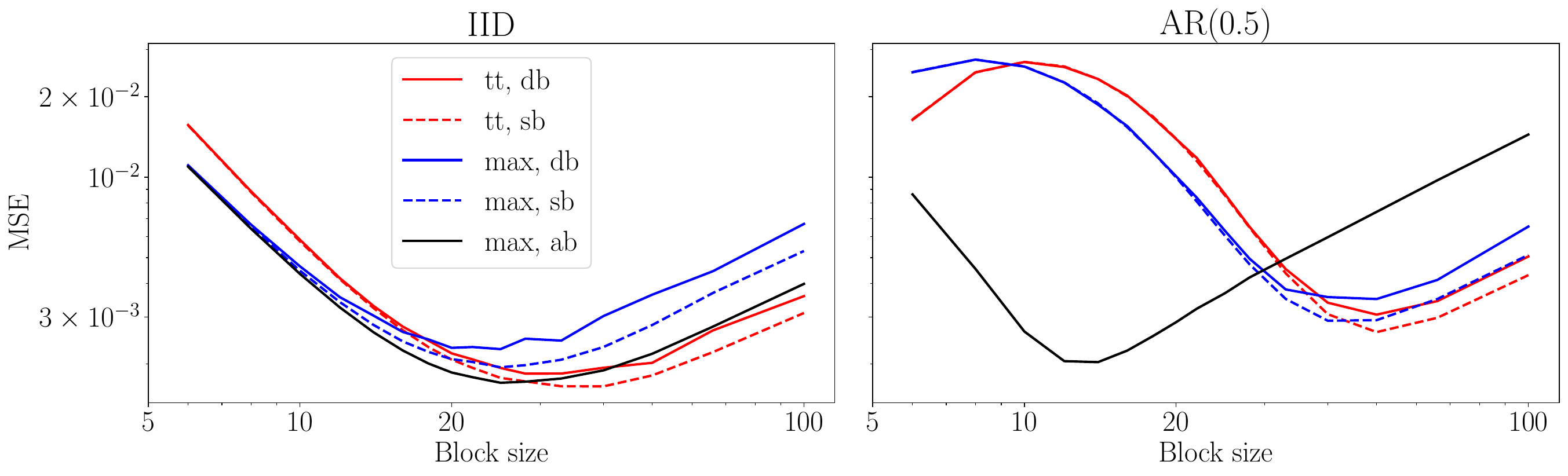}
    \caption{Estimation of $\alpha_0$ for fixed $n =10\,000$.}
    \label{fig:fixed_n}
\end{figure}

\subsection{Bootstrap approximations for the top-two estimator}
\label{subsec:simulation-bootstrap}

In practical applications, an estimator must typically be provided with an estimate of the uncertainty, for instance in the form of a confidence interval. In principle, the bootstrap offers a universal solution. As recently shown by \cite{bucher2025bootstrapping}, bootstrapping estimators based on disjoint block maxima is straightforward: one may just resample with replacement from the disjoint blocks. The situation is more complicated for sliding block maxima, where the simple disjoint blocks solution is inconsistent but where a certain `circular block bootstrap' can be shown to be consistent \cite{bucher2025bootstrapping}. In this section, we apply that circular block bootstrap to our sliding top-two estimators and provide some indication of its validity. 
Unfortunately, a mathematical proof of its validity is beyond the scope of this paper and must be postponed to future research.

We only present results for the AR(0.5)-$\Pareto$-model with $\alpha=3$ and with $k=r=100$. Specifically, we proceed as follows: we first assess the shape estimators' error distribution, i.e., the distribution of $\hat \alpha_{\TT}^{(\sbl)}-3$, based on $3\,000$ simulation runs and visualize it empirically using histograms (see Figure~\ref{fig:simu_bootstrap}). Then, for 100 runs, we employ the circular block bootstrap approach to assess the bootstrap error distribution, i.e., the distribution of $\hat \alpha_{\TT}^{(\sbl), *}- \hat \alpha_{\TT}^{(\sbl)}$, based on $500$ bootstrap estimates $\hat \alpha_{\TT}^{(\sbl), *}$ for each run. We also visualize that distribution using histograms. We repeat the same for return level estimation with $r=T=100$ and the botw-estimator, which was found to be best among all competitors in Section~\ref{sec:sim-fixed-r}.

The results in Figure \ref{fig:simu_bootstrap} provide empirical evidence that the bootstrap approach works as intended: the histograms of the estimators' error distribution closely resemble the histograms of the bootstrap estimation error, both for shape and for return level estimation. Overall, we consider these results to be sufficiently convincing to also use the circular block bootstrap in the following case study.

\begin{figure}[t]
    \centering
    \includegraphics[width = 0.85 \textwidth]{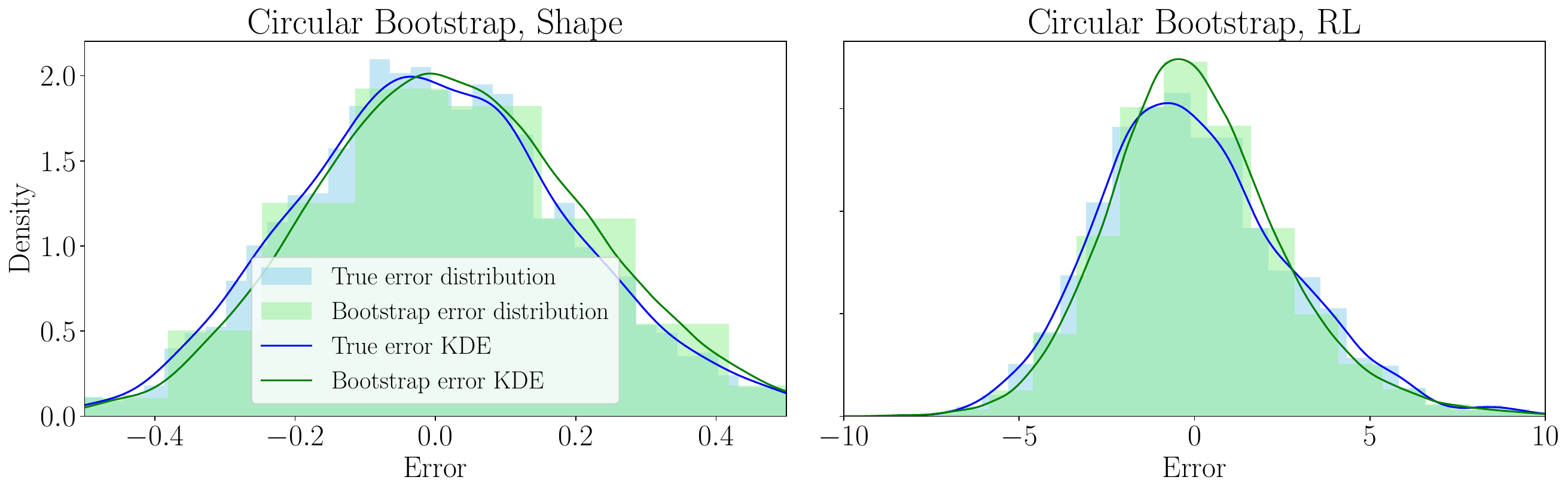}
    \caption{Histograms of estimation error (blue) and (circular block) bootstrap estimation errors (green) together with associated kernel density estimates. Left: shape estimation. Right: RL(100,100)-estimation.}
    \label{fig:simu_bootstrap}
\end{figure}

\section{Case Study}
\label{sec:case-study}

We provide a small case study to illustrate the usefulness of the new methods in a typical practical application from climate science. Our starting point is the recent extreme precipitation event that caused the heavy flooding in Ahrtal in June 2021; see \cite{tradowsky2023attribution} for a respective extreme event attribution study. Among the 2000 DWD weather stations in Germany, the largest daily cumulative precipitation amount in June 2021 was observed on June~14 in Köln-Stammheim (154mm). We hence choose to work with the respective univariate time series of daily precipitation at that station, for which the DWD provides data since 1945. The respective annual top two observations are illustrated in Figure~\ref{fig:sample}. 

Fitting the Fr\'echet distribution to the annual maxima using the botw-method, we obtain estimates of $\hat \alpha=3.3093$ and $\hat \sigma_{365}=27.9754$, which results in an estimate for the 100-year return level $\RL(365, 100)$ of about 112mm. Respective results for the max-only and the top-two estimators can be found in Table~\ref{tab:return_level_bootstrap}, alongside with $95\%$-basic bootstrap confidence intervals \cite{DavHin97} based on the circular block bootstrap from Section~\ref{subsec:simulation-bootstrap}. It can be seen that all five estimators yield similar point estimates, but that the confidence interval for the botw-estimator is the smallest among the five methods under consideration.  The results for the botw estimator are further illustrated in Figure~\ref{fig:sample}, where we depict the function that maps $T$ to the respective estimated $T$-year return level. Note that the preimage of that function at a given threshold corresponds to the return period of observing an event larger than that threshold. For the Ahrtal-event, the estimated return period is 280. The confidence region in Figure~\ref{fig:sample} is defined as $C=\{(T,c):T\in(0,\infty),c\in C(T)\}$ with 
\begin{align*}
    C(T)=[2\widehat\RL_{\mathrm{botw}}(T,100)-\widehat\RL{\,\!}^*_{\mathrm{botw}}(T,100)_{0.975}, 2\widehat\RL_{\mathrm{botw}}(T,100)-\widehat\RL{\,\!}^*_{\mathrm{botw}}(T,100)_{0.025}],
\end{align*}
where $\widehat\RL_{\mathrm{botw}}^*(T,100)_{q}$ denotes the empirical $q$-quantile of the bootstrap sample.

\begin{table}[t]
\label{tab:return_level_bootstrap}
\begin{center}
\begin{tabular}{lrrrrr}
\toprule
 & Return Level & Lower CI & Upper CI & CI Width & Relative CI Width \\
\midrule
max,dbm & 119.93 & 77.08 & 151.37 & 74.29 & 1.00 \\
max,sbm & 116.73 & 86.77 & 147.07 & 60.30 & 0.81 \\
tt,dbm & 113.93 & 84.53 & 134.94 & 50.41 & 0.68 \\
tt,sbm & 113.35 & 88.90 & 132.78 & 43.88 & 0.59 \\
botwe & 112.32 & 88.06 & 130.38 & 42.32 & 0.57 \\
\bottomrule
\end{tabular}
\caption{Estimated 100-year return level at Köln-Stammheim with $95\%$-basic bootstrap confidence intervals.}
\end{center}
\end{table}

\begin{figure}[t]
    \centering
    \includegraphics[width=0.9\linewidth]{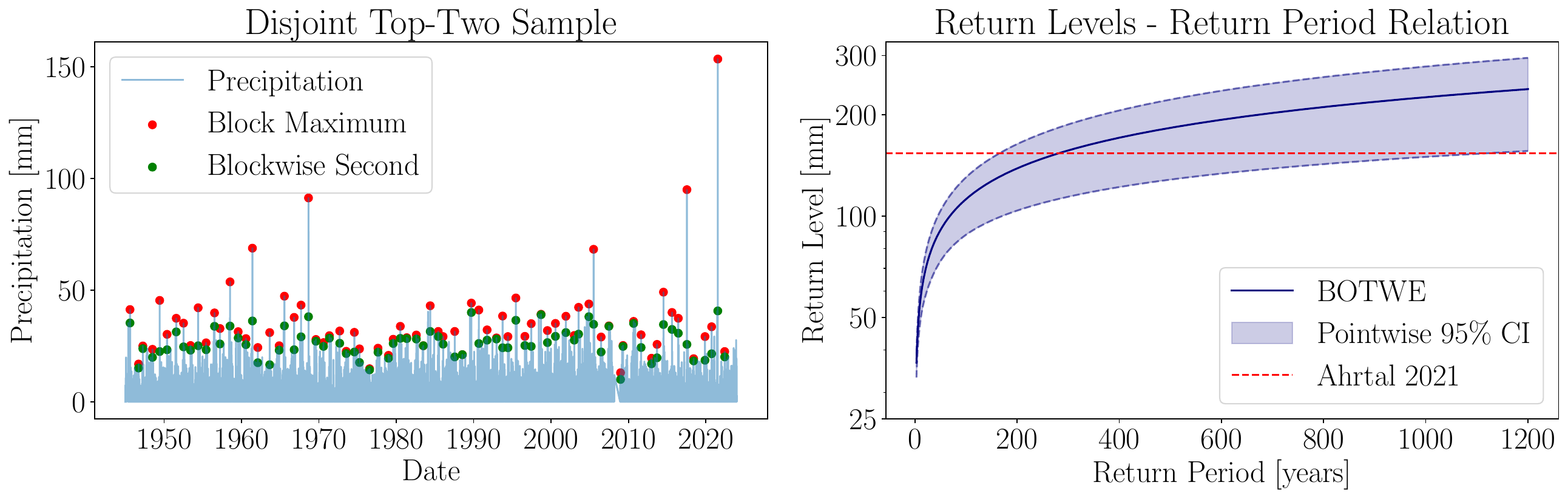}
    \caption{Left: annual top-two sample of daily precipitation amounts at Köln-Stammheim. Right: The estimated mapping $T\mapsto\widehat\RL_{\mathrm{botw}}(365, T)$ together with its bootstrap confidence region.}
    \label{fig:sample}
\end{figure}

\section{Conclusion}
\label{sec:conclusion}
Asymptotic theory for fitting models to a block maximum distribution has concentrated so far on the sample of block-wise maxima. This paper exploits existing mathematical theory for the two largest order statistics of a heavy-tailed stationary time series to develop a pseudo-maximum likelihood estimator based on the block-wise top-two order statistics. It is found that this approach typically outperforms existing methods based on just the block-wise maxima, both in terms of mathematical theory and in finite sample simulation experiments. Furthermore, it is demonstrated that taking into account overlapping `sliding' blocks leads to even more efficient estimators. As the estimator's asymptotic variance is unknown in practice, the adaptation of a circular bootstrap approach is proposed to access estimation uncertainty.

Several topics for future research emerge from the results of this work:
\begin{compactenum}[(1)]
\item It would be interesting to generalize the presented findings to the sample of block-wise top-$m$ order statistics with $m \ge 3$, and to provide a data-adaptive criterion for the choice of $m$.

\item The developed theory is so far limited to the two-parametric $\Frechet$ case. For more flexibility, it would be worthwhile to additionally include a location parameter $\mu$, or to even  fit the three-parametric GEV distribution to allow for non-positive shape parameters.  A particular challenge would then be to derive a suitable bias correction.

\item Asymptotic theory for the circular block bootstrap approach has only been studied for block maxima so far \cite{BucSta24}. The generalization of their results to high order statistics would mathematically legitimize its use in the present work.
\end{compactenum}

\subsection*{Acknowledgements}

The authors are grateful to two unknown referees and an associate
editor for their constructive comments that helped to improve the presentation substantially.
\subsection*{Funding}

This work has been supported by the integrated project ``Climate Change and Extreme Events  -- ClimXtreme Module B Statistics Phase II'' (project B3.3, grant number 01LP2323L) funded by the German Federal Ministry of Education and Research (BMFTR).
Erik Haufs is grateful for support by the Studienstiftung des deutschen Volkes.
This work used resources of the Deutsches Klimarechenzentrum (DKRZ) granted by its Scientific Steering Committee (WLA) under project ID bb1152.

\appendix

\section{Proofs for Section \ref{sec:estimation-general}}
\label{sec:proofs-estimation-general}

\begin{proof}[Proof of Lemma~\ref{lemm:ex_uni}]
To obtain the maximum of the log-likelihood, one needs the root of the score functions
\begin{align*}
s_\alpha(\alpha,\sigma|\z)
&:=
\pderiv{\ell(\alpha,\sigma|\z)}{\alpha}= \frac{2k}{\alpha} +2k\log \sigma -\sum_{i=1}^k \Big\{ \log (x_iy_i)-\sigma^\alpha y_i^{-\alpha}\log\frac{y_i}{\sigma} \Big\} \\
s_\sigma(\alpha,\sigma|\z)
&:=
\pderiv{\ell(\alpha,\sigma|\z)}{\sigma}= \frac{2k\alpha }{\sigma}-\alpha\sum_{i=1}^k y_i^{-\alpha}\sigma^{\alpha-1}
= \frac{k\alpha}{\sigma} \big( 2 - \sigma^\alpha M_{-\alpha}^{-\alpha} (\y)\big).
\end{align*}
For fixed $\alpha$, the function $\sigma \mapsto s_\sigma(\alpha,\sigma|\z)$ changes its sign exactly once at its zero $\hat\sigma(\alpha| \z) =  2^{1/\alpha}M_{-\alpha}(\y)$. As a consequence, $ \sigma \mapsto \ell(\alpha, \sigma|\z)$ is maximized at $\sigma=\hat\sigma(\alpha| \z)$. It is therefore sufficient to maximize $\alpha \mapsto \ell_\alpha(\alpha,\hat\sigma(\alpha| \z)|\z)$ with respect to $\alpha$. We find that
\begin{align*}
    \pderiv{\ell(\alpha,\hat\sigma(\alpha| \z)|\z)}{\alpha} 
    &= \pderiv{\ell(\alpha,\sigma|\z)}{\alpha}\Big|_{\sigma=\hat\sigma(\alpha|\z)}+ \pderiv{\ell(\alpha,\sigma|\z)}{\sigma}\Big|_{\sigma=\hat\sigma(\alpha|\z)}\cdot \pderiv{\hat\sigma(\alpha|\z)}{\alpha}.
\end{align*}
The second summand evaluates to 0 by definition of $\hat\sigma(\alpha|\z)$, 
whence, recalling the definition of $\Psi_k$ from \eqref{eq:def_psik},
\[
\pderiv{\ell(\alpha,\hat\sigma(\alpha| \z)|\z)}{\alpha} 
=
s_\alpha(\alpha,\hat\sigma(\alpha|\z)|\z) = k \Psi_k(\alpha|\z),
\]
where the last equation follows from a straightforward calculation.
Differentiating once more gives
\begin{align}\label{eq:llderiv}
    \pderiv{\!\!^2 \ell(\alpha,\hat\sigma(\alpha|\z)|\z)}{\alpha}=-\frac{2k}{\alpha^{2}}-2 M_{-\alpha}^{2\alpha} (\y)\bigg\{M_{-\alpha}^{-\alpha} (\y)\sum_{i=1}^k y_i^{-\alpha}\log^2 y_i-\Big(\sum_{i=1}^k y_i^{-\alpha}\log y_i\Big)^2\bigg\}.
\end{align}
The term in curly brackets is non-negative by the Cauchy-Schwarz inequality, such that 
\begin{align*}
    \pderiv{\!\!^2 \ell(\alpha,\hat\sigma(\alpha|\z)|\z)}{\alpha}\leq -2k\alpha^{-2}<0,
\end{align*}
whence $\alpha \mapsto \Psi_k(\alpha|\z)$ is strictly decreasing.
Discussing the cases $\alpha\to 0$ and $\alpha\to\infty$ in analogy to \cite{Bücher2018-disjoint} shows this function has a unique zero, which then is the global maximum of  $\alpha \mapsto \ell_\alpha(\alpha,\hat\sigma(\alpha| \z)|\z)$. This allows to conclude.
\end{proof}

\begin{proof}[Proof of Lemma~\ref{lem:digamma2}]
Using Lemma \ref{lem:mom}, we immediately get the first claim. Standard curve sketching shows that $\Pi_{\rho_0}$ is a continuous decreasing bijection from $(0,\infty)$ to $\R$ that satisfies $\Pi_{\rho_0}(1) = - \rho_0(1-\rho_0) / \{2(1+\rho_0)\} \le 0$ by a straightforward calculation. 
This expression is strictly smaller than 0 iff $\rho_0\notin\{0,1\}$. 
As a consequence, $\varpi_{\rho_0}=1$ if and only if $\rho_0\in\{0,1\}$, which in turn is equivalent to $\rho \in \{\rho_\indi,0\}$ by the properties of $\rho$.

Finally, regarding the claim about the smoothness of $\rho_0 \mapsto \varpi_{\omega_0}$,
consider the function $F(\rho_0, y) := \Pi_{\rho_0}(y)$, defined on $[0,1] \times (0,\infty)$. Clearly, $F$ is continuously differentiable on $(0,1) \times (0, \infty)$ with $F(\rho_0, \varpi_{\rho_0})=0$ for all $\rho_0 \in (0,1)$. Since $\partial_{y} F(\rho_0, y) < 0$ for all $\rho_0 \in (0,1)$, the implicit function theorem implies that $\rho_0 \mapsto \varpi_{\rho_0}$ is continuously differentiable on $(0,1)$ with derivative $-\partial_{\rho_0} F(\rho_0, \varpi_{\rho_0}) /  \partial_{y} F(\rho_0, \varpi_{\rho_0})$, which can be shown to be bounded; see Figure~\ref{fig:varpi}.

Suppose $\rho_0 \mapsto \varpi_{\rho_0}$ was not continuous at $0$. Then there exists a sequence of positive numbers $a_n$ converging to zero such that $\liminf_{n\to\infty} \varpi_{a_n} < \varpi_0 = 1$. In particular, for some $\eps\in(0,1)$, we have $\varpi_{a_{n}(k)}<1-\eps$ along a subsequence $a_n(k)$, for all $k\in \N$.
Hence, by monotonicity of $\Pi_{a_{n}(k)}$ and continuity of $\rho_0 \mapsto \Pi_{\rho_0}(1-\eps)$,
\[
0 = \Pi_{a_n(k)}(\varpi_{a_n(k)}) > \Pi_{a_n(k)}(1-\eps) \to \Pi_0(1-\eps)>0 \qquad(k \to\infty),
\]
which is a contradiction. A similar argument shows continuity at 1. Finally, since the derivative of $\rho_0 \mapsto \varpi_{\rho_0}$ was found to be bounded on $(0,1)$, the function must be Lipschitz continuous on $[0,1]$ by the mean-value theorem.
\end{proof}

\begin{proof}[Proof of Theorem \ref{thm:consist}] 
Define a random function $\Psi_n$ on $(0,\infty)$ by 
\begin{align}\label{eq:psin}
    \Psi_n(\alpha)=\Psi_{k_n}(\alpha|\bm Z_n)=\Psi_{k_n}(\alpha|\bm Z_n/\sigma_n),
\end{align}
with $\Psi_k(\cdot|\cdot)$ as in \eqref{eq:def_psik} being scale-invariant in the second component. Condition \ref{cond:1} implies that, for each $\alpha\in (\alpha_-, \alpha_+)$ and as $n\to\infty$, $\Psi_n(\alpha)\rightsquigarrow \Psi_\infty^{(\rho, \alpha_0)}(\alpha)$
with $\Psi_\infty^{(\rho, \alpha_0)}$ from \eqref{eq:psi}.
By Lemma \ref{lem:digamma2}, the limit $\Psi_\infty^{(\rho, \alpha_0)}(\alpha)$ is positive, zero or negative according to whether $\alpha$ is smaller, equal to, or greater than $\alpha_1$. Moreover, Lemma \ref{lemm:ex_uni} and its proof implies that the function $\Psi_n$ is decreasing with $\Psi_n(\hat\alpha_n)=0$.

Fix $\delta>0$ such that $\alpha_-<\alpha_1-\delta<\alpha_1+\delta<\alpha_+$. Since $\Psi_n(\alpha_1-\delta)\rightsquigarrow \Psi(\alpha_1-\delta)>0$ as $n\to\infty$, we find that
\begin{align*}
    \Prob\big(\hat\alpha_n\leq \alpha_1-\delta\big)\leq \Prob\big(\Psi_n(\alpha_1-\delta)\leq 0\big) = o(1), \qquad n\to\infty. 
\end{align*}
Similarly, $\Prob\big(\hat\alpha_n\geq \alpha_1+\delta\big)=o(1)$ as $n\to\infty$. Since $\delta$ was arbitrary, we can conclude that $\hat\alpha_n\rightsquigarrow\alpha_1$ as $n\to\infty$. 

It remains to show weak convergence of $\hat\sigma_n/\sigma_n$. Condition \ref{cond:1} implies that, for each $\alpha\in (\alpha_-,\alpha_+)$ and as $n\to\infty$,
\begin{align*}
    \frac{1}{\sigma_n}\bigg(\frac{1}{k_n}\sum_{i=1}^{k_n}Y_{n,i}^{-\alpha}\bigg)^{-1/\alpha}&=\bigg(\frac{1}{k_n}\sum_{i=1}^{k_n}(Y_{n,i}/\sigma_n)^{-\alpha}\bigg)^{-1/\alpha}\\
    &\rightsquigarrow \bigg(\int_0^\infty x^{-\alpha}\diff H^{(2)}_{\rho,\alpha_0,1}(x)\bigg)^{-1/\alpha}=\Upsilon_{\rho_0}\big(\alpha/\alpha_0\big)^{-1/\alpha}
\end{align*}
where we used Lemma \ref{lem:mom} for the last identity. Both the left-hand and right-hand sides are continuous, non-increasing functions of $\alpha$. Since $\hat\alpha_n\rightsquigarrow\alpha_1$ as $n\to\infty$, a standard argument then yields, as $n\to\infty$,
\begin{align*}
    \frac{\hat\sigma_n}{\sigma_n}=2^{1/\hat\alpha_n}\frac{1}{\sigma_n}\bigg(\frac{1}{k_n}\sum_{i=1}^{k_n}Y_{n,i}^{-\hat\alpha_n}\bigg)^{-1/\hat\alpha_n}\rightsquigarrow 2^{1/\alpha_1}\cdot\Upsilon_{\rho_0}\big(\alpha_1/\alpha_0\big)^{-1/\alpha_1}
\end{align*}

Finally, the last assertions about $\rho \in \{\rho_\indi, \rho_\pd\}$ are immediate consequences of Lemma~\ref{lem:digamma2} and straightforward calculations.
\end{proof}

The proof of Theorem \ref{thm:asymptotic} is decomposed into a sequence of lemmas. Recall $\Psi_n$ and $\Psi_\infty^{(\rho, \alpha_0)}$ in Equations \eqref{eq:psin} and \eqref{eq:psi}, respectively, and define $\Dot\Psi_n(\alpha)=\partial_\alpha\Psi_n(\alpha)$ and $\Dot\Psi_\infty^{(\rho, \alpha_0)}(\alpha)=\partial_\alpha\Psi_\infty^{(\rho, \alpha_0)}(\alpha)$. 
For $f:(0,\infty)^2 \to \R$, write
\begin{align*}
    \mathbb{P}_n f:= \frac1{k_n} \sum_{i=1}^{k_n}f\Big(\frac{X_{n,i}}{\sigma_n}, \frac{Y_{n,i}}{\sigma_n}\Big),
\end{align*}
and note that
\begin{align*}
    \Dot\Psi_n(\alpha)
    &=
    -\frac{2}{\alpha^2}-2\frac{\mathbb{P}_n[(x,y) \mapsto y^{-\alpha}\log^2 y]\mathbb{P}_n[(x,y) \mapsto y^{-\alpha}]-\{ \mathbb{P}_n[(x,y) \mapsto y^{-\alpha}\log y] \}^2}{ \{ \mathbb{P}_n[(x,y) \mapsto y^{-\alpha}] \} ^2}.
\end{align*} 
by \eqref{eq:llderiv}.
It turns out that the asymptotic distribution of $v_n\big(\hat\alpha_n-\alpha_1\big)$ can be derived from the asymptotic behavior of $\Dot\Psi_n$
and $v_n\Psi_n$, which will be discussed in the next two lemmas, respectively.

\begin{lemma}[Slope]\label{lem:slope}
Suppose that the conditions of Theorem \ref{thm:asymptotic} are met. If $\tilde \alpha_n$ is a random sequence in $(0,\infty)$ such that $\tilde \alpha_n\rightsquigarrow \alpha_1$ as $n\to\infty$, then
\begin{align*}
    \Dot\Psi_n(\tilde\alpha_n)
    \rightsquigarrow 
    \Dot \Psi_\infty^{(\rho, \alpha_0)}(\alpha_1) = -\frac{2}{\alpha_1^2}-2\frac{\Upsilon_{\rho_0}''(\varpi_{\rho_0})\Upsilon_{\rho_0}(\varpi_{\rho_0})-\Upsilon_{\rho_0}'(\varpi_{\rho_0})^2}{\alpha_0^2 \Upsilon_{\rho_0}(\varpi_{\rho_0})^2}
\end{align*}
as $n\to\infty$, where $\Upsilon_{\rho_0}$ is defined in \eqref{eq:upsilon_NEW}.
\end{lemma}

\begin{proof}
The claimed equality in the limit follows directly from Lemma~\ref{lem:digamma2}, whence we only need to show the weak convergence.
For $\alpha\in(0,\infty)$ and $m\in\{0,1,2\}$ define 
\begin{align*}
    f_{m,\alpha}(x,y) :=y^{-\alpha}(\log y)^m, \qquad (x,y) \in(0,\infty)^2.
\end{align*}
It can be shown analogously to Lemma A.2 in \cite{Bücher2018-disjoint} that, for $m\in\{0,1,2\}$ and some $\varepsilon>0$, 
\begin{align*}
    \sup_{\alpha:|\alpha-\alpha_0|<\varepsilon}\bigg|\mathbb{P}_n\big[f_{m,\alpha}\big]-\int_{(0,\infty)^2} f_{m,\alpha}(x,y)\diff H_{\rho,\alpha_0,1}(x,y)\bigg|\rightsquigarrow 0, \qquad n\to\infty.
\end{align*}
It then follows from weak convergence of $\tilde \alpha_n$ to $\alpha_1$, Slutsky's lemma and Lemma \ref{lem:mom} that
\begin{align*}
    \Dot\Psi_n(\tilde\alpha_n)\rightsquigarrow -\frac{2}{\alpha_1^2}-2\frac{\Upsilon_{\rho_0}''(\varpi_{\rho_0})\Upsilon_{\rho_0}(\varpi_{\rho_0})-\Upsilon_{\rho_0}'(\varpi_{\rho_0})^2}{\alpha_0^2 \Upsilon_{\rho_0}(\varpi_{\rho_0})^2}
\end{align*}
as $n\to\infty$.
\end{proof}

\begin{lemma}[Asymptotics of $v_n\Psi_n$]\label{lem:psin_asy}
    Assume Condition \ref{cond:2}. Then, as $n\to\infty$,
    \begin{align*}
        v_n\Psi_n(\alpha_1) 
        = 
        \frac{2}{\Upsilon_{\rho_0}(\varpi_{\rho_0})}\G_n f_1 
        +\frac{2\Upsilon_{\rho_0}'(\varpi_{\rho_0})}{\alpha_0\Upsilon_{\rho_0}(\varpi_{\rho_0})^2}\G_n f_2 
        -\G_nf_3-\G_nf_4 + o_{\mathbb P}(1),
    \end{align*}
    with $f_j$ as defined in  \eqref{eq:fct_H}.
    The expression on the right converges weakly to
    \begin{align*}
        W=\frac{2}{\Upsilon_{\rho_0}(\varpi_{\rho_0})} W_1+\frac{2\Upsilon_{\rho_0}'(\varpi_{\rho_0})}{\alpha_0\Upsilon_{\rho_0}(\varpi_{\rho_0})^2}W_2-W_3-W_4
    \end{align*}
\end{lemma}

\begin{proof}
Recall that, from the definition of $\Psi_k$ in \eqref{eq:def_psik}, 
\begin{align*}
    \Psi_n(\alpha_1)
    =
    \Psi_{k_n}(\alpha_1|\zz_n/\sigma_n)
    =
    \frac{2}{\alpha_1}+2\,\frac{\mathbb{P}_nf_1 }{\mathbb{P}_nf_2}-\mathbb{P}_n f_3-\mathbb{P}_n f_4.
\end{align*}
Define $\phi:\R\times(0,\infty)\times\R\times\R\to\R$ by
\begin{align*}
    \phi(\bm w):= \frac{2}{\alpha_1}+2\frac{w_1}{w_2}-w_3-w_4, \qquad \bm w = (w_1,w_2,w_3,w_4),
\end{align*}
which allows to write
$\Psi_n(\alpha_1)=\phi\big(\mathbb{P}_nf_1, \mathbb{P}_nf_2, \mathbb{P}_nf_3, \mathbb{P}_nf_4\big)$.
Next, define
\begin{align*}
    \bm v = (v_1, v_2, v_3, v_4) = \Big(\frac{-\Upsilon_{\rho_0}'(\varpi_{\rho_0})}{\alpha_0},\Upsilon_{\rho_0}(\varpi_{\rho_0}),\frac{\gamma-\rho_0}{\alpha_0},\frac{\gamma}{\alpha_0}\Big) 
\end{align*}
and note that $v_j = \Exp[f_j(X,Y)]$ for $(X,Y) \sim \mathcal W(\rho, \alpha_0,1)$ and $j\in\{1,2,3,4\}$ by Lemma~\ref{lem:mom}.
Further, by the representation of $\Psi_\infty^{(\rho, \alpha_0)}$ in Lemma~\ref{lem:digamma2} and the definition of $\alpha_1$ in \eqref{eq:alpha1_NEW}, we have $\phi(\bm v) = \Psi_\infty^{(\rho, \alpha_0)}(\alpha_1) = 0$.
As a consequence, 
\begin{align*}
    v_n\Psi_n(\alpha_1)=v_n\Big\{ \phi\big(\mathbb{P}_nf_1, \mathbb{P}_nf_2, \mathbb{P}_nf_3, \mathbb{P}_nf_4\big)-\phi(\bm v)\Big\} .
\end{align*}
In view of Condition \ref{cond:2} and the delta method, we hence obtain that
\begin{align*}
    v_n\Psi_n(\alpha_1) 
    = 
    \Dot\phi_1(\bm v)\G_nf_1 +
    \Dot\phi_2(\bm v)\G_n f_2+
    \Dot\phi_3(\bm v)\G_n f_3+
    \Dot\phi_4(\bm v)\G_nf_4 + o_{\mathbb P}(1)
\end{align*}
as $n\to\infty$, where $\Dot\phi_j$ denotes the $j$th first-order partial derivative of $\phi$. Evaluating these partial derivatives at $\bm v$ gives
\begin{align*}
    \Dot\phi_1(\bm v)=\frac{2}{\Upsilon_{\rho_0}(\varpi_{\rho_0})}, \qquad
    \Dot\phi_2(\bm v)=\frac{2\Upsilon_{\rho_0}'(\varpi_{\rho_0})}{\alpha_0\Upsilon_{\rho_0}(\varpi_{\rho_0})^2}, \qquad
    \Dot\phi_3(\bm v) = \Dot\phi_4(\bm v) = -1.
\end{align*}
This implies the assertions.
\end{proof}

\begin{proposition}[Asymptotic expansion for the shape parameter]\label{prop:asy}
Assume that the conditions of Theorem \ref{thm:asymptotic} are met. Then, for $n\to\infty$ and with $W$ as defined in Lemma \ref{lem:psin_asy} and $\Dot\Psi(\alpha_1)$ as in Lemma \ref{lem:slope}, 
\begin{align*}
    v_n(\hat \alpha_n-\alpha_1)
    &=
    -\frac{1}{\Dot \Psi_\infty^{(\rho, \alpha_0)}(\alpha_1)}v_n\Psi_n(\alpha_1) + o_{\mathbb P}(1)
    \rightsquigarrow 
    -\frac{1}{\Dot \Psi_\infty^{(\rho, \alpha_0)}}(\alpha_1)W.
\end{align*}
\end{proposition}

\begin{proof}
The result follows from Lemmas \ref{lem:slope} and \ref{lem:psin_asy} in total analogy to the proof of Proposition A.4 in \cite{Bücher2018-disjoint}.
\end{proof}

\begin{proof}[Proof of Theorem \ref{thm:asymptotic}]
Combining Lemma \ref{lem:psin_asy} and Proposition \ref{prop:asy} yields
\begin{multline} \label{eq:Gn1}
    G_{n1} := 
    v_n(\hat\alpha_n-\alpha_1)
    = 
    -\frac{1}{\Dot \Psi_\infty^{(\rho, \alpha_0)}(\alpha_1)}\Big(\frac{2}{\Upsilon_{\rho_0}(\varpi_{\rho_0})}\G_n f_1+\frac{2\Upsilon_{\rho_0}'(\varpi_{\rho_0})}{\alpha_0\Upsilon_{\rho_0}(\varpi_{\rho_0})^2}\G_n f_2 
    \\
    -\G_n f_3-\G_n f_4 \Big)+o_{\mathbb P}(1)
\end{multline}
as $n\to\infty$. The first row of $M_{\rho_0}(\alpha_0)=(\beta_{jk})_{j=1,2, k=1,2,3,4} \in \R^{2 \times 4}$ is hence given by
\begin{align}
\label{eq:mrho1}
(\beta_{11}, \beta_{12}, \beta_{13}, \beta_{14} ) = \frac{1}{\Dot \Psi_\infty^{(\rho, \alpha_0)}(\alpha_1)}
\Big( 
- \frac{2}{\Upsilon_{\rho_0}(\varpi_{\rho_0})},  
-\frac{2\Upsilon_{\rho_0}'(\varpi_{\rho_0})}{\alpha_0\Upsilon_{\rho_0}(\varpi_{\rho_0})^2},
1, 
1
\Big).
\end{align}

Next, define $Z_n=(\hat\sigma_n/\sigma_n)^{-\hat\alpha_n}$ and $z_0 =\frac{1}{2}\Upsilon_{\rho_0}(\varpi_{\rho_0})=s_1^{-\alpha_1}$.
The mean value theorem then allows to write 
\begin{align*} 
    G_{n2} := 
    v_n\Big( \frac{\hat\sigma_n}{\sigma_n} - s_1\Big)
    &= \nonumber
    v_n\big(Z_n^{-1/\hat\alpha_n}-z_0^{-1/\alpha_1}\big)
    \\&=\nonumber
    v_n\big(Z_n^{-1/\hat\alpha_n}-(z_0^{\hat\alpha_n/\alpha_1})^{-1/\hat \alpha_n}\big)
    \\&=
    v_n\big(Z_n-z_0^{\hat\alpha_n/\alpha_1}\big)(-1/\hat\alpha_n)\tilde{Z}_n^{-1/\hat\alpha_n-1},
\end{align*}
where $\tilde Z_n$ is a convex combination of $Z_n$ and $z_0^{\hat\alpha_n/\alpha_1}$. We will show below that $Z_n=z_0+o_{\mathbb P}(1)$. Hence, since $z_0^{\hat\alpha_n/\alpha_1}=z_0 + o_{\mathbb P}(1)$ by Theorem~\ref{thm:consist}, we also have $\tilde{Z}_n=z_0 + o_{\mathbb P}(1)$. Therefore,
\begin{align} \label{eq:Gn2}
    G_{n2} = - \frac1{\alpha_1} z_0^{-1/\alpha_1-1}
    v_n\big(Z_n-z_0^{\hat\alpha_n/\alpha_1}\big)+o_{\mathbb P}(1).
\end{align}
Next,
\begin{align}
    v_n \big(Z_n - z_0^{\hat\alpha_n/\alpha_1}\big)
    &= \nonumber
    v_n \big(Z_n - z_0\big)+v_n \big(z_0 - z_0^{\hat\alpha_n/\alpha_1}\big) \\
    &= \label{eq:znz0alpha}
    v_n \big(Z_n - z_0\big) - z_0 v_n \big(
    z_0^{\hat\alpha_n/\alpha_1-1} - 1 \big).
\end{align}
We discuss both terms on the right-hand side separately. First, 
by the representation of $\hat\sigma_n$ from Lemma~\ref{lemm:ex_uni}, we have
\begin{align*}
    Z_n = \Big(\frac{\hat\sigma_n}{\sigma_n}\Big)^{-\hat\alpha_n}
    =
    \frac{1}{2}\mathbb{P}_n\big[(x,y) \mapsto y^{-\hat\alpha_n}\big] =: \frac{1}{2}\mathbb{P}_n\big[y^{-\hat\alpha_n}\big].
\end{align*}
We may thus write the first expression on the right-hand side of \eqref{eq:znz0alpha} as
\begin{align}
v_n(Z_n - z_0) 
    &= \nonumber
    \frac{v_n}{2}\Big\{ \mathbb{P}_n\big[y^{-\hat\alpha_n}\big]-\mathbb{P}_n\big[y^{-\alpha_1}\big]\Big \} 
    +\frac{1}{2}v_n\Big\{ \mathbb{P}_n\big[y^{-\alpha_1}\big]-\Upsilon_{\rho_0}(\varpi_{\rho_0})\Big\}
    \\& \label{eq:sn1sn2}
    \equiv \frac12(S_{n1} + S_{n2}).
\end{align}
In view of Lemma~\ref{lem:mom}, we may write $S_{n2} = \G_n[y^{-\alpha_1}] =  \G_nf_2$. Regarding $S_{n1}$, 
by the mean value theorem, there exists a convex combination $\bar\alpha_n$ of $\hat\alpha_n$ and $\alpha_1$ such that
\begin{align*}
    S_{n1} = v_n \big\{ \mathbb{P}_n\big[y^{-\hat\alpha_n}\big]-\mathbb{P}_n\big[y^{-\alpha_1}\big] \big\}
    =-v_n(\hat\alpha_n-\alpha_1)\mathbb{P}_n\big[y^{-\bar\alpha_n}\log y\big].
\end{align*}
Similar to the proof of Lemma \ref{lem:slope}, arguing as in the proof of Lemma A.2 in  \cite{Bücher2018-disjoint}, we have
\begin{align*}
    \mathbb{P}_n\big[y^{-\bar\alpha_n}\log y\big]\rightsquigarrow 
    \int_{(0,\infty)} y^{-\alpha_1} \log y \diff H_{\rho, \alpha_0}^{(2)}(y) = - \frac{\Upsilon_{\rho_0}'(\varpi_{\rho_0})}{\alpha_0},   \qquad n\to\infty,
\end{align*}
where the last equality follows from Lemma~\ref{lem:mom}.
Hence, by the previous two displays, Proposition \ref{prop:asy} and Lemma \ref{lem:psin_asy}, it follows that, as $n\to\infty$, 
\begin{align*}
S_{n1}
&= 
v_n(\hat\alpha_n-\alpha_1)\frac{\Upsilon_{\rho_0}'(\varpi_{\rho_0})}{\alpha_0}+o_{\mathbb P}(1)
\\&= 
\frac{-\Upsilon_{\rho_0}'(\varpi_{\rho_0})}{\alpha_0\Dot \Psi_\infty^{(\rho, \alpha_0)}(\alpha_1)}v_n\Psi_n(\alpha_1)+o_{\mathbb P}(1)
\\&= 
\frac{-\Upsilon_{\rho_0}'(\varpi_{\rho_0})}{\alpha_0\Dot \Psi_\infty^{(\rho, \alpha_0)}(\alpha_1)}\bigg\{ \frac{2}{\Upsilon_{\rho_0}(\varpi_{\rho_0})}\G_nf_1+\frac{2\Upsilon_{\rho_0}'(\varpi_{\rho_0})}{\alpha_0\Upsilon_{\rho_0}(\varpi_{\rho_0})^2}\G_nf_2-\G_nf_3-\G_nf_4\bigg\} 
+o_{\mathbb P}(1).
\\&= 
\frac{-\Upsilon_{\rho_0}'(\varpi_{\rho_0})}{\alpha_0\Dot \Psi_\infty^{(\rho, \alpha_0)}(\alpha_1)}\bigg\{ \frac{1}{z_0}\G_nf_1+\frac{\Upsilon_{\rho_0}'(\varpi_{\rho_0})}{2\alpha_0z_0^2}\G_nf_2-\G_nf_3-\G_nf_4\bigg\} 
+o_{\mathbb P}(1),
\end{align*}
where we used $z_0 =\frac{1}{2}\Upsilon_{\rho_0}(\varpi_{\rho_0})$ at the last equality.
Combining the expansions for $S_{n1}$ and $S_{n2}$ with \eqref{eq:sn1sn2}, we obtain that
\begin{align}
\label{eq:sigma_asy}
v_n(Z_n - z_0) 
& =  \nonumber
\frac{-\Upsilon_{\rho_0}'(\varpi_{\rho_0})}{2\alpha_0\Dot \Psi_\infty^{(\rho, \alpha_0)}(\alpha_1)}
\bigg\{ \frac{1}{z_0} \G_nf_1+\frac{\Upsilon_{\rho_0}'(\varpi_{\rho_0})}{2\alpha_0z_0^2}\G_nf_2 
\\&\hspace{5.2cm}
-\G_nf_3-\G_nf_4\bigg\} +
\frac{1}{2}\G_n f_2 +o_{\mathbb P}(1).
\end{align}
Note that this implies $Z_n=z_0+o_{\mathbb P}(1)$ as required earlier.

Next, regarding the second expression on the right-hand side of \eqref{eq:znz0alpha}, note that the delta method implies that, for suitable random $T_n$, deterministic $\theta$ and continuously differentiable $g$ with $g'(\theta) \ne 0$,
\begin{align*}
    v_n(T_n-\theta) = v_n\frac{g(T_n)-g(\theta)}{g'(\theta)}+o_{\mathbb P}(1).
\end{align*}
Applying this with $g\equiv \log, T_n=z_0^{\hat\alpha_n/\alpha_1-1},\theta=1,g'(1)=1$, we obtain
\begin{align}
\label{eq:sigma_asy2}
    z_0 v_n \big(z_0^{\hat\alpha_n/\alpha_1-1} - 1\big) &=  v_n\log(z_0)\big(\hat\alpha_n/\alpha_1-1\big)+o_{\mathbb P}(1) \nonumber \\
    &=  \frac{z_0 \log(z_0)}{\alpha_1}v_n\big(\hat\alpha_n-\alpha_1\big)+o_{\mathbb P}(1) \nonumber \\
    &= 
    -\frac{z_0 \log(z_0)}{\alpha_1 \Dot \Psi_\infty^{(\rho, \alpha_0)}(\alpha_1)}\Big(\frac{1}{z_0}\G_n f_1+\frac{\Upsilon_{\rho_0}'(\varpi_{\rho_0})}{2\alpha_0z_0^2}\G_n f_2 
    \nonumber \\ &\hspace{5.5cm}
    -\G_n f_3-\G_n f_4 \Big)+o_{\mathbb P}(1),
\end{align}
where we have used \eqref{eq:Gn1} and $z_0 =\frac{1}{2}\Upsilon_{\rho_0}(\varpi_{\rho_0})$ at the last equality.

Overall, combining \eqref{eq:sigma_asy} and \eqref{eq:sigma_asy2} with \eqref{eq:znz0alpha} and then \eqref{eq:Gn2}, we obtain that
\[
G_{n2} = \sum_{k=1}^4 \beta_{2k} \G_n f_k + o_{\mathbb P}(1),
\]
where, recalling $z_0 =\frac{1}{2}\Upsilon_{\rho_0}(\varpi_{\rho_0})$,
\begin{align}
\label{eq:mrho2}
\beta_{21} 
&=\nonumber
\frac{z_0^{-1/\alpha_1}}{\alpha_1 \Dot \Psi_\infty^{(\rho, \alpha_0)}(\alpha_1)} \Big\{ \frac{\Upsilon_{\rho_0}'(\varpi_{\rho_0})}{2\alpha_0 z_0^2}   -  \frac{\log z_0}{\alpha_1 z_0} \Big\},
\\
\beta_{22} 
&= 
\frac{z_0^{-1/\alpha_1}}{\alpha_1 \Dot \Psi_\infty^{(\rho, \alpha_0)}(\alpha_1)} \Big\{ \frac{\Upsilon_{\rho_0}'(\varpi_{\rho_0})^2}{4\alpha_0^2 z_0^3}  - \frac{\log z_0\Upsilon_{\rho_0}'(\varpi_{\rho_0})}{2\alpha_0\alpha_1z_0^2} \Big\} - \frac{z_0^{-1/\alpha_1-1}}{2\alpha_1},
\\
\beta_{23} =\beta_{24} 
&= \nonumber
\frac{z_0^{-1/\alpha_1}}{\alpha_1 \Dot \Psi_\infty^{(\rho, \alpha_0)}(\alpha_1)} 
\Big\{ \frac{\log z_0}{\alpha_1} - \frac{\Upsilon_{\rho_0}'(\varpi_{\rho_0})}{2\alpha_0 z_0} \Big\}.
\end{align}
This proves the claimed expansion in \eqref{eq:thmasymp}, and the weak convergence follows immediately from Condition~\ref{cond:2}. 

If $\rho_0=1$, we have $\rho_0=1, z_0=1, \varpi_{\rho_0}=1$ and $\alpha_1=\alpha_0$. Hence, since $\Upsilon_1(1)=\Gamma(3)=2$, $\Upsilon_1'(1)=\Gamma'(3)=3-2\gamma$ and $\Upsilon_1''(1)=\Gamma''(3)=2-6\gamma+2\gamma^2+\pi^2/3$, we obtain that
\begin{align*}
    \Dot\Psi_{1,\alpha_0}(\alpha_0) 
    &=-\frac{2}{\alpha_0^2}-2\frac{\Upsilon_1''(1)\Upsilon_1(1)-\Upsilon_1'(1)^2}{\alpha_0^2\Upsilon_1(1)^2}
    \\& 
    = -\frac{2}{\alpha_0^2}-\frac{2-6\gamma+2\gamma^2+\pi^2/3 -(9-12\gamma+4\gamma^2)/2}{\alpha_0^2} 
    = \frac{3-2\pi^2}{6\alpha_0^2} , 
\end{align*}
which implies \eqref{eq:malpha} by plugging the previous expressions into \eqref{eq:mrho1} and \eqref{eq:mrho2}.
\end{proof}

\begin{proof}[Proof of Theorem~\ref{thm:bias-corrected-asymp}]
We can prove \eqref{eq:bias-corrected-asymp} coordinate-wise. 
First, since $\alpha_1 = \varpi_{\rho_0} \alpha_0$,
\begin{align*}
v_n(\widetilde \alpha_n-\alpha_0) 
&=
v_n(\hat \alpha_n-\alpha_1) \hat \varpi_n^{-1} + v_n (\hat \varpi_n^{-1} - \varpi_{\rho_0}^{-1}) \alpha_1
\\&=
v_n(\hat \alpha_n-\alpha_1) \varpi_{\rho_0}^{-1} + o_{\mathbb P}(1)
= \varpi_{\rho_0}^{-1} (M_{\rho_0}(\alpha_0))_1 (\G_n f_1, \dots, \G_nf_4)^\top + o_{\mathbb P}(1)
\end{align*}
by Slutsky's Lemma and Theorem \ref{thm:asymptotic}; here, $(M_{\rho_0}(\alpha_0))_1$ denotes the first row of $M_{\rho_0}(\alpha_0)$.

Next, for $(\varrho,\alpha) \in [0,1] \times (0,\infty)$, define $\varphi(\varrho, \alpha) = \{ \Upsilon_{\varrho}(\varpi_{\varrho})/2\}^{1/\alpha}$, and note that $\varphi(\rho_0, \alpha_1) = z_0^{1/\alpha_1}=1/s_1$. Then
    \begin{align} \label{eq:decomp-bc}
    v_n\Big( \frac{\widetilde \sigma_n}{\sigma_n}-1\Big)
    &= \nonumber
    v_n\Big( \frac{\hat \sigma_n}{\sigma_n}  \varphi (\hat \rho_{0,n}, \hat \alpha_n)-1\Big) \\
    &=
   \varphi(\rho_0, \alpha_1) v_n\Big( \frac{\hat \sigma_n}{\sigma_n} - \varphi (\rho_0, \alpha_1)^{-1} \Big) 
   +  
   \frac{\hat \sigma_n}{\sigma_n} v_n( \varphi (\hat \rho_{0,n}, \hat \alpha_n) -  \varphi (\rho_{0}, \alpha_1)).
    \end{align}
By Theorem~\ref{thm:asymptotic}, the first summand on the right can be written as 
\[
z_0^{1/\alpha_1}   (M_{\rho_0}(\alpha_0))_2 (\G_n f_1, \dots, \G_nf_4)^\top + o_{\mathbb P}(1).
\]
For the second summand on the right-hand side of \eqref{eq:decomp-bc}, note that 
$\hat \sigma_n/\sigma_n
= z_0^{-1/\alpha_1} + o_{\mathbb P}(1)$, 
and write 
\[
v_n \{ \varphi( \hat \rho_{0,n}, \hat \alpha_n) -  \varphi (\rho_{0}, \alpha_1) \}
=
v_n\{ \varphi (\hat \rho_{0,n}, \hat \alpha_n) -  \varphi(\hat \rho_{0,n}, \alpha_1) \}
+
v_n\{ \varphi (\hat \rho_{0,n}, \alpha_1) -  \varphi (\rho_{0}, \alpha_1) \}
\]
By Lipschitz continuity of $\varrho \mapsto \varphi(\varrho, \alpha_1)$, the second summand on the right is of the order $O_{\mathbb P}(v_n(\hat \rho_{0,n} - \rho_0))=o_{\mathbb P}(1)$. Regarding the first summand, the mean value theorem allows to write
\[
v_n\{ \varphi (\hat \rho_{0,n}, \hat \alpha_n) -  \varphi(\hat \rho_{0,n}, \alpha_1) \}
=
\partial_{\alpha}  \varphi(\hat \rho_{0,n}, \xi_n) v_n (\hat \alpha_n - \alpha_1)
\]
for some intermediate value $\xi_n$ between $\hat \alpha_n$ and $\alpha_1$.
A straightforward calculation shows that 
    \[
    \partial_\alpha \varphi(\varrho, \alpha) = - \frac1{\alpha^2} \Big( \frac{\Upsilon_{\varrho}(\varpi_\varrho)}{2}\Big)^{1/\alpha} \log\Big( \frac{\Upsilon_{\varrho}(\varpi_\varrho)}{2}\Big).
    \]
    This function is clearly continuous in $(\varrho, \alpha)$, so that $\hat \xi_n = \alpha_1 + o_{\mathbb P}(1)$ (a consequence of the fact that $\xi_n$ is an intermediate value between $\hat \alpha_n$ and $\alpha_1$ and that $\hat \alpha_n=\alpha_1+o_{\mathbb P}(1)$) and $\hat\rho_{0,n}= \rho_0 + o_{\mathbb P}(1)$ (by assumption) together with the continuous mapping theorem implies that 
    \[
    \partial_\alpha \varphi(\hat\rho_{0,n}, \xi_n)  = \partial_\alpha \varphi(\rho_{0}, \alpha_1) + o_{\mathbb P}(1).
    \]
    Consequently,
\begin{align*}
v_n\{ \varphi (\hat \rho_{0,n}, \hat \alpha_n) -  \varphi(\hat \rho_{0,n}, \alpha_1) \}
&=
\big(\partial_{\alpha} \varphi(\rho_0, \alpha_1) + o_{\mathbb P}(1)\big)v_n(\hat \alpha_n - \alpha_1) + o_{\mathbb P}(1)
\\ 
&=
\partial_{\alpha} \varphi(\rho_0, \alpha_1) v_n(\hat \alpha_n - \alpha_1) + o_{\mathbb P}(1)
\\ &=
-\alpha_1^{-2}z_0^{1/{\alpha_1}} \log(z_0) 
(M_{\rho_0}(\alpha_0))_1 (\G_n f_1, \dots, \G_nf_4)^\top + o_{\mathbb P}(1),
\end{align*}
where we used Theorem~\ref{thm:asymptotic} again. Assembling terms, observing that $-\alpha_1^{-2}z_0^{1/{\alpha_1}} \log(z_0) = \alpha_1^{-1}s_1^{-1} \log (s_1)$, yields \eqref{eq:bias-corrected-asymp}. 

If $\rho=\rho_\indi$, we have $\varpi_{1}=1$ and $z_0=1$ and hence the matrix in front of $M_{\rho_0}(\alpha_0)$ in \eqref{eq:malpha-bc} is the identity the matrix.
\end{proof}

\section{Proofs for Section \ref{sec:estimation-blockmaxima}}
\label{sec:proofs-estimation-blockmaxima}

\subsection{Disjoint Blocks: Proof of Theorem \ref{thm:blocks}}

The proof of Theorem \ref{thm:blocks} needs some lemmas as preparation.

\begin{lemma}[Largest two order statistics rarely show ties]\label{lem:no_ties}
Under Conditions \ref{cond:doa} and \ref{cond:alphamixing}, for every $c\in (0,\infty)$, we have
\begin{align*}
    \lim_{n\to\infty} 
    \Prob\big((M_{r_n,1} \vee c,S_{r_n,1} \vee c) = (M_{r_n,3} \vee c,S_{r_n,3} \vee c)\big) 
    = 0. 
\end{align*}
\end{lemma}

\begin{proof}
Since the event in question is contained in the event $\{M_{r_n,1}\vee c=M_{r_n,3} \vee c\}$, the result is an immediate consequence of Lemma~A.5 in \cite{Bücher2018-disjoint}.
\end{proof}

\begin{lemma}[Moment convergence]\label{lem:conv_of_mom_of_block}
Under Conditions \ref{cond:doa} and \ref{cond:mom}, we have, for every $c\in(0,\infty)$
\begin{align*}
&\lim_{r\to\infty} \Exp\!\big[f\big((M_{r}\vee c)/\sigma_{r}\big)\big]=\int_0^\infty f(x)\diff H^{(1)}_{\rho,\alpha_0,1}(x),
\\
&\lim_{r\to\infty} \Exp\!\big[f\big((S_{r}\vee c)/\sigma_{r}\big)\big]=\int_0^\infty f(y)\diff H^{(2)}_{\rho,\alpha_0,1}(y) ,
\end{align*}
for every measurable function $f:(0,\infty)\to \R$ which is continuous almost everywhere and for which there exist $0<\kappa<\nu$ such that $|f(x)|\leq g_{\kappa,\alpha_1}(x)$, where
\begin{align} 
\label{eq:g_kappa}
g_{\kappa,\alpha_1}(x) = \big(x^{-\alpha_1}\indic(x\leq \mathrm e)+\log x\indic(x>\mathrm e)\big)^{2+\nu}.
\end{align}
\end{lemma}

\begin{proof}
Since $c/\sigma_r\to 0$ as $r\to\infty$, the sequence $(M_r\vee c, S_r\vee c)/\sigma_r$ converges weakly to the $\mathcal{W}(\rho,\alpha_0,1)$ distribution in view of Condition \ref{cond:doa}. In particular, $(S_r\vee c)/\sigma_n$ and $(M_r\vee c)/\sigma_r$ converge to the required marginal distributions. The result then follows from Example 2.21 in \cite{vdVaart}, observing that we may replace the constant $1$ by $c$ and both $h_\nu$ and $h_{\nu, \alpha_1}$ by $g_{\nu, \alpha_1}$ in the bounds in \eqref{eq:intfinite} (since $S_r \le M_r$).
\end{proof}

A clipping technique is applied to show that the two largest observations from consecutive blocks are approximately independent. 
For integer $1< \ell < r$, define
\begin{align}
    M_{r,i}^{[\ell]}&=\max\{\xi_t:(i-1)r+1\leq t\leq ir-\ell+1\}\label{eq:M_trunc}\\
    S_{r,i}^{[\ell]}&=\max\big(\{\xi_t:(i-1)r+1\leq t\leq ir-\ell+1\}\setminus\{M_{r,i}^{[\ell]}\}\big)\label{eq:S_trunc}.
\end{align}
Clearly, $M_{r,i}\geq M_{r,i}^{[\ell]}$ and $S_{r,i}\geq S_{r,i}^{[\ell]}$. With the next three lemmas, we show that the probability that the largest two observations over a block of size $r$ are likely attained within the subblock of the first $r-\ell$ observations.

\begin{lemma}[Revisiting Lemma 7.1 from \cite{Bücher2014}]
\label{lem:revis}
Assume Condition~\ref{cond:doa}. Let $F_r$ be the cumulative distribution function of $S_r$. If $\ell_n=o(r_n)$ and $(r_n/\ell_n)\alpha(\ell_n)\to 0$, then, for every $u>0$,
\begin{align*}
    \Prob(F_{r_n}(S_{\ell_n})>u) = O(\ell_n/r_n), \qquad n\to\infty.
\end{align*}
\end{lemma}

\begin{proof}
Throughout, we write $r=r_n$ and $\ell=\ell_n$; all convergences are for $n\to\infty$.
Decompose the block of length $r$ into $\lfloor r/l\rfloor$ successive blocks of length $\ell$, and let $S_{\ell,1}, \dots, S_{\ell,\lfloor r/l\rfloor}$ denote the respective second-largest values in each sub-block.
Of these sub-blocks, only keep those with an odd index. Since the distribution of $S_r$ is continuous by assumption,
we find, for $u\in(0,1)$,
\begin{align*}
    0<u
    =
    \Prob\big(F_r(S_{r})\leq u\big)
    \leq 
    \Prob\Big({\max_{\substack{1\leq i\leq \lfloor r/\ell\rfloor \\\text{$i$ is odd}}}}F_r(S_{\ell,i})\leq u\Big).
\end{align*}
Observing that the odd blocks are separated by a lag $\ell$ we obtain, by induction,
\begin{align*}
    \Big|\Prob\Big({\max_{\substack{1\leq i\leq \lfloor r/\ell\rfloor \\\text{$i$ is odd}}}}F_r(S_{\ell,i})\leq u\Big)
    -
    \prod_{\substack{1\leq i\leq \lfloor r/\ell\rfloor \\\text{$i$ is odd}}} \Prob\big(F_r(S_{\ell,i})\leq u\big)\Big|\leq \frac{r}{\ell}\alpha(\ell) = o(1).
\end{align*}
Since the number of indices $i$ in the product is at least $\lfloor r/\ell\rfloor/2$,  we obtain
\begin{align*}
    \Big\{ 1-\Prob\big(F_r(S_{\ell,1})>u\big)\Big\}^{\lfloor r/\ell\rfloor/2}\geq u+o(1), \qquad n\to\infty.
\end{align*}
But $r/\ell \to\infty$, and thus
\begin{align*}
    \limsup_{n\to\infty}\frac{r}{\ell}\Prob\big(F_r(S_{\ell,1}>u)\big)< \infty,
\end{align*}
as required.
\end{proof}

\begin{lemma}[Short blocks are small]\label{lem:shortsmall}
Assume Condition \ref{cond:doa}. If $\ell_n=o(r_n)$ and if $\alpha(\ell_n)=o(\ell_n/r_n)$ as $n\to\infty$, then, for all $\eps>0$,
\begin{align*}
    \Prob\big(S_{\ell_n}\geq \eps\sigma_{r_n}\big)=O(\ell_n/r_n), \qquad  n\to\infty.
\end{align*}
\end{lemma}

\begin{proof}
Throughout, we write $r=r_n$ and $\ell=\ell_n$; all convergences are for $n\to\infty$.
Fix $\eps>0$ and let $F_r$ be the cumulative distribution function of $S_r$. By assumption and \eqref{eq:w_margCDF2}, we have
\begin{align*}
    \lim_{n\to\infty} F_{r}(\eps\sigma_{r}) = \exp\big(-\eps^{-\alpha_0}\big)\big(1+\rho_0 \eps^{-\alpha_0}\big).
\end{align*}
For sufficiently large $n$, we have
\begin{align*}
    \Prob\big(S_{\ell}\geq \eps\sigma_r\big)
    \leq 
    \Prob\big(F_{r}(S_{\ell})\geq F_{r}(\eps\sigma_r)\big)\leq \Prob\Big(F_{r}(S_{\ell})\geq \exp\big(-\eps^{-\alpha_0}\big)\big(1+\rho_0 \eps^{-\alpha_0}\big)/2\Big).
\end{align*}
Now apply Lemma \ref{lem:revis} for $u= \exp\big(-\eps^{-\alpha_0}\big)\big(1+\rho_0 \eps^{-\alpha_0}\big)/2$ to arrive at the claim.
\end{proof}

\begin{lemma}[Clipping doesn't hurt]\label{lem:clip}
Assume Condition \ref{cond:doa}. If $\ell_n=o(r_n)$ and if $\alpha(\ell_n)=o(\ell_n/r_n)$ as $n\to\infty$, then
\begin{align*}
    \Prob\big( \{ M_{r_n}>M_{r_n-\ell_n} \} \cup \{ S_{r_n}>S_{r_n-\ell_n} \} \big)\to 0, 
    \qquad n\to\infty.
\end{align*}
\end{lemma}

\begin{proof}
Throughout all convergences are for $n\to\infty$.
Since $\Prob\big(M_{r_n}>M_{r_n-\ell_n}\big)=o(1)$ by Lemma A.8 in \cite{Bücher2018-disjoint}, it is sufficient to show that
$\Prob\big(S_{r_n}>S_{r_n-\ell_n}\big) = o(1)$.
For that purpose, we have, by Lemma \ref{lem:shortsmall} and stationarity, for every $\eps>0$,
\begin{align*}
    \Prob\big( S_{r_n}>S_{r_n-\ell_n}\big)\leq \Prob\big(S_{r_n-\ell_n}\leq \eps\sigma_{r_n}\big)+\Prob\big( S_{r_n}>\eps\sigma_{r_n}\big).
\end{align*}
Since $\sigma_{r_n-\ell_n}/\sigma_{r_n}\to1$ as a consequence of Condition \ref{cond:doa} and the fact that $\ell_n = o(r_n)$, the first term converges to $\exp(-\eps^{-\alpha_0} )$ as $n\to\infty$, whereas the second one converges to $0$ by Lemma \ref{lem:shortsmall}. Since $\eps >0$ was arbitrary, the claim follows.
\end{proof}

\begin{proof}[Proof of Theorem \ref{thm:blocks}]
Throughout, we omit the upper index $\dbl$.
The result follows from an application of Theorem~\ref{thm:asymptotic}.
Recall $Z_{n,i}$ from \eqref{eq:trunc_toptwo}. Subsequently, we may fix $c=c_0$ with $c_0$ from Condition~\ref{cond:bias}. Indeed, as a consequence of Condition~\ref{cond:all_diverge}, this redefinition of $c$ does not change the estimator on a sequence of events whose probability converges to one. Hence, the asymptotic distribution does not change either.

Now, Lemma \ref{lem:no_ties} implies that, with probability tending to one, not all $Z_{n,i}$ are equal (and hence $\hat \theta_n$ is well-defined and unique by Lemma~\ref{lemm:ex_uni}); this is \eqref{eq:neglect}. 
It remains to check Condition~\ref{cond:2}, with the weak limit $\bm W$ from \eqref{eq:w} being $\mathcal{N}_4(\bm B,\Sigma)$-distributed.
As in \cite{Bücher2018-disjoint}, proof of Theorem 4.2, the proof is based on Bernstein's big-block-small-block method in combination with the Lindeberg central limit theorem. 

Recall the sequence $\ell_n$ from Condition \ref{cond:alphamixing}.
Define clipped versions of $Z_{n,i}$ from \eqref{eq:trunc_toptwo} by
\[
Z_{n,i}^{[\ell_n]} := \big( M_{r_n,i}^{[\ell_n]}\vee c_0, S_{r_n,i}^{[\ell_n]}\vee c_0 \big)
\]
with $M_{r,i}^{[\ell]}$ and $S_{r,i}^{[\ell]}$ from \eqref{eq:M_trunc} and \eqref{eq:S_trunc}, respectively.  Next, define
\begin{align}
   \mathbb P_n f &= \frac1{k_n} \sum_{i=1}^{k_n} f(Z_{n,i} / \sigma_{r_n}),
   &  
   P_n f &= \Exp\big[ f(Z_{n,i} / \sigma_{r_n}) \big], \label{eq:Pn}\\ \notag
   \mathbb P_n^{[\ell_n]} f &= \frac1{k_n} \sum_{i=1}^{k_n} f(Z_{n,i}^{[\ell_n]} / \sigma_{r_n}),
   &  
   P_n^{[\ell_n]} f &= \Exp\big[ f(Z_{n,i}^{[\ell_n]} / \sigma_{r_n}) \big], 
\end{align}
and write $P=\mathcal {W}(\rho, \alpha_0, 1)$ for the limit distribution of $Z_{n,i} / \sigma_{r_n}$. Define empirical processes
\begin{align}\label{eq:Gn}
    \Gb_n = \sqrt{k_n} (\mathbb P_n - P), \qquad 
    \tilde \Gb_n = \sqrt{k_n} (\mathbb P_n - P_n), \qquad 
    \tilde \Gb_n^{[\ell_n]} = \sqrt{k_n} (\mathbb P_n^{[\ell_n]} - P_n^{[\ell_n]})
\end{align}
and let $B_n = \sqrt{k_n} (P_n - P)$.

We need to check the assumptions of Condition~\ref{cond:2}, and we start by proving that there exist $0<\alpha_-<\alpha_1 < \alpha_+<\infty$ such that \eqref{eq:moments_convergence} from Condition~\ref{eq:F1} is met for any $f\in\mathcal F_2(\alpha_-, \alpha_+)$ from \eqref{eq:F2}. For that purpose, choose $\eta \in (2/\omega,\nu)$ and $0<\alpha_-<\alpha_1<\alpha_+$ (further constraints on $\alpha_+$ will imposed below), and let $f \in \mathcal{F}_2(\alpha_-, \alpha_+)$. We need to show that $\mathbb P_n f = Pf + o_{\mathbb P}(1)$, for $n\to\infty$. Observing that $|f|$ is bounded by a multiple of $g_{0,\alpha_1}$ from \eqref{eq:g_kappa} if $\alpha_+<2\alpha_1$, we obtain from Lemma \ref{lem:conv_of_mom_of_block} that
\begin{align*}
     \Exp\big[ \mathbb P_n f\big] = P_n f \to Pf, \qquad n\to\infty.
\end{align*}
Below we will show that
\begin{align}
\label{eq:tildegn}
    \tilde \Gb_n f = \tilde \Gb_n^{[\ell_n]} f + o_{\mathbb P}(1) = O_{\mathbb P}(1)+o_{\mathbb P}(1) = O_{\mathbb P}(1), \qquad n\to\infty,
\end{align}
which implies 
\begin{align*}
    \mathbb P_n f = k_n^{-1/2} \tilde \Gb_n f + P_n f = Pf + o_{\mathbb P}(1), \qquad  n\to\infty
\end{align*}
as required.

It remains to show the weak convergence in \eqref{eq:w} with $\bm W \sim \mathcal N_4(\bm B, \Sigma)$ as specified in Theorem~\ref{thm:blocks}. For that purpose write $\Gb_n = \tilde \Gb_n + B_n$, and note that $B_n f_j = Bf_j + o(1)$ by Condition \ref{cond:bias}, for $j\in\{1,2,3,4\}$. It hence remains to treat $\tilde \Gb_n f_j$, and for that purpose, we will in fact show that the first equality in \eqref{eq:tildegn} is met for any $f\in \mathcal F_2 := \mathcal F_2(\alpha_-, \alpha_+)$ and that the finite-dimensional distributions of $(\tilde \Gb_n^{[\ell_n]} f)_{f \in\mathcal F_2}$ converge weakly to the finite-dimensional distributions of $(\Gb f)_{f \in\mathcal F_2}$, where $\Gb$ is a $P$-Brownian  bridge; that is, a zero-mean Gaussian process with covariance function
\[
\Cov(\Gb f, \Gb g) 
= 
\Cov_{(X,Y) \sim \mathcal W(\rho, \alpha_0,1)}
\big(f(X,Y),g(X,Y)\big), 
\qquad 
f,g \in \mathcal F_2
\]

We start by showing that the first equality in \eqref{eq:tildegn} holds for any $f\in \mathcal F_2$. Write $\Delta_n =  \tilde \Gb_n - \tilde \Gb_n^{[\ell_n]}$, and note that 
\begin{align*}
    \Exp\big[(\Delta_n f)^2\big]=\Var(\Delta_n f) = \frac{1}{k_n}\Var\Big(\sum_{i=1}^{k_n} \Delta_{n,i}^{[\ell_n]}f\Big),
\end{align*}
where $ \Delta_{n,i}^{[\ell_n]}f = f(Z_{n,i} / \sigma_{r_n})-f(Z_{n,i}^{[\ell_n]} / \sigma_{r_n})$. By stationarity and the Cauchy-Schwarz inequality, we have
\begin{align}
     \Exp\!\big[(\Delta_n f)^2\big]
     &=\notag
     \Var\big(\Delta_{n,1}^{[\ell_n]} f\big)+\frac{2}{k_n}\sum_{h=1}^{k_n-1}(k_n-h)\Cov\Big(\Delta_{n,1}^{[\ell_n]} f, \Delta_{n,1+h}^{[\ell_n]}f\Big)
     \\ &\leq 
     3\Var\big(\Delta_{n,1}^{[\ell_n]} f\big) +2\sum_{h=2}^{k_n-1}\Big|\Cov\Big(\Delta_{n,1}^{[\ell_n]} f, \Delta_{n,1+h}^{[\ell_n]}f\Big)\Big|.
\label{eq:cov_terms}
\end{align}
Since $\ell_n=o(r_n)$ as $n\to\infty$ by Condition~\ref{cond:alphamixing}, we have ${\sigma_{r_n-\ell_n+1}}/{\sigma_{r_n}}\to 1$ as $n\to\infty$ by Condition~\ref{cond:doa}. 
The asymptotic moment bound in Condition \ref{cond:mom} then ensures that we may choose $\delta\in (2/\omega,\nu)$ and $\alpha_+>\alpha_1$, such that, for every $f\in\mathcal{F}_2(\alpha_-,\alpha_+)$, by Lemma \ref{lem:conv_of_mom_of_block}, 
\begin{align} \label{eq:deltan1mom}
    \limsup_{n\to\infty}\Exp\!\Big[\big|\Delta_{n,1}^{[\ell_n]} f\big|^{2+\delta}\Big]<\infty.
\end{align}
Further, on the event that $(M_{r_n,1},S_{r_n,1})=(M_{r_n-\ell_n+1,1}, S_{r_n-\ell_n+1})$, we have $\Delta_{n,1}^{[\ell_n]} f=0$, whence $\Delta_{n,1}^{[\ell_n]} f=o_{\mathbb P}(1)$ by Lemma \ref{lem:clip}. Hence, by \eqref{eq:deltan1mom},
\begin{align*}
    \lim_{n\to\infty}\Exp\!\Big[\big|\Delta_{n,1}^{[\ell_n]} f\big|^{2+\delta}\Big]=0, 
    \qquad 
    f\in \mathcal{F}_2(\alpha_-,\alpha_+).
\end{align*}
Finally, recall Lemma 3.11 in \cite{Dehling2002}: for random variables $\xi$ and $\eta$ and for numbers $p,q\in [1,\infty]$ such that $1/p+1/q<1$,
\begin{align*}
    \big|\Cov(\xi,\eta)\big|\leq 10 \|\xi\|_p\|\eta\|_q\big\{ \alpha(\sigma(\xi),\sigma(\nu))\big\}^{1-1/p-1/q},
\end{align*}
where $\alpha(\mathcal{A}_1,\mathcal{A}_2)$ denotes the strong mixing coefficient between two sigma-fields $\mathcal{A}_1$ and $\mathcal{A}_2$. Using this inequality with $p=q=2+\delta$ for the covariance terms in \eqref{eq:cov_terms} yields
\begin{align*}
    \Exp\!\big[(\Delta_n f)^2\big]\leq 3\big\|\Delta_{n,1}^{[\ell_n]} f\big\|_2^2+20k_n\big\|\Delta_{n,1}^{[\ell_n]} f\big\|_{2+\delta}^2(\alpha(r_n))^{\delta/(2+\delta)}.
\end{align*}
The expression on the right-hand side converges to $0$ by Condition \ref{cond:alphamixing} and \eqref{eq:deltan1mom}, observing that $\omega<2/\delta$. The proof of the first equality in \eqref{eq:tildegn} is hence finished.

It remains to show fidi-convergence of $\tilde \Gb_n^{[\ell_n]}$. By the Cram\'er-Wold device, it suffices to show that $\tilde \Gb_n^{[\ell_n]}g \rightsquigarrow \Gb g$, where $g$ is an arbitrary linear combination of functions $f\in \mathcal{F}_2(\alpha_-, \alpha_+)$. A standard argument involving characteristic functions, using that $k_n\alpha(\ell_n) = o(1)$ as a consequence of Condition \ref{cond:alphamixing}, shows that we may assume that the $Z_{n,i}^{[\ell_n]}$ are independent (see, for instance, the argumentation on the bottom of page 1453 in \cite{Bücher2018-disjoint}). Moreover, by similar (but easier) arguments that lead to the first equality in \eqref{eq:tildegn}, we may then pass back to the process $\tilde \Gb_n$, but with $Z_{n,i}$ independent over $i$. 
Hence, in view of Ljapunov's central limit theorem, it is sufficient to show that
\begin{align}
\label{eq:varconv}
    \Var\big(g(Z_{n,i}/\sigma_{r_n})\big) = P_ng^2-\big( P_ng\big)^2  = \Var(\Gb g) + o(1),
    \qquad
    n \to \infty,
\end{align}
and that Lyapunov's Condition is satisfied: 
\begin{align}
\label{eq:lyapunov}
    \lim_{n\to\infty} \frac{1}{k_n^{1+\delta/2}}
    \sum_{i=1}^{k_n}\Exp\!\Big[\big|g\big(Z_{n,i}/\sigma_{r_n}\big)- P_ng\big|^{2+\delta}\Big] = 0
\end{align}
for some $\delta>0$. First, \eqref{eq:varconv} follows immediately from Lemma \ref{lem:conv_of_mom_of_block}. Next, \eqref{eq:lyapunov} follows from Lemma \ref{lem:conv_of_mom_of_block} as well, observing that $|g|^{2+\delta}$ can be bounded by a multiple of $g_{\nu/2,\alpha_1}$ from \eqref{eq:g_kappa} if $\delta$ and $\alpha_+$ are chosen sufficiently small.
\end{proof}

\subsection{Sliding Blocks: Proof of Theorem \ref{thm:sl_asy}}

For $c\ge 0$ and integers $s,t$ such that $1 \le s \le t \le n$, define
\[
(X_{s:t}, Y_{s:t}) 
=
(X^{(n, c)}_{s:t}, Y^{(n, c)}_{s:t})
=
\Big(\frac{M_{s:t} \vee c}{\sigma_{r_n}}, \frac{S_{s:t} \vee c}{\sigma_{r_n}} \Big).
\]
For $\zeta \in[0,1]$, define
\begin{align*}
F_{n,\zeta, c} (x,y,\tilde x, \tilde y)
=
\Prob\Big( 
    X_{1:r_n}^{(n,c)} \le x,
    Y_{1:r_n}^{(n,c)} \le y, 
    X_{\flo{r_n\zeta}+1: \flo{r_n\zeta} + r_n}^{(n,c)} \le \tilde x,
    Y_{\flo{r_n\zeta}+1: \flo{r_n\zeta} + r_n}^{(n,c)} \le \tilde y 
\Big).
\end{align*}
We are interested in weak convergence of the bivariate margins. For that purpose, define 
\begin{align}
F_{\alpha, \zeta}(x, \tilde x) := \exp\big(-\zeta x^{-\alpha}-(1-\zeta)(x\wedge \tilde x)^{-\alpha}-\zeta \tilde x^{-\alpha}\big),
\end{align}
which appeared in Lemma~5.1 in \cite{Bücher2018-sliding} as the limit of $F_{n,\zeta,c}(x, \infty, \tilde x, \infty)$.

\begin{lemma}[Joint weak convergence of sliding block Top-Two]\label{lem:JoiWeaKon}
Suppose that Condition~\ref{cond:doa} is met and that there exists an integer sequence $(\ell_n)_n$ such that $\ell_n=o(r_n)$ and $\alpha(\ell_n)=o(\ell_n/r_n)$ as $n\to\infty$. Write $\alpha=\alpha_0$ for brevity. Then, for any $\zeta \in [0,1]$ and any $c\ge 0$, 
the limit 
\begin{align}\label{eq:sl_distr}
    K_{\rho,\alpha,\zeta}(x,y,\tilde x,\tilde y):=\lim_{n\to\infty}
    F_{n,\zeta, c}(x, y, \tilde x, \tilde y),
\end{align}
exists for all $(x,y,\tilde x, \tilde y)\in(0,\infty]^4$ such that at least one of $x,y$ and one of $\tilde x, \tilde y$ is infinite. Specifically, we have
\begin{align*}
\mathrm{[a]} & \
    K_{\rho,\alpha,\zeta}(x, \infty, \tilde x, \infty)
    = 
    F_{\alpha, \zeta}(x, \tilde x)
\\
\mathrm{[b]} & \
    K_{\rho,\alpha,\zeta}(\infty, y, \tilde x, \infty)
     =
    \begin{cases}
        F_{\alpha, \zeta}(y, \tilde x)
        \big\{ 1+ \zeta \rho_0 y^{-\alpha} + (1-\zeta) y^{-\alpha}\rho\big((y/\tilde x)^\alpha\big) \big\}
         ,& \tilde x\geq y \\
        F_{\alpha, \zeta}(y, \tilde x)\big(1+ \zeta \rho_0 y^{-\alpha}\big), & y\geq \tilde x
    \end{cases}
\\
\mathrm{[c]} & \
    K_{\rho,\alpha,\zeta}(x, \infty, \infty, \tilde y)
    =
    \begin{cases}
       F_{\alpha, \zeta}(x, \tilde y)
        \big\{ 1+ \zeta \rho_0 \tilde y^{-\alpha} + (1-\zeta) \tilde y^{-\alpha}\rho\big((\tilde y/x)^\alpha \big) \big\}
         ,& x\geq \tilde y \\
        F_{\alpha, \zeta}(x, \tilde y)
        \big(1+ \zeta \rho_0 \tilde y^{-\alpha}\big), & \tilde y\geq x
    \end{cases}
\\
\mathrm{[d]} & \
    K_{\rho,\alpha,\zeta}(\infty, y, \infty, \tilde y)
    =
    F_{\alpha, \zeta}(y, \tilde y)  \cdot \Big\{1+\zeta\rho_0 y^{-\alpha}+\zeta\rho_0\tilde y^{-\alpha}+(1-\zeta)\rho_0(y\wedge \tilde y)^{-\alpha}
    \\&\hspace{6.5cm}
   +\zeta\rho_0 y^{-\alpha}\tilde y^{-\alpha}\Big[\zeta\rho_0+(1-\zeta)\rho\Big(\big( \frac{y \wedge \tilde y}{y\vee \tilde y}\big)^\alpha\Big)\Big]\Big\}.
\end{align*}
\end{lemma}

\begin{proof} 
Throughout the proof, we write $r=r_n$ and $\ell=\ell_n$ for brevity, and all convergences are for $n\to\infty$. Since $c/\sigma_r = o(1)$, it is sufficient to consider the case $c=0$. The upper index $(n,c)=(n,0)$ will be suppressed.

Part [a] is Lemma 5.1 in \cite{Bücher2018-sliding}. Concerning [b], note that
\begin{align}
&\phantom{{}={}} \nonumber
F_{n,\zeta, c}(\infty, y, \tilde x, \infty)
\\ &= \nonumber
\Prob \big( Y_{1:r}\leq y,X_{\floor{r\zeta}+1:\floor{r\zeta}+r}\leq \tilde  x \big)
\\&= \label{eq:5.5a}
\Prob \big( X_{1:r}\leq y, X_{\floor{r\zeta}+1:\floor{r\zeta}+r}\leq \tilde x \big)
+ 
\Prob \big(Y_{1:r}\leq y < X_{1:r}, X_{\floor{r\zeta}+1:\floor{r\zeta}+r}\leq \tilde x \big)
\end{align}
The first probability on the right is equal to $F_{n,\zeta,c}(y, \infty, \tilde x, \infty)$, whose convergence has been treated in [a]. Regarding the second, we have
\begin{align}
    \Prob \big( Y_{1:r}\leq y < X_{1:r}, X_{\floor{r\zeta}+1:\floor{r\zeta}+r}\leq \tilde x \big)
&=\nonumber
    \Prob\big( Y_{1:r}\leq y < X_{1:r}, X_{\floor{r\zeta}+1:r}\leq \tilde x, X_{r+1:r+\floor{r\zeta}}\leq \tilde x \big)
\\&= \label{eq:5.5b}
A_{n,\zeta}(y, \tilde x)
    \cdot 
    \Prob\big(X_{r+1:r+\floor{r\zeta}}\leq \tilde x \big) + o(1),
\end{align}
where
\begin{align} \label{eq:anzeta}
A_{n,\zeta}(y, \tilde x)
\equiv 
\Prob\big( Y_{1:r}\leq y < X_{1:r}, X_{\floor{r\zeta}+1:r}\leq \tilde x \big) 
\end{align}
and
where we used asymptotic independence at the last equality, following the arguments in the proof of Lemma 5.1 in 
\cite{Bücher2018-sliding}. More precisely, we have
\begin{align*}
&\phantom{{}={}}
\Prob\big( Y_{1:r}\leq y < X_{1:r}, X_{\floor{r\zeta}+1:r}\leq \tilde x, X_{r+1:r+\floor{r\zeta}}\leq \tilde x \big)\\
&=
\Prob\big( Y_{1:r-\ell}\leq y < X_{1:r-\ell}, X_{\floor{r\zeta}+1:r-\ell}\leq \tilde x, X_{r+1:r+\floor{r\zeta}}\leq \tilde x \big)+o(1)\\
&=
\Prob\big( Y_{1:r-\ell}\leq y < X_{1:r-\ell}, X_{\floor{r\zeta}+1:r-\ell}\leq \tilde x\big)\Prob\big( X_{r+1:r+\floor{r\zeta}}\leq \tilde x \big)+o(1)\\
&=
\Prob\big( Y_{1:r}\leq y < X_{1:r}, X_{\floor{r\zeta}+1:r}\leq \tilde x\big)\Prob\big( X_{r+1:r+\floor{r\zeta}}\leq \tilde x \big)+o(1)\\
&=
A_{n,\zeta}(y, \tilde x)
    \cdot 
    \Prob\big(X_{r+1:r+\floor{r\zeta}}\leq \tilde x \big) + o(1),
\end{align*}
where we applied Lemma \ref{lem:clip} at the first and third equality, and $\alpha(\ell)=o(1)$ at the second equality.

Now, in \eqref{eq:5.5b}, the second factor on the right-hand side can be written as
\begin{align}
     \Prob\big(X_{r+1:r+\floor{r\zeta}}\leq \tilde x \big) 
     = 
     \Prob \big( X_{1:\floor{\zeta r}}\leq \tilde x \big)
     \label{eq:5.5c}
\end{align}
where we have used stationarity.
It remains to look at $A_{n,\zeta}(y, \tilde x)$, for which we split up the set $\{1,\dots,r\}$ at $\floor{\zeta r}$ to obtain that
\begin{align}
&\phantom{{}={}} \nonumber
A_{n,\zeta}(y, \tilde x)
\\&= \nonumber
    \Prob\big( 
    \big[ 
    X_{1:\flo{r\zeta}}>y, Y_{1:r}\leq y, X_{\floor{r\zeta}+1:r}\leq \tilde x
    \big] 
    \cup
    \big[ 
    X_{\flo{r\zeta}+1:r}>y, Y_{1:r}\leq y, X_{\floor{r\zeta}+1:r}\leq \tilde x
    \big] 
    \big)
\\&=\nonumber
\Prob\big( 
    \big[ 
    X_{1:\flo{r\zeta}}>y, Y_{1:\flo{r\zeta}}\leq y, X_{\floor{r\zeta}+1:r}\leq \tilde x \wedge y
    \big] 
\\&\hspace{5cm}\label{eq:5.5d}
    \cup
    \big[ 
    \tilde x \ge X_{\flo{r\zeta}+1:r}>y, Y_{\flo{r\zeta}+1:r}\leq y, X_{1:\floor{r\zeta}}\leq y
    \big] 
    \big).
\end{align}
Here, at the last equality, we have used 
the following event equalities, which follow from straightforward reflection: 
\begin{align*}
\{X_{1:\flo{r\zeta}}>y, Y_{1:r}\le y\} 
&=
\{X_{1:\flo{r\zeta}}>y, Y_{1:\flo{r\zeta}}\le y, X_{\flo{r\zeta}+1:r} \le y\},\\
\{X_{\flo{r\zeta}+1:r}>y, Y_{1:r}\le y\} 
&=
\{X_{\flo{r\zeta}+1:r}>y, Y_{\flo{r\zeta}+1:r}\le y, X_{1:\flo{r\zeta}} \le y\}.
\end{align*}
We proceed by distinguishing the cases $\tilde x\leq y$ and $\tilde x>y$. 
First, if $\tilde x\leq y$, the second event inside the probability on the right-hand side of \eqref{eq:5.5d} is impossible. Hence,
\begin{align*}
A_{n,\zeta}(y, \tilde x)
&=
\Prob\big( 
    X_{1:\flo{r\zeta}}>y, Y_{1:\flo{r\zeta}}\leq y, X_{\floor{r\zeta}+1:r}\leq \tilde x.
\big)
\end{align*}
We may now use asymptotic independence to obtain that, for $\tilde x\leq y$,
\begin{align}
A_{n,\zeta}(y, \tilde x)
&=\label{eq:5.5e}
\Prob\big(Y_{1:\flo{r\zeta}}\leq y<X_{1:\flo{r\zeta}}\big) 
\Prob\big( X_{\floor{r\zeta}+1:r}\leq \tilde x\big) 
+ o(1).
\end{align}
Next, if $\tilde x>y$, \eqref{eq:5.5d} yields
\begin{align}
A_{n,\zeta}(y, \tilde x)
    &= \nonumber
    \Prob\big( 
    \big[ 
    X_{1:\flo{r\zeta}}>y, Y_{1:\flo{r\zeta}}\leq y, X_{\floor{r\zeta}+1:r}\leq  y
    \big] 
\\&\hspace{2cm}\nonumber
    \cup
    \big[ 
    \tilde x \ge X_{\flo{r\zeta}+1:r}>y, Y_{\flo{r\zeta}+1:r}\leq y, X_{1:\floor{r\zeta}}\leq y
    \big] 
    \big)
\\&= \nonumber
\Prob\big( 
    X_{1:\flo{r\zeta}}>y, Y_{1:\flo{r\zeta}}\leq y, X_{\floor{r\zeta}+1:r}\leq  y
    \big)
\\&\hspace{2cm}\nonumber
    + \Prob
    \big(
    \tilde x \ge X_{\flo{r\zeta}+1:r}>y, Y_{\flo{r\zeta}+1:r}\leq y, X_{1:\floor{r\zeta}}\leq y 
    \big)
\\&= \nonumber
\Prob\big( 
    Y_{1:\flo{r\zeta}}\leq y<X_{1:\flo{r\zeta}}\big) \cdot \Prob\big(X_{\floor{r\zeta}+1:r}\leq  y
    \big)
\\&\hspace{2cm}\label{eq:5.5g}
    + \Prob
    \big(
    Y_{\flo{r\zeta}+1:r}\leq y < X_{\flo{r\zeta}+1:r} \le \tilde x\big) \cdot \Prob\big( X_{1:\floor{r\zeta}}\leq y 
    \big) + o(1),
\end{align}
where we used asymptotic independence at the last equality, and the fact that the two events in question are disjoint at the second to last equality.

Inserting \eqref{eq:5.5c} and \eqref{eq:5.5e} into \eqref{eq:5.5b} and then into \eqref{eq:5.5a}, we obtain, for the case $\tilde x \le y$,
\begin{align}
&\phantom{{}={}}\nonumber
F_{n,\zeta, c}(\infty, y, \tilde x, \infty)
\\&=  \nonumber
\Prob\big( X_{1:r}\leq y,X_{\floor{r\zeta}+1:r+\floor{r\zeta}}\leq \tilde x\big) 
\\& \hspace{1cm}+
\Prob\big( X_{1:\floor{r \zeta}}\leq \tilde x \big) \cdot 
\Prob\big(Y_{1:\flo{r\zeta}}\leq y<X_{1:\flo{r\zeta}}\big) \cdot
\Prob\big( X_{\floor{r \zeta}+1:r}\leq \tilde x\big) 
+o(1). \label{eq:5.5f}
\end{align}
Likewise, using \eqref{eq:5.5g} instead of \eqref{eq:5.5e}, for the case $\tilde x > y$,
\begin{align}
&\phantom{{}={}}\nonumber
F_{n,\zeta, c}(\infty, y, \tilde x, \infty)
\\&=  \nonumber
\Prob\big( X_{1:r}\leq y,X_{\floor{r\zeta}+1:r+\floor{r\zeta}}\leq \tilde x\big) 
\\& \hspace{.6cm}+ \nonumber
\Prob\big( X_{1:\floor{r \zeta}}\leq \tilde x \big) \cdot 
\Prob\big( Y_{1:\flo{r\zeta}}\leq y < X_{1:\flo{r\zeta}}\big) \cdot \Prob\big(X_{\floor{r\zeta}+1:r}\leq  y\big)
\\& \hspace{.6cm}+
\Prob\big( X_{1:\floor{r \zeta}}\leq \tilde x \big) \cdot 
\Prob\big(Y_{\flo{r\zeta}+1:r}\leq y < X_{\flo{r\zeta}+1:r} \le \tilde x\big) \cdot 
\Prob\big( X_{1:\floor{r\zeta}}\leq y \big) 
+o(1). \label{eq:5.5h}
\end{align}

It remains to show convergence of the probabilities on the right-hand side of \eqref{eq:5.5f} and \eqref{eq:5.5h}, which follows from the domain-of-attraction Condition~\ref{cond:doa}. First, note that
$
\lim_{n\to 0} \sigma_{\floor{r \zeta}} / {\sigma_r}  = \zeta^{1/\alpha}
$
for any $\zeta>0$
by regular variation of $(\sigma_r)_r$. As a consequence, by Condition~\ref{cond:doa}, for any $x,y>0$ and as $n\to\infty$,
\begin{align}
\Prob\big(X_{1:\floor{r \zeta}}\leq x, Y_{1:\floor{r \zeta}}\leq y\big) 
&= \nonumber
\Prob\Big(
M_{1:\floor{r \zeta}}\leq \sigma_\floor{r \zeta}\Big(\frac{\sigma_r}{\sigma_\floor{r \zeta}}x\Big),
S_{1:\floor{r \zeta}}\leq \sigma_\floor{r \zeta}\Big(\frac{\sigma_r}{\sigma_\floor{r \zeta}}y\Big)
\Big) \\
&= \label{eq:xyjoint-left}
H\big(\zeta^{-1/\alpha}x,\zeta^{-1/\alpha}y\big)
+ o(1), 
\end{align}
where we write $H=H_{\rho,\alpha, 1}$ for simplicity.
Likewise, by stationarity,
\begin{align}
\Prob\big(X_{\floor{r \zeta}+1:r}\leq x, Y_{\floor{r \zeta}+1:r}\leq y\big) 
&= \label{eq:xyjoint-right}
H\big((1-\zeta)^{-1/\alpha}x,(1-\zeta)^{-1/\alpha}y\big)
+ o(1).
\end{align}
Recalling the marginal cdfs of $H=H_{\rho, \alpha, 1}$ from \eqref{eq:w_margCDF1} and \eqref{eq:w_margCDF2}, Equation \eqref{eq:xyjoint-left} implies 
\begin{align}
\Prob\big(Y_{1:\floor{r\zeta}}\leq y<X_{1:\floor{r\zeta}}\big)
&= \nonumber
\Prob\big( Y_{1:\floor{r\zeta}}\leq y\big) - \Prob\big(X_{1:\floor{r\zeta}} \leq y, X_{1:\floor{r\zeta}}\leq y\big)
\\&=\nonumber
\Prob\big( Y_{1:\floor{r\zeta}}\leq y\big) - \Prob\big(X_{1:\floor{r\zeta}}\leq y\big)
\\&=\nonumber
H^{(2)}(\zeta^{-1/\alpha} y) - H^{(1)}(\zeta^{-1/\alpha} y) +o(1)
\\&=\label{eq:xy-sandwich}  
\exp\big(-\zeta y^{-\alpha}\big) \rho_0 \zeta y^{-\alpha} +o(1).
\end{align}
Hence, using part [a] with $\tilde x \le y$, \eqref{eq:xyjoint-left},  \eqref{eq:xyjoint-right} and \eqref{eq:xy-sandwich},
the expression in \eqref{eq:5.5f} satisfies
\begin{align*}
&\phantom{{}={}} 
F_{n,\zeta, c}(\infty, y, \tilde x, \infty)
\\&= 
\exp\big(-\zeta y^{-\alpha}- \tilde x^{-\alpha} \big)
+
\exp\big(-\zeta \tilde x^{-\alpha}\big) \cdot \exp\big(-\zeta y^{-\alpha}\big) \rho_0 \zeta y^{-\alpha} \cdot \exp\big(-(1-\zeta) \tilde x^{-\alpha}\big) + o(1)
\\&= 
\exp\big(-\zeta y^{-\alpha}- \tilde x^{-\alpha} \big) \big(1+  \rho_0 \zeta y^{-\alpha} \big) +o(1),
\end{align*}
where we have used the marginal cdfs of $H$ from \eqref{eq:w_margCDF1} and \eqref{eq:w_margCDF2} again. This is exactly the claim in [b], for $\tilde x \le y$.

Regarding the case $\tilde x \ge y$, we start by noting that, in view of \eqref{eq:xyjoint-right},
\begin{align}
B_{n,\zeta}(y, \tilde x)
&\equiv \label{eq:bnzeta}
\Prob\big(Y_{\floor{r\zeta}+1:r}\leq y < X_{\floor{r\zeta}+1:r} \le \tilde x\big) 
\\&= \notag
\Prob\big(X_{\floor{r\zeta}+1:r}\leq \tilde x, Y_{\floor{r\zeta}+1:r}\leq y\big)
-
\Prob\big( X_{\floor{r\zeta}+1:r}\leq y, Y_{\floor{r\zeta}+1:r}\leq y\big) 
\\&=\notag 
\Prob\big(X_{\floor{r\zeta}+1:r}\leq \tilde x, Y_{\floor{r\zeta}+1:r}\leq y\big)
-
\Prob\big( X_{\floor{r\zeta}+1:r}\leq y\big), 
\\&=\notag
H\big((1-\zeta)^{-1/\alpha}\tilde x,(1-\zeta)^{-1/\alpha}y \big) 
    - 
    H^{(1)}((1-\zeta)^{-1/\alpha} y) + o(1)
\\&=
\exp\big(-(1-\zeta)y^{-\alpha}\big)(1-\zeta)y^{-\alpha}\rho\big((y/\tilde x)^\alpha\big) + o(1)\label{eq:5.5o}
\end{align}
by the definition of $H$ from \eqref{eq:H_Frech}.
Hence, using part [a] with $\tilde x \ge y$, \eqref{eq:xyjoint-left}, \eqref{eq:xyjoint-right}, \eqref{eq:xy-sandwich} and \eqref{eq:5.5o}, the expression in \eqref{eq:5.5h} satisfies
\begin{align}
&\phantom{{}={}}\nonumber
F_{n,\zeta, c}(\infty, y, \tilde x, \infty)
\\&=  \nonumber
\exp\big(-y^{-\alpha}-\zeta \tilde x^{-\alpha} \big)
\\& \hspace{.4cm}+ \nonumber
\exp\big(-\zeta \tilde x^{-\alpha}\big) \cdot \exp\big(-\zeta y^{-\alpha}\big) \rho_0 \zeta y^{-\alpha} \cdot \exp\big(-(1-\zeta) \tilde y^{-\alpha}\big)
\\& \hspace{.4cm}+ \nonumber
\exp\big(-\zeta \tilde x^{-\alpha}\big) \cdot
\exp\big(-(1-\zeta)y^{-\alpha}\big)(1-\zeta)y^{-\alpha}\rho\big((y/\tilde x)^\alpha\big) \cdot
\exp\big(-\zeta y^{-\alpha}\big) + o(1)
\\&= 
\exp\big(-y^{-\alpha}-\zeta \tilde x^{-\alpha} \big) 
\big\{ 1 + \zeta \rho_0  y^{-\alpha} + (1-\zeta)y^{-\alpha}\rho\big((y/\tilde x)^\alpha\big)\big\} + o(1),
\end{align}
which is the claim in [b], for $\tilde x \ge y$.

Part [c] follows from part [b] by stationarity and symmetry reasons.

Concerning part [d],  note that
\begin{align}
    F_{n,\zeta, c}(\infty, y, \infty,  \tilde y)
    &= \notag
    \P{Y_{1:r}\leq y,Y_{\floor{r\zeta}+1:\floor{r\zeta}+r}\leq \tilde y} 
    \\&= 
    \P{Y_{1:r}\leq y,X_{\floor{r\zeta}+1:\floor{r\zeta}+r}\leq \tilde y}\notag
    \\& \hspace{1cm} \notag
    +\P{Y_{1:r}\leq y,Y_{\floor{r\zeta}+1:\floor{r\zeta}+r}\leq \tilde y< X_{\floor{r\zeta}+1:\floor{r\zeta}+r}}
    \\&= 
    F_{n,\zeta, c}(\infty, y, \tilde y, \infty) + p_1 + p_2    
    \label{eq:fnc-case-d}
\end{align}
where $F_{n,\zeta, c}(\infty, y, \tilde y, \infty)$ has been calculated in part [b] and where
\begin{align*}
    p_1 &= 
    \P{X_{1:r}\leq y,Y_{\floor{r\zeta}+1:\floor{r\zeta}+r}\leq \tilde y<X_{\floor{r\zeta}+1:\floor{r\zeta}+r}}, \\
    p_2 &= \P{Y_{1:r}\leq y< X_{1:r},Y_{\floor{r\zeta}+1:\floor{r\zeta}+r}\leq \tilde y<X_{\floor{r\zeta}+1:\floor{r\zeta}+r}}.
\end{align*}
Regarding $p_1$, we have
\begin{align}
\label{eq:p1}
p_1 
&= \notag
\P{X_{1:r}\leq y, Y_{\floor{r\zeta}+1:\floor{r\zeta}+r}\leq \tilde y}-\P{X_{1:r}\leq y, X_{\floor{r\zeta}+1:\floor{r\zeta}+r}\leq \tilde y} 
\\&= \notag
F_{n,\zeta, c}(y,\infty, \infty,\tilde y)
-
F_{n,\zeta, c}(y,\infty,\tilde y,\infty)
\\&=
F_{\alpha, \zeta}(y, \tilde y)
        \big\{  \zeta \rho_0 \tilde y^{-\alpha} + \bm1( y \ge \tilde y) (1-\zeta) \tilde y^{-\alpha}\rho\big((\tilde y/y)^\alpha \big) \big\}
        +o(1).
\end{align}

The term $p_2$ is more difficult. First, note that the event 
$\{Y_{\flo{r\zeta}+1:\flo{r\zeta}+r}\leq \tilde y<X_{\flo{r\zeta}+1:\flo{r\zeta}+1}\}$
requires exactly one exceedance $\xi_{j_0} > \tilde y\sigma_r$,  for some unique $j_0\in\{\flo{r\zeta}+1, \dots , \flo{r\zeta}+r\}$, among all indices $j=\flo{r\zeta}+1, \dots, \flo{r\zeta}+r$. 
Distinguishing the cases $j_0 \le r$ or $j_0 > r$, we obtain that
the event 
$\{Y_{\flo{r\zeta}+1:\flo{r\zeta}+r}\leq \tilde y<X_{\flo{r\zeta}+1:\flo{r\zeta}+r}\}$ 
is the disjoint union of the two events 
$\{ Y_{\floor{r\zeta}+1:r}\leq \tilde y< X_{\floor{r\zeta}+1:r}, X_{r+1:\floor{r\zeta}+r}\leq \tilde y \}$ 
and 
$\{Y_{r+1: \floor{r\zeta}+r}\leq \tilde y<X_{r+1:\floor{r\zeta}+r},X_{\floor{r\zeta}+1:r}\leq \tilde y\}$.
Hence, by asymptotic independence, stationarity,  and \eqref{eq:xyjoint-left} and \eqref{eq:xy-sandwich},
\begin{align}
p_2 
&= \notag
\P{Y_{1:r}\leq y<X_{1:r},Y_{\floor{r\zeta}+1:\floor{r\zeta}+r}\leq \tilde y<X_{\floor{r\zeta}+1:\floor{r\zeta}+r}} 
\\&= \notag
\P{Y_{1:r}\leq y<X_{1:r},  Y_{\floor{r\zeta}+1:r}\leq \tilde y< X_{\floor{r\zeta}+1:r}, X_{r+1:\floor{r\zeta}+r}\leq \tilde y}
\\&\hspace{.6cm} + \notag
\P{Y_{1:r}\leq y<X_{1:r}, Y_{r+1: \floor{r\zeta}+r}\leq \tilde y<X_{r+1:\floor{r\zeta}+r},X_{\floor{r\zeta}+1:r}\leq \tilde y} 
\\& = \notag
\P{X_{r+1:\floor{r\zeta}+r}\leq \tilde y} \cdot
\P{ Y_{1:r}\leq y<X_{1:r},  Y_{\floor{r\zeta}+1:r}\leq \tilde y< X_{\floor{r\zeta}+1:r}}
\\&\hspace{.6cm} + \notag
\P{Y_{r+1: \floor{r\zeta}+r}\leq \tilde y<X_{r+1:\floor{r\zeta}+r}} \cdot
\P{Y_{1:r}\leq y<X_{1:r}, X_{\floor{r\zeta}+1:r}\leq \tilde y} 
\\&\hspace{.6cm} + \notag
o(1) 
\\&= \notag
\P{X_{1:\floor{r\zeta}}\leq \tilde y} \cdot p_{21} 
+ \P{Y_{1: \floor{r\zeta}}\leq \tilde y<X_{1:\floor{r\zeta}}} p_{22} + o(1)
\\&= \label{eq:p2-decomp}
\exp\big(-\zeta \tilde y^{-\alpha}\big)\cdot p_{21}
+
\exp\big(-\zeta \tilde y^{-\alpha}\big) \rho_0 \zeta \tilde y^{-\alpha} \cdot p_{22} +o(1),
\end{align}
where
\begin{align*}
p_{21} 
&= 
\P{ Y_{1:r}\leq y<X_{1:r},  Y_{\floor{r\zeta}+1:r}\leq \tilde y< X_{\floor{r\zeta}+1:r}} \\
p_{22} 
&= 
\P{Y_{1:r}\leq y<X_{1:r}, X_{\floor{r\zeta}+1:r}\leq \tilde y} 
\end{align*}
We start by treating the term $p_{22}$, which is exactly the term $ A_{n,\zeta}(y, \tilde y)$ from  \eqref{eq:anzeta}. Hence, in view of \eqref{eq:5.5e}, for the case $\tilde y \le y$
\begin{align}
p_{22} 
&= \notag
\Prob\big(Y_{1:\flo{r\zeta}}\leq y<X_{1:\flo{r\zeta}}\big) 
\cdot
\Prob\big( X_{\floor{r\zeta}+1:r}\leq \tilde y\big) +o(1)
\\&= \notag
\exp\big(-\zeta  y^{-\alpha}\big) \rho_0 \zeta  y^{-\alpha}
\cdot 
\exp\big(-(1-\zeta) \tilde y^{-\alpha}\big) + o(1)
\\&= \label{eq:p22-ytilde<y}
\exp\big(-\zeta  y^{-\alpha} -(1-\zeta) \tilde y^{-\alpha} \big) \rho_0 \zeta y^{-\alpha} +o(1)
\end{align}
by \eqref{eq:xy-sandwich} and \eqref{eq:xyjoint-right}. Likewise, for the case $\tilde y>y$, and in view of \eqref{eq:5.5g},
\begin{align}
p_{22} 
&= \notag
\Prob\big( 
    Y_{1:\flo{r\zeta}}\leq  y<X_{1:\flo{r\zeta}}\big) \cdot \Prob\big(X_{\floor{r\zeta}+1:r}\leq   y
    \big)
\\&\hspace{2cm}\nonumber
    + \Prob
    \big(
    Y_{\flo{r\zeta}+1:r}\leq y < X_{\flo{r\zeta}+1:r} \le \tilde y \big) \cdot \Prob\big( X_{1:\floor{r\zeta}}\leq  y 
    \big) + o(1),
\\&=\notag
\exp\big(-\zeta  y^{-\alpha}\big) \rho_0 \zeta  y^{-\alpha}
\cdot 
\exp\big(-(1-\zeta) y^{-\alpha}\big)
\\&\hspace{2cm}\nonumber
+
\exp\big(-(1-\zeta) y^{-\alpha}\big)(1-\zeta) y^{-\alpha}\rho\big(( y/\tilde y)^\alpha\big) 
\cdot 
\exp\big(-\zeta y^{-\alpha}\big) + o(1)
\\&= \label{eq:p22-ytilde>y}
\exp\big(-  y^{-\alpha} \big)  y^{-\alpha}
\big\{ \zeta \rho_0 + (1-\zeta) \rho\big(( y/\tilde y)^\alpha\big)  \big\} + o(1)
\end{align}
by \eqref{eq:xy-sandwich}, \eqref{eq:xyjoint-right}, \eqref{eq:5.5o} and \eqref{eq:xyjoint-left}. 

It remains to treat $p_{21}$, for which we use the fact that the event $\{Y_{1:r} \le y < X_{1:r}\}$ is the disjoint union of the two events $\{Y_{1:\flo{r\zeta}} \le y < X_{1:\flo{r\zeta}}, X_{\flo{r\zeta}+1:r}\le y \}$ and $\{Y_{\flo{r\zeta}+1:r} \le y < X_{\flo{r\zeta}+1:r}, X_{1:\flo{r\zeta}}\le y \}$. Hence,
\begin{align*}
p_{21} 
&= 
\P{ Y_{1:r}\leq y<X_{1:r},  Y_{\floor{r\zeta}+1:r}\leq \tilde y< X_{\floor{r\zeta}+1:r}} 
\\&=
\P{Y_{1:\flo{r\zeta}} \le y < X_{1:\flo{r\zeta}}, X_{\flo{r\zeta}+1:r}\le y, Y_{\floor{r\zeta}+1:r}\leq \tilde y< X_{\floor{r\zeta}+1:r}}
\\&\hspace{.6cm} + \notag
\P{Y_{\flo{r\zeta}+1:r} \le y < X_{\flo{r\zeta}+1:r}, X_{1:\flo{r\zeta}}\le y,  Y_{\floor{r\zeta}+1:r}\leq \tilde y< X_{\floor{r\zeta}+1:r}}
\\&=
\P{Y_{1:\flo{r\zeta}} \le y < X_{1:\flo{r\zeta}}} \cdot
\P{Y_{\floor{r\zeta}+1:r}\leq \tilde y< X_{\floor{r\zeta}+1:r} \le y}
\\&\hspace{.6cm} + \notag
\P{X_{1:\flo{r\zeta}}\le y} \cdot 
\P{Y_{\flo{r\zeta}+1:r} \le y < X_{\flo{r\zeta}+1:r},  Y_{\floor{r\zeta}+1:r}\leq \tilde y< X_{\floor{r\zeta}+1:r}}
+o(1)
\\&=\notag
\P{Y_{1:\flo{r\zeta}} \le y < X_{1:\flo{r\zeta}}} \cdot B_{n,\zeta}(\tilde y, y) 
+
\P{X_{1:\flo{r\zeta}}\le y} \cdot C_{n,\zeta}(\tilde y, y) +o(1),
\end{align*}
with $B_{n,\zeta}(\tilde y, y) $ from \eqref{eq:bnzeta} and with 
\begin{align*}
C_{n,\zeta}(\tilde y, y) 
&=
\P{Y_{\flo{r\zeta}+1:r} \le y \wedge \tilde y,  X_{\flo{r\zeta}+1:r}> y \vee \tilde y }
\\&=
\P{Y_{\flo{r\zeta}+1:r} \le y \wedge \tilde y} - \P{X_{\flo{r\zeta}+1:r} \le y \vee \tilde y , Y_{\flo{r\zeta}+1:r} \le y \wedge \tilde y} 
\\&=
H^{(2)}\big((1-\zeta)^{-1/\alpha}(y \wedge \tilde y)\big) - H\big((1-\zeta)^{-1/\alpha} (y \vee \tilde y), (1-\zeta)^{-1/\alpha} (y \wedge \tilde y)\big) + o(1)
\\&=
\exp\big( -(1-\zeta) (y \wedge \tilde y)^{-\alpha} \big)  (y \wedge \tilde y)^{-\alpha} (1-\zeta) \big\{ \rho_0 -  \rho\big( (\tfrac{y \wedge \tilde y}{y \vee \tilde y})^{\alpha} \big\} + o(1),
\end{align*}
by \eqref{eq:xyjoint-right} and the definition of $H=H_{\rho,\alpha,1}$ in \eqref{eq:H_Frech}.

Overall, if $\tilde y > y$, then $B_{n,\zeta}(\tilde y, y) $ from \eqref{eq:bnzeta} equals zero, and the previous two displays together with \eqref{eq:xyjoint-left} yield
\begin{align}
p_{21} 
&= \notag
\exp(-\zeta y^{-\alpha}) \cdot \exp( - (1-\zeta)y^{-\alpha}) y^{-\alpha} (1-\zeta) \big\{ \rho_0 - \rho( (y / \tilde y)^\alpha)\big\}+ o(1)
\\&=\label{eq:p21-ytilde>y}
\exp(-y^{-\alpha}) y^{-\alpha} (1-\zeta) \big\{ \rho_0 - \rho( (y / \tilde y)^\alpha)\big\}+ o(1).
\end{align}
Otherwise, if $\tilde y \le y$, then $B_{n,\zeta}(\tilde y, y) $ has been calculated in \eqref{eq:5.5o}, and we obtain, using \eqref{eq:xy-sandwich},
\begin{align}
p_{21}
&=\notag
\exp\big(-\zeta y^{-\alpha}\big) \rho_0 \zeta y^{-\alpha} \cdot 
\exp\big(-(1-\zeta)\tilde y^{-\alpha}\big)(1-\zeta)\tilde y^{-\alpha}\rho\big((\tilde y/y)^\alpha\big)
\\&\hspace{.6cm} + \notag
\exp(-\zeta y^{-\alpha}) \cdot
\exp\big( -(1-\zeta) \tilde y^{-\alpha} \big)  \tilde y^{-\alpha} (1-\zeta) \big\{ \rho_0 -  \rho\big( (\tilde y/y)^{\alpha}\big) \big\} + o(1),
\\&=\label{eq:p21-ytilde<y}
\exp\big( - \zeta y^{-\alpha}-(1-\zeta) \tilde y^{-\alpha} \big) \tilde y^{-\alpha} (1-\zeta)
\big\{ \zeta \rho_0 y^{-\alpha}\rho\big((\tilde y/y)^\alpha\big) + \rho_0 - \rho\big( (\tilde y/y)^{\alpha}\big)
\big\} + o(1).
\end{align}

Finally, we need to assemble terms. First, if $\tilde y \le y$, then, from \eqref{eq:p2-decomp}, \eqref{eq:p22-ytilde<y} and \eqref{eq:p21-ytilde<y},
\begin{align}
p_2
&= \notag
\exp\big(-\zeta \tilde y^{-\alpha}\big)\cdot 
\exp\big( - \zeta y^{-\alpha}-(1-\zeta) \tilde y^{-\alpha} \big) \tilde y^{-\alpha} (1-\zeta) 
\\&\hspace{6cm} \times \notag
\big\{ \zeta \rho_0 y^{-\alpha}\rho\big((\tilde y/y)^\alpha\big) + \rho_0 - \rho\big( (\tilde y/y)^{\alpha}\big)
\big\} 
\\&\hspace{1cm} + \notag
\exp\big(-\zeta \tilde y^{-\alpha}\big) \rho_0 \zeta \tilde y^{-\alpha} \cdot 
\exp\big(-\zeta  y^{-\alpha} -(1-\zeta) \tilde y^{-\alpha} \big) \rho_0 \zeta y^{-\alpha}
+o(1),
\\&= \notag
\exp\big( - \zeta y^{-\alpha}- \tilde y^{-\alpha} \big) \tilde y^{-\alpha} 
\\&\hspace{1cm} \times \notag
\Big\{ \rho_0 \zeta y^{-\alpha} \big\{ \zeta  \rho_0 + (1-\zeta)\rho\big( (\tilde y/y)^{\alpha} \big) \big\}  + (1-\zeta) \big\{ \rho_0 -  \rho\big( (\tilde y/y)^{\alpha} \big) \big\} \Big\} +o(1).
\end{align}
Likewise, if $\tilde y > y$, then, from \eqref{eq:p2-decomp}, \eqref{eq:p22-ytilde>y} and \eqref{eq:p21-ytilde>y},
\begin{align}
p_2
&= \notag
\exp\big(-\zeta \tilde y^{-\alpha}\big)
\cdot 
\exp(-y^{-\alpha}) y^{-\alpha} (1-\zeta) \big\{ \rho_0 - \rho( (y / \tilde y)^\alpha)\big\}
\\&\hspace{1cm} + \notag
\exp\big(-\zeta \tilde y^{-\alpha}\big) \rho_0 \zeta \tilde y^{-\alpha} 
\cdot 
\exp\big(-  y^{-\alpha} \big)  y^{-\alpha}
\big\{ \zeta \rho_0 + (1-\zeta) \rho\big(( y/\tilde y)^\alpha\big)  \big\} +o(1),
\\&= \notag
\exp\big( - \zeta \tilde y^{-\alpha}- y^{-\alpha} \big)  y^{-\alpha} 
\\&\hspace{1cm} \times \notag
\Big\{ \rho_0 \zeta \tilde y^{-\alpha} \big\{ \zeta  \rho_0 + (1-\zeta)\rho\big( (y/\tilde y)^{\alpha} \big) \big\}  + (1-\zeta) \big\{ \rho_0 -  \rho\big( (y/ \tilde y)^{\alpha} \big) \big\} \Big\} + o(1).
\end{align}
The expressions for the two cases $\tilde y \le y$ and $\tilde y > y$ can be unified in one formula as follows:
\begin{align}
p_2
&= \label{eq:p2}
F_{\alpha, \zeta}(y, \tilde y) (y \wedge \tilde y)^{-\alpha} 
\\&\hspace{.3cm} \times \notag
\Big\{ \rho_0 \zeta (y \vee \tilde y)^{-\alpha} \big\{ \zeta  \rho_0 + (1-\zeta)\rho\big( (\tfrac{y \wedge \tilde y}{y \vee \tilde y})^{\alpha} \big) \big\}  + (1-\zeta) \big\{ \rho_0 -  \rho\big( (\tfrac{y \wedge \tilde y}{y \vee \tilde y})^{\alpha} \big) \big\} \Big\} + o(1).
\end{align}
Finally, from \eqref{eq:fnc-case-d}, the convergence in part [b], \eqref{eq:p1}, and \eqref{eq:p2},
\begin{align*}
&\phantom{{}={}}
F_{n,\zeta, c}(\infty, y, \infty,  \tilde y)
\\&=
F_{n,\zeta, c}(\infty, y, \tilde y, \infty) + p_1 + p_2 
\\&=
F_{\alpha, \zeta}(y, \tilde y) \Big[ 1 + \zeta \rho_0 (y^{-\alpha} + \tilde y^{-\alpha}) + (1-\zeta) (y \wedge \tilde y)^{-\alpha}  \rho\big( (\tfrac{y \wedge \tilde y}{y \vee \tilde y})^{\alpha} \big)  
\\&\hspace{1cm}+
(y \wedge \tilde y)^{-\alpha}  \Big\{ \rho_0 \zeta (y \vee \tilde y)^{-\alpha} \big\{ \zeta  \rho_0 + (1-\zeta)\rho\big( (\tfrac{y \wedge \tilde y}{y \vee \tilde y})^{\alpha} \big) \big\}  + (1-\zeta) \big\{ \rho_0 -  \rho\big( (\tfrac{y \wedge \tilde y}{y \vee \tilde y})^{\alpha} \big) \big\} \Big\} \Big]
\\&=
F_{\alpha, \zeta}(y, \tilde y) \Big[ 1 + \zeta \rho_0 (y^{-\alpha} + \tilde y^{-\alpha}) + (1-\zeta) \rho_0 (y \wedge \tilde y)^{-\alpha}  
\\&\hspace{7cm} 
+\zeta \rho_0 y^{-\alpha} \tilde y^{-\alpha} \big\{ \zeta  \rho_0 + (1-\zeta)\rho\big( (\tfrac{y \wedge \tilde y}{y \vee \tilde y})^{\alpha} \big) \big\} \big],
\end{align*}
which is the asserted formula.
\end{proof}

\begin{lemma}[Asymptotic covariances of functions of sliding block maxima]\label{lem:asy_cov_slbm}
    Suppose Conditions \ref{cond:doa} and \ref{cond:mom} are met and that there exists an integer sequence $(\ell_n)_n$ such that $\ell_n = o(r_n)$ and $\alpha(\ell_n) = o(\ell_n / r_n)$ as $n \to \infty$. Then, for any $c > 0$, $\zeta \in [0, 1]$ and any pair of measurable functions $f, g$ on $(0, \infty)$ which are continuous almost everywhere and satisfy
    \begin{align*}
    (|f| \lor |g|)^2 \leq g_{\eta,\alpha_1}(x) = \{x^{-\alpha_1} \indic(x \leq \mathrm{e}) + \log(x) \indic(x > \mathrm{e})\}^{2 + \eta}
    \end{align*}
    for some $0 < \eta < \nu$, we have
    \begin{align*}
        \lim_{n \to \infty} \Cov(f(X_{1:r_n}), g(Y_{\lfloor r_n\zeta\rfloor+1:\lfloor r_n\zeta\rfloor + r_n})) &= \Cov(f(X_\zeta), g(\tilde Y_\zeta))
    \end{align*}
    where $(X_\zeta, \tilde Y_\zeta) \sim K_{\rho,\alpha,\zeta}(x,\infty,\infty,\tilde y)$ and
    \begin{align*}
        \lim_{n \to \infty} \Cov(f(Y_{1:r_n}), g(Y_{\lfloor r_n\zeta\rfloor+1:\lfloor r_n\zeta\rfloor + r_n})) &= \Cov(f(Y_\zeta), g(\tilde Y_\zeta))
    \end{align*}
    where $(Y_\zeta, \tilde Y_\zeta) \sim K_{\rho,\alpha,\zeta}(\infty,y, \infty, \tilde y)$ with $K_{\rho, \alpha, \zeta}$ from \eqref{eq:sl_distr}.
\end{lemma}

\begin{proof}
    The result follows from Lemma \ref{lem:JoiWeaKon} and the Cauchy–Schwarz inequality, together with Example 2.21 in \cite{vdVaart}.
\end{proof}

\begin{lemma}[Asymptotic covariances of sliding block maxima empirical process]
\label{lem:covs_emp_proc}
    Suppose Conditions \ref{eq:doa}, \ref{cond:alphamixing} 
    and \ref{cond:mom} are met. Then, for any pair of measurable functions $f, g$ on $(0,\infty)$ which are continuous almost everywhere and satisfy
    \begin{align*}
    (|f| \lor |g|)^2 \leq g_{\eta, \alpha_1}(x) = \{x^{-\alpha_1} \indic(x \leq \mathrm e) + \log(x)\indic(x > \mathrm e)\}^{2 + \eta}
    \end{align*}
    for some $0 < \eta < \nu$, we have,
    with $\Gb_n^{(\sbl)}$ as defined in the paragraph before Theorem~\ref{thm:sl_asy},
    \begin{align*}
    \lim_{n \to \infty} \Cov\Big(\G_n^{(\sbl)}\big[(x,y) \mapsto f(x)\big], \G_n^{(\sbl)}\big[(x,y) \mapsto g(y)\big]\Big) 
    = 2 \int_0^1 \Cov(f(X_\zeta), g(\tilde Y_\zeta)) \diff\zeta
    \end{align*}
    where $(X_\zeta, \tilde Y_\zeta) \sim K_{\rho,\alpha,\zeta}(x,\infty,\infty,\tilde y)$ and
    \begin{align*}
     \lim_{n \to \infty} \Cov\Big(\G_n^{(\sbl)}\big[(x,y) \mapsto f(y)\big], \G_n^{(\sbl)}\big[(x,y) \mapsto g(y)\big]\Big) 
     = 2 \int_0^1 \Cov(f(Y_\zeta), g(\tilde Y_\zeta)) \diff\zeta.
    \end{align*}
    where $(Y_\zeta, \tilde Y_\zeta) \sim K_{\rho,\alpha,\zeta}(\infty,y, \infty, \tilde y)$ with $K_{\rho, \alpha, \zeta}$ from \eqref{eq:sl_distr}
\end{lemma}

\begin{proof}
    The proof applies the same strategies as the proof of Lemma 5.3 in \cite{Bücher2018-sliding}. It is omitted for the sake of brevity.
\end{proof}

\begin{proof}[Proof of Theorem \ref{thm:sl_asy}]
Throughout, we omit the upper index $\sbl$.
The result follows from an application of Theorem~\ref{thm:asymptotic}.
Recall $Z_{n,i}$ from \eqref{eq:trunc_toptwo2}, $k_n = n - r_n + 1$, $v_n=\sqrt{n/r_n}$  and define $\Pb_n$ and $P_n$ as in \eqref{eq:Pn}, such that $
\Gb_n f  = v_n \big(\mathbb P_n f - Pf\big)$.
Here and in the remaining parts of the proof, we may assume that $c=c_0$, as argued at the beginning of the proof of Theorem~\ref{thm:blocks}. 
For the application of Theorem~\ref{thm:asymptotic}, we need to show the following three properties:
\begin{compactenum}[(1)]
    \item $\lim_{n \to \infty} \Prob(Z_{n,1} = \cdots = Z_{n,n-r_n+1}) = 0$.
    \item There exist constants $0 < \alpha_- < \alpha_1 < \alpha_+ < \infty$ such that
    $\Pb_n f \rightsquigarrow P f$
    for all $f \in \mathcal F_2(\alpha_-, \alpha_+)$, where $\mathcal F_2(\alpha_-, \alpha_+)$ is as in \eqref{eq:F2}.
    \item We have $\W_n=(\Gb_n f_1, \dots, \Gb_n f_4)^\top \rightsquigarrow \mathcal{N}_4(\bm B,\Sigma_{\rho,\alpha_0}^{(\sbl)})$, where $\bm B$ and $\Sigma_{\rho,\alpha_0}^{(\sbl)}$ are as in Theorem \ref{thm:sl_asy}.
\end{compactenum}
The ``not-all-tied'' property in (1) follows immediately from Lemma \ref{lem:no_ties}.

For the proof of (2), choose $\eta \in (2/\omega, \nu)$ with $\omega$ and $\nu$ from Conditions \ref{cond:alphamixing} and \ref{cond:mom}, respectively. Define $\alpha_+ := 2\alpha_1$ and let $0 < \alpha_- < \alpha_1$ be arbitrary. Any $f \in \mathcal F_2(\alpha_-, \alpha_+)$ can then be bounded in absolute value by $g_{0,\alpha_1}$ from \eqref{eq:g_kappa}, whence $\lim_{n \to \infty} \Exp[\mathbb P_n f] = Pf$ by Lemma \ref{lem:conv_of_mom_of_block}.
Further,
$
    \mathbb P_n f - \Exp[\mathbb P_n f] = O_{\mathbb P}(v_n^{-1}) = o_{\mathbb P}(1),
$
as will be shown in the proof of (3).
These two facts imply (2). 

To show (3),  we start by decomposing
\begin{align*}
    \G_n = v_n (\mathbb P_n - P_n) + v_n (P_n - P) \equiv \widetilde{\G}_n + B_n.
\end{align*}
For $j = 1,...,4$, we have $B_n(f_j) \to B(f_j)$ by Condition \ref{cond:bias}. It remains to show that the finite-dimensional distributions of $\widetilde{\G}_n(f)$ for $f \in \mathcal F_2(\alpha_-, \alpha_+)$ converge weakly to those of a zero-mean Gaussian process $\Gb$ with covariance
\begin{align}
\label{eq:sl-asy-cov}
    \Cov(\G f, \G g) = 2 \int_0^1 \Cov_{K_{\rho,\alpha_0,\zeta}} (f(U_1), g(U_2)) \, d\zeta, \quad f, g \in \mathcal F_2(\alpha_-, \alpha_+),
\end{align}
with $K_{\rho,\alpha_0,\zeta}$ as defined in \eqref{eq:sl_distr}.
Indeed, this implies (3) and additionally closes the gap in the proof of (2).

The proof of the claimed weak convergence now follows analogously to the proof of Theorem 2.6 in \cite{Bücher2018-sliding}, page 117-119, with the asymptotic covariance in \eqref{eq:sl-asy-cov} arising from Lemma~\ref{lem:covs_emp_proc} (which replaces Lemma 5.3 in \cite{Bücher2018-sliding}). Details are omitted for the sake of brevity.
\end{proof}

\subsection{Proof of Example~\ref{ex:mori}}
\label{subsec:proof-of-mori-example}

\begin{proof}[Proof of Example~\ref{ex:mori}]
We will show below that, for any $x \ge 1$,
\begin{align}
\Prob(M_r \le  x) 
&=\label{eq:cdf-mr-mori}
\Big(1-\frac1{x^\alpha}\Big)^r  \Big( 1 - \frac{1-\rho_0}{x^\alpha}  \Big)
\\
\Prob(S_r \le x) 
&= \label{eq:cdf-sr-mori}
\Big(1-\frac{1}{x^\alpha}\Big)^{r-1} \Big(1-\frac{1-\rho_0}{x^\alpha}\Big)  \Big(1+\frac1{x^\alpha} + (r-2) \frac{\rho_0}{x^\alpha} \Big)
\end{align}
and that, for all $x,y \in \R$
\begin{align} \label{eq:joint-cdf-mr-sr}
\lim_{r \to \infty} \Prob(r^{-1/\alpha}M_r \le x, r^{-1/\alpha}S_r  \le y)  = H_{\rho,\alpha,1}(x,y).
\end{align}
The latter equation immediately yields Condition~\ref{cond:doa} with $\alpha_0=\alpha, \sigma_r=r$ and with the given~$\rho$.

Next, regarding Condition~\ref{cond:all_diverge}, it is sufficient to consider $c \ge 2^{1/\alpha}$. By the union bound and \eqref{eq:cdf-sr-mori}, we have
\begin{align*}
    \Prob(\min\{S_{r_n,1},\dots,S_{r_n,k_n} \} \le c)
    \le
    k_n \Prob(S_{r_n} \le c)
    \le 
    k_n r_n\Big(1-\frac 1{c^\alpha}\Big)^{r_n-1} \le 2k_n r_n2^{-r_n},
\end{align*}
where we have used that $c \ge 2^{1/\alpha}$ and $\rho_0 \in [0,1]$. The expression on the right-hand side approaches zero provided $\log(k_n) = o(r_n)$, which is easily met if $n = O(r_n^3)$.

Condition~\ref{cond:alphamixing} in fact holds for any $r_n \in [n]$ satisfying $r_n\to \infty, r_n=o(n)$ and for any $\omega>0$; this follows immediately from 1-dependence.

Regarding Condition~\ref{cond:mom}, it is sufficient to consider $\nu=2$ and $\omega=1$. Using \eqref{eq:cdf-mr-mori} and \eqref{eq:cdf-sr-mori} and a computer algebra system, one obtains
\begin{align*}
    \lim_{r\to\infty} \Exp\big[\log^4(r^{-1/\alpha}M_r)\big] &=\alpha^{-4}[8 \gamma  \zeta (3)+\gamma ^4+{3 \pi ^4}/{20}+\gamma ^2 \pi ^2]
    , \\
    \lim_{r \to \infty} \Exp\big[(r^{-1/\alpha}S_r)^{-4}\big] &= \alpha^{-1}(\alpha+4\rho_0) \Gamma \left(1+4/{\alpha}\right), 
\end{align*}
where $\zeta (3)$ is Ap\'ery's constant.
Using straightforward monotonicity arguments (note that $\alpha_1 \le 1$), it can be shown that these two limits are sufficient to deduce Condition \ref{cond:mom} with $\nu=2$ and $\omega=1$.

Finally, regarding Condition~\ref{cond:bias}, we fix $c_0=1$, and note that $M_r \vee 1 = M_r$ and $S_r \vee 1 = S_r$. 
The functions $f=f_j$ from \eqref{eq:fct_H} are given by
\begin{align*}
    (f_1,f_2,f_3,f_4)= \big((x,y)\mapsto y^{-\alpha_1}\log y,(x,y)\mapsto y^{-\alpha_1},(x,y)\mapsto \log y, (x,y) \mapsto \log x\big).
\end{align*}
Let
\[
B_r'(f) = \Exp\!\big[f\big( r^{-1/\alpha}M_{r} , r^{-1/\alpha}S_{r}\big)\big]-\int_{(0,\infty)^2} f(x,y)\diff H_{\rho,\alpha,1}(x,y).
\]
Using \eqref{eq:cdf-mr-mori} and \eqref{eq:cdf-sr-mori} and a computer algebra system (CAS), it can be shown that
\begin{align*}
B'(f) := \lim_{r \to \infty} r B'_r(f) 
\end{align*}
is given by 
\begin{align*}
B'(f_1) 
&= 
\alpha^{-1}\frac{\Gamma (\varpi_{\rho_0}+1)}{2} \Big[ \psi ^{(0)}({\varpi_{\rho_0}}+1)\Big(-2 (1-\rho_0)^2 {\varpi_{\rho_0}} ({\varpi_{\rho_0}}+1)
\\&\hspace{2cm} 
+(1-\rho_0) \big\{{\varpi_{\rho_0}} (12-\varpi_{\rho_0}^2 + 5 {\varpi_{\rho_0}})+8\big\}+{\varpi_{\rho_0}} ({\varpi_{\rho_0}}+1)^2\Big) 
\\&\hspace{2cm} 
+3 \rho_0 \varpi_{\rho_0}^2+2 (1-\rho_0) (3+2\rho_0) {\varpi_{\rho_0}} +2 (5+\rho_0)(1-\rho_0) +4 {\varpi_{\rho_0}}+1\Big]
\\
B'(f_2) 
&= 
\frac{1}{2} \varpi_{\rho_0} \Big[(1-\rho_0) \big\{ 2 (1-\rho_0) (\varpi_{\rho_0}+1)+(\varpi_{\rho_0}-1) \varpi_{\rho_0} \big\} \Gamma (\varpi_{\rho_0}+1)  
\\&\hspace{2cm} 
- (\varpi_{\rho_0}+1) \Gamma (\varpi_{\rho_0}+2) \Big]
\\
B'(f_3) 
&=
\alpha^{-1} \Big[\frac{1}{2}-(1-\rho_0)^2 \Big]
\\
B'(f_4) 
&=
\alpha^{-1}\Big[\frac{3}{2}-\rho_0 \Big].
\end{align*}
As a consequence, since $(n/r_n^3)^{1/2} \to \lambda_1 \ge 0$, we obtain that Condition 4.5 is met with
\begin{align} \label{eq:bias-mori}
B(f_j) 
= 
\lim_{n \to \infty} \sqrt{n/r_n} B'_{r_n}(f_j) 
=
\lim_{n \to \infty}
\sqrt{n/r_n^3} \cdot r_n B'_{r_n}(f_j)  
= \lambda_1 \cdot B'(f_j).
\end{align}

\noindent
\textit{Proof of \eqref{eq:cdf-mr-mori}}. This part of the proof, we only conduct for $\alpha=1$. The general case may be obtained by replacing $M_r$ by $M_r^{1/\alpha}$. We have 
\begin{align*}
    M_r&=\max\{Z_0,Z_1,\dots, Z_{r-1},\zeta_1Z_1,\dots,\zeta_rZ_r\}, 
\end{align*}
As a consequence, since $\zeta_t \le 1$,
\begin{align*}
    \Prob(M_r\leq x)
    &=
    \Prob(Z_0\leq x)\Prob(Z_1\leq x,\zeta_1 Z_1\leq x)^{r-1}\Prob(\zeta_rZ_r\leq x)\\
    &=\Prob(Z_0\leq x)^r\Prob(\zeta_1Z_1\leq x)=\Big(1-\frac{1}{x}\Big)^r\Prob(\zeta_1 Z_1\leq x).
\end{align*}
The last probability evaluates to
\begin{align*}
\Prob(\zeta_1 Z_1\leq x)
=
\int_1^\infty \Prob( \zeta_1 \le x/z ) \frac1{z^2} \diff z
&= 
\int_1^x \frac1{z^2} \diff z - \int_x^\infty \rho'(x/z) \frac1{z^2}\diff z 
\\&= 
1-\frac1x - \frac1x \int_0^1\rho'(u) \diff u
= 1-\frac{1-\rho_0}x,
\end{align*}
where we used the substitution $x/z=u$ and the fact that $\int_0^1 \rho'(u) \diff u = \rho(1) - \rho(0) = -\rho(0)=-\rho_0$ by the fundamental theorem of calculus for Lebesgue integrals. Equation \eqref{eq:cdf-mr-mori} follows.

\medskip\noindent
\textit{Proof of \eqref{eq:cdf-sr-mori}}. For $x \ge 1$, we have
\begin{align} \label{eq:sr-desomposition}
\Prob(S_r \le x) = \Prob(M_r \le x) + \Prob(S_r \le x< M_r).
\end{align}
Here, 
\begin{align} \label{eq:sr-mr-decomposition}
\Prob(S_r \le x< M_r)
=
\sum_{j=1}^{r} \Prob(\xi_j >x, \xi_i \le x \, \forall j \ne i).
\end{align}
For $j=1$, we have
\begin{align*}
&\phantom{{}={}}\Prob( \xi_1 >x, \xi_i \le x \,\forall i \ge 2)
\\&=
\Prob(\max (Z_0, \zeta_1Z_1)>x, Z_1 \le y, \dots, Z_{r-1} \le x, \zeta_2Z_2 \le x, \dots, \zeta_rZ_r \le x)
\\&=
\Prob(\max (Z_0, \zeta_1Z_1)>x \ge  Z_1, Z_2 \le x, \dots, Z_{r-1} \le x, \zeta_rZ_r \le x)
\\&=
\Prob(\max (Z_0, \zeta_1Z_1)>x\ge  Z_1) \Prob(Z_0 \le x)^{r-2} \Prob(\zeta_r Z_r \le x)
\\&=
\Prob(Z_0 >x) \Prob(Z_0 \le x)^{r-1} \Prob(\zeta_r Z_r\le x).
\end{align*}
For $j \in \{2, \dots, r-1\}$, 
\begin{align*}
&\phantom{{}={}}\Prob( \xi_j >x, \xi_i \le x \,\forall i \ne j)
\\&=
\Prob\big( \{ \max (Z_{j-1}, \zeta_jZ_j) >x \} 
\\ & \hspace{2cm} \cap \{ Z_i \le x \, \forall i\in\{0, \dots,r-1\} \setminus \{j-1\} \} \cap \{\zeta_iZ_i \le x \, \forall i\in\{1, \dots,r\} \setminus \{j\} \}\big)
\\&=
\Prob\big( \{ \max (Z_{j-1}, \zeta_jZ_j) >x , Z_{j} \le x, \zeta_{j-1}Z_{j-1} \le x\} 
\\ & \hspace{2cm} \cap \{ Z_i \le x \, \forall i\in\{0, \dots,r-1\} \setminus \{j-1, j\} \} \cap \{\zeta_rZ_r \le x \}\big)
\\&=
\Prob\big(\max (Z_{j-1}, \zeta_jZ_j) >x , Z_{j} \le x, \zeta_{j-1}Z_{j-1} \le x) \Prob(Z_0 \le x)^{r-2} \Prob(\zeta_rZ_r \le x)
\\&=
\Prob\big( \zeta_{j-1}Z_{j-1}\le x < Z_{j-1}) \Prob(Z_0 \le x)^{r-1} \Prob(\zeta_rZ_r \le x).
\end{align*}
Finally, for $j=r$, we have
\begin{align*}
&\phantom{{}={}}\Prob( \xi_r >x, \xi_i \le x \,\forall i \le r-1)
\\&=
\Prob(\max (Z_{r-1}, \zeta_rZ_r) >x, Z_0 \le x, \dots, Z_{r-2} \le x, \zeta_1Z_1\le x, \dots, \zeta_{r-1}Z_{r-1} \le x)
\\&=
\Prob(\zeta_{r-1}Z_{r-1}  \le x <\max (Z_{r-1}, \zeta_rZ_r))\Prob( Z_0 \le x)^{r-1}.
\end{align*}
All probabilities on the right-hand sides of the previous three displays have already been calculated explicitly, except for the following two: first,
\begin{align*}
\Prob\big( \zeta_{j-1}Z_{j-1}\le x < Z_{j-1})
= 
\int_{x}^\infty \Prob(\zeta \le x/z)\frac1{z^2} \diff z 
&= 
-\int_{x}^\infty \rho'(x/z) \frac1{z^2} \diff z
\\&=
-\frac1x \int_{0}^1 \rho'(u) \diff u
= \frac{\rho_0}{x} 
\end{align*}
where we used the substitution $y/z=u$ and the fundamental theorem of calculus again.
Second, 
\begin{align*}
\Prob(A) 
&\equiv
\Prob(\zeta_{r-1}Z_{r-1}  \le x <\max (Z_{r-1}, \zeta_rZ_r))
\\&=
\Prob(A\cap  \{\zeta_rZ_r\le x\})
+\Prob(A \cap \{ \zeta_rZ_r  >x\} ) 
\\&=
\Prob(\zeta_{r-1}Z_{r-1}  \le x < Z_{r-1}, \zeta_r Z_r \le x)
+
\Prob(\zeta_{r-1}Z_{r-1} \le x< \zeta_rZ_r)
\\&=
\Prob(\zeta_{r-1}Z_{r-1} \le x<Z_{r-1} ) \Prob(\zeta_rZ_r \le x) 
+
\Prob(\zeta_{r-1}Z_{r-1} \le x) \Prob(\zeta_rZ_r >x) 
\\&=
\frac{\rho_0}x \Big(1-\frac{1-\rho_0}x\Big) + \frac{1-\rho_0}x \Big(1-\frac{1-\rho_0}x\Big)
=
\frac1x \Big(1-\frac{1-\rho_0}x\Big),
\end{align*}
Overall, 
\begin{align*}
\Prob( \xi_1 >x, \xi_i \le x \,\forall i \ne 1)
&= 
\frac1x \Big(1-\frac1x\Big)^{r-1}\Big(1-\frac{1-\rho_0}x\Big) 
\\
\Prob( \xi_j >x, \xi_i \le x \,\forall i \ne j)
&= \frac{\rho_0}x \Big(1-\frac1x\Big)^{r-1}\Big(1-\frac{1-\rho_0}x\Big)
\\
\Prob( \xi_r >x, \xi_i \le x \,\forall i \ne r)
&=
\frac1x \Big(1-\frac1x\Big)^{r-1}\Big(1-\frac{1-\rho_0}x\Big).
\end{align*}
Hence, by \eqref{eq:sr-mr-decomposition},
\[
\Prob(S_r \le x<M_r) = \Big(1-\frac{1}x\Big)^{r-1} \Big(1-\frac{1-\rho_0}x\Big) \Big(\frac2x + (r-2) \frac{\rho_0}{x} \Big),
\]
which in turn implies
\begin{align*}
\Prob(S_r \le x) 
&= 
\Prob(M_r \le x) + \Prob(S_r \le x < M_r) 
\\&= 
\Big(1-\frac{1}x\Big)^{r-1} \Big(1-\frac{1-\rho_0}x\Big)  \Big(1+\frac1x + (r-2) \frac{\rho_0}{x} \Big)
\end{align*}
as asserted.

\medskip
\noindent
\textit{Proof of \eqref{eq:joint-cdf-mr-sr}}. The proof is similar to the one for \eqref{eq:cdf-sr-mori}, but due to fact that we are only interested in the limit, some complicated negligible terms do not need to be calculated explicitly. First,
for $1 \le x \le y$, we have
\[
\Prob(M_r \le x, S_r \le y) = \Prob(M_r\le x),
\]
which immediately yields \eqref{eq:joint-cdf-mr-sr} for $1 \le x \le y$ after using \eqref{eq:cdf-mr-mori}.
Next, for $x >y \ge 1$, we have
\begin{align} \label{eq:joint-cdf-dec}
\Prob(M_r \le x, S_r \le y) = \Prob(M_r \le y) + \Prob(S_r \le y< M_r \le x).
\end{align}
Here, by a similar decomposition as in \eqref{eq:sr-mr-decomposition},
\[
\Prob(S_r \le y< M_r \le x)
=
\sum_{j=1}^{r} \Prob(\xi_j \in (y,x], \xi_i \le y \, \forall j \ne i).
\]
We need these expressions with $x$ and $y$ replaced by $rx$ and $ry$, and then the summands with $j=1$ and $j=r$ are negligible. Indeed,
\begin{align*}
\Prob( \xi_1 \in (ry,rx], \xi_i \le ry \,\forall i \ne 1)
&\le
\Prob( \xi_1 > ry, \xi_2 \le ry) 
\\&= \Prob( \max(Z_0, \zeta_1 Z_1)>ry, \max(Z_1, \zeta_2Z_2) \le ry)
\\ &\le  
\Prob(Z_0>ry) = 1/(ry) = o(1)
\end{align*}
for $r\to\infty$. A similar calculation shows that
$\Prob( X_r \in (ry,rx], X_i \le ry \,\forall i \ne r) = o(1)$. It remains to consider
$j \in \{2, \dots, r-1\}$, where
\begin{align*}
&\phantom{{}={}}\Prob( \xi_j \in (y,x], \xi_i \le y \,\forall i \ne j)
\\&=
\Prob\big( \{ \max (Z_{j-1}, \zeta_jZ_j) \in (y,x] \} 
\\ & \hspace{2cm} \cap \{ Z_i \le y \, \forall i\in\{0, \dots,r-1\} \setminus \{j-1\} \} \cap \{\zeta_iZ_i \le y \, \forall i\in\{1, \dots,r\} \setminus \{j\} \}\big)
\\&=
\Prob\big( \{ \max (Z_{j-1}, \zeta_jZ_j) \in (y,x] , Z_{j} \le y, \zeta_{j-1}Z_{j-1} \le y\} 
\\ & \hspace{2cm} \cap \{ Z_i \le y \, \forall i\in\{0, \dots,r-1\} \setminus \{j-1, j\} \} \cap \{\zeta_rZ_r \le y \}\big)
\\&=
\Prob\big(\max (Z_{j-1}, \zeta_jZ_j) \in (y,x] , Z_{j} \le y, \zeta_{j-1}Z_{j-1} \le y) \Prob(Z_0 \le y)^{r-2} \Prob(\zeta_rZ_r \le y)
\\&=
\Prob\big( \zeta_{j-1}Z_{j-1}\le y < Z_{j-1} \le x) \Prob(Z_0 \le y)^{r-1} \Prob(\zeta_rZ_r \le y).
\end{align*}
The only unknown expression is
\begin{align*}
\Prob\big( \zeta_{j-1}Z_{j-1}\le y < Z_{j-1} \le x)
&= 
\int_{y}^x \Prob(\zeta \le y/z)\frac1{z^2} \diff z 
\\&= 
-\int_{y}^x \rho'(y/z) \frac1{z^2} \diff z
\\&=
-\frac1y \int_{y/x}^1 \rho'(u) \diff u
= \frac{\rho(y/x)}{y} ,
\end{align*}
where we used the substitution $y/z=u$ and the fundamental theorem of calculus again. 
Overall, for $1 \le y <x$, 
\begin{align*}
\Prob(S_r \le yr< M_r \le xr) 
&=
(r-2) \Prob( \xi_2 \in (y,x], \xi_i \le y \,\forall i \ne 2) + o(1)
\\& =
(r-2) \frac{\rho(y/x)}{yr} \Big(1-\frac1{yr}\Big)^{r-1} \Big(1-\frac{1-\rho_0}{yr}\Big) + o(1)
\\&=
\frac{\rho(y/x)}{y} \exp(-1/y) + o(1),
\end{align*}
which in turn implies, by \eqref{eq:joint-cdf-dec} and \eqref{eq:cdf-mr-mori},
\[
\Prob(M_r \le xr, S_r \le ry) =   \exp(-1/y) \Big\{ 1+ \frac{\rho(y/x)}{y}\Big\}+o(1)
=H_{\rho,1,1}(x,y) +o(1)
\]
as asserted in \eqref{eq:joint-cdf-mr-sr}.  

It remains to prove \eqref{eq:block-maxima-mle-bias-mori}. By Theorem 3.6 in \cite{Bücher2018-disjoint}, we need to calculate $\boldsymbol B_{\max} = M(\alpha) (B(f_5), B(f_6), B(f_4))^\top$ with
\[
M(\alpha) = \frac{6}{\pi^2}
    \begin{pmatrix}
         \alpha^2 & (1-\gamma)\alpha & -\alpha^2 \\ 
         \gamma-1 & -(\Gamma''(2)+1)/\alpha & 1-\gamma
     \end{pmatrix}
\]
from Formula (2.16) in \cite{Bücher2018-disjoint} and $B(f_j) = \lambda_1 B'(f_j)$ as in \eqref{eq:bias-mori}, with $(f_5,f_6)=\big((x,y)\mapsto x^{-\alpha}\log x,(x,y)\mapsto x^{-\alpha}\big)$.
Similar calculations as before yield
\begin{align*}
    B'(f_5)&= \alpha^{-1}\Big(\frac{9}{2}-2\gamma+(\gamma-2)\rho_0\Big), \qquad 
    B'(f_6)= \rho_0-2
\end{align*}
which yields \eqref{eq:block-maxima-mle-bias-mori} by straightforward calculations.
\end{proof}

\section{Proofs for Section \ref{sec:estimation-blockmaxima-iid}}
\label{sec:proofs-estimation-blockmaxima-iid}

\begin{proof}[Proof of Theorem \ref{thm:IID}] 
We start with the disjoint blocks estimator, $\mbl=\dbl$, for which the assertion follows from an application of Theorem \ref{thm:blocks}. Hence, we only have to verify its conditions.

\medskip
\noindent
(i) Proof of Condition \ref{cond:doa}.
Second-order regular variation from Condition \ref{cond:sorv} implies first-order regular variation in \eqref{eq:regvar}, which in turn is equivalent to weak convergence of block maxima as in \eqref{eq:iid_doa} with $a_r$ as in \eqref{eq:aniid}. We claim that Condition \ref{cond:doa} is met with $\sigma_r = a_r$ and $\rho=\rho_\indi$. First, $a_r$ is regularly varying with index $1/\alpha_0$ by Proposition 1.11 in \cite{Resnick1987}. Finally, the weak convergence in \eqref{eq:doa} follows for instance from Theorem 3.5 in \cite{Col01}.

\medskip
\noindent
(ii) Proof of Condition \ref{cond:all_diverge}.
Choose your favorite $c\in(0,\infty)$. Note that, for any $r\in\N$,
\begin{align} \nonumber
\Prob(S_{r} \le c) 
=
\Prob(S_{r} \le c, M_{r} > c) + \Prob(S_{r} \le c, M_{r} \le c)
&=
r F^{r-1}(c)(1-F(c)) + F^r(c) 
\\&\le \label{eq:sr-cdf-iid}
2rF^{r-1}(c).
\end{align}
Hence, 
since $\log F(c)<0$ and $\log k_n=o(r_n)$ by Remark 4.5 in \cite{Bücher2018-disjoint}, we have, by the union-bound,
\begin{align*}
    \Prob\big(\min\{S_{r_n,1},\dots,S_{r_n,k_n}\}\leq c\big)
    & \le
    2k_n r_n F^{r_n-1}(c)
    \\&=
    \exp\big\{ \log k_n + \log r_n + r_n \log F(c) \big\} =o(1), \qquad n \to \infty.
\end{align*}

\medskip
\noindent
(iii) Proof of Condition \ref{cond:alphamixing}.
This is trivial, as $\alpha(\ell)=0$ for integer $\ell\geq 1$.

\medskip
\noindent
(iv) Proof of Condition \ref{cond:mom}. Both bounds in \eqref{eq:intfinite} hold for arbitrary $\nu>0$
as a consequence of Lemma~\ref{lem:c1-disjoint}.

\medskip
\noindent
(v) Proof of Condition \ref{cond:bias}.
This condition, in particular the explicit computation of the bias vector, will take the majority of effort within this proof. For $x>0$ such that $F(x)>0$, write $L(x) = -\log F(x) x^{\alpha_0}$. Elementary calculations then allow to write \eqref{eq:secorderrv} as
\begin{align}\label{eq:secorderrv-L}
    \lim_{u\to\infty} \frac{1}{A(u)}\Big(\frac{L(ux)}{L(u)}-1\Big) = h_\tau(x), \qquad  x \in (0,\infty).
\end{align}
As argued in the proof of Theorem 4.2 in \cite{Bücher2018-disjoint} (beginning of the proof of Condition 3.5), we can find, for any fixed $\delta\in(0,\alpha_0)$ ,  constants $x(\delta)\ge 1$ and $c(\delta)>0$ such that, for all $u\ge x(\delta)$ and $x\ge x(\delta)/u$,
\begin{align} \label{eq:potter}
    \frac{L(u)}{L(ux)} \leq (1+\delta) \max\{x^{-\delta},x^\delta\}, \qquad     
    \Big|\frac{L(ux)-L(u)}{g(u)}\Big|\leq c(\delta)\max\{x^{\tau-\delta},x^{\tau+\delta}\},
\end{align}
where $g(u) = A(u) L(u)$. 
Moreover, by increasing $x(\delta)$ if necessary, we also have
\begin{align}
\label{eq:potter2}
    \frac{1-F(ux)}{1-F(u)}\leq (1+\delta) \max\big\{x^{-\alpha_0+\delta}, x^{-\alpha_0-\delta}\big\},
    \qquad
    \frac{L(ux)}{L(u)} \leq (1+\delta) \max\{x^{-\delta},x^\delta\}
\end{align}
for all $u\ge x(\delta)$ and $x \ge x(\delta)/u$ by the Potter bounds; see Theorem 1.5.7 in \cite{Bingham1987}.
We are going to show Condition \ref{cond:bias} for $c_0:=c := x(\delta)$ and $\sigma_{r_n}=a_{r_n}$.

Recall the definition of $Z_{n,i}=(X_{n,i}, Y_{n,i})$ from \eqref{eq:trunc_toptwo}, and let $P_n$ denote the distribution of $Z_{n,i}/a_{r_n}$, whose limit distribution $P$ is the iid \Frechet-Welsch distribution $P=\mathcal {W}_\indi(\alpha_0,1)$ by the proof of Condition~\ref{cond:doa} at the beginning of this proof. For $f=f_j$ from \eqref{eq:fct_H}, write $B_n(f)=\sqrt{k_n}(P_n f  - Pf)$. We need to show that, for $j\in\{1, 2,3,4\}$,
\begin{align}
\label{eq:b-limit}
B(f_j) = \lim_{n\to\infty} B_n(f_j) =  B_j(\alpha_0, \tau)
\end{align}
with $B(\alpha_0, \tau) \in \R^4$ from \eqref{eq:bias-iid}.

For $m\in\{1,2\}$, write $P^{(m)}$ and $P_n^{(m)}$ for the $m$th marginal of $P$ and $P_n$, respectively, and note that
\begin{align*}
B_n(f_1) &= \sqrt{k_n}(P_n^{(2)} - P^{(2)})[y^{-\alpha_0} \log y],
& 
B_n(f_2) &= \sqrt{k_n}(P_n^{(2)} - P^{(2)})[y^{-\alpha_0}] \\
B_n(f_3) &= \sqrt{k_n}(P_n^{(2)} - P^{(2)})[ \log y],
& 
B_n(f_4) &= \sqrt{k_n}(P_n^{(1)} - P^{(1)})[\log x].
\end{align*}
For the case $\lambda_3=0$, convergence of $B_n(f_4)$ to $B_4(\alpha_0, \tau)$ has been shown in \cite{Bücher2018-disjoint}, Formula (A.24). The more general case is treated in Lemma~\ref{lem:correction-bias}. It remains to treat $B_n(f_j)$ for $j\in\{1,2,3\}$. For that purpose, let $G_n$ and $G$ denote the cdf of $P_n^{(2)}$ and $P^{(2)}$, respectively, which are given by
\begin{align*}
    G_n(y)&= \Big\{F^{r_n}(a_{r_n}y)+r_n F^{r_n-1}(a_{r_n}y)\big(1-F(a_{r_n}y)\big)\Big\} \indic_{[c/a_{r_n},\infty)}(y) \\
    &= F^{r_n}(a_{r_n}y) \Big\{1 + r_n \Big(\frac{1}{F(a_{r_n}y)} - 1 \Big) \Big\} \indic_{[c/a_{r_n},\infty)}(y) \\
    G(y)&= \exp\big(-y^{-\alpha_0}\big)\big(1+ y^{-\alpha_0}\big)\indic_{(0,\infty)}(y).
\end{align*}
Here, the former follows from similar calculations as in \eqref{eq:sr-cdf-iid}, while the latter follows immediately from \eqref{eq:H_iid}. Now, by the display on top of page 1457 in \cite{Bücher2018-disjoint}, we have
\begin{align*}
    B_n (f_j)  
    = -\int_0^\infty \sqrt{k_n} \big\{ G_n(y)-G(y) \big \} f_j'(y)\diff y
\end{align*}
for $j\in\{1, 2,3\}$.

Let us rewrite
\[
F^{r_n}(a_{r_n}y)= \exp\big(r_n \log F(a_{r_n})\big)= \exp\Big(-y^{-\alpha_0} \big(-r_n \log F(a_{r_n}) \big)\frac{L(a_{r_n}y)}{L(a_{r_n})}\Big)
\]
As a consequence, $B_n(f) = J_{n,1}(f)+J_{n,2}(f)+J_{n,3}(f)$, where

\begin{align*}
    J_{n,1}(f) 
    &=
    \sqrt{k_n} \int_0^{c/a_{r_n}} \exp\big(-y^{-\alpha_0}\big)\big(1+y^{-\alpha_0}\big)f'(y)\diff y,
    \\
    J_{n,2}(f)
    &= 
    - \sqrt{k_n} \int_{c/a_{r_n}}^\infty \exp\Big(-y^{-\alpha_0}\big(-r_n \log F(a_{r_n})\big)\frac{L(a_{r_n}y)}{L(a_{r_n})}\Big)
    \\ & \hspace{1cm} \times
    \bigg[1+\frac{r_n(1-F(a_{r_n}y))}{F(a_{r_n}y)}-\Big\{ 1 + y^{-\alpha_0}\big(-r_n \log F(a_{r_n})\big)\frac{L(a_{r_n}y)}{L(a_{r_n})}\Big\}\bigg]f'(y)\diff y, 
    \\
    J_{n,3}(f)
    &= 
    -\sqrt{k_n} \int_{c/a_{r_n}}^\infty \bigg[\exp\Big(-y^{-\alpha_0}\big(-r_n \log F(a_{r_n})\big)\frac{L(a_{r_n}y)}{L(a_{r_n})}\Big)
    \\ & \hspace{1cm} \times
    \Big\{1+y^{-\alpha_0}\big(-r_n \log F(a_{r_n})\big)\frac{L(a_{r_n}y)}{L(a_{r_n})}\Big\} 
    -\exp\big(-y^{-\alpha_0}\big)\big(1+y^{-\alpha_0}\big)\bigg]f'(y)\diff y.
\end{align*}We start by showing that $J_{n,1}(f_j)$ 
converges to zero, for any $j\in\{1,2,3\}$. For that purpose, we decompose 
\begin{align*}
    J_{n,1}(f_j)&= \sqrt{k_n} \int_0^{c/a_{r_n}} \exp\big(-y^{-\alpha_0}\big)f_j'(y)\diff y+\sqrt{k_n} \int_0^{c/a_{r_n}} \exp\big(-y^{-\alpha_0}\big)y^{-\alpha_0}f_j'(y)\diff y.
\end{align*}
The first integral on the right-hand side has been treated similarly in \cite{Bücher2018-disjoint}, page 1457. The second integral can be treated analogously,
as the multiplication with $y^{-\alpha_0}$ does not change the decay of the integrand at zero being dominated by the exponential term.

Regarding $J_{n,2}(f_j)$, recall $L(x) = - \log F(x) x^{\alpha_0}$. 
We start by bounding
\begin{align*}
g_n(y) 
&:= 
r_n\frac{1-F(a_{r_n} y)}{F(a_{r_n}y)} - y^{-\alpha} \big(-r_n \log F(a_{r_n})\big)\frac{L(a_{r_n}y)}{L(a_{r_n})}
\\&=
r_n \Big[ \frac{1-F(a_{r_n} y)}{F(a_{r_n}y)} - (a_ry)^{-\alpha_0} L(a_{r_n}y) \Big]
\\&=
r_n \Big[ \frac{1-F(a_{r_n} y)}{F(a_{r_n}y)} - \log\Big( \frac{1}{F( a_ry)} \Big) \Big],
\end{align*}

A Taylor expansion of $x \mapsto  \log (x)$ 
around 1 allows to write
\begin{align*}
g_r(y) 
&= 
r_n \Big[ \frac{1-F(a_{r_n} y)}{F(a_{r_n}y)} - \Big(\frac1{F(a_{r_n}y)} - 1 -   \frac12 \Big\{\frac1{F(a_{r_n}y)} - 1 \Big\}^2  + R_n(y)\Big) \Big] \\
&=
r_n \Big[ \frac{\{F(a_{r_n}y)-1\}^2}{2F(a_{r_n}y)^2}  - R_n(y)\Big],
\end{align*}
where, for some $1 \le \xi_{n,y}  \le 1/F(a_{r_n}y)$,
\[
R_n(y)
= \frac1{3 \xi_{n,y}^2} \Big\{ \frac1{F(a_{r_n}y)}-1\Big\}^3.
\]
We have
\[
|R_n(y)|
\le 
\frac13 \Big\{ \frac{F(a_{r_n}y)-1}{F(a_{r_n}y)}\Big\}^3
=
O(r_n^{-3}),
\]
where the last bound follows from $F(a_{r_n}y)=1+o(1)$ and $r_n\{F(a_{r_n}y)-1\} = y^{-\alpha_0}+o(1)$.
As a consequence, since $\sqrt {k_n}/r_n=\lambda_1+o(1)$ by \eqref{eq:lambda1},
\begin{align*}
\sqrt{k_n} g_n(y) 
&= 
\frac{\sqrt{k_n}}{r_n} \Big[ \frac{r_n^2\{F(a_{r_n}y)-1\}^2}{2F(a_{r_n}y)^2}  + r_n^2R_n(y)\Big) \Big]
\\ &=
\{\lambda_1+o(1)\} \big[ y^{-2\alpha_0}/2 + o(1)\big] = \lambda_1 y^{-2\alpha_0} /2 + o(1).
\end{align*}
Consequently, the integrand of $J_{n,2}(f)$ converges pointwise to
\begin{align*}
    -(\lambda_1/2)\cdot\exp\big(-y^{-\alpha_0}\big)y^{-2\alpha_0}f'(y)
\end{align*}
If we now show that 
\begin{align*}
    h_{n,j}(y)= \sqrt{k_n}\exp\Big(-y^{-\alpha_0}\frac{L(a_{r_n}y)}{L(a_{r_n})}\Big)g_n(y)f_j'(y)\indic_{[c/a_{r_n},\infty)}(y)
\end{align*}
may be bounded by an integrable function on $(0,\infty)$, we would conclude 
\begin{align}
\label{eq:jn2-limit}
    \lim_{n\to\infty} J_{n,2}(f_j) 
    = 
    -(\lambda_1/2)\cdot\int_0^\infty \exp\big(-y^{-\alpha_0}\big)y^{-2\alpha_0}f_j'(y)\diff y
    =: -(\lambda_1/2)\cdot J_2(f_j)
\end{align}
where 
\begin{align*}
J_2(f_j)
=
\alpha_0^{-1}\Exp[Yf_j'(Y)]
= 
\begin{cases}
        (5-2\gamma)\alpha_0^{-1}, & f_1(y)=y^{-\alpha_0}\log y,\\
        -2,& f_2(y)=y^{-\alpha_0},\\
        \alpha_0^{-1}, & f_3(y)=\log y,
    \end{cases}
\end{align*}
with $Y\sim H_{\rho_\indi,\alpha_0,1}^{(2)}$, and where the last identity follows from Lemma~\ref{cor:mom}, using that $\Gamma(3)=2$ and $\Gamma'(3)=(3-2\gamma)$.

For that purpose, we start by deriving a majorant for $\sqrt{k_n} g_n(y)$ for $y \in [c/a_{r_n}, \infty)$.
By Taylor's theorem with Lagrange remainder applied to $x\mapsto \log x$, we have
\begin{align*}
    \sqrt{k_n}g_n(y) 
    &= 
    \sqrt{k_n}\cdot r_n\Big[ \frac{1-F(a_{r_n} y)}{F(a_{r_n}y)} - \log\Big(\frac1{F(a_{r_n}y)} \Big) \Big]\\
    &= 
    \sqrt{k_n}\cdot r_n\Big[ \frac{1-F(a_{r_n} y)}{F(a_{r_n}y)} - \Big(\frac{1}{F(a_{r_n}y)}-1\Big)+ \frac{1}{2\xi_{n,y}^2}\Big(\frac{1}{F(a_{r_n}y)}-1\Big)^2 \Big]\\
    &= 
    \frac{\sqrt{k_n}}{r_n}\cdot r_n^2\Big[\frac{1}{2\xi_{n,y}^2}\Big(\frac{1-F(a_{r_n}y)}{1-F(a_{r_n})}\Big)^2\cdot\frac{1}{F(a_{r_n}y)^2}\cdot \big(1-F(a_{r_n})\big)^2 \Big]
\end{align*}
for some $1 \le \xi_{n,y} \le 1/ F(a_{r_n}y)$. Using that $1/ F(a_{r_n}y)\le 1/F(c)$, we have $(\xi_{n,y}F(a_{r_n}y) ) ^{-2} \le F(c)^{-4}$. Further, for sufficiently large $n$, we have $r_n^2(1-F(a_{r_n}))^2<2$. 
Finally, by \eqref{eq:potter2} with $u=a_{r_n}$ and $x=y$, we have
\[
\frac{1-F(a_{r_n}y)}{1-F(a_{r_n})}
\leq (1+\delta)
\max\big\{y^{-\alpha_0+\delta}, y^{-\alpha_0-\delta}\big\}.
\]
Altogether, we have found a constant $C=C(\delta,\lambda_1)$ such that 
\begin{align}
\label{eq:bound-gn}
    \sqrt{k_n}g_n(y) \le C \max\big\{y^{-2\alpha_0+2\delta}, y^{-2\alpha_0-2\delta}\big\} \qquad \forall y \ge c/a_{r_n}
\end{align}
for all sufficiently large $n$.

We will now bound $h_{n,j}$ separately on $[c/a_{r_n},1)$ and $[1,\infty)$, respectively. First, for $y \in [c/a_{r_n},1)$ we have 
\begin{align*} 
    \exp\Big(-y^{-\alpha_0}\frac{L(a_{r_n}y)}{L(a_{r_n})}\Big)\leq \exp\big(-(1+\delta)^{-1}y^{-\alpha_0+\delta}\big)
\end{align*}
by \eqref{eq:potter}. 
Hence, in view of \eqref{eq:bound-gn} and the fact that there exists a constant $C'$ such that $f_j'(y) \le C'y^{-\alpha_0-\delta-1}$ for all $y \in (0,1)$, we obtain that
\begin{align*}
    h_{n,j}(y) \leq C \cdot C' \cdot y^{-3\alpha_0-3\delta-1}\exp\big(-(1+\delta)^{-1}y^{-\alpha_0+\delta}\big) \qquad \forall y \in (0,1)
\end{align*}
for all sufficiently large $n$
The upper bound is clearly integrable on $(0,1)$.

Second, for $y \in [1, \infty)$, we have
\begin{align*}
    \exp\Big(-y^{-\alpha_0}\frac{L(a_{r_n}y)}{L(a_{r_n})}\Big)\leq \exp\big(-(1+\delta)^{-1}y^{-\alpha_0-\delta}\big)
\end{align*}
by \eqref{eq:potter}. Hence, since $f'(y)$ is bounded by a multiple of $y^{-1}$ for $y\in[1,\infty)$, we have, again using \eqref{eq:bound-gn},
\begin{align*}
    h_{n,j}(y)\leq C''\cdot y^{-1-2\alpha_0+2\delta}\exp\big(-(1+\delta)^{-1}y^{-\alpha_0-\delta}\big) \qquad \forall y \ge 1
\end{align*}
for some constant $C''=C''(\delta, \lambda_1)$ and for all sufficiently large $n$. The upper bound is integrable on $[1,\infty)$ by our choice of $\delta< \alpha_0$.

It remains to treat $J_{n,3}(f_j)$. In view of the mean value theorem, applied to the function $z \mapsto \exp(-y^{-\alpha_0} z)(1+y^{-\alpha_0}z)$, there exists some $\xi_{n,y}$ between  $\big(-r_n \log F(a_{r_n})\big)L(a_{r_n}y)/L(a_{r_n})$ and $1$ such that
\begin{align*}
    J_{n,3}(f)
    &= 
    -\sqrt{k_n} \int_{c/a_{r_n}}^\infty \bigg[\exp\Big(-y^{-\alpha_0}\big(-r_n \log F(a_{r_n})\big)\frac{L(a_{r_n}y)}{L(a_{r_n})}\Big)
    \\ & \hspace{1cm} \times
    \Big\{1+y^{-\alpha_0}\big(-r_n \log F(a_{r_n})\big)\frac{L(a_{r_n}y)}{L(a_{r_n})}\Big\} 
    -\exp\big(-y^{-\alpha_0}\big)\big(1+y^{-\alpha_0}\big)\bigg]f'(y)\diff y
    \\&= 
    \sqrt{k_n} \int_{c/a_{r_n}}^\infty \Big[\big(-r_n \log F(a_{r_n})\big)\frac{L(a_{r_n}y)}{L(a_{r_n})}-1\Big]\exp\big(-\xi_{n,y} y^{-\alpha_0}\big)\xi_{n,y}y^{-2\alpha_0}f'(y)\diff y.
\end{align*}
Adding and subtracting $L(a_{r_n}y)/L(a_{r_n})$, we may write
\begin{align*}
&\phantom{{}={}}
\sqrt{k_n}
\Big[\big(-r_n \log F(a_{r_n})\big) \frac{L(a_ry)}{L(a_r)} - 1 \Big]
\\&=
\sqrt{k_n} \Big[\big(-r_n \log F(a_{r_n})\big) -1 \Big]\frac{L(a_{r_n}y)}{L(a_{r_n})}  + \sqrt{k_n} A(a_{r_n}) \cdot \frac1{A(a_{r_n}) }\Big[\frac{L(a_{r_n}y)}{L(a_{r_n})} - 1 \Big]
\\&=
(\lambda_3+o(1)(1+o(1)) + (\lambda_2+o(1)) (h_\tau(y)+o(1)) 
\\&=
\lambda_3 + \lambda_2 h_\tau(y) + o(1),
\end{align*}
where we have used  \eqref{eq:aniid}, \eqref{eq:lambda1} and Condition \ref{cond:sorv} at the second equality. 
As a consequence,
the integrand in the penultimate display converges pointwise in $y\in(0,\infty)$ to
\begin{align*}
    \{\lambda_3+\lambda_2h_\tau(y)\} y^{-2\alpha_0} \exp\big(-y^{-\alpha_0}\big)f'(y)
\end{align*}
Hence, in view of the dominated convergence theorem, we obtain that
\begin{align}
\label{eq:jn3-limit}
    \lim_{n\to\infty} J_{n,3}(f_j) 
    &= 
     \int_0^\infty \{\lambda_3+\lambda_2h_\tau(y)\} y^{-2\alpha_0} \exp\big(-y^{-\alpha_0}\big)f_j'(y)\diff y
    \notag\\&=:
     \lambda_3J_{31}(f_j)+\lambda_2J_{32}(f_j,\tau)
\end{align}
provided we show that

\begin{align*}
    f_n(y):= \sqrt{k_n}\Big[\big (-r_n \log F(a_{r_n})\big)\frac{L(a_{r_n}y)}{L(a_{r_n})}-1\Big]\exp\big(-\xi_{n,y} y^{-\alpha_0}\big)\xi_{n,y}y^{-2\alpha_0}f'(y) \indic_{[c/a_{r_n},\infty)}(y)
\end{align*}
can be bounded by an integrable function on $(0,\infty)$. The latter follows analogous to the argumentation on top of page 1459 in \cite{Bücher2018-disjoint}: first, by \eqref{eq:potter} and \eqref{eq:potter2}, we have, for sufficiently large $n$,
\begin{align*}
&\phantom{{}={}} \sqrt{k_n}\Big|\big(-r_n \log F(a_{r_n}))\frac{L(a_{r_n}y)}{L(a_{r_n})}-1\Big|
\\&\leq
\Big|\sqrt{k_n} \Big[ \big(-r_n \log F(a_{r_n})\big) - 1 \Big] \frac{L(a_{r_n}y)}{L(a_{r_n})}\Big|  + \Big|\sqrt{k_n} A(a_{r_n}) \cdot \frac1{A(a_{r_n}) }\Big(\frac{L(a_{r_n}y)}{L(a_{r_n})} - 1 \Big)\Big|
\\&\le 
(|\lambda_3|+\delta)(1+\delta)\max\{y^{-\delta}, y^\delta\} + (|\lambda_2| + \delta) 
c(\delta) \max\big\{ y^{\tau-\delta}, y^{\tau+\delta} \big\}
\end{align*}and
\begin{align*}
\xi_{n,y} 
&\geq 
\min\Big\{1,\frac{L(a_{r_n}y)}{L(a_{r_n})}\Big\}
\geq 
(1+\delta)^{-1} \min\big\{y^\delta,y^{-\delta}\big\},
\\
\xi_{n,y}
&\leq 
\max\Big\{1,\frac{L(a_{r_n}y)}{L(a_{r_n})}\Big\}
\leq (1+\delta)\max\big\{y^\delta,y^{-\delta}\big\}.
\end{align*}
Hence, in view of the bounds on $f_j'$, we conclude that there exists a finite constant $c'(\delta)$ such that, for $1\geq y\geq c/a_{r_n}$
\begin{align*}
    f_n(y)\leq c'(\delta)(1+y^{\tau})\exp\big\{-(1+\delta)^{-1}y^{-\alpha_0+\delta}\big\}y^{-3\alpha_0-3\delta-1}
\end{align*}
and the function is integrable since $\delta<\alpha_0$. On the other hand, for $y\geq 1$ we find the bound
\begin{align*}
    f_n(y)\leq c''(\delta)(1+y^{\tau})y^{2\delta-2\alpha_0-1}
\end{align*}
which is easily integrable on $[1,\infty)$.

It remains to calculate the limit on the right-hand side of \eqref{eq:jn3-limit}. Note that we may write

\begin{align*}
    J_{31}(f_j) = \alpha_0^{-1}\Exp\!\big[f_j'(Y)Y\big] = \begin{dcases}
        \Exp\!\Big[Y^{-\alpha_0}\big(\alpha_0^{-1}-\log Y\big)\Big], & f_1(y)=y^{-\alpha_0}\log y, \\
        -\Exp\!\Big[Y^{-\alpha_0}\Big], & f_2(y)=y^{-\alpha_0}, \\
        \alpha_0^{-1}, & f_3(y)=\log y,
    \end{dcases}
\end{align*}

\begin{align*}
    J_{32}(f_j,\tau) = \alpha_0^{-1}\Exp\!\big[h_\tau(Y)f_j'(Y)Y\big] = \begin{dcases}
        \Exp\!\Big[h_\tau(Y)Y^{-\alpha_0}\big(\alpha_0^{-1}-\log Y\big)\Big], & f_1(y)=y^{-\alpha_0}\log y, \\
        -\Exp\!\Big[h_\tau(Y)Y^{-\alpha_0}\Big], & f_2(y)=y^{-\alpha_0}, \\
        \alpha_0^{-1}\Exp\!\big[h_\tau(Y)\big], & f_3(y)=\log y,
    \end{dcases}
\end{align*}
where $Y\sim H_{\rho_\indi,\alpha_0,1}^{(2)}$. The expectations may again be calculated explicitly using Lemma \ref{cor:mom}. First,
\begin{align*}
    J_{31}(f_j)=\begin{cases}
        \frac{\Gamma(3)}{\alpha_0} + \frac{\Gamma'(3)} {\alpha_0} = \frac{5-2\gamma}{\alpha_0}, & f_1(y)=y^{-\alpha_0}\log y \\
        -\Gamma(3)=-2, & f_2(y)=y^{-\alpha_0} \\
        \frac1{\alpha_0}, & f_3(y)=\log y.
    \end{cases}
\end{align*}
Next, regarding $J_{32}(f_j,\tau)$, for $\tau =0$, we have $h_\tau(y)=\log y$, whence
\begin{align*}
    J_{32}(f_j, 0)=\begin{cases}
            -\frac{\Gamma'(3)}{\alpha_0^2}-\frac{\Gamma''(3)}{\alpha_0^2} =
            \frac{8\gamma-5-2\gamma^2-\pi^2/3}{\alpha_0^2}
        , & f_1(y)=y^{-\alpha_0}\log y \\
        \frac{\Gamma'(3)}{\alpha_0}=\frac{3-2\gamma}{\alpha_0}, & f_2(y)=y^{-\alpha_0} \\
        -\frac{\Gamma'(2)}{\alpha_0^2}=\frac{\gamma-1}{\alpha_0^2}, & f_3(y)=\log y.
    \end{cases}
\end{align*}
For $\tau <0$, we have $h_\tau(y)=(y^\tau-1)/\tau$, whence
\begin{align*}
    J_{32}(f_1, \tau)
    &=
    \frac{1}{\alpha_0\tau}\Exp\!\big[Y^{\tau-\alpha_0}-Y^{-\alpha_0}\big]+\frac{1}{\tau}\Exp\!\big[Y^{-\alpha_0}\log Y-Y^{\tau-\alpha_0}\log Y\big]  \\
    &=
    \frac{1}{\tau\alpha_0} \Big\{ \Gamma\Big(3+\frac{|\tau|}{\alpha_0}\Big)-\Gamma(3) \Big\} + \frac1{\tau}
    \Big\{ -\frac{\Gamma'(3)}{\alpha_0}+\frac{1}{\alpha_0}\Gamma'\Big(3+\frac{|\tau|}{\alpha_0}\Big)\Big\} \\
    &=
    \frac{1}{\tau\alpha_0}\Big\{ \Gamma\Big(3+\frac{|\tau|}{\alpha_0}\Big)-5+2\gamma+\Gamma'\Big(3+\frac{|\tau|}{\alpha_0}\Big)\Big\}, \\
    J_{32}(f_2, \tau)
    &=
    \frac{1}{\tau}\Exp\!\big[Y^{-\alpha_0}-Y^{\tau-\alpha_0}\big]
    =
    \frac{1}{\tau}\Big\{ 2 - \Gamma\Big(3+\frac{|\tau|}{\alpha_0}\Big) \Big\} \\
    J_{32}(f_3, \tau)
    &=
    \frac{1}{\tau\alpha_0}\Exp\!\big[Y^\tau - 1\big]
    =
    \frac{1}{\tau\alpha_0}\Big\{ \Gamma\Big(2+\frac{|\tau|}{\alpha_0}\Big)-1\Big\}.
\end{align*}

Overall, since $B_n(f)= J_{n,1}(f)+J_{n,2}(f)+J_{n,3}(f)$, we obtain from \eqref{eq:jn2-limit} and \eqref{eq:jn3-limit} and the subsequent calculations that Condition~\ref{cond:bias} is met with
\[
B(f_j) 
= 
\lim_{n \to \infty} B_n(f_j) 
=
-(\lambda_1/2) J_2(f_j)+\lambda_3 J_{31}(f_j)+ \lambda_2 J_{32}(f_j, \tau)
=
B_j(\alpha_0, \tau),
\]
with $B_j(\alpha_0, \tau)$ from \eqref{eq:bias-iid} (note that $|\tau|=-\tau$),
as claimed in \eqref{eq:b-limit}. Hence, the proof for the $\mbl=\dbl$ is finished. 

We next prove the claim regarding the sliding blocks maxima estimator, $\mbl=\sbl$, for which we apply Theorem \ref{thm:sl_asy}. In view of the proof for disjoint blocks, the only condition left to be validated is Condition~\ref{cond:all_diverge2}. For that purpose, we apply \eqref{eq:sr-cdf-iid} with $r=\tilde r_n$ to obtain that, for any $c>0$,
\begin{align*}
    \Prob\big(\min\{S_{1:\tilde r_n},\dots,S_{(k-1)\tilde r_n+1: \tilde r_n \tilde k_n}\}\leq c\big)
    & \le
    2\tilde k_n \tilde r_n F^{\tilde r_n-1}(c)
    \\&=
    \exp\big\{ \log \tilde k_n + \log \tilde r_n + \tilde r_n \log F(c) \big\}.
\end{align*}
The upper bound converges to zero since $\log k_n = o(r_n)$ (see the sentences after \eqref{eq:sr-cdf-iid}) implies $\log \tilde k_n = o(\tilde r_n)$.

Finally, the result regarding the bias-corrected estimators is an immediate consequence of Theorem~\ref{thm:bias-corrected-asymp}. 
\end{proof}

\begin{lemma}[Correction of the bias formula in Theorem 4.2 of \cite{Bücher2018-disjoint}] \label{lem:correction-bias}
Assume the notations and conditions from Theorem 4.2 in \cite{Bücher2018-disjoint}, and note that their $\lambda$ corresponds to our $\lambda_2$ from \eqref{eq:lambda2} and their $\rho$ corresponds to our $\tau$. Additionally, assume that   \eqref{eq:lambda3} is met. Then, the mean parameter of the limiting normal distribution in (4.10) of \cite{Bücher2018-disjoint} is given by $M(\alpha_0) B(\alpha_0, \tau)$, where $M(\alpha_0)$ is from their Equation (2.16) and where
\begin{align*}
B(\alpha,\tau)=
\frac{\lambda_2}{\bar\tau \alpha_0^2}\begin{pmatrix}
            2-\gamma-\Gamma(2+\bar\tau)-\Gamma'(2+\bar\tau) \\
            \alpha_0(\Gamma(2+\bar\tau)-1) \\
            1-\Gamma(1+\bar\tau)
        \end{pmatrix}+\frac{\lambda_3}{\alpha_0}\begin{pmatrix}
            2-\gamma\\
            -\alpha_0\\
            1
        \end{pmatrix},
    \end{align*}
for $\tau<0$ and
    \begin{align*}
        B(\alpha,0)=\frac{\lambda_2}{ \alpha_0^2}\begin{pmatrix}
            \gamma-(1-\gamma)^2-\pi^2/6 \\
            \alpha_0(1-\gamma) \\
            \gamma.
        \end{pmatrix}+\frac{\lambda_3}{\alpha_0}\begin{pmatrix}
            2-\gamma\\
            -\alpha_0\\
            1
        \end{pmatrix}.
    \end{align*}
    The expression is the same as the one in \cite{Bücher2018-sliding} if and only if $\lambda_3=0$.
\end{lemma}

\begin{proof}
A careful inspection of the proof of Theorem 4.2 in \cite{Bücher2018-disjoint} shows that the integrand in their $J_{n2}$ integral on page 1457 converges to $(\lambda_3+\lambda_2 h_\tau(y)) \exp(-y^{-\alpha_0})f'(y)$ rather than $\lambda_2 h_\tau(y) \exp(-y^{-\alpha_0})f'(y)$. This effectively implies that an additional bias term with coordinates $\lambda_3\alpha_0^{-1}\Exp[f_j'(Z)Z]$ appears, where $Z$ is Fréchet-distributed with parameter $(\alpha_0,1)$. Using their Lemma B.1, the three expectations $\Exp[f_j'(Z)Z]$ are $\Exp[Y^{-\alpha_0}(1-\alpha_0 \log Y)]=\Gamma(2)+\Gamma'(2)=2-\gamma$ for $f_j(y)=y^{-\alpha_0} \log(y)$, $\Exp[(-\alpha_0)Y^{-\alpha_0}]=-\alpha_0 \Gamma(2)=-\alpha_0$ for  $f_j(y)=y^{-\alpha_0}$ and $\Exp[1]=1$ for $f_3(y)=1/y$. The results follows by carefully assembling terms.
\end{proof}

\section{Further properties of the Fr\'echet-Welsch-distribution}
\label{sec:frechet-welsch-auxiliary}

\begin{lemma}[Existence of a Lebesgue-density]\label{lem:density-new3}
Suppose that $\rho\in\mathcal C$ is twice differentiable on $[0,1]$ at all but finitely many points.
Then $\mathcal W(\alpha, \rho, 1)$ has a Lebesgue density if and only if 
$\int_0^1 \rho'(z) + z \rho''(z) \diff z = -1$.
In that case, if $D$ denotes the finite set of points at which $\rho$ is not twice differentiable,
the density is given by
\[
h_{\rho, \alpha, 1}(x,y) 
=
-\alpha^2
\exp(-y^{-\alpha})\Big\{ (xy)^{-\alpha-1} \rho'\big( (y/x)^\alpha \big)+x^{-2\alpha-1}y^{\alpha-1}\rho''\big( (y/x)^\alpha \big)\Big\}
\]
for all $(x,y) \in S_{\rho,\alpha} = \{(x,y) \in (0,\infty)^2: x<y \text{ and } y \ne z^{1/\alpha}x \text{ for all } z \in D\}$ and $h_{\rho, \alpha, 1}(x,y)=0$ for all $(x,y) \notin S_{\rho,\alpha}$. 

Addendum: if $\rho$ is twice continuously differentiable on $[0,1]$, the condition $\int_0^1 \rho'(z) + z \rho''(z) \diff z = -1$ is equivalent to $\rho'(1)=-1$.
\end{lemma}

\begin{proof}
Note that $h_{\rho, \alpha, 1}(x,y)\ge 0$ by non-increasingness and concavity of $\rho$, and that $ S_{\rho,\alpha}^c$ is a Lebesgue null set. Substituting $z=(y/x)^\alpha$ with $\diff z = \alpha y^{\alpha-1} x^{-\alpha}\diff y$ and then $u=x^{-\alpha}/z$ with $\diff u = - \alpha x^{-\alpha-1}/z \diff x$, we obtain that
\begin{align*}
\int_{\R^2} h_{\rho, \alpha, 1}(x,y)  \diff (x,y)
&=
\int_0^\infty \int_0^x h_{\rho, \alpha, 1}(x,y) \diff y \diff x
\\&=
- \alpha \int_0^\infty \int_0^1 \exp(-x^{-\alpha} z^{-1}) \Big\{ x^{-2\alpha-1} z^{-2} \rho'(z) + x^{-\alpha-1} \rho''(z) \Big\} \diff z \diff x
\\&=
- \int_0^\infty \int_0^1 \exp(-u) \big\{ u \rho'(z) + z \rho''(z)\big\} \diff z \diff u 
\\&=
- \int_0^1 \rho'(z) + z \rho''(z) \diff z \ge 0.
\end{align*}
Hence, $B \mapsto \mu(B) := \int_B h_{\rho, \alpha, 1}(x,y)  \diff (x,y)$ defines a finite Borel measure on $\R^2$. It is a probability measure if and only if $\int_0^1 \rho'(z) + z \rho''(z) \diff z = -1$. 

Now, elementary calculations show that, for all $(x,y) \in S_{\rho,\alpha}$, we have
$
 \frac{\partial^2}{\partial x\partial y} H_{\rho,\alpha,1}(x,y) = h_{\rho, \alpha, 1}(x,y). 
$
As a consequence, the measures $\mu$ and $\mathcal W(\rho, \alpha,1)$ assign the same measure to all rectangles in $(0,\infty)^2$ that are completely contained in $S_{\rho,\alpha}$. Since $(0,\infty)^2\setminus S_{\rho,\alpha}$ consists of finitely many straight lines intersecting at the origin, the two measures must coincide on $S_{\rho,\alpha}$. This implies the assertion.

The addendum follows straightforwardly from partial integration.
\end{proof}

Recall the gamma function $\Gamma(x)=\int_0^\infty t^{x-1}e^{-t}\diff t$ and let $\Gamma'$ denote its first derivative. Note  that $\Gamma'(1) = - \gamma$, with $\gamma\approx 0.5772$ the Euler-Mascheroni constant.

\begin{lemma}[Moments]\label{cor:mom} 
    Fix $\alpha_0\in(0,\infty)$ and let  
    $H^{(1)}_{\alpha_0,1}$ and $H^{(2)}_{\alpha_0,1}$ denote the marginal cdfs of the $\mathcal{W}_\indi(\alpha_0,1)$ distribution; see \eqref{eq:w_margCDF1} and \eqref{eq:w_margCDF2} with $\rho_0=1$, respectively. Then
    \begin{align*}
        \mathrm{(a)} && \int_0^\infty y^{-\alpha}\diff H^{(2)}_{\alpha_0,1}(y)&=\Gamma\Big(2+\frac{\alpha}{\alpha_0}\Big) && \alpha\in(-2\alpha_0,\infty),\\
        \mathrm{(b)} && \int_0^\infty y^{-\alpha}\log y\diff H^{(2)}_{\alpha_0,1}(y) &=\frac{-1}{\alpha_0}\Gamma'\Big(2+\frac{\alpha}{\alpha_0}\Big) && \alpha\in(-2\alpha_0,\infty),\\
        \mathrm{(c)} && \int_0^\infty y^{-\alpha}\log^2 y\diff H^{(2)}_{\alpha_0,1}(y) &=\frac{1}{\alpha_0^2}\Gamma''\Big(2+\frac{\alpha}{\alpha_0}\Big) && \alpha\in(-2\alpha_0,\infty),\\
        \mathrm{(d)} && \int_0^\infty \log x\diff H^{(1)}_{\alpha_0,1}(x)&=\frac{-1}{\alpha_0}\Gamma'(1)=\frac{\gamma}{\alpha_0} && \alpha\in(-\alpha_0,\infty).
    \end{align*}
\end{lemma}

\begin{proof}
    Define the substitution $z=y^{-\alpha_0}$. Then we have for part $\mathrm{(a)}$
    \begin{align*}
        \int_0^\infty y^{-\alpha}\diff H^{(2)}_{\alpha_0,1}(y)&=\int_0^\infty y^{-\alpha}\cdot \alpha_0 y^{-1-2\alpha_0}\exp\big(-y^{-\alpha_0}\big)\diff y \\&= \int_0^\infty z^{\alpha/\alpha_0}z^{(1+2\alpha_0)/\alpha_0}\exp(-z)z^{-(\alpha_0+1)/\alpha_0}\diff z\\
        &= \int_0^\infty z^{(\alpha+\alpha_0)/\alpha_0}\exp(-z)\diff z
        \\&= \Gamma\Big(2+\frac{\alpha}{\alpha_0}\Big).
    \intertext{With the same substitution for part $\mathrm{(b)}$,}
        \int_0^\infty y^{-\alpha}\log y\diff H^{(2)}_{\alpha_0,1}(y)&=\int_0^\infty y^{-\alpha}\log y\cdot \alpha_0 y^{-1-2\alpha_0}\exp\big(-y^{-\alpha_0}\big)\diff y \\
        &= \int_0^\infty z^{\alpha/\alpha_0}\log\big(z^{-1/\alpha}\big)z^{(1+2\alpha_0)/\alpha_0}\exp(-z)z^{-(\alpha_0+1)/\alpha_0}\diff z\\
        &= \frac{-1}{\alpha_0}\int_0^\infty z^{(\alpha+\alpha_0)/\alpha_0}\log z\cdot\exp(-z)\diff z
        \\&= \frac{-1}{\alpha_0}\Gamma'\Big(2+\frac{\alpha}{\alpha_0}\Big).
    \intertext{Similarly, we receive for part $\mathrm{(c)}$,}
        \int_0^\infty y^{-\alpha}\log^2 y\diff H^{(2)}_{\alpha_0,1}(y)&=\int_0^\infty y^{-\alpha}\log^2 y\cdot \alpha_0 y^{-1-2\alpha_0}\exp\big(-y^{-\alpha_0}\big)\diff y \\
        &= \int_0^\infty z^{\alpha/\alpha_0}\log^2\big(z^{-1/\alpha}\big)z^{(1+2\alpha_0)/\alpha_0}\exp(-z)z^{-(\alpha_0+1)/\alpha_0}\diff z\\
        &= \frac{1}{\alpha_0^2}\int_0^\infty z^{(\alpha+\alpha_0)/\alpha_0}\log z\cdot\exp(-z)\diff z
        \\&= \frac{1}{\alpha_0^2}\Gamma''\Big(2+\frac{\alpha}{\alpha_0}\Big).
    \end{align*}
    For part $\mathrm{(d)}$, we refer to \cite{Bücher2018-disjoint}, Lemma B.1.
\end{proof}

\begin{lemma}[Moments, more general] \label{lem:mom}
Fix $\alpha_0\in(0,\infty)$ and let $H_{\rho, \alpha_0, 1}^{(1)}$ and $H_{\rho, \alpha_0, 1}^{(2)}$ denote the marginal cdfs of the $\mathcal{W}(\rho,\alpha_0,1)$ distribution; see \eqref{eq:w_margCDF1} and \eqref{eq:w_margCDF2}, respectively. Then, for any $\alpha\in (-\alpha_0,\infty)$,
    \begin{align*}
        \mathrm{(a)} && \int_0^\infty y^{-\alpha}\diff H_{\rho, \alpha_0, 1}^{(2)}(y)&=\Upsilon_{\rho_0}(\alpha/\alpha_0), \\
        \mathrm{(b)} && \int_0^\infty y^{-\alpha}\log y\diff H_{\rho, \alpha_0, 1}^{(2)}(y)&=\frac{-1}{\alpha_0}\Upsilon_{\rho_0}'(\alpha/\alpha_0), \\
        \mathrm{(c)} && \int_0^\infty y^{-\alpha}\log^2 y\diff H_{\rho, \alpha_0, 1}^{(2)}(y)&=\frac{1}{\alpha_0^2}\Upsilon_{\rho_0}''(\alpha/\alpha_0), \\
        \mathrm{(d)} && \int_0^\infty \log x\diff H_{\rho, \alpha_0, 1}^{(1)}(x)&=\frac{-1}{\alpha_0}\Gamma'(1)=\frac{\gamma}{\alpha_0},
    \end{align*}
    where $\Upsilon_{\rho_0}(x):=\rho_0\Gamma(x+2)+(1-\rho_0)\Gamma(x+1)$ 
    and where $\rho_0:=\rho(0)$.
\end{lemma}

\begin{proof}
    Recall the marginal densities in \eqref{eq:general_dens}. One quickly notices that 
    \begin{align*}
        p_{\rho,\alpha_0,1}^{(2)}(y)&=\alpha_0  y^{-\alpha-1}\exp\big(-y^{-\alpha}\big)\big[1-\rho_0+\rho_0 (y)^{-\alpha}\Big]\\
        &=(1-\rho_0)p_{\indi,\alpha_0,1}^{(1)}(y)+\rho_0 p_{\indi,\alpha_0,1}^{(2)}(y).
    \end{align*}
    Consequently, if $(X,Y)\sim\mathcal{W}(\rho, \alpha_0,1)$ and $(X',Y')\sim\mathcal{W}_\indi(\alpha_0,1)$,
    \begin{align*}
    \Exp[f(Y)]=\rho_0 \Exp[f(Y')] + (1- \rho_0)  \Exp[f(X')]
\end{align*}
Now the claim directly follows from Lemma \ref{cor:mom} and Lemma B.1 in \cite{Bücher2018-disjoint}. 
\end{proof}

\section{Asymptotic covariance formulas}
\label{sec:asymptotic-covariances}

\begin{lemma}[Asymptotic covariance for the disjoint blocks top-two estimator]\label{lem:cov_disjoint_general}
    Suppose $(X, Y)$ is a random vector from the Fréchet-Welsch-distribution $\mathcal W(\alpha, \rho, 1)$ with joint cdf $H_{\rho, \alpha, 1}$ as in \eqref{eq:H_Frech}. 
    Let $\varpi = \varpi_{\rho_0}$ be as in \eqref{eq:varpi} and let $(f_1,f_2,f_3,f_4)$ be defined as in \eqref{eq:fct_H} with $\alpha_1=\varpi \alpha$, that is
    \[
    f_1(x,y) = y^{-\varpi \alpha}\log y,  \quad
    f_2(x,y) = y^{-\varpi \alpha}, \quad
    f_3(x,y) = \log y,  \quad
    f_4(x,y) = \log x.
    \]
    Then, for $i,j \in \{1, \dots, 4\}$,
    \[
    \sigma_{ij}^{(\dbl)}  
    := 
    \Cov_{(X,Y) \sim \mathcal W(\rho, \alpha,1)}(f_i(X,Y),f_j(X,Y))
    \]
    is given by {\small
\begin{align*}
    \sigma_{11}^{(\dbl)} 
    &= 
    \alpha^{-2}\Big\{\Gamma (2 \varpi+1) \Big((2 \rho_0 \varpi+1) \psi ^{(0)}(2 \varpi+1)^2+2 \rho_0 \psi ^{(0)}(2 \varpi+1)
    \\& \hspace{1.6cm}
    +(2 \rho_0 \varpi+1) \psi ^{(1)}(2 \varpi+1)\Big)-\Gamma (\varpi+1)^2 ((\rho_0 \varpi+1) \psi ^{(0)}(\varpi+1)+\rho_0)^2\Big\}
    \\ \sigma_{12}^{(\dbl)} 
    &=  
    \alpha^{-1}\Big\{(\rho_0 \varpi+1) \Gamma (\varpi+1)^2 ((\rho_0 \varpi+1) \psi ^{(0)}(\varpi+1)+\rho_0) 
    \\& \hspace{1.6cm}
    -\Gamma (2 \varpi+1) ((2 \rho_0 \varpi+1) \psi ^{(0)}(2 \varpi+1)+\rho_0)\Big\}
    \\ \sigma_{13}^{(\dbl)} 
    &= 
    \alpha^{-2}\Big\{(\gamma -\rho_0) \Gamma (\varpi+1) ((\rho_0 \varpi+1) \psi ^{(0)}(\varpi+1)+\rho_0)
    \\& \hspace{1.6cm}
    +\varpi \Gamma (\varpi) \Big((\rho_0 \varpi+1) \psi ^{(0)}(\varpi+1)^2+2 \rho_0 \psi ^{(0)}(\varpi+1)+(\rho_0 \varpi+1) \psi ^{(1)}(\varpi+1)\Big)\Big\} 
    \\ \sigma_{14}^{(\dbl)}
    &=
    \alpha^{-2}\Gamma (\varpi) \big\{\gamma+2\psi ^{(0)}(\varpi)+\varpi \psi ^{(0)}(\varpi)^2+\varpi\gamma \psi ^{(0)}(\varpi)+\varpi \psi ^{(1)}(\varpi)\big\} 
    \\& \hspace{1.6cm}
    + \alpha^{-2} \Gamma(\varpi+1) \Big[ [\rho_0\gamma+\rho_0\psi^{(0)}(\varpi+1)-\rho_1] \big\{ 1+\varpi \psi^{(0)}(\varpi+1)\big\} + \rho_0 \varpi\psi^{(1)}(\varpi+1) \Big]
    \\ \sigma_{22}^{(\dbl)} 
    &= 
    (2 \rho_0 \varpi+1) \Gamma (2 \varpi+1)-(\rho_0 \varpi+1)^2 \Gamma (\varpi+1)^2 
    \\ \sigma_{23}^{(\dbl)} 
    &= 
    -\alpha^{-1}\Big\{\Gamma (\varpi+1) \left(\rho_0^2 (-\varpi)+\gamma  \rho_0 \varpi+(\rho_0 \varpi+1) \psi ^{(0)}(\varpi+1)+\gamma \right)\Big\}
    \\ \sigma_{24}^{(\dbl)}
    &=
    - \alpha^{-1}\Gamma(\varpi) \Big[1 + \varpi(\gamma+\psi ^{(0)}(\varpi)) + \varpi^2[\rho_0\gamma+\rho_0\psi^{(0)}(\varpi+1)-\rho_1] \Big]
    \\ \sigma_{33}^{(\dbl)} 
    &= 
    \frac{\pi ^2-6 \rho_0^2}{6 \alpha^2}
    \\ \sigma_{34}^{(\dbl)}
    &= 
    \frac{\pi^2}{6\alpha^2}-\frac{\rho_1}{\alpha^2}
    \\ \sigma_{44}^{(\dbl)} 
    &=
    \frac{\pi^2}{6\alpha^2},
\end{align*}
}where $\rho_0=\rho(0)$ and  $\rho_1 = \int_0^1z^{-1}[\rho_0 - \rho(z)]\diff z \ge 0$.

Moreover, we have
\begin{align} \label{eq:bound_on_rho_1}
    \rho_0 + (1-\rho_0)\log(1-\rho_0) \le \rho_1 \le \rho_0
\end{align}
with equality on the left for $\rho(z) = \min(\rho_0, 1-z)$ and equality on the right for $\rho(z) = \rho_0 \cdot (1-z)$. Finally, if $\rho=\rho_\indi$, the matrix simplifies to
\begin{align*}
    \big(\sigma_{ij}^{(\dbl)} \big)_{i,j=1}^4 
    =\frac{1}{\alpha^2}
    \begin{pmatrix}
        3-10\gamma+2\gamma^2+\pi^2 & \alpha(2\gamma-5) & \frac{\pi^2}{3}-1-\gamma & \frac{\pi^2}{3}-1-\gamma\\
        \alpha(2\gamma-5) & 2\alpha^2 & -\alpha & -\alpha \\
        \frac{\pi^2}{3}-1-\gamma & -\alpha & \frac{\pi^2}{6}-1 & \frac{\pi^2}{6}-1\\
        \frac{\pi^2}{3}-1-\gamma & -\alpha & \frac{\pi^2}{6}-1 & \frac{\pi^2}{6}
    \end{pmatrix} .
\end{align*}
\end{lemma}

\begin{proof}[Proof of Lemma~\ref{lem:cov_disjoint_general}]
The claimed (in)equalities in \eqref{eq:bound_on_rho_1} are immediate. 
The assertion for $\rho=\rho_\indi$ follows from the general formulas by a straightforward calculation, noting that $\rho_{\indi,0}=\rho_{\indi,1}=1$.
Next, $\sigma_{44}^{(\dbl)}=\pi^2/(6\alpha^2)$ by \cite[Lemma B.2]{Bücher2018-disjoint}. All expressions with $i,j \in \{1,2,3\}$ only depend on the marginal distribution $H^{(2)}_{\rho, \alpha,1}$ from \eqref{eq:w_margCDF2} and can be calculated explicitly using a CAS. The remaining entries $\sigma_{i4}^{(\dbl)}$ with $i=1,2,3$ are more challenging and require some manipulation before they can be evaluated using a CAS. 

First, by Hoeffding's covariance formula for absolutely continuous functions, \cite[Theorem 3.1]{Lo2017},
\begin{align*}
    \sigma_{i4}^{(\dbl)} &= \int_0^\infty \Big\{ \int_0^x \Big( H_{\rho,\alpha,1}(x,y) -H_{\rho,\alpha,1}^{(1)}(x)H^{(2)}_{\rho,\alpha,1}(y)\Big) f_i'(y)\diff y \\
    &\hspace{2cm}+H_{\rho,\alpha,1}^{(1)}(x)\int_x^\infty\Big( 1-H_{\rho,\alpha,1}^{(2)}(y)\Big) f_i'(y)\diff y \Big\} \frac{1}{x}\diff x=:\int_0^\infty (I_{1i}+I_{2i})(x)\frac{\diff x}{x}. 
\end{align*}
We have
\begin{align*}
    I_{1i}(x) &= \int_0^x \Big[\exp(-y^{-\alpha})\big\{1+\rho\big((y/x)^\alpha\big)y^{-\alpha}\big\} -\exp(-x^{-\alpha}-y^{-\alpha})\big\{1+\rho_0y^{-\alpha}\big\}\Big] f_i'(y)\diff y 
    \\&= 
    (1-\exp(-x^{-\alpha}))\int_0^x \exp(-y^{-\alpha}) f_i'(y)\diff y 
    \\& \hspace{1cm}
    +\int_0^x \Big[\rho\big((y/x)^\alpha\big) -\rho_0\exp(-x^{-\alpha})\Big]\exp(-y^{-\alpha})y^{-\alpha} f_i'(y)\diff y =: J_{1i}(x)+J_{2i}(x)\\
    I_{2i}(x)
    &=
    \exp(-x^{-\alpha})\int_x^\infty\Big( 1-\exp(-y^{-\alpha})\big\{1+\rho_0y^{-\alpha}\big\}\Big) f_i'(y)\diff y
    \\&=
    \exp(-x^{-\alpha})\int_x^\infty\big( 1-\exp(-y^{-\alpha})\big) f_i'(y)\diff y 
    \\& \hspace{1cm}
    -\rho_0\exp(-x^{-\alpha})\int_x^\infty\exp(-y^{-\alpha})y^{-\alpha} f_i'(y)\diff y =:J_{3i}(x)+J_{4i}(x).
\end{align*}
Note that $J_{1i}$ and $J_{3i}$ do not depend on $\rho$, whereas $J_{2i}$ and $J_{4i}$ do. Thus, we split
\begin{align*}
    \sigma_{i4}^{(\dbl)} &= \int_0^\infty (J_{1i}+J_{3i})(x) \frac{\diff x}{x}+\int_0^\infty (J_{2i}+J_{4i})(x)\frac{\diff x}{x}. 
\end{align*}
and evaluate both summands separately.

\medskip
\noindent
\underline{\textit{The first summand}.}
Starting with the first summand, we have for $i=3$, using the CAS, 
\begin{align} \label{eq:J13-J33}
    \int_0^\infty (J_{13}+J_{33})(x)\frac{\diff x}{x} 
    \caseq 
    \frac{\pi ^2}{6 \alpha^2}.
\end{align}
The terms for $i=1,2$ turn out to be a bit trickier, as they involve $\alpha_1=\alpha\varpi$ appearing in $f_i$. First,
\begin{align*}
    \int_0^\infty (J_{12}+J_{32})(x)\frac{\diff x}{x}
    \caseq 
    -\alpha^{-1}\Gamma(\varpi)+\varpi\int_0^\infty \exp(-x^{-\alpha})   \Gamma (\varpi) -   \Gamma (\varpi,x^{-\alpha}) \frac{\diff x}{x},
\end{align*}
with $\Gamma(x,y)=\int_y^\infty t^{x-1}e^{-t} \diff t$ being the incomplete Gamma function. Substituting $u=x^{-\alpha}$ with $\diff x / x=-\alpha^{-1}\diff u/u$, we obtain
\begin{align} \label{eq:J12-J32}
    \int_0^\infty (J_{12}+J_{32})(x)\frac{\diff x}{x} 
    &= \nonumber
    -\alpha^{-1}\Gamma(\varpi)+\varpi\alpha^{-1}\int_0^\infty \exp(-u)   \Gamma (\varpi) - \Gamma (\varpi,u) \frac{\diff u}{u}
    \\&\caseq 
    -\alpha^{-1}\Gamma(\varpi)-\varpi\alpha^{-1}\Gamma (\varpi) (\gamma+\psi ^{(0)}(\varpi) )
\end{align}
with $\psi^{(0)}(z)=\Gamma'(z)/\Gamma(z)$ the digamma function.

The case $i=1$ remains challenging. Recalling  the definition of $f_i$ in \eqref{eq:fct_H}, we have
\begin{align*}
    f_1'(y)=y^{-\alpha\varpi-1}(1-\alpha\varpi \log y) = -\alpha^{-1}\varpi^{-1} f_2'(y)-\alpha\varpi y^{-\alpha\varpi-1}\log y,
\end{align*}
so that
\begin{align*}
    \int_0^\infty (J_{11}+J_{31})(x)\frac{\diff x}{x}
    &=
    -\alpha^{-1}\varpi^{-1}\int_0^\infty (J_{12}+J_{32})(x)\frac{\diff x}{x}
    \\& \hspace{1cm} 
    -\alpha\varpi\int_0^\infty \Big[ (1-\exp(-x^{-\alpha}))\int_0^x \exp(-y^{-\alpha}) y^{-\alpha\varpi-1}\log y\diff y
    \\&\hspace{2cm}
    +\exp(-x^{-\alpha})\int_x^\infty\big( 1-\exp(-y^{-\alpha})\big) y^{-\alpha\varpi-1}\log y\diff y \Big]\frac{\diff x}{x}.
\end{align*}
Let us begin with tackling the inner integrals. Write $\Gamma_1(x,y)=\partial\Gamma(x,y)/ \partial x$.
\begin{align*}
    \int_0^x \exp(-y^{-\alpha}) y^{-\alpha\varpi-1}\log y\diff y 
    &\caseq 
    -\alpha^{-2}\Gamma_1(\varpi,x^{-\alpha})\\
    \int_x^\infty\big( 1-\exp(-y^{-\alpha})\big) y^{-\alpha\varpi-1}\log y\diff y
    & \caseq 
    \alpha^{-2}\big[\varpi^{-2}x^{-\alpha \varpi} (\alpha \varpi \log (x)+1) \\
    & \hspace{2cm}+\Gamma (\varpi) \psi ^{(0)}(\varpi)-\Gamma_1(\varpi,x^{-\alpha}) \big].
\end{align*}
Let us split the outer integral into two parts, the first one being
\begin{align*}
    \int_0^\infty \exp(-x^{-\alpha})\alpha^{-2}\varpi^{-2}x^{-\alpha \varpi} (\alpha \varpi \log (x)+1)\frac{\diff x}{x}
    \caseq 
    \alpha^{-3}\varpi^{-2}\Gamma (\varpi) [1-\varpi \psi ^{(0)}(\varpi)].
\end{align*}
It remains to calculate
\begin{align*}
    &\phantom{{}={}}\alpha^{-2}\int_0^\infty \exp(-x^{-\alpha}) \big[\Gamma (\varpi) \psi ^{(0)}(\varpi)-\Gamma_1(\varpi,x^{-\alpha})\big]-\Gamma_1(\varpi,x^{-\alpha})(1-\exp(-x^{-\alpha}))\frac{\diff x}{x}\\
    &=\alpha^{-2}\int_0^\infty \exp(-x^{-\alpha})\Gamma (\varpi) \psi ^{(0)}(\varpi)-\Gamma_1(\varpi,x^{-\alpha})\frac{\diff x}{x}\\
    &=\alpha^{-3}\int_0^\infty \exp(-u)\Gamma (\varpi) \psi ^{(0)}(\varpi)-\Gamma_1(\varpi,u)\frac{\diff u}{u}\\
    &\caseq 
    -\alpha^{-3}\Gamma (\varpi) \{\psi ^{(0)}(\varpi) (\psi ^{(0)}(\varpi)+\gamma )+\psi ^{(1)}(\varpi)\},
\end{align*}
where we used the substitution $u=x^{-\alpha}$ with $\diff x / x=-\alpha^{-1}\diff u/u$ again.
Assembling terms,
\begin{align}
    \int_0^\infty (J_{11}+J_{31})(x) \frac{\diff x}{x}
    &= \nonumber
    \alpha^{-1}\varpi^{-1}\{\alpha^{-1}\Gamma(\varpi)+\varpi\alpha^{-1}\Gamma (\varpi) (\gamma+\psi ^{(0)}(\varpi) )\}
    \\& \hspace{1cm}\nonumber  
    - \alpha^{-2}\varpi^{-1}\Gamma (\varpi) [1-\varpi \psi ^{(0)}(\varpi)]
    \\& \hspace{1cm} \nonumber
    +\varpi\alpha^{-2}\Gamma (\varpi) \{\psi ^{(0)}(\varpi) (\psi ^{(0)}(\varpi)+\gamma )+\psi ^{(1)}(\varpi)\}\\
    &= \label{eq:J11-J31}
    \alpha^{-2}\Gamma (\varpi) \big\{\gamma+2\psi ^{(0)}(\varpi)+\varpi \psi ^{(0)}(\varpi)^2+\varpi\gamma \psi ^{(0)}(\varpi)+\varpi \psi ^{(1)}(\varpi)\big\}.
\end{align}

\medskip
\noindent\underline{\textit{The second summand}.}
We now consider the second summand, which is given by
\begin{align}
    (J_{2i}+J_{4i})(x) 
    &= \nonumber
    \int_0^x \rho\big((y/x)^\alpha\big) \exp(-y^{-\alpha})y^{-\alpha} f_i'(y)\diff y
    \\&\hspace{2cm} \nonumber
    -\rho_0 \exp(-x^{-\alpha})\int_0^\infty \exp(-y^{-\alpha})y^{-\alpha} f_i'(y)\diff y
    \\&= \label{eq:j2j4-j5j6}
    J_{5i}(x) - \rho_0 \exp(-x^{-\alpha}) J_{6i}(x),
\end{align}
where, substituting $z=(y/x)^\alpha$,
\begin{align}
    J_{5i}(x) 
    &:= \nonumber
    \int_0^x \rho\big((y/x)^\alpha\big) \exp(-y^{-\alpha})y^{-\alpha} f_i'(y)\diff y
    \\&\phantom{:}= \label{eq:J5}
    \alpha^{-1}\int_0^1 \rho(z) \exp(-x^{-\alpha}z^{-1})f_i'(xz^{1/\alpha})x^{1-\alpha}z^{1/\alpha-2}\diff z
\end{align}
and where 
\begin{align}\label{eq:J6}
    J_{6i}(x) := \int_0^\infty \exp(-y^{-\alpha})y^{-\alpha} f_i'(y)\diff y 
    \caseq 
    \begin{cases*}
        \alpha^{-1}\Gamma(1+\varpi) \{1+\varpi \psi ^{(0)}(1+\varpi)\}
        , & $i=1$, \\
        -\varpi \Gamma(1+\varpi), & $i=2$, \\
        \alpha^{-1}, & $i=3$.
    \end{cases*}
\end{align}

\medskip
\noindent\underline{\textit{The second summand for $i=3$.}}
Since $f_3'(y)=1/y$, the previous three displays yield
\begin{align*}
    \int_0^\infty (J_{23}+J_{43})(x) \frac{\diff x}{x} &=\alpha^{-1}\int_0^\infty \int_0^1 \rho(z) \exp(-x^{-\alpha}z^{-1})x^{-\alpha} z^{-2}\diff z-\rho_0\exp(-x^{-\alpha}) \frac{\diff x}{x}
\end{align*}
The inner integral can be expressed as follows
\begin{align*}
     \int_0^1 \rho(z) \exp(-x^{-\alpha}z^{-1})x^{-\alpha} z^{-2}\diff z
     &=
     \int_0^1 \rho(z) \frac{\partial}{\partial z} \exp(-x^{-\alpha}z^{-1})\diff z\\
      &= \exp(-x^{-\alpha}z^{-1})\rho(z)\bigg|_{z=0}^1-\int_0^1  \exp(-x^{-\alpha}z^{-1})\rho'(z)\diff z \\
      &=
      -\int_0^1  \exp(-x^{-\alpha}z^{-1})\rho'(z)\diff z.
\end{align*}
It follows that
\begin{align*}
    \int_0^\infty (J_{23}+J_{43})(x) \frac{\diff x}{x} &=-\alpha^{-1}\int_0^\infty \int_0^1 \exp(-x^{-\alpha}z^{-1}) \rho'(z) \diff z+ \rho_0\exp(-x^{-\alpha})\frac{\diff x}{x}\\
    &\caseq 
    -\alpha^{-1}\int_0^\infty \int_0^1 \exp(-x^{-\alpha}z^{-1})\big( \rho'(z)+\rho_0\big)\diff z+ \rho_0x^{-\alpha}\mathrm{Ei}(-x^{-\alpha})\frac{\diff x}{x}\\
    &\caseq 
    -\alpha^{-1}\int_0^\infty \int_0^1 \exp(-x^{-\alpha}z^{-1})\big( \rho'(z)+\rho_0\big)\diff z\frac{\diff x}{x}-\frac{\rho_0}{\alpha^2},
\end{align*}
where, for $z>0$, $\mathrm{Ei}(-z)= -E_1(z)=-\int_1^\infty e^{-tz}/t \diff t$ denotes the exponential integral. In view of the fact that $\int_0^1 \big(\rho'(z)+\rho_0 \big)\diff z=0$, we have
\begin{align*}
    \int_0^\infty \int_0^1 \exp(-x^{-\alpha})\big(\rho'(z)+\rho_0 \big)\diff z\frac{\diff x}{x}=0
\end{align*}
and we may add this as a `productive zero' to interchange the order of integration to see
\begin{align}
    &\phantom{{}={}}\nonumber
    \int_0^\infty \int_0^1 \exp(-x^{-\alpha}/z)\big(\rho'(z)+\rho_0 \big)\diff z\frac{\diff x}{x}
    \\&=\nonumber
    \int_0^\infty \int_0^1 [\exp(-x^{-\alpha}/z)-\exp(-x^{-\alpha})]\big(\rho'(z)+\rho_0 \big)\diff z\frac{\diff x}{x}\\
    &=\nonumber
    \int_0^1 \int_0^\infty [\exp(-x^{-\alpha}/z)-\exp(-x^{-\alpha})]\frac{\diff x}{x}\big(\rho'(z)+\rho_0 \big)\diff z\\
    & \caseq  \nonumber
    \alpha^{-1}\int_0^1\log(z)\big(\rho'(z)+\rho_0 \big)\diff z
    \\&=\nonumber
    -\rho_0\alpha^{-1}+\alpha^{-1}\int_0^1\log(z)\rho'(z)\diff z\\
    &=\nonumber
    -\rho_0\alpha^{-1}+\alpha^{-1}\Big[\log(z)\big(\rho(z)-\rho_0\big)\Big]_0^1
    -\rho_0\alpha^{-1}\int_0^1\frac{\rho(z)-\rho_0}{ z}\diff z\\
    &=\label{eq:rho1-calc}
    -\rho_0\alpha^{-1}-\alpha^{-1}\int_0^1\frac{\rho(z)-\rho_0}{z}\diff z.
\end{align}
Here, the last identity holds because $\lim_{z\downarrow0}\log(z)(\rho_0-\rho(z))=0$ as a consequence of the fact that $\rho_0(1-z) \le \rho(z) \le \rho_0$. In total, we get
\begin{align*}
    \int_0^\infty (J_{23}+J_{43})(x) \frac{\diff x}{x}=\frac{1}{\alpha^2}\int_0^1\frac{\rho(z)-\rho_0}{ z}\diff z= -\frac{\rho_1}{\alpha^2}.
\end{align*}
Together with \eqref{eq:J13-J33} this implies the claimed formula for $\sigma_{34}^{(\dbl)}$.

\medskip
\noindent\underline{\textit{The second summand for $i=2$.}}
Since $f'_2(y)=-\varpi\alpha y^{-\varpi\alpha-1}$, the term $J_{52}$ from \eqref{eq:J5} can be written as
\begin{align}
    J_{52}(x)
    &= \label{eq:J52}
    \alpha^{-1}\int_0^1 \rho(z) \exp(-x^{-\alpha}z^{-1})f_2'(xz^{1/\alpha})x^{1-\alpha}z^{1/\alpha-2}\diff z
    \\&= \nonumber
    -\varpi\int_0^1 \rho(z) \exp(-x^{-\alpha}z^{-1})x^{-(1+\varpi)\alpha}z^{-2-\varpi}\diff z
    \\& \caseq \nonumber
    -\varpi\Big[\rho(z) \Gamma(\varpi+1,x^{-\alpha}/z)\Big]_{z=0}^1 
    +\varpi \int_0^1  \Gamma(\varpi+1,x^{-\alpha}/z) \rho'(z)\diff z
    \\& \caseq \nonumber
    \varpi \int_0^1  \Gamma(\varpi+1,x^{-\alpha}/z)\rho'(z)\diff z.
\end{align}
Thus, by \eqref{eq:J6}, 
\begin{align} \label{eq:int_J2J4}
    \int_0^\infty (J_{22}+J_{42})(x)\frac{\diff x}{x} 
    &= \nonumber
    \int_0^\infty J_{52}(x) - \rho_0 \exp(-x^{-\alpha}) J_{62}(x)\frac{\diff x}{x} 
    \\&=\nonumber
    \varpi\int_0^\infty  \int_0^1 \Gamma(\varpi+1,x^{-\alpha}/z) \rho'(z) \diff z+ \rho_0 \Gamma(\varpi+1)\exp(-x^{-\alpha}) \frac{\diff x}{x}
    \\&=
    \varpi \int_0^\infty (J_{72} + J_{82})(x) \frac{\diff x}{x},
\end{align}
where
\begin{align*}
    J_{72}(x) &= \int_0^1 \Gamma(\varpi+1,x^{-\alpha}/z) \big(\rho'(z)+\rho_0\big) \diff z, \\
    J_{82}(x) &= \rho_0\Big[ \Gamma(\varpi+1)\exp(-x^{-\alpha}) - \int_0^1 \Gamma(\varpi+1,x^{-\alpha}/z) \diff z \Big].
\end{align*}
The second integral can be calculated explicitly using a CAS: first,
\begin{align*}
    J_{82}(x)
    & \caseq 
    \rho_0 \Big[\Gamma(\varpi+1)\exp(-x^{-\alpha}) - \varpi^{-1}\exp(-x^{-\alpha})x^{-(1+\varpi)\alpha} 
    \\ &\hspace{5cm}
    - \varpi^{-1}(\varpi-x^{-\alpha})\Gamma(\varpi+1,x^{-\alpha}) \Big],
\end{align*}
which implies
\begin{align} \label{eq:int_J7}
    \int_0^\infty J_{82}(x) \frac{\diff x}{x}
    \caseq 
    \rho_0\alpha^{-1}\varpi [1-\gamma-\psi^{(0)}(\varpi+1)] \Gamma (\varpi).
\end{align}
Next, regarding the integral over $J_{72}$, we aim at applying Fubini's theorem, which requires some preparation. First, recall that $\int_0^1 \ell(x)\big(\rho'(z) +\rho_0 \big)\diff z=0$ for any expressions $\ell(x)$ not depending on $z$. Choosing 
\begin{align*}
\ell(x)
&= 
\int_0^1 \Gamma(\varpi+1,x^{-\alpha}/z)\diff z
\\& \caseq 
   \varpi^{-1}\exp(-x^{-\alpha})x^{-(1+\varpi)\alpha} 
    + \varpi^{-1}(\varpi-x^{-\alpha})\Gamma(\varpi+1,x^{-\alpha})
\end{align*}
and adding this as a productive zero, we obtain that
\begin{align*}
    \int_0^\infty J_{72}(x) \frac{\diff x}{x}
    &=
    \int_0^\infty \int_0^1 \big[ \Gamma(\varpi+1,x^{-\alpha}/z) - \ell(x)\big] \big(\rho'(z)+\rho_0\big) \diff z  \frac{\diff x}{x}
    \\&=
     \int_0^1 \int_0^\infty\big[ \Gamma(\varpi+1,x^{-\alpha}/z) - \ell(x)\big]\frac{\diff x}{x} \big(\rho'(z)+\rho_0\big) \diff z  
\end{align*}
by Fubini's theorem.
Treating the remaining inner integral, we start with substituting $u=x^{-\alpha}$
\begin{align*}
    &\phantom{{}={}}
    \int_0^\infty\big[ \Gamma(\varpi+1,x^{-\alpha}/z) - \ell(x)\big]\frac{\diff x}{x}
    \\&=
    \int_0^\infty \Gamma(\varpi+1,u/z)-\varpi^{-1}\{e^{-u}u^{1+\varpi}+(\varpi-u)\Gamma(\varpi+1,u)\}\frac{\diff u}{\alpha u}
    \\&=
    \alpha^{-1}\varpi^{-1} \int_0^\infty \Gamma(\varpi+1,u)\diff u+\int_0^\infty \Gamma(\varpi+1,u/z)-\Gamma(\varpi+1,u)\frac{\diff u}{\alpha u} -\alpha^{-1}\Gamma(\varpi)\\
    & \caseq 
    \alpha^{-1}\{\varpi^{-1} \Gamma(\varpi+2)+\Gamma(\varpi+1)\log(z) -\Gamma(\varpi)\}.
\end{align*}
As a consequence,
\begin{align*}
    \int_0^\infty J_{72}(x) \frac{\diff x}{x}
    &=
    \alpha^{-1} \int_0^1\{\varpi^{-1} \Gamma(\varpi+2)+\Gamma(\varpi+1)\log(z) -\Gamma(\varpi)\}\big(\rho'(z)+\rho_0\big) \diff z
    \\&=
    \alpha^{-1} \Gamma(\varpi+1)\int_0^1\log(z) \big(\rho'(z)+\rho_0\big) \diff z
    \\&=
    -\alpha^{-1}\Gamma(\varpi+1) (\rho_0- \rho_1) = -\alpha^{-1}\varpi\Gamma(\varpi) (\rho_0- \rho_1)
\end{align*}
where the last line follows as in \eqref{eq:rho1-calc}. Together with \eqref{eq:int_J2J4} and \eqref{eq:int_J7}, we obtain that
\begin{align} \label{eq:j2j4-res}
    \int_0^\infty (J_{22}+J_{42})(x)\frac{\diff x}{x} 
    &= \nonumber
    \rho_0\alpha^{-1}\varpi^2 [1-\gamma-\psi^{(0)}(\varpi+1)] \Gamma (\varpi) - \alpha^{-1}\varpi^2\Gamma(\varpi)(\rho_0-\rho_1)
    \\&=
    -\alpha^{-1}\varpi^2 [\rho_0\gamma+\rho_0\psi^{(0)}(\varpi+1)-\rho_1] \Gamma (\varpi).
\end{align}
Together with \eqref{eq:J12-J32} this implies the assertion about  $\sigma_{24}^{(\dbl)}$.

\medskip
\noindent\underline{\textit{The second summand for $i=1$.}}
Since $f'_1(y)=y^{-\varpi\alpha-1} ( 1-\varpi\alpha \log y)$, the term $J_{51}$ from \eqref{eq:J5} can be written as
\begin{align*}
    J_{51}(x)
    &=
    \alpha^{-1}\int_0^1 \rho(z) \exp(-x^{-\alpha}z^{-1})f_1'(xz^{1/\alpha})x^{1-\alpha}z^{1/\alpha-2}\diff z
    \\&= 
    \alpha^{-1}\int_0^1 \rho(z)\exp(-x^{-\alpha}z^{-1})x^{-(1+\varpi)\alpha}z^{-2-\varpi}\{1+\varpi\log (x^{-\alpha}z^{-1})\}\diff z
    \\&=
    -\alpha^{-1}\varpi^{-1} J_{52}(x) + \alpha^{-1} \varpi\int_0^1 \rho(z)\exp(-x^{-\alpha}z^{-1})x^{-(1+\varpi)\alpha}z^{-2-\varpi}\log (x^{-\alpha}z^{-1})\diff z,
\end{align*}
with $J_{52}(x)$ from \eqref{eq:J52}. Next, using partial integration and properties of the incomplete gamma function, 
\begin{align*}
    &\phantom{{}={}}
    \int_0^1 \rho(z)\exp(-x^{-\alpha}z^{-1})x^{-(1+\varpi)\alpha}z^{-2-\varpi}\log (x^{-\alpha}z^{-1})\diff z
    \\ & \caseq 
    \Big[\rho(z)\partial_\varpi \Gamma(\varpi+1,x^{-\alpha}/z)\Big]_{z=0}^1-\int_0^1\partial_\varpi \Gamma(\varpi+1,x^{-\alpha}/z) \rho'(z)\diff z \\
    &\caseq 
    -\int_0^1\partial_\varpi\Gamma(\varpi+1,x^{-\alpha}/z)\rho'(z)\diff z.
\end{align*}
Overall, by \eqref{eq:j2j4-j5j6} and \eqref{eq:J6},
\begin{align*}
    &\phantom{{}={}} 
    \int_0^\infty (J_{21}+J_{41})(x)  \frac{\diff x}x
    \\&=
    - \alpha^{-1}\int_0^\infty \Big[ \varpi^{-1} J_{52}(x) + \varpi\int_0^1\partial_\varpi\Gamma(\varpi+1,x^{-\alpha}/z)\rho'(z)\diff z \\
    &\hspace{3cm} + \rho_0  \exp(-x^{\alpha})\Gamma(1+\varpi) \big\{1+\varpi \psi^{(0)}(1+\varpi) \big\} \Big] \frac{\diff x}x
    \\&=
    - \alpha^{-1}\varpi^{-1}\int_0^\infty \Big[  J_{52}(x) - \rho_0 \exp(-x^{-\alpha})J_{62}(x) \Big] \frac{\diff x}x
    \\&\hspace{.5cm}
    - \alpha^{-1}\varpi \int_0^\infty\Big[ \int_0^1\partial_\varpi\Gamma(\varpi+1,x^{-\alpha}/z)\rho'(z)\diff z + \rho_0  \exp(-x^{-\alpha})\Gamma(1+\varpi) \psi^{(0)}(1+\varpi) \Big] \frac{\diff x}x
\end{align*}
where we have used that $J_{62}(x)=-\varpi \Gamma(1+\varpi)$ by \eqref{eq:J6}. We have seen before, see \eqref{eq:int_J2J4} and \eqref{eq:j2j4-res}, that 
\begin{align*}
    \int_0^\infty \Big[  J_{52}(x) - \rho_0 \exp(-x^{-\alpha})J_{62}(x) \Big] \frac{\diff x}x
    &=
    \int_0^\infty (J_{22}+J_{42})(x)  \frac{\diff x}x
    \\&=
    -\alpha^{-1}\varpi^2 [\rho_0\gamma+\rho_0\psi^{(0)}(\varpi+1)-\rho_1] \Gamma (\varpi)
    \\&=
    -\alpha^{-1}\varpi[\rho_0\gamma+\rho_0\psi^{(0)}(\varpi+1)-\rho_1] \Gamma (\varpi+1),
\end{align*} 
where the last identity follows from $\varpi \Gamma(\varpi)=\Gamma(\varpi+1)$.
It remains to calculate the second integral in the penultimate display; in view of $\Gamma(\varpi+1)\psi^{(0)}(\varpi+1)=\Gamma'(1+\varpi)$, \eqref{eq:int_J2J4} and \eqref{eq:j2j4-res}, it can be written as
\begin{align*}
    &\phantom{{}={}}\int_0^\infty \int_0^1\partial_\varpi\Gamma(\varpi+1,x^{-\alpha}/z) \rho'(z) \diff z + \rho_0 \exp(-x^{-\alpha})\Gamma(\varpi+1)\psi^{(0)}(\varpi+1)\frac{\diff x}{x}\\
    &=\frac{\partial}{\partial \varpi} \int_0^\infty \int_0^1\Gamma(\varpi+1,x^{-\alpha}/z)\rho'(z)\diff z + \rho_0\Gamma(\varpi+1)\exp(-x^{-\alpha})\frac{\diff x}{x}
    \\&= 
    \frac{\partial}{\partial \varpi} \int_0^\infty\varpi^{-1} (J_{22}+J_{42})(x)  \frac{\diff x}x
    \\&= 
    -\frac{\partial}{\partial \varpi} \Big[ \alpha^{-1} [\rho_0\gamma+\rho_0\psi^{(0)}(\varpi+1)-\rho_1] \Gamma (\varpi+1) \Big]
    \\&=
    -\alpha^{-1} \Big[ \rho_0 \psi^{(1)}(\varpi+1)\Gamma(\varpi+1) + [\rho_0\gamma+\rho_0\psi^{(0)}(\varpi+1)-\rho_1] \Gamma'(\varpi+1) \Big]
    \\&=
    -\alpha^{-1}\Gamma(\varpi+1) \Big[ \rho_0 \psi^{(1)}(\varpi+1) + [\rho_0\gamma+\rho_0\psi^{(0)}(\varpi+1)-\rho_1] \psi^{(0)}(\varpi+1) \Big]
\end{align*}
Assembling terms, we get
\begin{align*}
    &\phantom{{}={}}
    \int_0^\infty (J_{21}+J_{41})(x) \frac{\diff x}{x} 
    \\&=
    \alpha^{-2}[\rho_0\gamma+\rho_0\psi^{(0)}(\varpi+1)-\rho_1] \Gamma (\varpi+1) 
    \\&\hspace{1cm}
    +\alpha^{-2} \varpi\Gamma(\varpi+1) \Big[ \rho_0 \psi^{(1)}(\varpi+1) 
    + [\rho_0\gamma+\rho_0\psi^{(0)}(\varpi+1)-\rho_1] \psi^{(0)}(\varpi+1) \Big] 
    \\&=
    \alpha^{-2} \Gamma(\varpi+1) \Big[ [\rho_0\gamma+\rho_0\psi^{(0)}(\varpi+1)-\rho_1] \big\{ 1+\varpi \psi^{(0)}(\varpi+1)\big\} + \rho_0 \varpi\psi^{(1)}(\varpi+1) \Big]
\end{align*}
Together with \eqref{eq:J11-J31}, this implies the claimed formula for $\sigma_{14}^{(\dbl)}$.

\end{proof}

\begin{lemma}[Asymptotic covariance for the sliding block top-two estimator under independence]\label{lem:cov_sl}
    Suppose $(X, Y, \tilde X, \tilde Y)$ is a random vector whose bivariate cdfs needed for evaluating the following covariances are given by $K_{\rho,\alpha,\zeta}$ from \eqref{eq:sl_distr}.
    Let $(f_1,f_2,f_3,f_4)$ be defined as in \eqref{eq:fct_H} with $\alpha_1=\alpha\varpi$, that is, 
    \[
    f_1(x,y) = y^{-\alpha\varpi}\log y,  \quad
    f_2(x,y) = y^{-\alpha\varpi}, \quad
    f_3(x,y) = \log y,  \quad
    f_4(x,y) = \log x.
    \]
    Then, for $i,j \in \{1, \dots, 4\}$,
    \[
    s_{ij}  
    := s_{ij}(\alpha) 
    := 
    \int_0^1 \Cov_{K_{\rho,\alpha,\zeta}}\big(f_i(X,Y), f_j(\tilde X, \tilde Y)\big) \diff\zeta,
    \]
    may be evaluated explicitly. Precise formulas are provided in the proof and a Mathematica notebook. For $\rho=\rho_{\indi}$, the formulas simplify to {\small
\begin{align*}
s_{11} &=
    \frac{-126 \zeta (3)-174+\pi ^2 (11+24 \log (2))-12 \gamma  \left(\pi ^2-23+(11-4 \log (2)) \log (8)\right)}{12\alpha^2} 
    \\ &\hspace{4.2cm}
    + \frac{6 \log (2) (46+\log (2) (\log (256)-33))+18 \gamma ^2 (\log (256)-5)}{12\alpha^2} \\
s_{12} &= 
    \frac{11.5-\pi^2/2 + 6\log(2)^2-16.5\log(2) + 1.5 \, \gamma \, (8\log(2) - 5)}{ \alpha} \\
s_{13} &= \frac{ 4-3.5 \zeta(3) + 3.5\log(2)^2 -9\log(2)  + 7 \pi^2/12 -  \gamma \left(\pi^2/3 + 2 - 7 \log(2)\right)}{\alpha^2} \\
s_{14} &=
    \frac{- 7 \zeta (3) / {4}+\pi ^2/{3}+4+\log (2) \{ \log (8)-8 \} +\gamma  \{ -{\pi ^2}/{6}-3+\log (64)\} }{\alpha^2} \\
s_{22} &= 12\log(2) - 15/2 \\
s_{23} &= -\frac{\pi^2/3 + 2 - 7 \log (2)}{ \alpha} \\
s_{24} &= -\frac{\pi^2/6 + 3 - 6 \log(2)}{ \alpha} \\
s_{33} &= -\frac{\pi^2/6 + 5 - 10 \log (2)}{\alpha^2} \\
s_{34} &= -\frac{\pi^2/12 + 3 - 6 \log(2)}{\alpha^2} \\
s_{44} &= 4\log(2)-2 
\end{align*}}
\end{lemma}

\begin{proof}[Proof of Lemma~\ref{lem:cov_sl}]
Throughout, ${}_pF_q(a_1, \dots, a_p; b_1, \dots, b_q;z)$ denotes the generalized hypergeometric function. Its regularized version is denoted by ${}_p\tilde F_q(a_1, \dots, a_p; b_1, \dots, b_q;z) ={}_pF_q(a_1, \dots, a_p; b_1, \dots, b_q;z) / [ \Gamma(b_1) \cdots \Gamma(b_q) ]$.

Moreover, we write $f_j(y) = f_j(x,y)$ for $j \in \{1,2,3\}$ and $f_4(x)=f_4(x,y)$.

The entry $s_{44}$ is known from \cite{Bücher2018-sliding}. For the others, let us start by calculating the entries with $i,j\neq 4$. 

\medskip
\noindent\underline{\textit{The entries $s_{ij}$ with $i,j \in \{1,2,3\}$.}}
Unlike in the disjoint case, these are no longer just moments of a univariate distribution, so we need to apply Hoeffding's covariance formula here already. We have 
\begin{align*}
    s_{ij}
    &= \int_0^\infty \Big\{ \int_0^y \Big( K_{\rho,\alpha,\zeta}(\infty,x,\infty,y) -H_{\rho,\alpha,1}^{(2)}(x)H^{(2)}_{\rho,\alpha,1}(y)\Big) f_i'(x)\diff x 
    \\ & \hspace{2cm}
    +\int_y^\infty\Big( K_{\rho,\alpha,\zeta}(\infty,x,\infty,y)-H_{\rho,\alpha,1}^{(2)}(x)H_{\rho,\alpha,1}^{(2)}(y)\Big) f_i'(x)\diff x \Big\} f_j'(y)\diff y
    \\&=\int_0^\infty (I_{1i}(y)+I_{2i}(y))f_j'(y)\diff y, 
\end{align*}
where, by Lemma~\ref{lem:JoiWeaKon},
\begin{align*}
    I_{1i}(y) 
    &= 
    \int_0^y \Big(\exp(-x^{-\alpha}-\zeta y^{-\alpha})\Big\{1+\zeta \rho_0 y^{-\alpha}+\rho_0 x^{-\alpha}+ \\
    &\hspace{2cm}\zeta\rho_0 x^{-\alpha}y^{-\alpha}\big[\zeta \rho_0+(1-\zeta)\rho\big((x/y)^\alpha\big)\big]\Big\}\\
    &\hspace{3cm} -\exp(-x^{-\alpha}- y^{-\alpha})\big\{1+\rho_0 x^{-\alpha}+\rho_0 y^{-\alpha}+\rho_0^2 x^{-\alpha}y^{-\alpha}\big\}\Big)f_i'(x) \diff x, \\
    I_{2i}(y)
    &=
    \int_y^\infty \Big(\exp(-\zeta x^{-\alpha}- y^{-\alpha})\Big\{1+ \rho_0 y^{-\alpha}+\rho_0\zeta x^{-\alpha}+ \\
    &\hspace{2cm} \zeta\rho_0 x^{-\alpha}y^{-\alpha}\big[\zeta \rho_0+(1-\zeta)\rho\big((y/x)^\alpha\big)\big]\Big\}\\
    &\hspace{3cm}-\exp(-x^{-\alpha}- y^{-\alpha})\big\{1+\rho_0 x^{-\alpha}+\rho_0 y^{-\alpha}+\rho_0^2 x^{-\alpha}y^{-\alpha}\big\}\Big)f_i'(x).
\end{align*}

For the evaluation of $I_{1i}$, let us substitute $z=x^\alpha y^{-\alpha}$ with $\diff x = \alpha^{-1}yz^{1/\alpha-1}\diff z$; and for $I_{2i}$, we will substitute $z=x^{-\alpha} y^{\alpha}$ with $\diff x = -\alpha^{-1}yz^{-1/\alpha-1}\diff z$. Further, put $u=y^{-\alpha}$.
\begin{align*}
    I_{1i}(u^{-1/\alpha})&=\alpha^{-1}\int_0^1 \Big[\exp(-u/z-\zeta u)\Big\{1+\zeta \rho_0 u+\rho_0 u/z+\zeta\rho_0 y^{-2\alpha}/z\big[\zeta \rho_0+(1-\zeta)\rho(z)\big]\Big\}\\
    &\hspace{1cm}-\exp(-u/z- u)\big\{1+\rho_0 u/z+\rho_0 u+\rho_0^2 u^2/z\big\}\Big]f_i'((z/u)^{1/\alpha})(z/u)^{1/\alpha}\frac{\diff z}{z} \\
    I_{2i}(u^{-1/\alpha})&=\alpha^{-1}\int_0^1 \Big[\exp(-uz\zeta- u)\Big\{1+ \rho_0 u+\rho_0 uz\zeta+\zeta\rho_0 y^{-2\alpha}z\big[\zeta \rho_0+(1-\zeta)\rho(z)\big]\Big\}\\
    &\hspace{1cm}-\exp(-uz- u)\big\{1+\rho_0 uz+\rho_0 u+\rho_0^2 u^2z\big\}\Big]f_i'((uz)^{-1/\alpha}) (uz)^{-1/\alpha}\frac{\diff z}{z} 
\end{align*}
With the substitution $u=y^{-\alpha}$, $\displaystyle f_j'(y)\diff y = -f_j'(u^{-1/\alpha})\frac{\diff u }{\alpha u^{1+1/\alpha}}$, we can write
\begin{align*}
    s_{ij} &= \alpha^{-2}\int_0^1\int_0^1\int_0^\infty \Big[\exp(-u/z-\zeta u)\Big\{1+\zeta \rho_0 u+\rho_0 u/z+\zeta\rho_0 u^2/z\big[\zeta \rho_0+(1-\zeta)\rho(z)\big]\Big\}\\
    &\hspace{.6cm}-\exp(-u/z- u)\big\{1+\rho_0 u/z+\rho_0 u+\rho_0^2 u^2/z\big\}\Big]f_i'((z/u)^{1/\alpha})(z/u)^{1/\alpha}\\
    &\hspace{.6cm}+\Big[\exp(-uz\zeta- u)\Big\{1+\rho_0 u+\rho_0 uz\zeta +\zeta\rho_0 u^2z\big[\zeta \rho_0+(1-\zeta)\rho(z)\big]\Big\}\\
    &\hspace{.6cm}-\exp(-uz- u)\big\{1+\rho_0 uz+\rho_0 u+\rho_0^2 u^2z\big\}\Big]f_i'((zu)^{-1/\alpha})(zu)^{-1/\alpha}\frac{f_j'(u^{-1/\alpha})\diff u }{u^{1+1/\alpha}}\diff z\diff \zeta.
\end{align*}
Let us split this integral into four parts, one independent from $\rho$, one linear / quadratic in $\rho_0$ and one being an integral over $\rho(z)$. More precisely, let us write $\alpha^2s_{ij}=J_{ij1}+J_{ij2}+J_{ij3}+J_{ij4}$, where
\begin{align*}
    J_{ij1}&= 
    \int_0^1\int_0^1\int_0^\infty \Big[\exp(-u/z-\zeta u)-\exp(-u/z- u)\Big]f_i'((z/u)^{1/\alpha})(z/u)^{1/\alpha}\\
    &\hspace{1cm}+\Big[\exp(-uz\zeta- u)-\exp(-uz- u)\Big]f_i'((zu)^{-1/\alpha})(zu)^{-1/\alpha}\frac{f_j'(u^{-1/\alpha})\diff u }{u^{1+1/\alpha}}\diff z\diff \zeta \\
    J_{ij2}&= \rho_0
    \int_0^1\int_0^1\int_0^\infty \Big[\exp(-u/z-\zeta u)\{\zeta  u+ u/z\}-\exp(-u/z- u)\{u/z+ u\}\Big] \\
    & \hspace{1cm} \times f_i'((z/u)^{1/\alpha})(z/u)^{1/\alpha}
    +\Big[\exp(-uz\zeta- u)\{u+uz\zeta \}-\exp(-uz- u)\{uz+ u\big\}\Big] \\
    & \hspace{1cm} \times f_i'((zu)^{-1/\alpha})(zu)^{-1/\alpha}\frac{f_j'(u^{-1/\alpha})\diff u }{u^{1+1/\alpha}}\diff z\diff \zeta\\
    J_{ij3}&= \rho_0^2
    \int_0^1\int_0^1\int_0^\infty \Big[\exp(-u/z-\zeta u)\zeta^2 u^2/z-\exp(-u/z- u) u^2/z\Big] \\
     &\hspace{1cm} \times f_i'((z/u)^{1/\alpha})(z/u)^{1/\alpha} +\Big[\exp(-uz\zeta- u)\zeta^2 u^2z-\exp(-uz- u) u^2z\Big] \\
      &\hspace{1cm} \times f_i'((zu)^{-1/\alpha})(zu)^{-1/\alpha}\frac{f_j'(u^{-1/\alpha})\diff u }{u^{1+1/\alpha}}\diff z\diff \zeta\\
    J_{ij4}&=\rho_0
    \int_0^1\int_0^1\int_0^\infty \big[\exp(-u/z-\zeta u)\zeta u^2/z(1-\zeta)\rho(z)\big]f_i'((z/u)^{1/\alpha})(z/u)^{1/\alpha}\\
    &\hspace{1cm}+\big[\exp(-uz\zeta- u)\zeta u^2z(1-\zeta)\rho(z)\big]f_i'((zu)^{-1/\alpha})(zu)^{-1/\alpha}\frac{f_j'(u^{-1/\alpha})\diff u }{u^{1+1/\alpha}}\diff z\diff \zeta.
\end{align*}
It is elementary to integrate with respect to $\zeta$ first, which yields
\begin{align*}
    J_{ij1}
    &= \int_0^1\int_0^\infty \Big[\frac{\e^u-1}{u}-1\Big]\e^{-u-u/z}f_i'((z/u)^{1/\alpha})(z/u)^{1/\alpha}\\
    &\hspace{1cm}+\Big[\frac{\e^{uz}-1}{uz}-1\Big]\e^{-uz- u}f_i'((zu)^{-1/\alpha})(zu)^{-1/\alpha}\frac{f_j'(u^{-1/\alpha})\diff u }{u^{1+1/\alpha}}\diff z
    \\
    J_{ij2}
    &= \rho_0\int_0^1\int_0^\infty \frac{\e^u (u+z)-u (u+1) (z+1)-z}{u z}\e^{-u/z- u}f_i'((z/u)^{1/\alpha})(z/u)^{1/\alpha}\\
    &\hspace{.5cm}+\frac{(u+1) e^{u z}-u (z+1) (u z+1)-1}{u z}\e^{-uz- u}f_i'((zu)^{-1/\alpha})(zu)^{-1/\alpha}\frac{f_j'(u^{-1/\alpha})\diff u }{u^{1+1/\alpha}}\diff z\\
    J_{ij3}
    &= \rho_0^2\int_0^1\int_0^\infty \frac{2 \e^u-(u+1) (u^2+2) }{u z}\e^{-u/z-u}f_i'((z/u)^{1/\alpha})(z/u)^{1/\alpha}\\
    &\hspace{1cm}+\frac{2 \e^{u z}-(u z+1) (u^2 z^2+2)}{u z^2}\e^{-uz-u} f_i'((zu)^{-1/\alpha})(zu)^{-1/\alpha}\frac{f_j'(u^{-1/\alpha})\diff u }{u^{1+1/\alpha}}\diff z\\
    J_{ij4}
    &=\rho_0\int_0^1\int_0^\infty\frac{e^u (u-2)+u+2}{u z}\e^{-u/z-u}\rho(z)f_i'((z/u)^{1/\alpha})(z/u)^{1/\alpha}\\
    &+\frac{u z+e^{u z} (u z-2)+2}{u z^2}\e^{-uz-u}\rho(z) f_i'((zu)^{-1/\alpha})(zu)^{-1/\alpha}\frac{f_j'(u^{-1/\alpha})\diff u }{u^{1+1/\alpha}}\diff z.
\end{align*}
 Next, we want to integrate with respect to $u$, for which we need to insert the concrete forms of $f_i$. More precisely, we have
 \begin{align} \label{eq:concrete-fis}
 \begin{aligned}
    f_1'((z/u)^{1/\alpha})(z/u)^{1/\alpha}&=  (u/z)^\varpi[1+\log((u/z)^\varpi)], \\
    f_1'((zu)^{-1/\alpha})(zu)^{-1/\alpha} &=(uz)^\varpi[1+\log((uz)^\varpi)], \\
    f_1'(u^{-1/\alpha}){u^{-(1+1/\alpha)}} & =[1+\varpi\log u]u^{\varpi-1} ,\\
    f_2'((z/u)^{1/\alpha})(z/u)^{1/\alpha} &= -\alpha\varpi (u/z)^\varpi, \\ f_2'((zu)^{-1/\alpha})(zu)^{-1/\alpha} & =-\alpha\varpi(uz)^\varpi, \\
    f_2'(u^{-1/\alpha}){u^{-(1+1/\alpha)}} &=-\varpi\alpha u^{\varpi-1}, \\
    f_3'((z/u)^{1/\alpha})(z/u)^{1/\alpha} & = 1, \\
    f_3'((zu)^{-1/\alpha})(zu)^{-1/\alpha} &=1,   \\
    f_3'(u^{-1/\alpha}){u^{-(1+1/\alpha)}} &=1/u.
\end{aligned}
 \end{align}

\medskip
\noindent\underline{\textit{The term $s_{33}$.}}
For $i=j=3$, the formulas in \eqref{eq:concrete-fis} yield 
\begin{align*}
    J_{331}&=\int_0^1\int_0^\infty \Big[\frac{\e^u-1}{u}-1\Big]\e^{-u-u/z}+\Big[\frac{\e^{uz}-1}{uz}-1\Big]\e^{-uz- u}\frac{\diff u}{u}\frac{\diff z}{z}=\log(16)-2 
\\
    J_{332}&= 
    \rho_0\int_0^1\int_0^\infty \frac{\e^u (u+z)-u (u+1) (z+1)-z}{u z}\e^{-u/z- u}\\
    &\hspace{2cm}+\frac{(u+1) e^{u z}-u (z+1) (u z+1)-1}{u z}\e^{-uz- u}\frac{\diff u }{u}\frac{\diff z}{z}=0
\\
    J_{333}&= \rho_0^2\int_0^1\int_0^\infty\frac{2 \e^u-(u+1) (u^2+2) }{u z}\e^{-u/z-u}+\frac{2 \e^{u z}-(u z+1) (u^2 z^2+2)}{u z^2}\e^{-uz-u}\frac{\diff u }{u}\frac{\diff z}{z}\\
    &=\log(4)-2
\\
    J_{334}&=
    \rho_0\int_0^1\rho(z)\int_0^\infty\frac{e^u (u-2)+u+2}{u z}\e^{-u/z-u}+\frac{u z+e^{u z} (u z-2)+2}{u z^2}\e^{-uz-u}\frac{\diff u }{u}\frac{\diff z}{z}. \\
    &=\rho_0\int_0^1\rho(z)\frac{2 (z+2) \log (z+1)-4 z}{z^3}\diff z =: \rho_0\rho_{33}
\end{align*}
In total, $\alpha^2 s_{33}=\log(16)-2+(\log(4)-2)\rho_0^2+ \rho_0\rho_{33}$.

\medskip
\noindent\underline{\textit{The term $s_{32}$.}}
For $i=2,j=3$, the formulas in \eqref{eq:concrete-fis} yield
\begin{align*}
&\hspace{-.3cm}-\varpi^{-1}\alpha^{-1} J_{321}
    \\&=\int_0^1\int_0^\infty \Big[\frac{\e^u-1}{u}-1\Big]\e^{-u-u/z}\frac{u^\varpi}{z^\varpi}+\Big[\frac{\e^{uz}-1}{uz}-1\Big]\e^{-uz- u}(uz)^\varpi\frac{\diff u}{u}\frac{\diff z}{z}\\
    &= 2^{-\varpi} \big[-2 (\varpi-1) \, _2F_1(1,1;\varpi+1;-1)+2 \varpi+2^\varpi ((\varpi-3) \varpi+3)-4\big] \frac{\Gamma (\varpi-1)}{\varpi-1}\\
    &=:  T_{321}(\varpi) \\
&\hspace{-.3cm}-\rho_0^{-1}\varpi^{-1}\alpha^{-1} J_{322}
    \\&= 
    \int_0^1\int_0^\infty \frac{\e^u (u+z)-u (u+1) (z+1)-z}{u z}\e^{-u/z- u}\frac{u^\varpi}{z^\varpi}\\
    &\hspace{2cm}+\frac{(u+1) e^{u z}-u (z+1) (u z+1)-1}{u z}\e^{-uz- u}(uz)^\varpi\frac{\diff u }{u}\frac{\diff z}{z}\\
    &=\frac{2^{-\varpi} \left(-2 (\varpi-1) \, _2F_1(1,1;\varpi+1;-1)+2 \varpi+2^\varpi ((\varpi-3) \varpi+3)-4\right) \Gamma (\varpi+1)}{(\varpi-1)^2}\\
    &=: T_{322}(\varpi) \\
&\hspace{-.3cm}-\rho_0^{-2}\varpi^{-1}\alpha^{-1} J_{323}
    \\&= \int_0^1\int_0^\infty\frac{2 \e^u-(u+1) (u^2+2) }{u z}\e^{-u/z-u}\frac{u^\varpi}{z^\varpi}\\
    &\hspace{2cm}+\frac{2 \e^{u z}-(u z+1) (u^2 z^2+2)}{u z^2}\e^{-uz-u}(uz)^\varpi\frac{\diff u }{u}\frac{\diff z}{z}\\
    &=2^{-\varpi-1} [2^\varpi (8-\varpi ((\varpi-3) \varpi+4))-4 (\varpi+2)] \frac{\Gamma (\varpi-1)}{\varpi-2}=:  T_{323}(\varpi) \\
    &\hspace{-.3cm}-\rho_0^{-1}\varpi^{-1}\alpha^{-1}J_{324}
    \\&=
    \int_0^1\rho(z)\int_0^\infty\frac{e^u (u-2)+u+2}{u z}\e^{-u/z-u}\frac{u^\varpi}{z^\varpi}+\frac{u z+e^{u z} (u z-2)+2}{u z^2}\e^{-uz-u}(uz)^\varpi\frac{\diff u }{u}\frac{\diff z}{z}. \\
    &= \int_0^1 \frac{(z+1)^{-\varpi} \left(z^\varpi+1\right) \left(((\varpi-1) z-2) (z+1)^\varpi+\varpi z+z+2\right) \Gamma (\varpi-1)}{z^3}\rho(z) \diff z\\
    &=:\rho_{32}(\varpi).
\end{align*}

In total, $-\alpha \varpi^{-1} s_{32}= T_{321}(\varpi)+ T_{322}(\varpi)\rho_0+ T_{323}(\varpi)\rho_0^2+\rho_0\rho_{32}(\varpi)$.

\medskip
\noindent\underline{\textit{The term $s_{31}$.}}
For $i=1,j=3$, the formulas in \eqref{eq:concrete-fis} yield
\begin{align*}
    -\varpi^{-1}\alpha^{-2} J_{311}&=\int_0^1\int_0^\infty \Big[\frac{\e^u-1}{u}-1\Big]\e^{-u-u/z}\frac{u^\varpi}{z^\varpi}[1+\log((u/z)^\varpi)]\\
    &\hspace{2cm}+\Big[\frac{\e^{uz}-1}{uz}-1\Big]\e^{-uz- u}(uz)^\varpi[1+\log((uz)^\varpi)]\frac{\diff u}{u}\frac{\diff z}{z}\\
    &=T_{321}(\varpi)+\varpi T_{321}'(\varpi)
\end{align*}
In analogy, we can conclude 
\begin{align*}
    \alpha^{-2}\rho_0^{-1} J_{312}&=T_{322}(\varpi)+\varpi T_{322}'(\varpi) \\
    \alpha^{-2}\rho_0^{-2} J_{313}&=T_{323}(\varpi)+\varpi T_{323}'(\varpi) \\
    \alpha^{-2}\rho_0^{-1} J_{314}&= \rho_{32}(\varpi)+ \varpi \rho_{32}'(\varpi).
\end{align*}
In total, $\alpha^2  s_{31}= T_{321}(\varpi)+\varpi T_{321}'(\varpi)+ [T_{322}(\varpi)+\varpi T_{322}'(\varpi)]\rho_0+ [T_{323}(\varpi)+\varpi T_{323}'(\varpi)]\rho_0^2+[\rho_{32}(\varpi)+ \varpi \rho_{32}'(\varpi)]\rho_0$.

\medskip
\noindent\underline{\textit{The term $s_{22}$.}}
For $i=j=2$, the formulas in \eqref{eq:concrete-fis} yield
\begin{align*}
    &\hspace{-.5cm}\alpha^{-2}\varpi^{-2} J_{221}
    \\&=\int_0^1\int_0^\infty \Big[\frac{\e^u-1}{u}-1\Big]\e^{-u-u/z}\frac{u^\varpi}{z^\varpi}+\Big[\frac{\e^{uz}-1}{uz}-1\Big]\e^{-uz- u}(uz)^\varpi\frac{\diff u}{u^{1-\varpi}}\frac{\diff z}{z}\\
    &=-\frac{2 (\, _2F_1(\varpi-1,2 \varpi;\varpi;-1)+2 (\varpi-1) \, _2F_1(\varpi,2 \varpi;\varpi+1;-1)-1) \Gamma (2 \varpi-1)}{\varpi-1}\\
    &=:  T_{221}(\varpi) \\
    &\hspace{-.5cm}\alpha^{-2}\rho_0^{-1}\varpi^{-2} J_{222} 
    \\&= 
    \int_0^1\int_0^\infty \frac{\e^u (u+z)-u (u+1) (z+1)-z}{u z}\e^{-u/z- u}\frac{u^\varpi}{z^\varpi}\\
    &\hspace{2cm}+\frac{(u+1) e^{u z}-u (z+1) (u z+1)-1}{u z}\e^{-uz- u}(uz)^\varpi\frac{\diff u }{u^{1-\varpi}}\frac{\diff z}{z}\\
    &=\frac{4 \Gamma (\varpi+1)^2}{1-2 \varpi}-8 \varpi (\, _2F_1(\varpi-1,2 \varpi;\varpi;-1)-1) \Gamma (2 \varpi-2)\\
    &=: T_{222}(\varpi) \\
    &\hspace{-.5cm}\alpha^{-2}\rho_0^{-2}\varpi^{-2} J_{223}
    \\&= \int_0^1\int_0^\infty\frac{2 \e^u-(u+1) (u^2+2) }{u z}\e^{-u/z-u}\frac{u^\varpi}{z^\varpi}\\
    &+\frac{2 \e^{u z}-(u z+1) (u^2 z^2+2)}{u z^2}\e^{-uz-u}(uz)^\varpi\frac{\diff u }{u^{1-\varpi}}\frac{\diff z}{z}\\
    &=-4 \Big[\varpi \left(\varpi^2-1\right) \, _2F_1(\varpi-2,2 (\varpi+1);\varpi-1;-1)\\
    &\qquad+(\varpi-1) \Big(\left(2 \varpi^2-3 \varpi-2\right) (\varpi+1)^2 \, _2F_1(\varpi,2 (\varpi+1);\varpi+1;-1)\\
    &\qquad+\varpi \left(2 \left(2 \varpi^2-3 \varpi-2\right) \varpi^2 \, _2F_1(\varpi+1,2 (\varpi+1);\varpi+2;-1)-\varpi-1\right)\Big)\\
    &\qquad+2 (\varpi-2) \varpi (\varpi+1)^2 \, _2F_1(\varpi-1,2 (\varpi+1);\varpi;-1)\Big]\frac{ \Gamma (2 \varpi-1)}{(\varpi-2) (\varpi-1) \varpi (\varpi+1)}\\
    &=:  T_{223}(\varpi) \\
    &\hspace{-.5cm}\alpha^{-2}\rho_0^{-1}\varpi^{-2}J_{224}
    \\&=
    \int_0^1\rho(z)\int_0^\infty\frac{e^u (u-2)+u+2}{u z}\e^{-u/z-u}\frac{u^\varpi}{z^\varpi}\\
    &+\frac{u z+e^{u z} (u z-2)+2}{u z^2}\e^{-uz-u}(uz)^\varpi\frac{\diff u }{u^{1-\varpi}}\frac{\diff z}{z}. \\
    &= \int_0^1 2 z^{\varpi-3} (z+1)^{-2 \varpi} \left(((2 \varpi-1) z-2) (z+1)^{2 \varpi}+2 \varpi z+z+2\right) \Gamma (2 \varpi-1)\rho(z) \diff z\\
    &=:\rho_{22}(\varpi)
\end{align*}
In total, $\alpha^2 \varpi^{-2} s_{22}= T_{221}(\varpi)+ T_{222}(\varpi)\rho_0+ T_{223}(\varpi)\rho_0^2+\rho_0\rho_{22}(\varpi)$.

\medskip
\noindent\underline{\textit{The term $s_{12}$.}}
For $i=1,j=2$, the formulas in \eqref{eq:concrete-fis} yield
\begin{align*}
    -\alpha^{-1}\varpi^{-1} J_{121}&=\int_0^1\int_0^\infty \Big[\frac{\e^u-1}{u}-1\Big]\e^{-u-u/z}\frac{u^\varpi}{z^\varpi}[1+\log(u/z)]\\
    &\hspace{2cm}+\Big[\frac{\e^{uz}-1}{uz}-1\Big]\e^{-uz- u}(uz)^\varpi[1+\log(uz)]\frac{\diff u}{u^{1-\varpi}}\frac{\diff z}{z}\\
    &=T_{221}(\varpi)+\int_0^1\int_0^\infty \Big[\frac{\e^u-1}{u}-1\Big]\e^{-u-u/z}\frac{u^\varpi}{z^\varpi}\log(u/z)\\
    &\hspace{4cm}+\Big[\frac{\e^{uz}-1}{uz}-1\Big]\e^{-uz- u}(uz)^\varpi\log(uz)\frac{\diff u}{u^{1-\varpi}}\frac{\diff z}{z}\\
    &=:T_{221}(\varpi)+T_{121}(\varpi)
\end{align*}
Instead of evaluating $T_{121}(\varpi)$ directly, let us phrase this expression as the special case $T_{121}(\varpi)=\partial_\vartheta\mathfrak{T}(\varpi,\vartheta)|_{\vartheta=\varpi}$, where
\begin{align*}
    \frac{\partial}{\partial\vartheta}\mathfrak{T}(\varpi,\vartheta) &:=\int_0^1\int_0^\infty \Big[\frac{\e^u-1}{u}-1\Big]\e^{-u-u/z}\frac{u^\vartheta}{z^\vartheta}\log(u/z)
    \\& \hspace{3cm}
    +\Big[\frac{\e^{uz}-1}{uz}-1\Big]\e^{-uz- u}(uz)^\vartheta\log(uz)\frac{\diff u}{u^{1-\varpi}}\frac{\diff z}{z}.
\end{align*}
Then we conclude
\begin{align*}
    \mathfrak{T}(\varpi,\vartheta)&=
    \int_0^1\int_0^\infty \Big[\frac{\e^u-1}{u}-1\Big]\e^{-u-u/z}\frac{u^\vartheta}{z^\vartheta}+\Big[\frac{\e^{uz}-1}{uz}-1\Big]\e^{-uz- u}(uz)^\vartheta\frac{\diff u}{u^{1-\varpi}}\frac{\diff z}{z}\\
    &=
    \Big[\frac{1}{\vartheta-1}+\frac{1}{\varpi-1}-\frac{\, _2F_1(\vartheta-1,\vartheta+\varpi;\vartheta;-1)}{\vartheta-1}-\frac{(\vartheta+\varpi) \, _2F_1(\vartheta,\vartheta+\varpi;\vartheta+1;-1)}{\vartheta}\\
    &\hspace{.5cm}-\frac{\, _2F_1(\varpi-1,\vartheta+\varpi;\varpi;-1)}{\varpi-1}-\frac{(\vartheta+\varpi) \, _2F_1(\varpi,\vartheta+\varpi;\varpi+1;-1)}{\varpi}\Big] \Gamma (\vartheta+\varpi-1)
\end{align*}
Consequently, 

\begin{align*}
    &\phantom{{}={}}
    T_{121}(\varpi)
    \\&=\lim_{\vartheta\to\varpi}\frac{\partial}{\partial\vartheta} \mathfrak{T}(\varpi,\vartheta)
    \\&
    =\Gamma (2 \varpi-1) 
    \Big(
    -\frac{1}{(\varpi-1)^2}
    +\frac{2 \psi ^{(0)}(2 \varpi-1)}{\varpi-1} -2 \, _2F_1^{(0,0,1,0)}(\varpi,2 \varpi,\varpi+1,-1)
    \\&\hspace{1.4cm} 
    -4 \, _2F_1^{(0,1,0,0)}(\varpi,2 \varpi,\varpi+1,-1)
    -2 \, _2F_1^{(1,0,0,0)}(\varpi,2 \varpi,\varpi+1,-1)
    \\&\hspace{1.4cm}
    -\frac{\, _2F_1^{(0,0,1,0)}(\varpi-1,2 \varpi,\varpi,-1)}{\varpi-1}
    -\frac{2 \, _2F_1^{(0,1,0,0)}(\varpi-1,2 \varpi,\varpi,-1)}{\varpi-1}
    \\&\hspace{1.4cm} 
    -\frac{\, _2F_1^{(1,0,0,0)}(\varpi-1,2 \varpi,\varpi,-1)}{\varpi-1}
    +\frac{\, _2F_1(\varpi-1,2 \varpi;\varpi;-1)}{(\varpi-1)^2}
    \\&\hspace{1.4cm}
    -\frac{2 \, _2F_1(\varpi-1,2 \varpi;\varpi;-1) \psi ^{(0)}(2 \varpi-1)}{\varpi-1}
    -4 \, _2F_1(\varpi,2 \varpi;\varpi+1;-1) \psi ^{(0)}(2 \varpi-1)
    \Big)
\end{align*}
With the same trick, we obtain expressions for $J_{122},J_{123},J_{124}$, which we omit for brevity.

\medskip
\noindent\underline{\textit{The term $s_{11}$.}}
For $i=1,j=1$, the formulas in \eqref{eq:concrete-fis} yield
\begin{align*}
    J_{111}&=\int_0^1\int_0^\infty \Big[\frac{\e^u-1}{u}-1\Big]\e^{-u-u/z}\frac{u^\varpi}{z^\varpi}[1+\log (u/z)]\\
    &\hspace{2cm}+\Big[\frac{\e^{uz}-1}{uz}-1\Big]\e^{-uz- u}(uz)^\varpi[1+\log(uz)]\frac{[1+\varpi\log u]\diff u}{u^{1-\varpi}}\frac{\diff z}{z}\\
    &=T_{121}(\varpi)+T_{221}(\varpi)
    \\& \qquad +\int_0^1\int_0^\infty \Big[\frac{\e^u-1}{u}-1\Big]\e^{-u-u/z}\frac{u^\varpi}{z^\varpi}
    +\Big[\frac{\e^{uz}-1}{uz}-1\Big]\e^{-uz- u}(uz)^\varpi\frac{\varpi\log u\diff u}{u^{1-\varpi}}\frac{\diff z}{z}\\
    & \qquad +\int_0^1\int_0^\infty \Big[\frac{\e^u-1}{u}-1\Big]\e^{-u-u/z}\frac{u^\varpi}{z^\varpi}\log (u/z)
    \\&\hspace{3cm}
    +\Big[\frac{\e^{uz}-1}{uz}-1\Big]\e^{-uz- u}(uz)^\varpi\log(uz)\frac{\varpi\log u\diff u}{u^{1-\varpi}}\frac{\diff z}{z}\\
    &=: T_{121}(\varpi)+T_{221}(\varpi)+T_{111}(\varpi)+\Tilde T_{111}(\varpi). 
\end{align*}
$T_{111}(\varpi)$ may be directly evaluated by a CAS to
\begin{align*}
    T_{111}(\varpi)&=\frac{\Gamma (2 \varpi-1) }{\varpi}\Big(2 \varpi^5 \Gamma (\varpi)^2 \, _3\tilde{F}_2(2 \varpi,\varpi+1,\varpi+1;\varpi+2,\varpi+2;-1)\\
    & \qquad +\, _3F_2(\varpi,\varpi,2 \varpi;\varpi+1,\varpi+1;-1)-1\Big)
    \\& \qquad +\varpi \, _2F_1(2 \varpi,\varpi+1;\varpi+2;-1) (2 \varpi \psi ^{(0)}(2 \varpi-1)+1))\\
    &\qquad-\frac{2 \Gamma (2 \varpi) }{2 \varpi^2+\varpi-1}((\varpi+1) (\, _2F_1(\varpi,2 \varpi;\varpi+1;-1)-1) \psi ^{(0)}(2 \varpi-1)\\
    &\qquad+\int_0^1 \frac{\varpi z^{\varpi-1} (z+1)^{-2 \varpi} \Gamma (2 \varpi) ((4 \varpi z+2) \log (z+1))}{2 \varpi-1} \diff z
\end{align*}
One can additionally check that 
\begin{multline*}
    \int_0^1 \frac{\varpi z^{\varpi-1} (z+1)^{-2 \varpi} \Gamma (2 \varpi) ((4 \varpi z+2) \log (z+1))}{2 \varpi-1} \diff z  \\ =
    -2 \varpi \Gamma (2 \varpi-1) \frac{\partial }{\partial q}\left((2 \varpi-1) B_{\frac{1}{2}}(\varpi+1,q)+B_{\frac{1}{2}}(\varpi,q)\right)\Big|_{q=\varpi-1}.
\end{multline*}
Concerning $\Tilde T_{111}(\varpi)$, borrow the ideas from entry $s_{12}$ and write $ \Tilde T_{111}(\varpi)=\mathfrak{K}(\varpi,\varpi)$, where
\begin{align*}
    \mathfrak{K}(\varpi,\vartheta)&:= \int_0^1\int_0^\infty \Big[\frac{\e^u-1}{u}-1\Big]\e^{-u-u/z}\frac{u^\vartheta}{z^\vartheta}\log (u/z)
    \\&\hspace{2cm}
    +\Big[\frac{\e^{uz}-1}{uz}-1\Big]\e^{-uz- u}(uz)^\vartheta\log(uz)\frac{\varpi\log u\diff u}{u^{1-\varpi}}\frac{\diff z}{z} \\
    &= \frac{\partial}{\partial\vartheta}\int_0^1\int_0^\infty \Big[\frac{\e^u-1}{u}-1\Big]\e^{-u-u/z}\frac{u^\vartheta}{z^\vartheta}+\Big[\frac{\e^{uz}-1}{uz}-1\Big]\e^{-uz- u}(uz)^\vartheta\frac{\varpi\log u\diff u}{u^{1-\varpi}}\frac{\diff z}{z}\\
    &= \varpi\frac{\partial^2}{\partial\vartheta\partial\varpi}\int_0^1\int_0^\infty \Big[\frac{\e^u-1}{u}-1\Big]\e^{-u-u/z}\frac{u^\vartheta}{z^\vartheta}+\Big[\frac{\e^{uz}-1}{uz}-1\Big]\e^{-uz- u}(uz)^\vartheta\frac{\diff u}{u^{1-\varpi}}\frac{\diff z}{z}\\
    &= \varpi \frac{\partial}{\partial\varpi} \mathfrak{T}(\varpi,\vartheta).
\end{align*}
Again, the same trick helps to derive expressions for $J_{112},J_{113},J_{114}$, which are omitted for brevity and may be found in supplementary notebooks.

\medskip
\noindent\underline{\textit{The entries $s_{i4}$.}}
In the case $j=4$, we need a different bivariate margin of $K$. Lemma~\ref{lem:JoiWeaKon} yields
\begin{align*}
    s_{i4}
    &= 
    \int_0^\infty \Big\{ \int_0^y \Big( K_{\rho,\alpha,\zeta}(\infty,x,y,\infty) -H_{\rho,\alpha,1}^{(2)}(x)H^{(1)}_{\rho,\alpha,1}(y)\Big) f_i'(x)\diff x \\
    & \qquad +\int_y^\infty\Big( K_{\rho,\alpha,\zeta}(\infty,x,y,\infty)-H_{\rho,\alpha,1}^{(2)}(x)H_{\rho,\alpha,1}^{(1)}(y)\Big) f_i'(x)\diff x \Big\} \frac{\diff y}{y} 
    \\&=: \int_0^\infty (I_{1i}(y)+I_{2i}(y))\frac{\diff y}{y}, 
\end{align*}
where, by Lemma~\ref{lem:JoiWeaKon},
\begin{align*}
    I_{1i}(y)&=\int_0^y \Big(\exp(-x^{-\alpha}-\zeta y^{-\alpha})\Big\{1+\zeta\rho_0 x^{-\alpha}+(1-\zeta)x^{-\alpha}\rho\big((x/y)^\alpha)\big)\Big\}\\
    &\hspace{6.5cm}-\exp(-x^{-\alpha}-y^{-\alpha})\big\{1+\rho_0 x^{-\alpha}\big\}\Big)f_i'(x)\diff x, \\
    I_{2i}(y)&=\int_y^\infty  \exp(-\zeta x^{-\alpha}- y^{-\alpha})\big\{1+\zeta\rho_0 x^{-\alpha}\big\}-\exp(-x^{-\alpha}-y^{-\alpha})\big\{1+\rho_0 x^{-\alpha}\big\} f_i'(x)\diff x.
\end{align*}

Apply again the substitutions  $z=x^\alpha y^{-\alpha}$ with $\diff x = \alpha^{-1}yz^{1/\alpha-1}\diff z$ to $I_{1i}$; and for $I_{2i}$, we will substitute $z=x^{-\alpha} y^{\alpha}$ with $\diff x = -\alpha^{-1}yz^{-1/\alpha-1}\diff z$. Additionally, write $u=y^{-\alpha}$. Then 
\begin{align*}
    I_{1i}(u^{-1/\alpha})&=\alpha^{-1}\int_0^1\Big[\exp(-u/z-\zeta u)\Big\{1+\zeta\rho_0 u/z+(1-\zeta)u/z\rho(z)\Big\} 
    \\ &\hspace{2cm} 
    -\exp(-u/z-u)\big\{1+\rho_0 u/z\big\}\Big]
    \cdot f_i'((z/u)^{1/\alpha})(z/u)^{1/\alpha}\frac{\diff z}{z} \\
    I_{2i}(u^{-1/\alpha})&=\alpha^{-1}\int_0^1\Big[\exp(-\zeta uz- u)\big\{1+\zeta\rho_0 uz\big\}-\exp(-uz-u)\big\{1+\rho_0 uz\big\}\Big]
    \\ &\hspace{2cm}
    f_i'((uz)^{-1/\alpha})(uz)^{-1/\alpha}\frac{\diff z}{z} 
\end{align*}
It follows that we may write $s_{i4}$ as $\alpha^2 s_{i4} = J_{1i}+J_{2i}+J_{3i}$, where, applying the substitution $u=y^{-\alpha}$, $-\alpha/u\diff u=1/y\diff y$,
\begin{align*}
    J_{i41}&= \int_0^1 \int_0^1 \int_0^\infty [\e^{-u/z-\zeta u}-\e^{-u/z-u}]f_i'((z/u)^{1/\alpha})(z/u)^{1/\alpha}\\
    &\hspace{2cm}+[\e^{-\zeta uz- u}-\e^{-uz-u}]f_i'((uz)^{-1/\alpha})(uz)^{-1/\alpha}\frac{\diff u}{u} \frac{\diff z}{z} \diff \zeta \\
    J_{i42}&= \rho_0\int_0^1 \int_0^1 \int_0^\infty [\e^{-u/z-\zeta u}\zeta-\e^{-u/z-u}]f_i'((z/u)^{1/\alpha})(z/u)^{1/\alpha-1}\\
    &\hspace{2cm}+[\e^{-\zeta uz- u}\zeta-\e^{-uz-u}]f_i'((uz)^{-1/\alpha})(uz)^{1-1/\alpha}\frac{\diff u}{u} \frac{\diff z}{z} \diff \zeta \\
    J_{i43}&= \int_0^1 \int_0^1 \int_0^\infty \e^{-u/z-\zeta u}(1-\zeta)\rho(z)f_i'((z/u)^{1/\alpha})(z/u)^{1/\alpha-1}\frac{\diff u}{u} \frac{\diff z}{z} \diff \zeta 
\end{align*}
It is elementary to integrate with respect to $\zeta$ first. We are left with three double integrals, namely
\begin{align*}
    J_{i41}&= \int_0^1 \int_0^\infty \frac{\e^u-u-1}{u}\exp(-u(1+z)/z)f_i'((z/u)^{1/\alpha})(z/u)^{1/\alpha}\\
    &\hspace{2cm}+\frac{\e^{uz}-uz-1}{uz}\exp(-u(1+z)) f_i'((uz)^{-1/\alpha})(uz)^{-1/\alpha}\frac{\diff u}{u} \frac{\diff z}{z} \\
    J_{i42}&= \rho_0\int_0^1 \int_0^\infty -\frac{1-\e^u+u+u^2}{u^2}\exp(-u(1+z)/z)f_i'((z/u)^{1/\alpha})(z/u)^{1/\alpha-1}\\
    &\hspace{2cm}+\frac{\e^{uz}-1-uz(1+uz)}{u^2z^2}\exp(-u(1+z)) f_i'((uz)^{-1/\alpha})(uz)^{1-1/\alpha}\frac{\diff u}{u} \frac{\diff z}{z} \\
    J_{i43}&= \int_0^1 \int_0^\infty \frac{1+u\e^u-\e^u}{u^2}\exp(-u(1+z)/z)\rho(z)f_i'((z/u)^{1/\alpha})(z/u)^{1/\alpha-1}\frac{\diff u}{u} \frac{\diff z}{z}.
\end{align*}
Next, we want to integrate with respect to $u$, for which we need to insert the concrete forms of $f_i$.

\medskip
\noindent\underline{\textit{The term $s_{34}$.}}
The formulas in \eqref{eq:concrete-fis} with $i=3$ yields
\begin{align*}
    J_{341}&= \int_0^1 \int_0^\infty \frac{\e^u-u-1}{u}\exp(-u(1+z)/z)+\frac{\e^{uz}-uz-1}{uz}\exp(-u(1+z)) \frac{\diff u}{u} \frac{\diff z}{z}\\ &=\log (16)-2 \\
    J_{342}&= \rho_0\int_0^1 \int_0^\infty -\frac{1-\e^u+u+u^2}{u^2}\e^{-u(1+z)/z}(z/u)^{-1} 
    \\& \hspace{2cm} +\frac{\e^{uz}-1-uz(1+uz)}{u^2z^2}\e^{-u(1+z)} (uz)\frac{\diff u}{u} \frac{\diff z}{z} \\
    &=\rho_0\int_0^1 \frac{z-(z+1) \log (z+1)}{z^2} \frac{\diff z}{z} \\
    J_{343}&= \int_0^1 \int_0^\infty \frac{1+u\e^u-\e^u}{u^2}\exp(-u(1+z)/z)\rho(z)(z/u)^{-1}\frac{\diff u}{u} \frac{\diff z}{z}\\
    &=\int_0^1 \frac{(z+1) \log (z+1)-z}{z^2}\rho(z)\frac{\diff z}{z}
\end{align*}
Together,
\begin{align*}
    J_{342}+J_{343}=\int_0^1 \frac{(z+1) \log (z+1)-z}{z^3}[\rho(z)-\rho_0]\diff z=: \rho_{34},
\end{align*}
which in summary yields $\alpha s_{34}=\log (16)-2+\rho_{34}$.

\medskip
\noindent\underline{\textit{The term $s_{24}$.}}
The formulas in \eqref{eq:concrete-fis} with $i=2$ yields
\begin{align*}
    -\alpha^{-1} \varpi^{-1} J_{241}&= \int_0^1 \int_0^\infty \frac{\e^u-u-1}{u}\e^{-u(1+z)/z}\frac{u^\varpi}{z^\varpi}+\frac{\e^{uz}-uz-1}{uz}\e^{-u(1+z)}(uz)^\varpi
    \frac{\diff u}{u} \frac{\diff z}{z}\\ 
    &= T_{241}(\varpi),
\end{align*}
where 
\begin{align*}
    T_{241}(\varpi)&=\frac{2^{-w} \left(-2 (w-1) \, _2F_1(1,1;w+1;-1)+2 w+2^w ((w-3) w+3)-4\right) \Gamma (w-1)}{w-1}.
\end{align*}
Regarding the second and third integral, note that
\begin{align*}
    \int_0^\infty \frac{(\e^u (u-1)+1) \e^{-\frac{u (z+1)}{z}}}{u^2} \frac{u^{1+\varpi}}{z^{1+\varpi}}\frac{\diff u}{u}&=\frac{((z+1)^{1-\varpi}+(\varpi-1) z-1) \Gamma (\varpi-1)}{z^2},
\end{align*}
which yields
\begin{align*}
    J_{242}+J_{342}&= \int_0^1 \frac{((z+1)^{1-\varpi}+(\varpi-1) z-1) \Gamma (\varpi-1)}{z^2} [\rho(z)-\rho_0]\frac{\diff z}{z}\\
    &\qquad+\rho_0 \int_0^1 \frac{((z+1)^{1-\varpi}+(\varpi-1) z-1) \Gamma (\varpi-1)}{z^2} \\
    &\qquad+\int_0^\infty \frac{-u z (u z+1)+\e^{u z}-1 }{u^2 z^2}\e^{-u (z+1)}(u z)^{\varpi+1}
    \\&\hspace{2cm}
    -\frac{u^2+u-\e^u+1 }{u^2}\e^{-\frac{u (z+1)}{z}}\frac{u^{\varpi+1}}{z^{\varpi+1}}\frac{\diff u}{u}\frac{\diff z}{z}\\
    &=\rho_{24}(\varpi)+\rho_0 T_{242}(\varpi),
\end{align*}
where
\begin{align*}
    \rho_{24}(\varpi) &:= \int_0^1 \frac{((z+1)^{1-\varpi}+(\varpi-1) z-1) \Gamma (\varpi-1)}{z^2} [\rho(z)-\rho_0]\frac{\diff z}{z},\\
    T_{242}(\varpi) &:= \Big((\varpi-1) \varpi-\frac{2^{-\varpi} \varpi \left(2 (\varpi-1) \, _2F_1(1,1;\varpi+1;-1)+\left(2^\varpi-2\right) (\varpi-2)\right)}{\varpi-1}\Big) \\
    &\hspace{11cm}\Gamma (\varpi-1).
\end{align*}
Together, $-\alpha \varpi^{-1}s_{24}= T_{241}(\varpi)+\rho_0 T_{242}(\varpi)+\rho_{24}(\varpi)$.

\medskip
\noindent\underline{\textit{The term $s_{14}$.}}
Again, we approach the case $i=1$ fundamentally different. In view of $f_1'((z/u)^{1/\alpha})(z/u)^{1/\alpha}= (u/z)^\varpi(1+\log(u/z))$ by \eqref{eq:concrete-fis}, we may write $J_{14\ell}=\mathcal{T}(J_{24\ell})$ for an operator $\mathcal{T}(f)=f+\partial_\varpi f$. We directly conclude $\alpha^{-2}\sigma_{14}=T_{241}(\varpi)+\rho_0 T_{242}(\varpi)+\partial_\varpi(T_{241}(\varpi)+\rho_0 T_{242}(\varpi))+\rho_{14}(\varpi)+\rho_{24}(\varpi)$, where
\begin{align*}
    \rho_{14}(\varpi)=\int_0^1 \frac{\partial}{\partial \varpi} \frac{((z+1)^{1-\varpi}+(\varpi-1) z-1) \Gamma (\varpi-1)}{z^2} [\rho(z)-\rho_0]\frac{\diff z}{z}.
\end{align*}
\end{proof}

\section{Finite moments of top two order statistics}

\begin{lemma}[Lemma C.1 in \cite{Bücher2018-disjoint} revisited]
\label{lem:c1-disjoint}
    Let $\xi_1,\xi_2,...\sim F$ be iid random variables satisfying \eqref{eq:iid_doa}. Let $M_r:=\xi_{r:r},S_n:=\xi_{r-1:r}$. For every $\beta \in (-\infty,\alpha_0)$ and any constant $c>0$, we have
    \begin{align*}
    \limsup_{r\to\infty} \Exp\big[\big((M_r\vee c)/a_r\big)^\beta\big]<\infty,
    \qquad
        \limsup_{r\to\infty} \Exp\big[\big((S_r\vee c)/a_r\big)^\beta\big]<\infty.
    \end{align*}
\end{lemma}

\begin{proof} 
The claim regarding $M_r$ is Lemma C.1 in \cite{Bücher2018-disjoint}. Regarding $S_r$, we
    distinguish the three cases $\beta=0,\beta>0,\beta<0$. The first case is trivial. The second case follows from the assertion regarding $M_r$, observing that $((S_r\vee c)/a_r)^\beta \le ((M_r\vee c)/a_r)^\beta$.
We are only left with the case $\beta<0$. Let $Z_r=(S_r\vee c)/a_r$ and note that
    \begin{align*}
        \Exp[Z_r^\beta] &= \int_0^\infty \Prob(Z_r^\beta>x)\diff x =\int_0^\infty \Prob(Z_r<x^{1/\beta})\diff x = \int_0^\infty \Prob(Z_r<y)|\beta|y^{\beta-1}\diff y 
        \\&= 
        \int_0^1 \Prob(Z_r<y)|\beta|y^{\beta-1}\diff y + 
\int_1^\infty \Prob(Z_r<y)|\beta|y^{\beta-1}\diff y.
    \end{align*}
    Using the bound $\Prob(Z_r<y) \le 1$, the second integral is bounded by $ \int_1^\infty |\beta|y^{\beta-1} \diff y=1$.
Regarding the first integral, note that
    \begin{align*}
        \Prob(Z_r<y)=\Prob(S_r\vee c < a_ry) = [r(1-F(a_ry))F^{r-1}(a_ry)+F^r(a_ry)]\indic_{(c/a_r,\infty)}(y)
    \end{align*}
    by similar arguments as in \eqref{eq:sr-cdf-iid}. 
As a consequence,
        \begin{multline*}
         \int_0^1 \Prob(Z_r<y)|\beta|y^{\beta-1}\diff y 
         \\=
        \int_{c/a_r}^1 F^r(a_ry)|\beta|y^{\beta-1}\diff y  + \int_{c/a_r}^1 r(1-F(a_ry))F^{r-1}(a_ry)|\beta|y^{\beta-1}\diff y.
    \end{multline*}
    The limes superior of the left integral has been shown to be finite in the proof of Lemma C.1 in \cite{Bücher2018-disjoint}. For the right integral, fix $\delta\in(0,\alpha_0)$. As in the proof of Lemma C.1 in \cite{Bücher2018-disjoint}, there exists a constant $c(\delta)>0$ such that 
    \begin{align*}
        F^{r-1}(a_ry)\leq \exp\big(-c(\delta)y^{-\alpha_0+\delta}\big)
    \end{align*}
    for all sufficiently $r$ and all $y\in (c/a_r,1]$.
    We proceed by bounding $r(1-F(a_ry))$. Observing that $1-F$ is regularly varying of index $\alpha_0$, we may apply Potter's theorem (Theorem 1.5.6 in \cite{Bingham1987}) to deduce that there exists a constant $x(\delta)>0$ such that, for all $r$ such that $a_r\ge x(\delta)$ and all $y \in (x(\delta)/a_r,1]$, 
    \[
    \frac{1-F(a_ry)}{1-F(a_r)} \le (1+\delta) y^{-\alpha_0+\delta}.
    \]
    Without loss of generality, we may choose $x(\delta)>c$.
    For $y \in (c/a_r, x(\delta)/a_r]$, we have, writing $ L_{c,\delta} = \{1-F(c)\}/\{1-F(x(\delta))\}$,
    \[
    \frac{1-F(a_ry)}{1-F(a_r)} 
    \le 
    \frac{1-F(c)}{1-F(a_r)}
    =
     L_{c,\delta} \frac{1-F(x(\delta))}{1-F(a_r)}
    \le
     L_{c,\delta} (1+\delta) (x(\delta)/a_r)^{-\alpha_0+\delta}
     \le L_{c,\delta} (1+\delta) y^{-\alpha_0+\delta}.
    \]
    Combing the previous two displays, and observing 
  that $\sup_{r\in\N} r\{1-F(a_r)\}<\infty$ as argued in the proof of Lemma C.1 in \cite{Bücher2018-disjoint},
    we find that, for sufficiently large $r$ and all $y \in(c/a_r,1]$, 
    \begin{align*}
        r(1-F(a_ry))= r(1-F(a_r))\frac{1-F(a_ry)}{1-F(a_r)}\leq K_{c,\delta}y^{-\alpha_0+\delta},
    \end{align*}
    where $K_{c, \delta}$ is a positive constant.
    Altogether we now have, for sufficiently large $r$,
     \begin{align*}
        \int_{c/a_r}^1 r(1-F(a_ry))F^{r-1}(a_ry)|\beta|y^{\beta-1}\diff y
         &\leq 
        K_{c,\delta} |\beta|\int_0^1  y^{-\alpha_0+\delta+\beta-1} \exp\big(-c(\delta)y^{-\alpha_0+\delta}\big)\diff y, 
    \end{align*}
which is finite.
\end{proof}

\section{Additional simulation results}

\label{sec:sim-additional}
\subsection{Illustrating theoretical bias and variance formulas}
\label{subsec:bias-variance-expansions}

In this section, we compare the theoretical asymptotic expansions for the bias and the variance obtained in Remark~\ref{rem:iid-mse} (iid case) and Example~\ref{ex:mori} (time series case) to the observed counterparts in Monte Carlo simulations.

\subsubsection{The IID Case}
\label{subsec:bias-variance-expansions-iid}

We consider the situation underlying the iid case in Figure \ref{fig:fixed_bs_shape_and_rl}. To apply the formulas derived in Remark~\ref{rem:iid-mse}, we need to derive an explicit second order expansion as required in Condition~\ref{cond:sorv} for the Pareto distribution. This is straightforward: first, since $\bar F(x) = x^{-\alpha}$, we have $-\log F(x) = - \log(1-x^{-\alpha})$. Hence, in view of the fact that $-\log(1-u)=u+u^2/2+\mathcal{O}(u^{3})$ as $u \to 0$, we obtain that
\begin{align*}
    \frac{-\log F(tx)}{-\log F(t)}
    &=
    \frac{(tx)^{-\alpha} + (tx)^{-2\alpha}/2 + \mathcal{O}(t^{-3\alpha})}{t^{-\alpha}+ t^{-2\alpha}/2 +\mathcal{O}(t^{-3\alpha})}
    \\&=
    x^{-\alpha} + x^{-\alpha} \Big\{  \frac{1 + (tx)^{-\alpha}/2 + \mathcal{O}(t^{-2\alpha})}{1+t^{-\alpha}/2 + \mathcal{O}(t^{-2\alpha})} - 1 \Big\}
    \\&=
    x^{-\alpha} + x^{-\alpha} \frac{t^{-\alpha}}2 \frac{x^{-\alpha}-1 + \mathcal{O}(t^{-\alpha})}{1+t^{-\alpha}/2 + \mathcal{O}(t^{-2\alpha})}
    \\&=
    x^{-\alpha} + x^{-\alpha} \frac{t^{-\alpha}}{2} (x^{-\alpha}-1) 
    \\ & \hspace{2cm} + x^{-\alpha} \frac{t^{-\alpha}}{2} (x^{-\alpha}-1) \Big\{ \frac{1}{1+t^{-\alpha}/2 + \mathcal{O}(t^{-2\alpha})} - 1 \Big\} + \mathcal{O}(t^{-2\alpha})
    \\&=
    x^{-\alpha} + x^{-\alpha} \frac{t^{-\alpha}}{2} (x^{-\alpha}-1)  + \mathcal{O}(t^{-2\alpha}).
\end{align*}
As a consequence, $A(t) = - \alpha t^{-\alpha}/2$, which means that $\bar \tau = 1$ and $c=-1/2$ in the notation of Remark~\ref{rem:iid-mse}. 

In Figures \ref{fig:iid_simu_mse} and \ref{fig:iid_simu_mse_bs}, we compare the theoretical expansions from Remark~\ref{rem:iid-mse} to their observed counterparts in the simulation experiments.  The respective curves align remarkably well with each other.

\begin{figure}[!thp]
    \centering
    \includegraphics[width=0.995\linewidth]{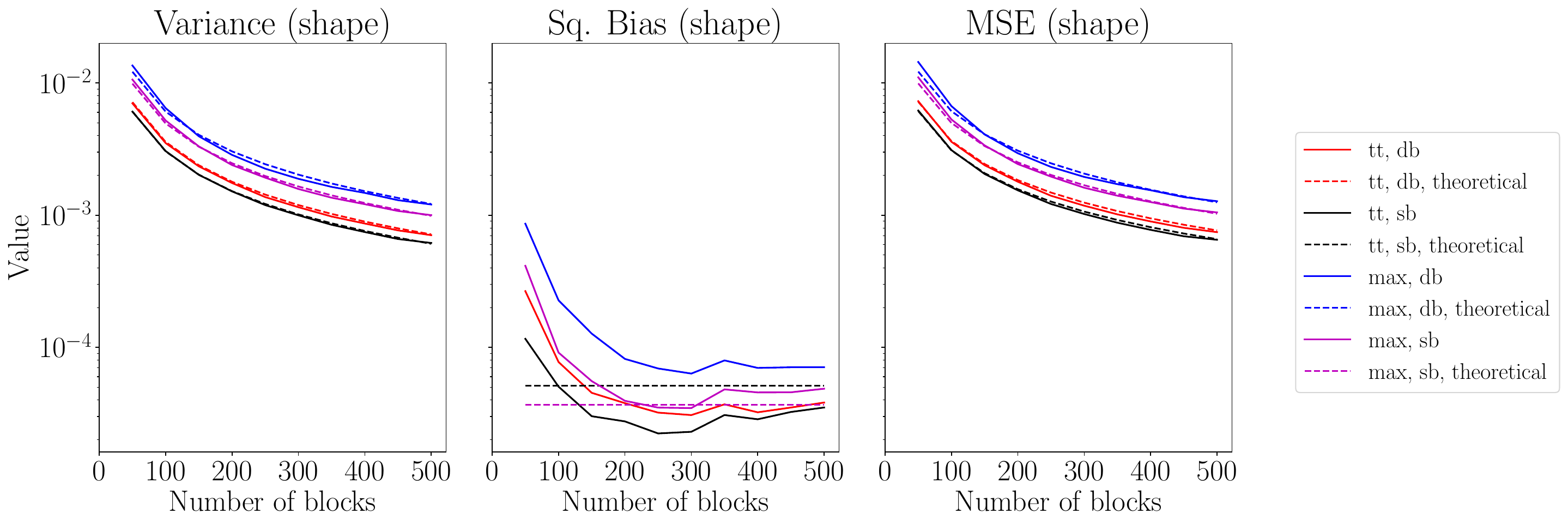}
    \caption{Revisiting Figure \ref{fig:fixed_bs_shape_and_rl} (iid case, standard Pareto, shape estimation). For a fixed block size of $r=100$, the simulated bias, variance and MSE of the shape estimators are compared to the respective asymptotic expansions from Remark~\ref{rem:iid-mse} (dashed lines).}
    \label{fig:iid_simu_mse}
\end{figure}

\begin{figure}[!thp]
    \centering
    \includegraphics[width=0.995\linewidth]{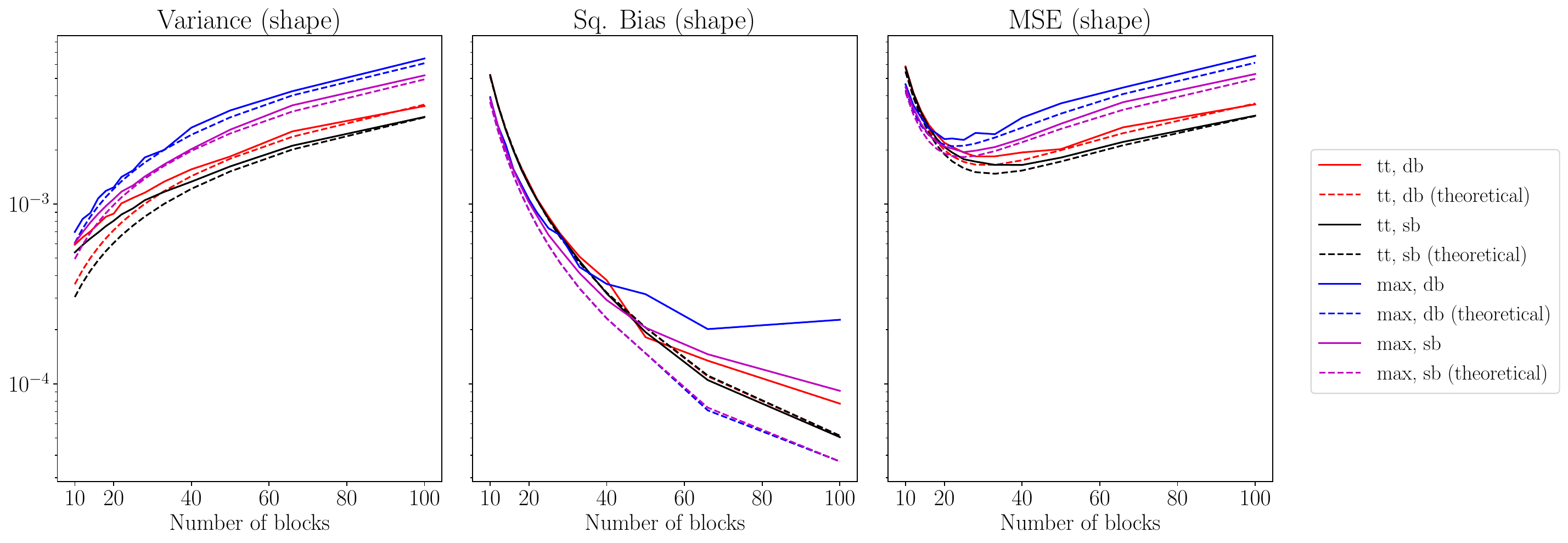}
    \caption{Revisiting the left-hand side of Figure \ref{fig:fixed_n} (iid case, standard Pareto). For a fixed total sample size of $n=10,000$, the simulated bias, variance and MSE of the shape estimators are compared to the respective asymptotic expansions from Remark~\ref{rem:iid-mse} (dashed lines).}
    \label{fig:iid_simu_mse_bs}
\end{figure}

\subsubsection{A Time Series Model}
\label{subsec:bias-variance-expansions-mori}

We consider the model from Example~\ref{ex:mori} with $\rho(\eta) = c(1-\eta)$ and $\alpha=1$. The stochastic construction simplifies: we have $(Z_t)_t$ iid standard $\Pareto(1)$,  $(\zeta_t)_t$ iid $\Bernoulli(1-c)$ and
\begin{align*}
    \xi_t = \max\{Z_{t-1},\zeta_tZ_t\}, \qquad t\in\N.
\end{align*}
Note that $\rho_0 = c$. Asymptotic expansions for the bias, variance and MSE have been derived in Example~\ref{ex:mori}. We illustrate them in Figures
\ref{fig:mori_amse_n} and \ref{fig:mori_amse_r} for the case of a fixed block size and fixed total sample size, respectively. Regarding the first setting, it is found that the minimum of the curves for the top-two estimators is consistently below that for the max-only estimators. Regarding the second setting,  the top-two shape estimator has a globally smaller MSE than the max-only estimator; for the scale estimation, both estimators show a comparable MSE.

\begin{figure}[!thp]
    \centering
    \includegraphics[width=0.98\linewidth]{Figures/amse_mori_vs_r_n_1e3_1e5.pdf}
    \includegraphics[width=0.98\linewidth]{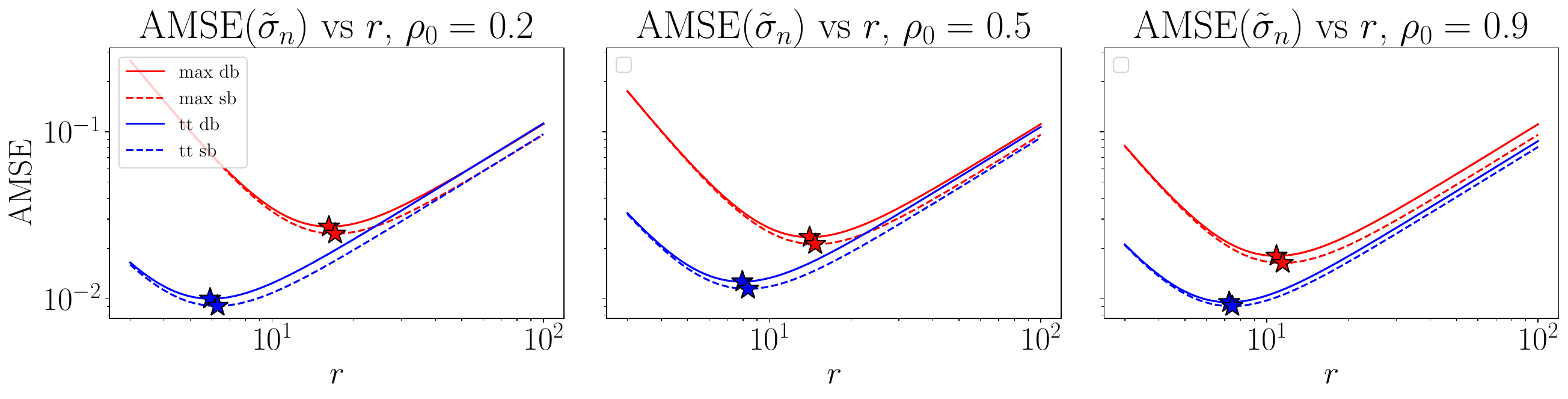}
    \caption{Asymptotic MSE for different choices of $\rho_0$, as a function of the block size $r$ for a fixed total sample size $n=1000$. Top: $\mathrm{AMSE}(\tilde\alpha_n)$, bottom: $\mathrm{AMSE}(\tilde\sigma_n)$.}
    \label{fig:mori_amse_n}
\end{figure}

\begin{figure}[!thp]
    \centering
    \includegraphics[width=0.98\linewidth]{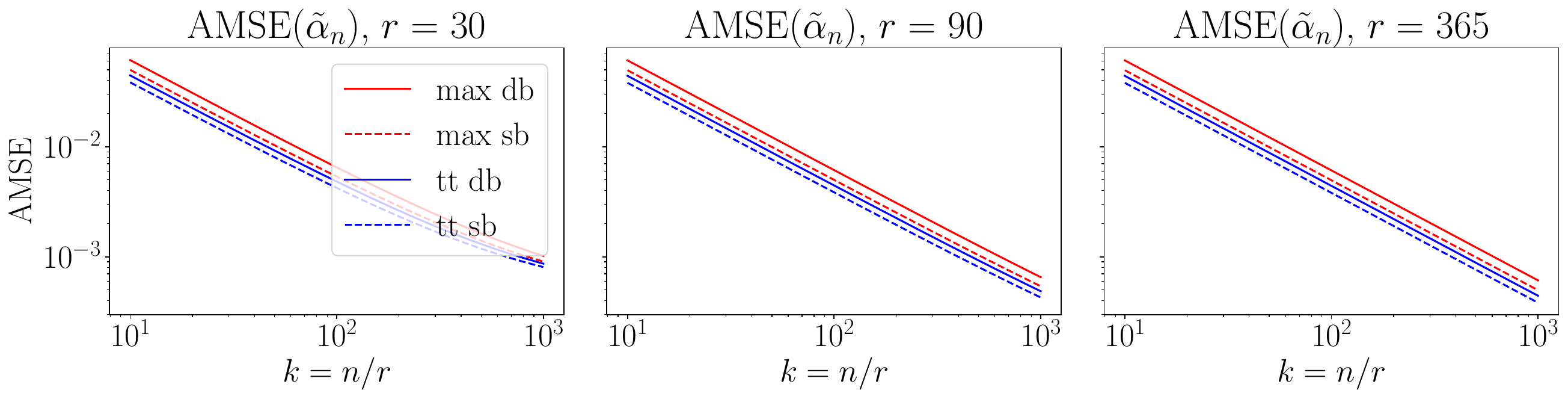}
    \includegraphics[width=0.98\linewidth]{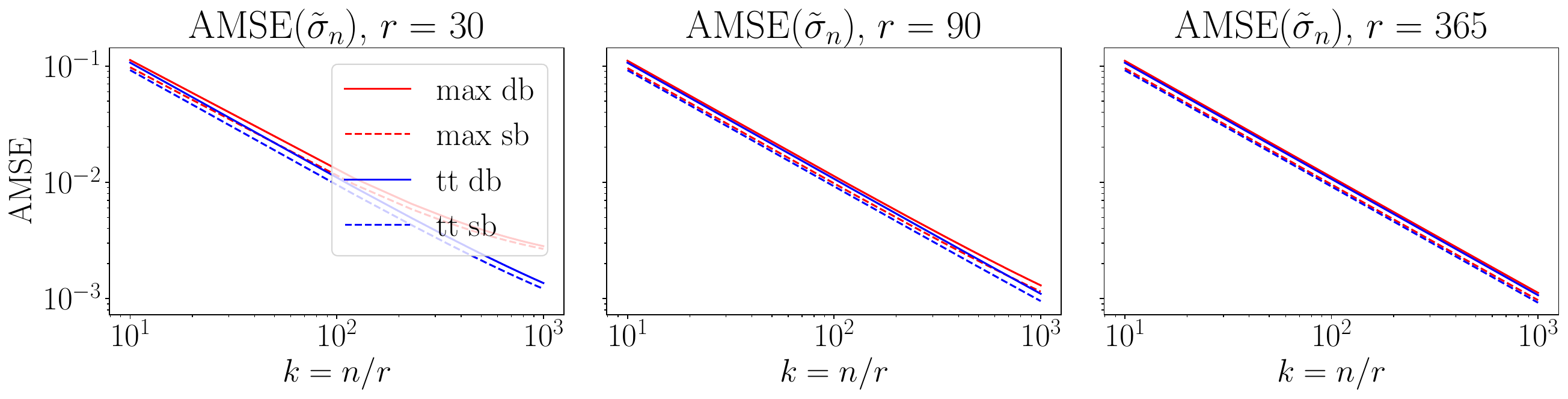}
    \caption{Asymptotic MSE for different choices of (fixed) block size $r=30,90,365$, as a function of the effective sample size $k$, fixed $\rho_0=0.5$. Top: $\mathrm{AMSE}(\tilde\alpha_n)$, bottom: $\mathrm{AMSE}(\tilde\sigma_n)$. }
    \label{fig:mori_amse_r}
\end{figure}

Next, we compare the asymptotic expansions to the observed values in simulation experiments. For simplicity, we only consider the disjoint blocks estimators with block size $r=100$ and effective sample size $k=200$. The results are presented in Figures~\ref{fig:mori_simu_var} (variance) and \ref{fig:mori_simu_bias} (bias).

\begin{figure}[!thp]
    \centering
    \includegraphics[width=0.95\linewidth]{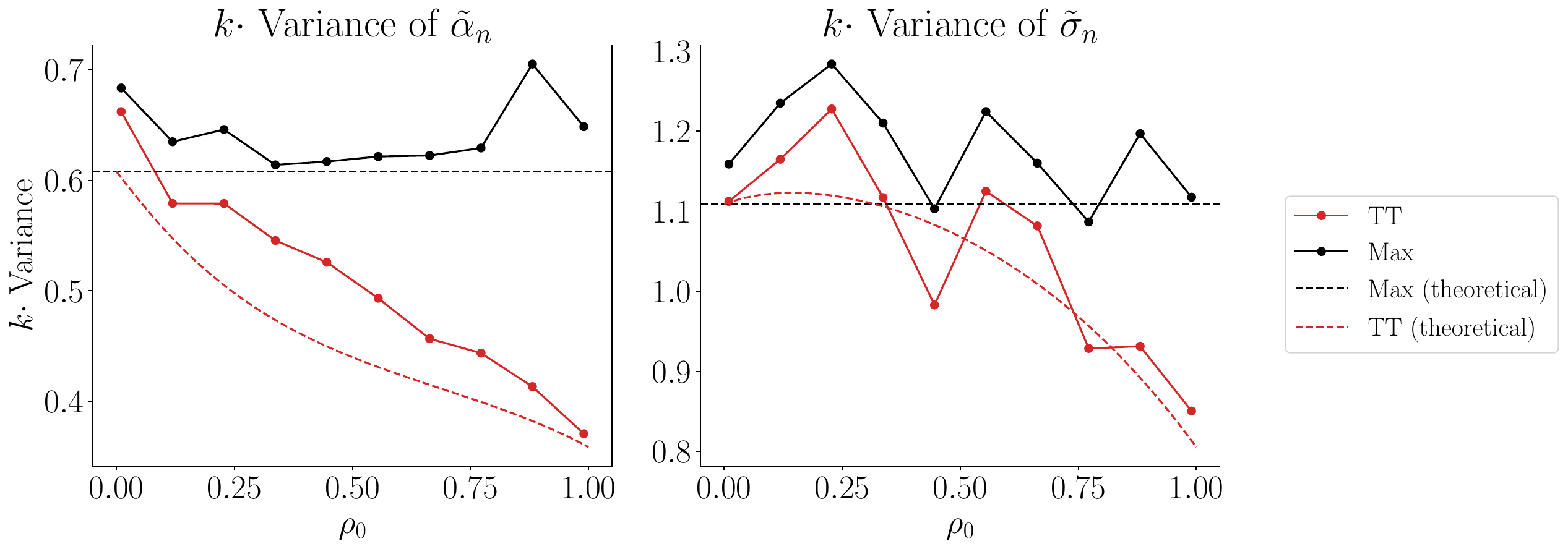}
    \caption{Simulated (1,000 repetitions) vs.\ theoretical rescaled variance of the disjoint blocks top-two and max-only estimators for the shape and scale parameter. The effective sample size is $k=200$ and the block size is $r=100$.}
    \label{fig:mori_simu_var}
\end{figure}

\begin{figure}[!thp]
    \centering
    \includegraphics[width=0.95\linewidth]{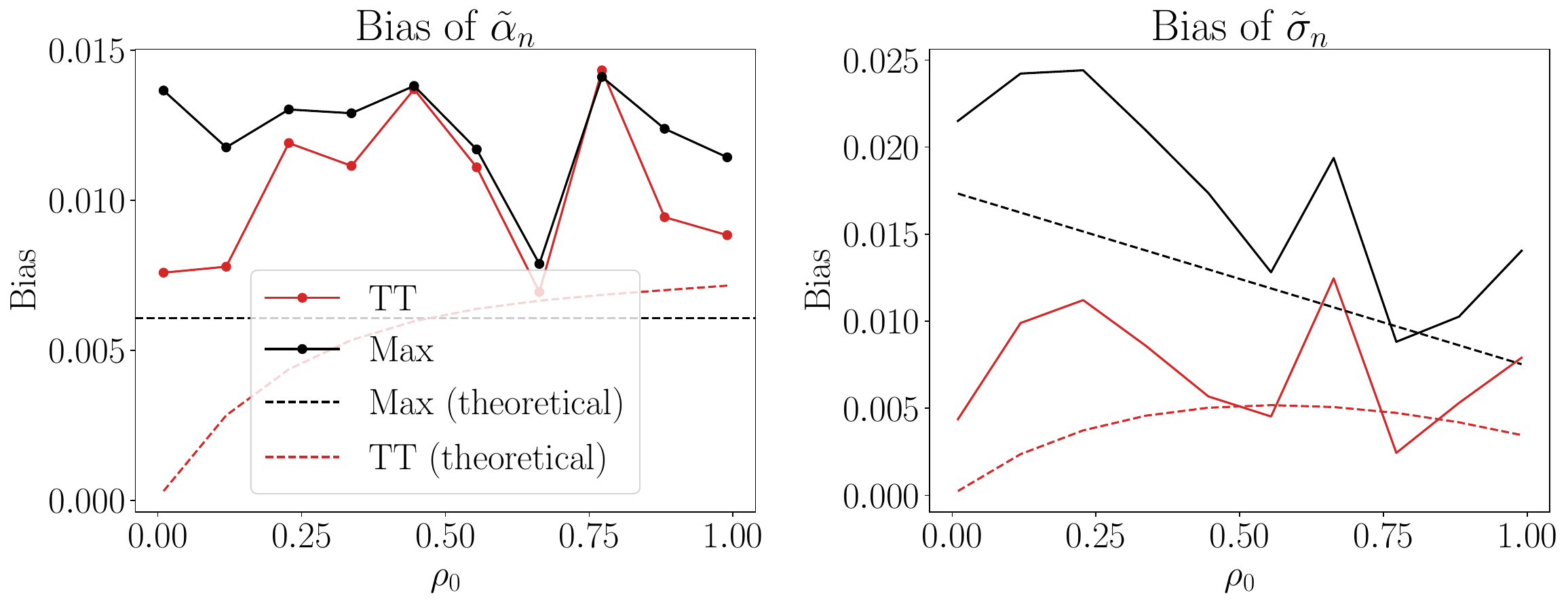}
    \caption{Simulated (1,000 repetitions) vs.\ theoretical rescaled bias of the disjoint blocks top-two and max-only estimators for the shape and scale parameter. The effective sample size is $k=200$ and the block size is $r=100$.}
    \label{fig:mori_simu_bias}
\end{figure}

\subsection{Bias correction}

In this section we study the effect of the additional estimation step needed for the bias-correction. We only consider the iid model and the ARMAX model, for which we know the true value of $\rho_0=1-\beta$ (with $\beta=0$ corresponding to the iid case). We can hence define an `oracle bias correction' by considering the estimator from \eqref{eq:bias-corrected-estimator} with the true value of $\rho_0$ and $\varpi_{\rho_0}$ instead of $\hat \rho_{0,n}$ and $\hat \varpi_{n}$. 

The difference between the estimated bias correction and the oracle  bias correction is illustrated in Figure~\ref{fig_supp:biascor}, where we consider shape estimation for fixed block sizes $r=50$ and $r=100$. The estimated bias correction is performed with respective block size parameter $r'=25$ and $r'=50$, respectively; see Section~\ref{sec:bias-corr-block-maxima} for the definition of $r'$. It can be seen that the oracle and the estimator perform quite similar, with small advantages for the estimated bias correction in some of the models.

\begin{figure}[t]
    \centering
    \includegraphics[width=.85\textwidth]{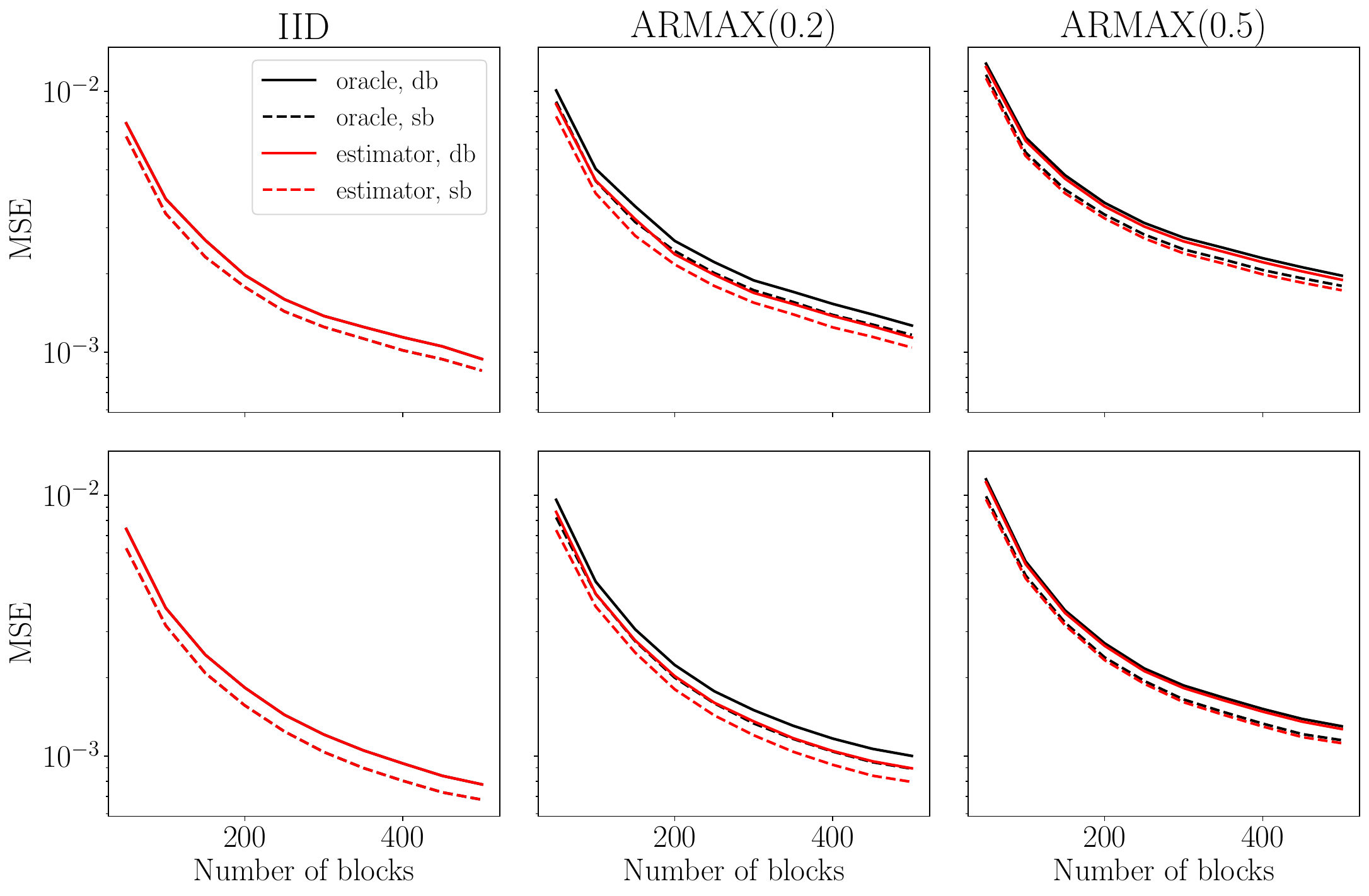}
    \caption{Shape estimation based on estimated bias correction (black) and oracle bias correction (red). Top row:  $r=50$ and $r'=25$. Bottom row: $r=100$ and $r'=50$.}
    \label{fig_supp:biascor}
\end{figure}

\subsection{Further results for fixed block sizes}
\label{subsec:further-results-fixed-block-size}

We present further details on the simulation results for the situation where the block size is fixed. In all the following results, the block size $r'$ for the bias correction from Section~\ref{sec:bias-corr-block-maxima} is chosen as $r'=25$ for $r=50$ and $r'=50$ for $r\in\{100, 200\}$. 

\paragraph{Estimating the scale parameter.}

We briefly present results for the estimation of the scale parameter. In view of the fact that the scale parameter is an asymptotic parameter that is not uniquely identifiable from the block size, we can only study the performance in terms of the estimation variance. The results are summarized in Figure~\ref{fig_supp:sigma}, where we restrict attention to the AR-model with block size $r=100$.  The results reveal that the sliding max-only estimator exhibits a smaller estimation variance than the top-two counterpart, which ultimately motivates the botw-estimator for the return level from \eqref{eq:botw}. While the ABM estimator shows a minimal variance in the iid case, it appears less competitive to the other estimators under temporal dependence.

\begin{figure}[th]
    \centering
    \includegraphics[width=0.85\linewidth]{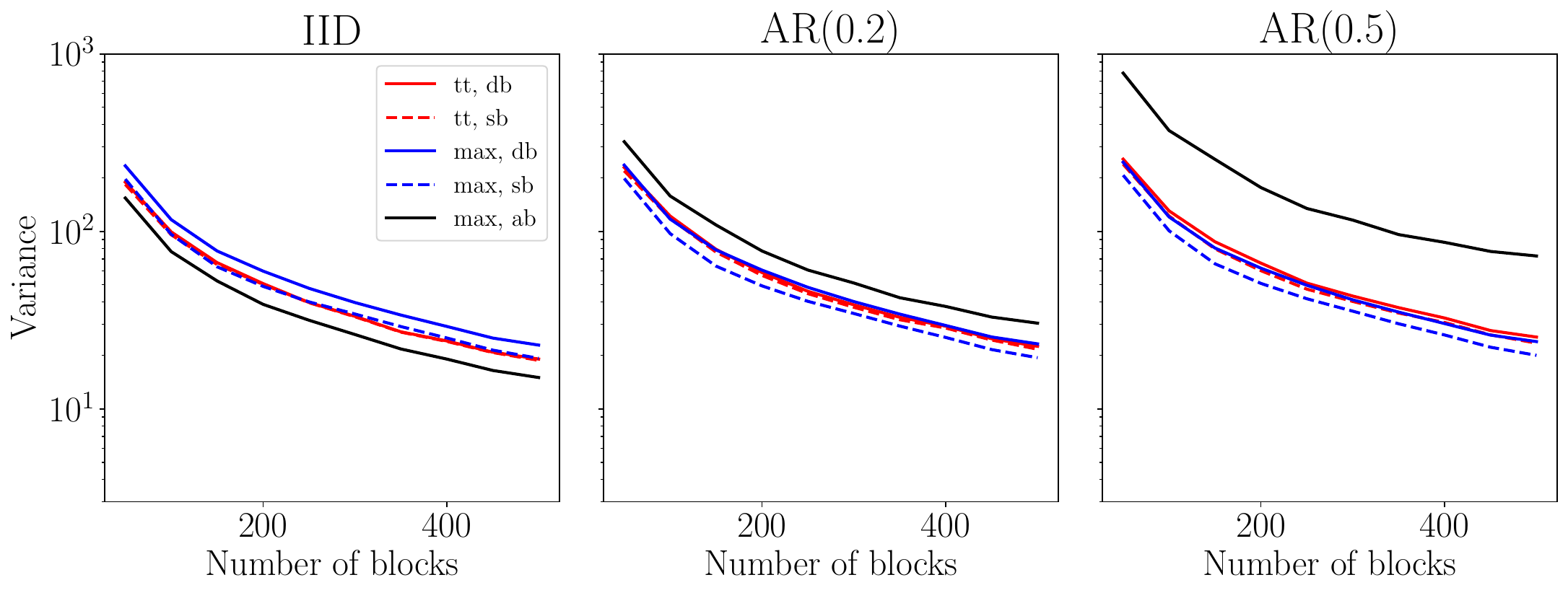}
    \caption{Scale estimation for fixed block size $r=100$. The estimation variance is shown here.}
    \label{fig_supp:sigma}
\end{figure}

\paragraph{Further block sizes.} 
We consider different block sizes, namely $r\in\{50,100,200\}$. 
The results are illustrated in  Figure~\ref{fig_supp:iid_fixedbs} (iid case) and Figure \ref{fig_supp:ar_fixedbs} (AR(0.5)-case). Overall, the results are consistent with those presented in Section~\ref{sec:sim-fixed-r}.

\begin{figure}[t]
    \centering
   \includegraphics[width=0.85\linewidth]{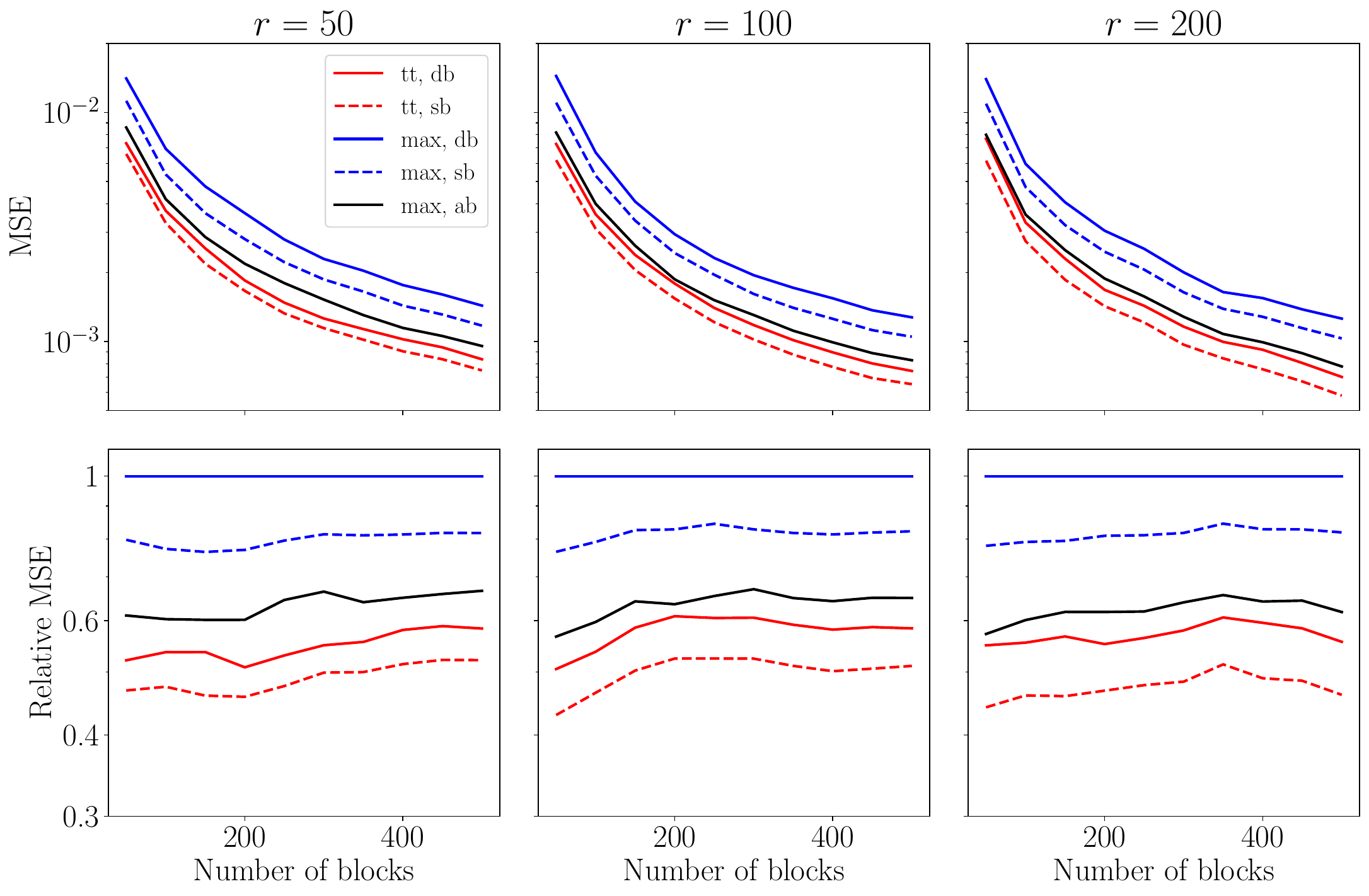}
    \caption{Shape estimation for the iid model with fixed block size.
    Top row: mean squared error. Bottom row: relative mean squared error with respect to the disjoint block maxima estimator, $\mathrm{MSE}(\,\cdot\,)/\mathrm{MSE}(\hat \alpha_{\max}^{(\dbl)})$.}
    \label{fig_supp:iid_fixedbs}
\end{figure}

\begin{figure}
    \centering
\includegraphics[width=0.85\linewidth]{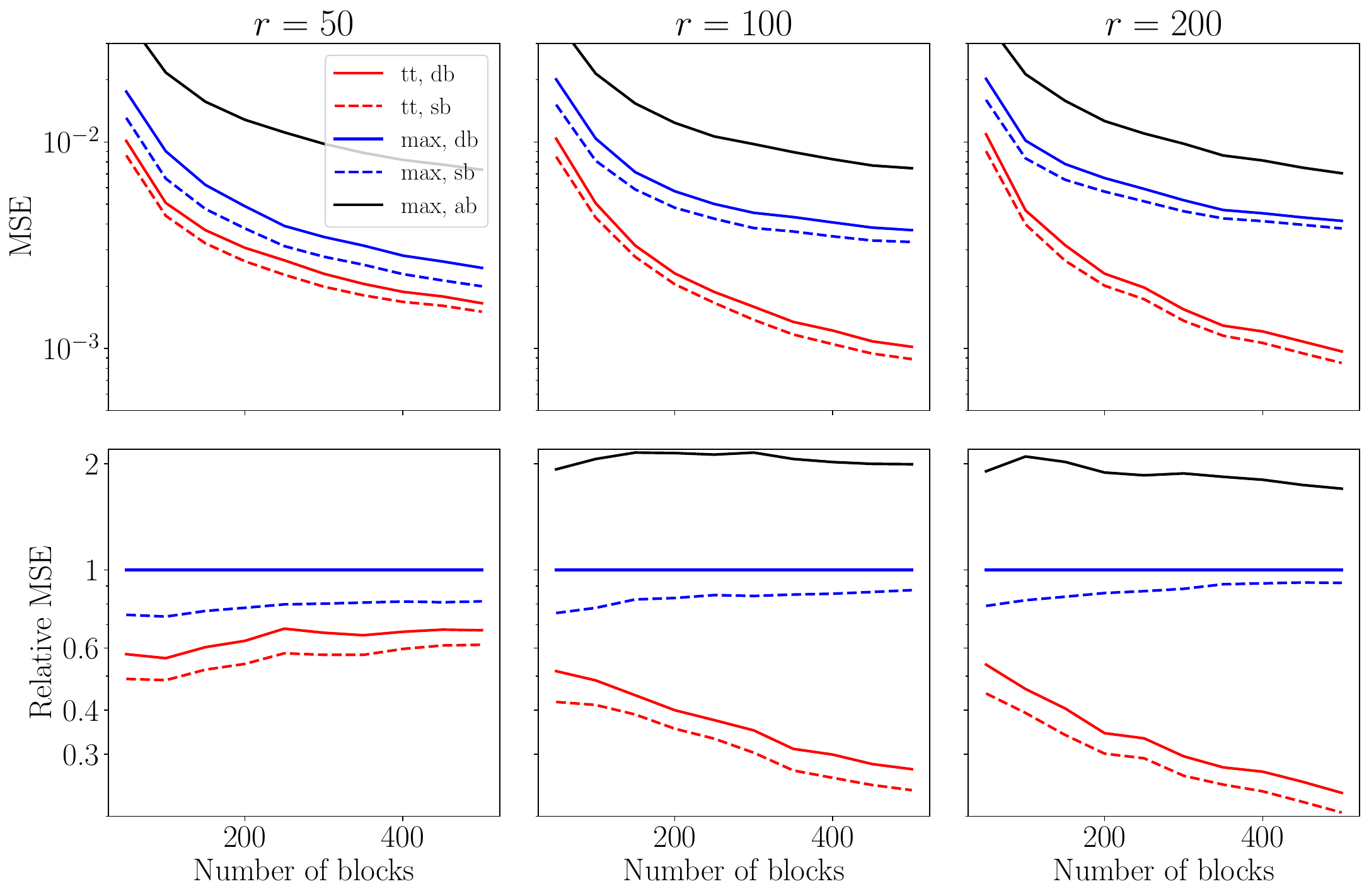}
\caption{Shape estimation for the AR(0.5)-model with fixed block size. 
Top row: mean squared error. Bottom row: relative mean squared error with respect to the disjoint block maxima estimator, $\mathrm{MSE}(\,\cdot\,)/\mathrm{MSE}(\hat \alpha_{\max}^{(\dbl)})$.}
    \label{fig_supp:ar_fixedbs}
\end{figure}

\paragraph{Further time series models.} 
We consider the remaining time series models that have been omitted in the presentation in Section~\ref{sec:sim-fixed-r}, namely, the AR-model with $\beta \ne 0.5$ and the ARMAX-model, both with fixed block size $r=100$. The results are presented in Figure~\ref{fig_supp:othermodels_shape} (shape estimation) and Figure~\ref{fig_supp:othermodels_rl} (return level estimation with $T=100$). The results are mostly consistent with those presented in Section~\ref{sec:sim-fixed-r}: unless the serial dependence is very strong, the top-two sliding estimator is best for shape estimation and the botw-estimator is best for return level estimation. For very strong serial dependence, the sliding max-only estimator wins. This can be explained by the fact that strong serial dependence decreases the effective block size and thus induces a comparably large bias for the top-two methods.

The ABM estimator exhibits a substantially larger MSE than all competing methods. Under moderate temporal dependence ($\beta = 0.2$), its MSE is approximately $50\%$ higher. For stronger dependence ($\beta = 0.2, 0.5$), the corresponding values lie outside the plotting range, indicating that the MSE exceeds that of the other estimators by more than a factor of two.

\begin{figure}
    \centering
\includegraphics[width=0.85\linewidth]{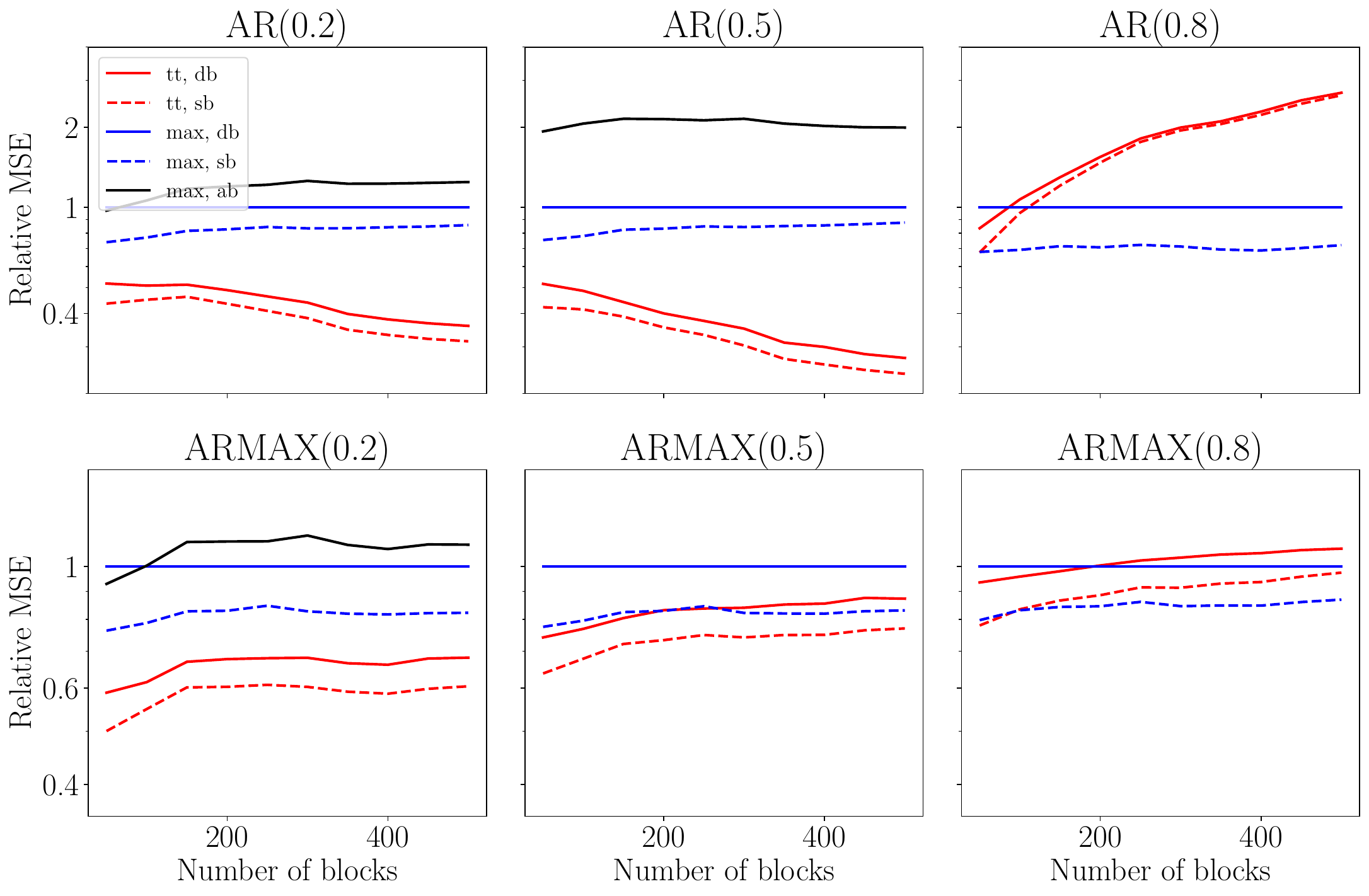}
    \caption{Shape estimation for fixed block size $r=100$. The curves represent the relative mean squared error with respect to the disjoint block maxima estimator, $\mathrm{MSE}(\,\cdot\,)/\mathrm{MSE}(\hat \alpha_{\max}^{(\dbl)})$. Top row: AR-models. Bottom row: ARMAX-models. The ABM estimator is only depicted on the left, as it is otherwise outside the plotting range.}
    \label{fig_supp:othermodels_shape}
\end{figure}

\begin{figure}
    \centering
\includegraphics[width=0.85\linewidth]{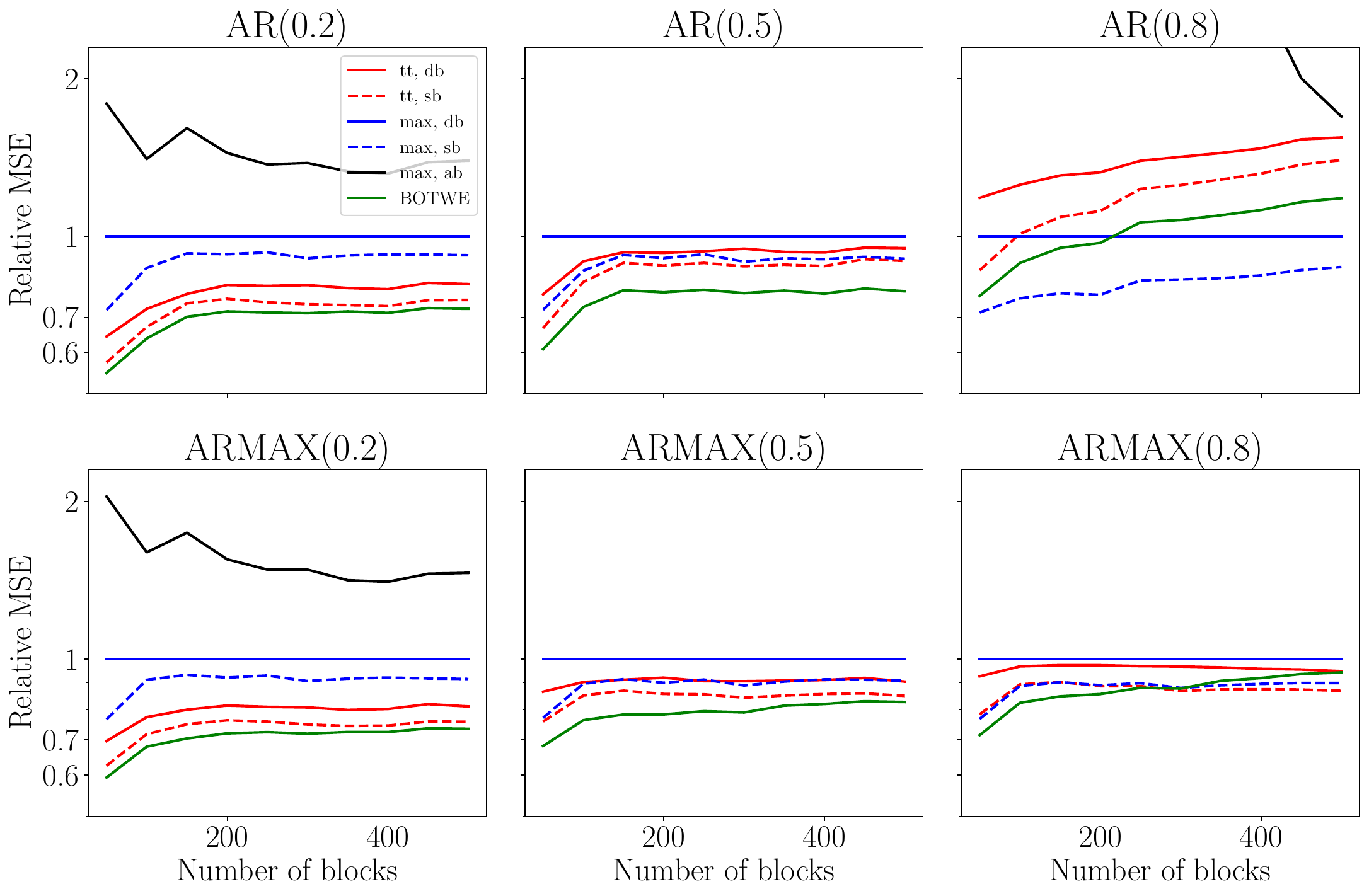}
    \caption{Return level estimation for fixed block size $r=100$ and for $T=100$. The curves represent the relative mean squared error with respect to the disjoint block maxima estimator, $\mathrm{MSE}(\,\cdot\,)/\mathrm{MSE}({\widehat {\RL}}\!\,_{\max}^{(\dbl)})$. Top row: AR-models. Bottom row: ARMAX-models. The ABM estimator is not shown when it falls outside the plotting range (with values roughly between 2 and 10 in the middle and between 2 and 200 on the right). 
    }
    \label{fig_supp:othermodels_rl}
\end{figure}

\paragraph{Further return levels.}
We finally consider the estimation of return levels with fixed block size $r=100$ and varying `annuality' $T\in\{50,100,200\}$. 
The results are summarized in Figure~\ref{fig_supp:RLs}, where we we restrict attention to the AR(0.5)-model for the sake of brevity. The botw-estimator is best in all scenarios under consideration.

\begin{figure}
    \centering
    \includegraphics[width=0.85\linewidth]{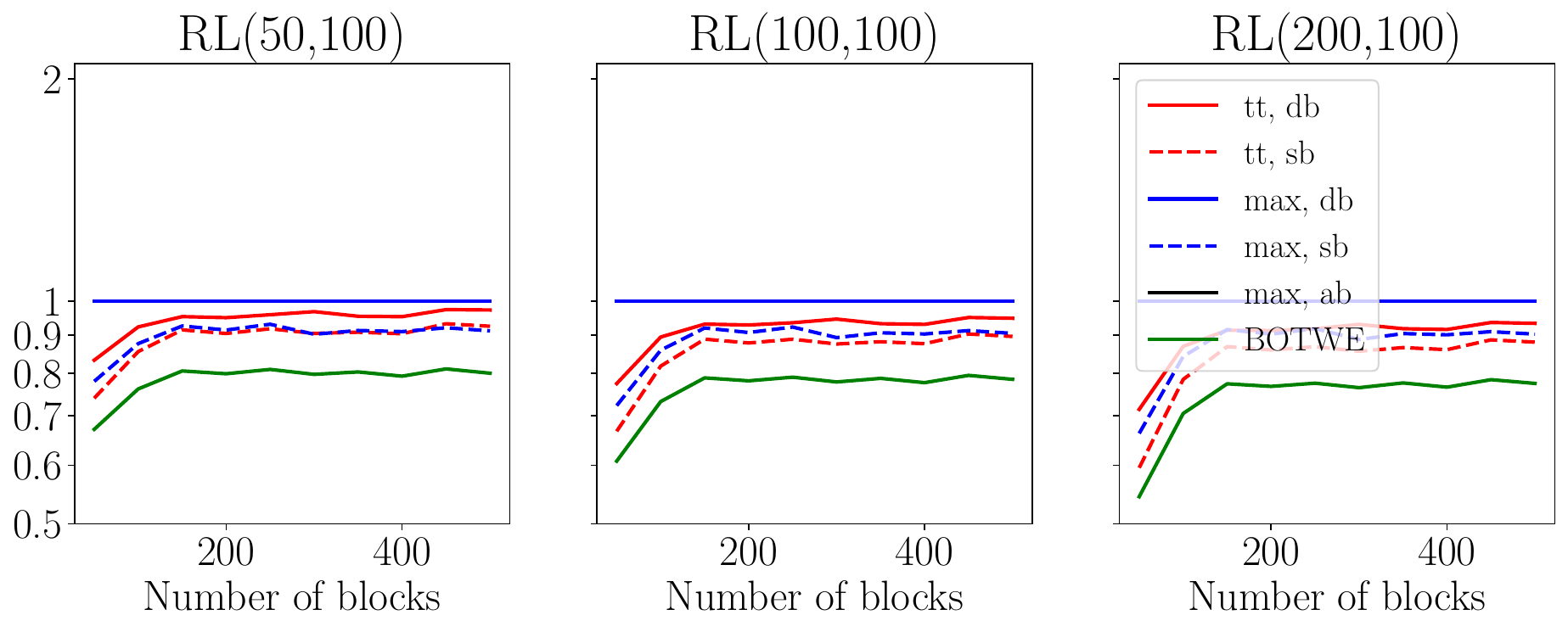}
    \caption{Return level estimation for the AR(0.5)-model with fixed block size $r=100$ and $T \in \{50,100,200\}$. The ABM estimator is not shown as it falls outside the plotting range (with values consistently between 2 and 15).}
    \label{fig_supp:RLs}
\end{figure}

\subsection{Further results for fixed total sample size}

Section~\ref{subsec:simulation-fixed-sample-size}, in particular Figure~\ref{fig:fixed_n}, shows that for fixed $n=10,000$ the minimum MSE is achieved by the sliding top-two estimator in the iid model, and by the all block maxima estimator in the AR model. The bias-variance decomposition in Figure~\ref{fig:fixed_n_decomp} clarifies the mechanism underlying these observations. In the AR model, the bias of the ABM estimator crosses zero at a relatively small block size (approximately $r \approx 20$), leading to an early MSE minimum at small $r$. This behavior is absent in the iid and ARMAX model, where no such zero crossing occurs. Instead, in these models the lower variance of the max-only and top-two estimators dominates, resulting in the overall minimal MSE at larger block sizes ($r\approx 50$).

\begin{figure}[!htp]
    \centering
    \includegraphics[width=0.95\linewidth]{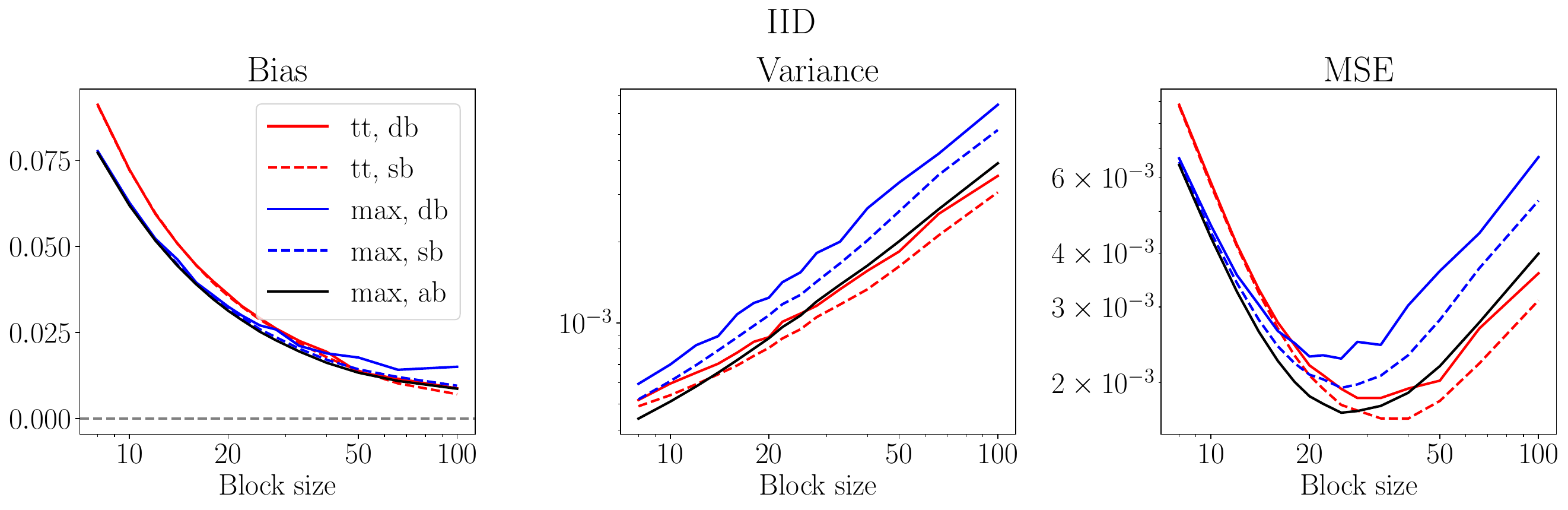}
    \includegraphics[width=0.95\linewidth]{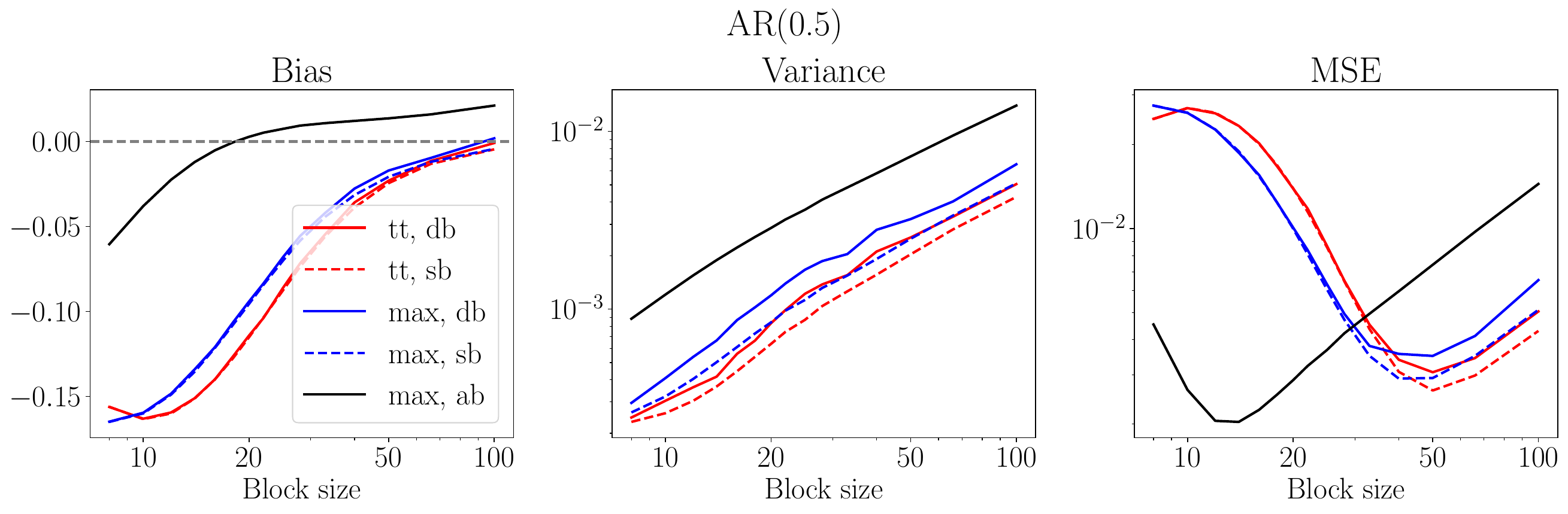}
    \includegraphics[width=0.95\linewidth]{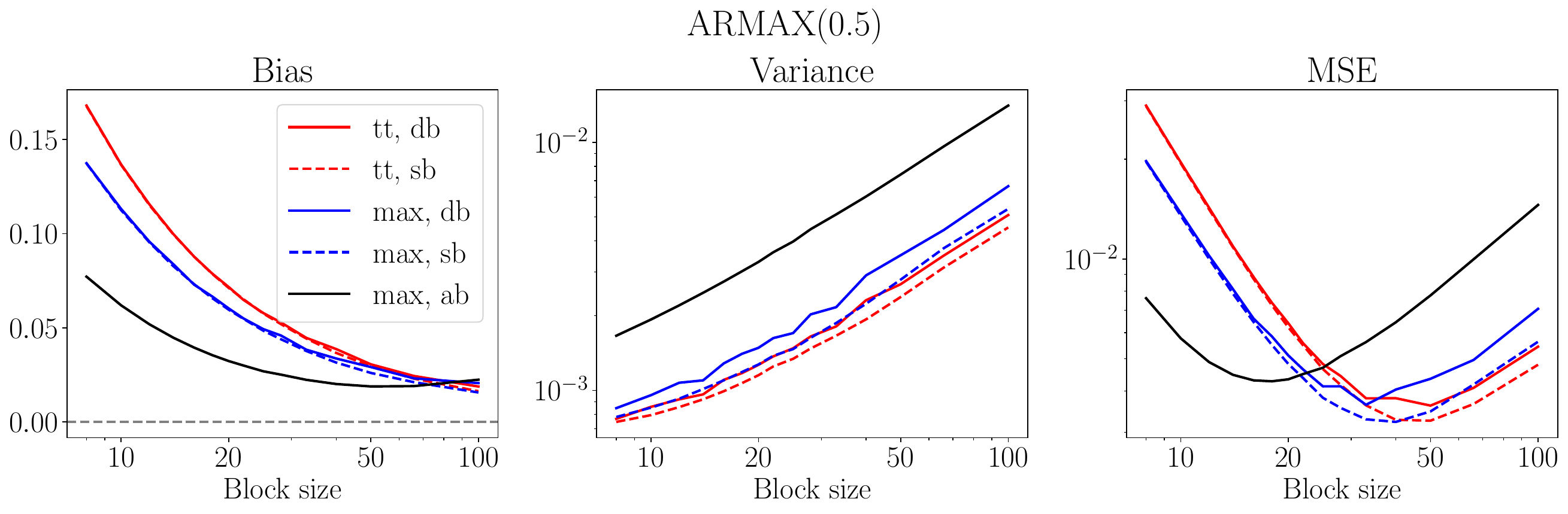}
    \caption{Estimation of $\alpha_0$ for fixed $n=10,000$. Left: bias, middle: variance, right: MSE (see also Figure~\ref{fig:fixed_n})}.
    \label{fig:fixed_n_decomp}
\end{figure}

\subsection{Comprehensive results for different block sizes and different numbers of blocks}

We finally present results for a more comprehensive range of block sizes and number of blocks, both ranging from $25$ to $500$. For the sake of brevity, we only report results for the iid-model (Figure \ref{fig_supp:heat_iid}) and the AR(0.5)-model (Figure \ref{fig_supp:heat_ar05}); results for the other models are qualitatively similar. 

The results are consistent with previous findings: the sliding blocks top-two estimator is the best estimator in most scenarios under consideration, except for very small block sizes, where the all block maxima method wins. The latter is not competitive in the case of serial dependence for $r\ge 50$.

\begin{figure}
    \centering
    \includegraphics[width=.95\textwidth]{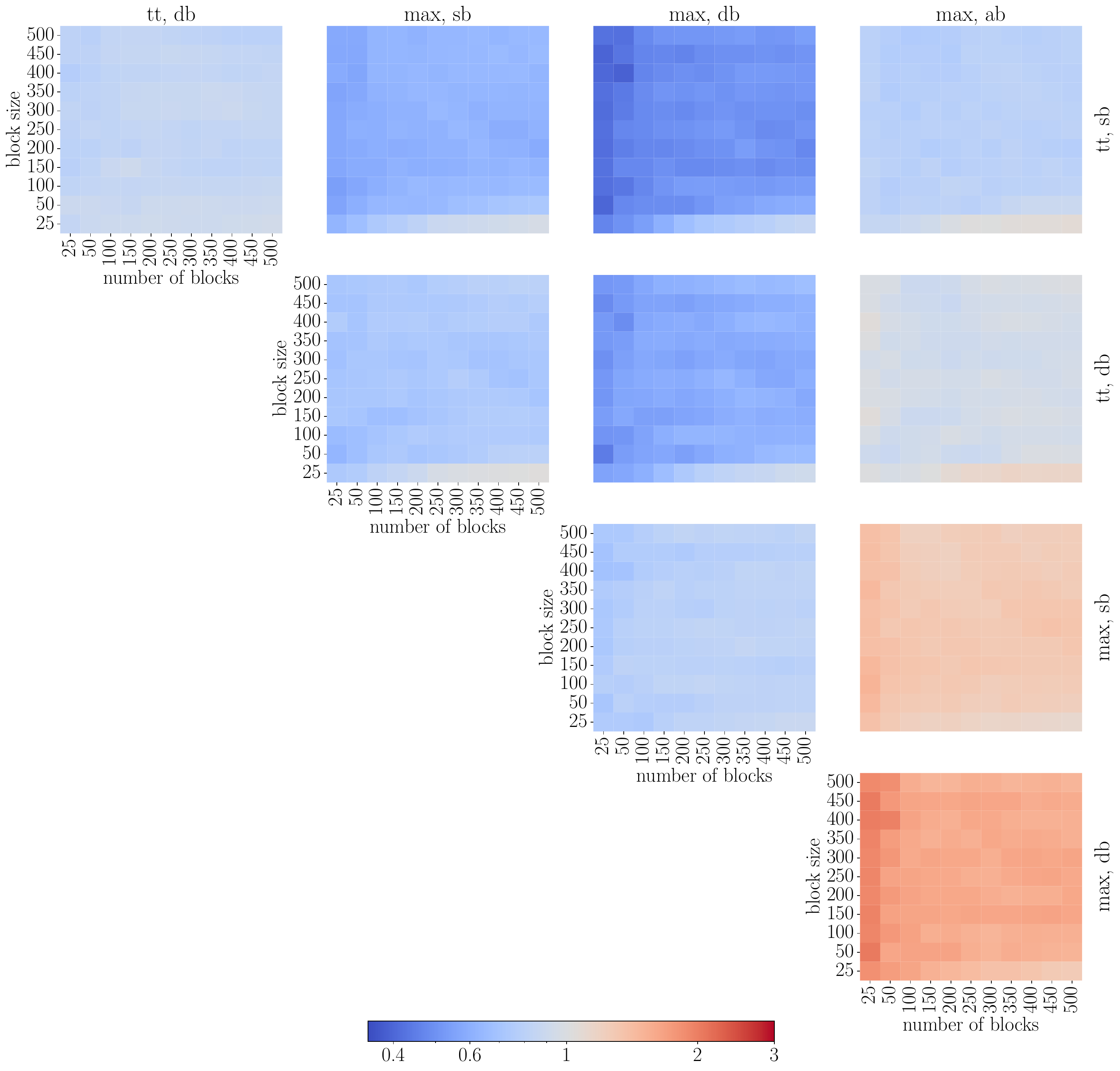}
    \caption{Shape estimation in the iid model for various combinations of the block size and the number of blocks ranging from 25 to 500. Depicted is the relative MSE, i.e., the MSE of the estimator indicated on right divided by the MSE of the estimator indicated at the top. Red color means that the top estimator performs better.}
    \label{fig_supp:heat_iid}
\end{figure}

\begin{figure}
    \centering
    \includegraphics[width=.95\textwidth]{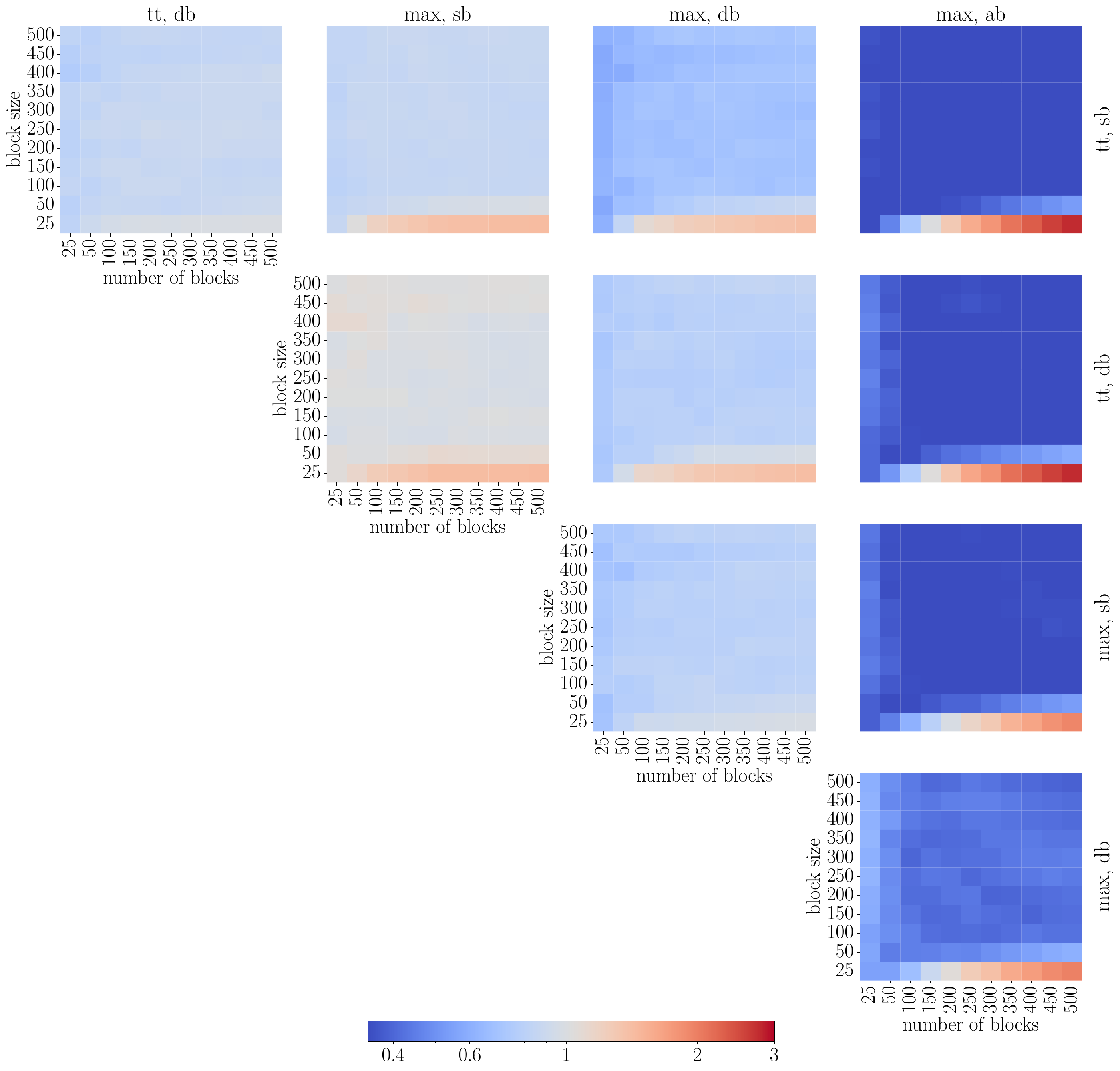}
    \caption{Shape estimation in the AR(0.5)-model for various combinations of the block size  and the number of blocks ranging from 25 to 500. Depicted is the relative MSE, i.e., the MSE of the estimator indicated on right divided by the MSE of the estimator indicated at the top. Red color means that the top estimator performs better.}
    \label{fig_supp:heat_ar05}
\end{figure}

\clearpage

\clearpage

\printbibliography

@article{philip2020protocol,
  title={A protocol for probabilistic extreme event attribution analyses},
  author={Philip, Sjoukje and others},
  journal={Advances in Statistical Climatology, Meteorology and Oceanography},
  volume={6},
  number={2},
  pages={177--203},
  year={2020},
  publisher={Copernicus GmbH}
}

@book{Roc97,
    AUTHOR = {Rockafellar, R. Tyrrell},
     TITLE = {Convex analysis},
    SERIES = {Princeton Landmarks in Mathematics},
      NOTE = {Reprint of the 1970 original,
              Princeton Paperbacks},
 PUBLISHER = {Princeton University Press, Princeton, NJ},
      YEAR = {1997},
     PAGES = {xviii+451},
      ISBN = {0-691-01586-4},
   MRCLASS = {49-02 (26-02 46-02 58-02 90-02)},
  MRNUMBER = {1451876},
}

@book{DavHin97, 
place={Cambridge}, 
series={Cambridge Series in Statistical and Probabilistic Mathematics}, 
title={Bootstrap Methods and their Application}, 
DOI={10.1017/CBO9780511802843}, 
publisher={Cambridge University Press}, 
author={Davison, A. C. and Hinkley, D. V.}, 
year={1997}, 
collection={Cambridge Series in Statistical and Probabilistic Mathematics},
}

@article{Ram02,
title = {Local models for exploratory analysis of hydrological extremes},
journal = {Journal of Hydrology},
volume = {256},
number = {1},
pages = {106-119},
year = {2002},
issn = {0022-1694},
doi = {10.1016/S0022-1694(01)00522-4},
url = {https://www.sciencedirect.com/science/article/pii/S0022169401005224},
author = {N.I. Ramesh and A.C. Davison}
}

@article {Wei78,
    AUTHOR = {Weissman, Ishay},
     TITLE = {Estimation of parameters and large quantiles based on the
              {$k$}\ largest observations},
   JOURNAL = {J. Amer. Statist. Assoc.},
  FJOURNAL = {Journal of the American Statistical Association},
    VOLUME = {73},
      YEAR = {1978},
    NUMBER = {364},
     PAGES = {812--815},
      ISSN = {0162-1459,1537-274X},
   MRCLASS = {62F12 (62G30)},
  MRNUMBER = {521329},
MRREVIEWER = {J.\ Tiago de Oliveira},
       URL = {http://www.jstor.org/stable/2286285},
 DOI = {10.2307/2286285},
}

@article{Smi86,
title = {Extreme value theory based on the $r$ largest annual events},
journal = {Journal of Hydrology},
volume = {86},
number = {1},
pages = {27-43},
year = {1986},
issn = {0022-1694},
doi = {10.1016/0022-1694(86)90004-1},
url = {https://www.sciencedirect.com/science/article/pii/0022169486900041},
author = {Richard L. Smith}
}

@article{Taw88,
title = {An extreme-value theory model for dependent observations},
journal = {Journal of Hydrology},
volume = {101},
number = {1},
pages = {227-250},
year = {1988},
issn = {0022-1694},
doi = {10.1016/0022-1694(88)90037-6},
url = {https://www.sciencedirect.com/science/article/pii/0022169488900376},
author = {Jonathan A. Tawn}
}

@article {Lea83,
    AUTHOR = {Leadbetter, M. R.},
     TITLE = {Extremes and local dependence in stationary sequences},
   JOURNAL = {Z. Wahrsch. Verw. Gebiete},
  FJOURNAL = {Zeitschrift f\"ur Wahrscheinlichkeitstheorie und Verwandte
              Gebiete},
    VOLUME = {65},
      YEAR = {1983},
    NUMBER = {2},
     PAGES = {291--306},
      ISSN = {0044-3719},
   MRCLASS = {60G10},
  MRNUMBER = {722133},
MRREVIEWER = {Robert\ J.\ Adler},
       DOI = {10.1007/BF00532484},
       URL = {https://doi.org/10.1007/BF00532484},
}

@article {BucSta24,
    AUTHOR = {B\"ucher, Axel and Staud, Torben},
     TITLE = {Limit theorems for non-degenerate {U}-statistics of block
              maxima for time series},
   JOURNAL = {Electron. J. Stat.},
  FJOURNAL = {Electronic Journal of Statistics},
    VOLUME = {18},
      YEAR = {2024},
    NUMBER = {2},
     PAGES = {2850--2885},
      ISSN = {1935-7524},
   MRCLASS = {62G32 (60G70 62G20 62M10)},
  MRNUMBER = {4772513},
       DOI = {10.1214/24-ejs2269},
       URL = {https://doi.org/10.1214/24-ejs2269},
}

@article {Zou21,
    AUTHOR = {Zou, Nan and Volgushev, Stanislav and B\"ucher, Axel},
     TITLE = {Multiple block sizes and overlapping blocks for multivariate
              time series extremes},
   JOURNAL = {Ann. Statist.},
  FJOURNAL = {The Annals of Statistics},
    VOLUME = {49},
      YEAR = {2021},
    NUMBER = {1},
     PAGES = {295--320},
      ISSN = {0090-5364,2168-8966},
   MRCLASS = {62G32 (60G70 62E20 62G20)},
  MRNUMBER = {4206679},
MRREVIEWER = {Holger\ Drees},
       DOI = {10.1214/20-AOS1957},
       URL = {https://doi.org/10.1214/20-AOS1957},
}

@article {Dom19,
    AUTHOR = {Dombry, Cl\'ement and Ferreira, Ana},
     TITLE = {Maximum likelihood estimators based on the block maxima
              method},
   JOURNAL = {Bernoulli},
  FJOURNAL = {Bernoulli. Official Journal of the Bernoulli Society for
              Mathematical Statistics and Probability},
    VOLUME = {25},
      YEAR = {2019},
    NUMBER = {3},
     PAGES = {1690--1723},
      ISSN = {1350-7265,1573-9759},
   MRCLASS = {62G32 (62E20)},
  MRNUMBER = {3961227},
MRREVIEWER = {Emanuele\ Taufer},
       DOI = {10.3150/18-BEJ1032},
       URL = {https://doi.org/10.3150/18-BEJ1032},
}

@article {Pre80,
    AUTHOR = {Prescott, P. and Walden, A. T.},
     TITLE = {Maximum likelihood estimation of the parameters of the
              generalized extreme-value distribution},
   JOURNAL = {Biometrika},
  FJOURNAL = {Biometrika},
    VOLUME = {67},
      YEAR = {1980},
    NUMBER = {3},
     PAGES = {723--724},
      ISSN = {0006-3444,1464-3510},
   MRCLASS = {62F10},
  MRNUMBER = {601119},
       DOI = {10.1093/biomet/67.3.723},
       URL = {https://doi.org/10.1093/biomet/67.3.723},
}

@article {Hos85,
    AUTHOR = {Hosking, J. R. M. and Wallis, J. R. and Wood, E. F.},
     TITLE = {Estimation of the generalized extreme-value distribution by
              the method of probability-weighted moments},
   JOURNAL = {Technometrics},
  FJOURNAL = {Technometrics. A Journal of Statistics for the Physical,
              Chemical and Engineering Sciences},
    VOLUME = {27},
      YEAR = {1985},
    NUMBER = {3},
     PAGES = {251--261},
      ISSN = {0040-1706,1537-2723},
   MRCLASS = {62G30},
  MRNUMBER = {797563},
       DOI = {10.2307/1269706},
       URL = {https://doi.org/10.2307/1269706},
}

@book{Dou94,
    AUTHOR = {Doukhan, Paul},
     TITLE = {Mixing},
    SERIES = {Lecture Notes in Statistics},
    VOLUME = {85},
      NOTE = {Properties and examples},
 PUBLISHER = {Springer-Verlag, New York},
      YEAR = {1994},
     PAGES = {xii+142},
      ISBN = {0-387-94214-9},
   MRCLASS = {60G05 (60E99 60F99)},
  MRNUMBER = {1312160},
MRREVIEWER = {Wlodzimierz\ Bryc},
       DOI = {10.1007/978-1-4612-2642-0},
       URL = {https://doi.org/10.1007/978-1-4612-2642-0},
}

@Article{BucZan23,
 Author = {B{\"u}cher, Axel and Zanger, Leandra},
 Title = {On the disjoint and sliding block maxima method for piecewise stationary time series},
 FJournal = {The Annals of Statistics},
 Journal = {Ann. Stat.},
 ISSN = {0090-5364},
 Volume = {51},
 Number = {2},
 Pages = {573--598},
 Year = {2023},
 Language = {English},
 DOI = {10.1214/23-AOS2260},
 zbMATH = {7714172},
 Zbl = {1539.62264}
}

@Misc{OorZho20,
      title={All Block Maxima method for estimating the extreme value index}, 
      author={Jochem Oorschot and Chen Zhou},
      year={2020},
      eprint={2010.15950},
      archivePrefix={arXiv},
      primaryClass={math.ST},
      url={https://arxiv.org/abs/2010.15950}, 
}

@article {Rob09,
    AUTHOR = {Robert, Christian Y.},
     TITLE = {Inference for the limiting cluster size distribution of
              extreme values},
   JOURNAL = {Ann. Statist.},
  FJOURNAL = {The Annals of Statistics},
    VOLUME = {37},
      YEAR = {2009},
    NUMBER = {1},
     PAGES = {271--310},
      ISSN = {0090-5364},
     CODEN = {ASTSC7},
   MRCLASS = {60G70 (62E20 62G32 62M09)},
MRREVIEWER = {Yasutaka Shimizu},
       DOI = {10.1214/07-AOS551},
       URL = {http://dx.doi.org/10.1214/07-AOS551},
}

@article {Rob09b,
    AUTHOR = {Robert, C. Y.},
     TITLE = {Asymptotic distributions for the intervals estimators of the
              extremal index and the cluster-size probabilities},
   JOURNAL = {J. Statist. Plann. Inference},
  FJOURNAL = {Journal of Statistical Planning and Inference},
    VOLUME = {139},
      YEAR = {2009},
    NUMBER = {9},
     PAGES = {3288--3309},
      ISSN = {0378-3758},
   MRCLASS = {62G32 (60F05 62E20)},
  MRNUMBER = {2535201},
       DOI = {10.1016/j.jspi.2009.03.010},
       URL = {https://doi.org/10.1016/j.jspi.2009.03.010},
}

@article {Hsi91,
    AUTHOR = {Hsing, Tailen},
     TITLE = {Estimating the parameters of rare events},
   JOURNAL = {Stochastic Process. Appl.},
  FJOURNAL = {Stochastic Processes and their Applications},
    VOLUME = {37},
      YEAR = {1991},
    NUMBER = {1},
     PAGES = {117--139},
      ISSN = {0304-4149},
   MRCLASS = {62F12 (60G70 62E20)},
  MRNUMBER = {1091698},
MRREVIEWER = {N. Leonenko},
       DOI = {10.1016/0304-4149(91)90064-J},
       URL = {https://doi.org/10.1016/0304-4149(91)90064-J},
}

@phdthesis{Fer03,
  title={Statistical Methods for Cluster of Extreme Values},
  author={Ferro, Christopher AT},
  year={2003},
  school={University of Lancaster}
}

@book {Bei04,
    AUTHOR = {Beirlant, Jan and Goegebeur, Yuri and Teugels, Jozef and
              Segers, Johan},
     TITLE = {Statistics of extremes},
    SERIES = {Wiley Series in Probability and Statistics},
      NOTE = {Theory and applications,
              With contributions from Daniel De Waal and Chris Ferro},
 PUBLISHER = {John Wiley \& Sons, Ltd., Chichester},
      YEAR = {2004},
     PAGES = {xiv+490},
      ISBN = {0-471-97647-4},
   MRCLASS = {62-01 (60G70 62G32)},
  MRNUMBER = {2108013},
MRREVIEWER = {L\'aszl\'o\ Viharos},
       DOI = {10.1002/0470012382},
       URL = {https://doi.org/10.1002/0470012382},
}

@book {Col01,
    AUTHOR = {Coles, Stuart},
     TITLE = {An introduction to statistical modeling of extreme values},
    SERIES = {Springer Series in Statistics},
 PUBLISHER = {Springer-Verlag London, Ltd., London},
      YEAR = {2001},
     PAGES = {xiv+208},
      ISBN = {1-85233-459-2},
   MRCLASS = {62-01 (60G70 62G32)},
  MRNUMBER = {1932132},
       DOI = {10.1007/978-1-4471-3675-0},
       URL = {https://doi.org/10.1007/978-1-4471-3675-0},
}

@article{Bücher2018-sliding,
    AUTHOR = {B\"ucher, Axel and Segers, Johan},
     TITLE = {Inference for heavy tailed stationary time series based on sliding blocks},
   JOURNAL = {Electron. J. Stat.},
  FJOURNAL = {Electronic Journal of Statistics},
    VOLUME = {12},
      YEAR = {2018},
    NUMBER = {1},
     PAGES = {1098--1125},
      ISSN = {1935-7524},
   MRCLASS = {62G32 (62F10 62M10)},
  MRNUMBER = {3780041},
       DOI = {10.1214/18-EJS1415},
       URL = {https://doi.org/10.1214/18-EJS1415},
}

@article {Bücher2018-disjoint,
    AUTHOR = {B\"{u}cher, Axel and Segers, Johan},
     TITLE = {Maximum likelihood estimation for the {F}r\'{e}chet
              distribution based on block maxima extracted from a time
              series},
   JOURNAL = {Bernoulli},
  FJOURNAL = {Bernoulli. Official Journal of the Bernoulli Society for
              Mathematical Statistics and Probability},
    VOLUME = {24},
      YEAR = {2018},
    NUMBER = {2},
     PAGES = {1427--1462},
      ISSN = {1350-7265,1573-9759},
   MRCLASS = {62M10 (60G70 62F10)},
  MRNUMBER = {3706798},
       DOI = {10.3150/16-BEJ903},
       URL = {https://doi.org/10.3150/16-BEJ903},
}

@article {Welsch_1972,
    AUTHOR = {Welsch, Roy E.},
     TITLE = {Limit laws for extreme order statistics from strong-mixing
              processes},
   JOURNAL = {Ann. Math. Statist.},
  FJOURNAL = {Annals of Mathematical Statistics},
    VOLUME = {43},
      YEAR = {1972},
     PAGES = {439--446},
      ISSN = {0003-4851},
   MRCLASS = {62E20 (62G30)},
  MRNUMBER = {307306},
MRREVIEWER = {Z.\ W.\ Birnbaum},
       DOI = {10.1214/aoms/1177692624},
       URL = {https://doi.org/10.1214/aoms/1177692624},
}

@article {Dombry2015,
    AUTHOR = {Dombry, Cl\'{e}ment},
     TITLE = {Existence and consistency of the maximum likelihood estimators
              for the extreme value index within the block maxima framework},
   JOURNAL = {Bernoulli},
  FJOURNAL = {Bernoulli. Official Journal of the Bernoulli Society for
              Mathematical Statistics and Probability},
    VOLUME = {21},
      YEAR = {2015},
    NUMBER = {1},
     PAGES = {420--436},
      ISSN = {1350-7265,1573-9759},
   MRCLASS = {62F10},
  MRNUMBER = {3322325},
       DOI = {10.3150/13-BEJ573},
       URL = {https://doi.org/10.3150/13-BEJ573},
}

@article {Ferreira2015,
    AUTHOR = {Ferreira, Ana and de Haan, Laurens},
     TITLE = {On the block maxima method in extreme value theory: {PWM}
              estimators},
   JOURNAL = {Ann. Statist.},
  FJOURNAL = {The Annals of Statistics},
    VOLUME = {43},
      YEAR = {2015},
    NUMBER = {1},
     PAGES = {276--298},
      ISSN = {0090-5364,2168-8966},
   MRCLASS = {62G32 (62G20 62G30)},
  MRNUMBER = {3285607},
MRREVIEWER = {Gilles\ Stupfler},
       DOI = {10.1214/14-AOS1280},
       URL = {https://doi.org/10.1214/14-AOS1280},
}

@article {Gnedenko1943,
    AUTHOR = {Gnedenko, B.},
     TITLE = {Sur la distribution limite du terme maximum d'une s\'{e}rie
              al\'{e}atoire},
   JOURNAL = {Ann. of Math. (2)},
  FJOURNAL = {Annals of Mathematics. Second Series},
    VOLUME = {44},
      YEAR = {1943},
     PAGES = {423--453},
      ISSN = {0003-486X},
   MRCLASS = {60.0X},
  MRNUMBER = {8655},
MRREVIEWER = {M.\ Kac},
       DOI = {10.2307/1968974},
       URL = {https://doi.org/10.2307/1968974},
}

@book {Bingham1987,
    AUTHOR = {Bingham, N. H. and Goldie, C. M. and Teugels, J. L.},
     TITLE = {Regular variation},
    SERIES = {Encyclopedia of Mathematics and its Applications},
    VOLUME = {27},
 PUBLISHER = {Cambridge University Press, Cambridge},
      YEAR = {1987},
     PAGES = {xx+491},
      ISBN = {0-521-30787-2},
   MRCLASS = {26A12 (11K65 11N60 30-02 40E05 60-02 60Fxx)},
  MRNUMBER = {898871},
MRREVIEWER = {R.\ A.\ Maller},
       DOI = {10.1017/CBO9780511721434},
       URL = {https://doi.org/10.1017/CBO9780511721434},
}

@book {vdVaart,
    AUTHOR = {van~der~Vaart, A. W.},
     TITLE = {Asymptotic statistics},
    SERIES = {Cambridge Series in Statistical and Probabilistic Mathematics},
    VOLUME = {3},
 PUBLISHER = {Cambridge University Press, Cambridge},
      YEAR = {1998},
     PAGES = {xvi+443},
      ISBN = {0-521-49603-9},
   MRCLASS = {62-02 (62E20 62F05 62F12 62G07 62G09 62G20)},
  MRNUMBER = {1652247},
MRREVIEWER = {Nancy\ Reid},
       DOI = {10.1017/CBO9780511802256},
       URL = {https://doi.org/10.1017/CBO9780511802256},
}

@article {Bücher2014,
    AUTHOR = {B\"{u}cher, Axel and Segers, Johan},
     TITLE = {Extreme value copula estimation based on block maxima of a
              multivariate stationary time series},
   JOURNAL = {Extremes},
  FJOURNAL = {Extremes. Statistical Theory and Applications in Science,
              Engineering and Economics},
    VOLUME = {17},
      YEAR = {2014},
    NUMBER = {3},
     PAGES = {495--528},
      ISSN = {1386-1999,1572-915X},
   MRCLASS = {62G32 (60F05 62H12)},
  MRNUMBER = {3252823},
       DOI = {10.1007/s10687-014-0195-8},
       URL = {https://doi.org/10.1007/s10687-014-0195-8},
}

@Inbook{Dehling2002,
author="Dehling, Herold
and Philipp, Walter",
editor="Dehling, Herold
and Mikosch, Thomas
and S{\o}rensen, Michael",
title="Empirical Process Techniques for Dependent Data",
bookTitle="Empirical Process Techniques for Dependent Data",
year="2002",
publisher="Birkh{\"a}user Boston",
address="Boston, MA",
pages="3--113",
isbn="978-1-4612-0099-4",
doi="10.1007/978-1-4612-0099-4_1",
url="https://doi.org/10.1007/978-1-4612-0099-4_1"
}

@book {Resnick1987,
    AUTHOR = {Resnick, Sidney I.},
     TITLE = {Extreme values, regular variation, and point processes},
    SERIES = {Applied Probability. A Series of the Applied Probability Trust},
    VOLUME = {4},
 PUBLISHER = {Springer-Verlag, New York},
      YEAR = {1987},
     PAGES = {xii+320},
      ISBN = {0-387-96481-9},
   MRCLASS = {60K99 (60G55)},
  MRNUMBER = {900810},
MRREVIEWER = {Charles\ M.\ Goldie},
       DOI = {10.1007/978-0-387-75953-1},
       URL = {https://doi.org/10.1007/978-0-387-75953-1},
}

@misc{haufs_xtremes_24,
  author       = {Haufs, Erik},
  title        = {\texttt{xtremes}, a Python package containing auxiliary EVA functionalities},
  year         = {2026},
  url          = {https://github.com/haufse/xtremes},
  howpublished = {\url{https://github.com/haufse/xtremes}},
}

@article {Mori_1976,
    AUTHOR = {Mori, Toshio},
     TITLE = {Limit laws for maxima and second maxima from strong-mixing
              processes},
   JOURNAL = {Ann. Probability},
  FJOURNAL = {The Annals of Probability},
    VOLUME = {4},
      YEAR = {1976},
    NUMBER = {1},
     PAGES = {122--126},
      ISSN = {0091-1798},
   MRCLASS = {60F05 (60G10 62G30)},
  MRNUMBER = {428392},
MRREVIEWER = {J.\ Pickands, see 11Kxx},
       DOI = {10.1214/aop/1176996190},
       URL = {https://doi.org/10.1214/aop/1176996190},
}

@article {Jennessen2022,
    AUTHOR = {B\"ucher, Axel and Jennessen, Tobias},
     TITLE = {Statistical analysis for stationary time series at extreme
              levels: new estimators for the limiting cluster size
              distribution},
   JOURNAL = {Stochastic Process. Appl.},
  FJOURNAL = {Stochastic Processes and their Applications},
    VOLUME = {149},
      YEAR = {2022},
     PAGES = {75--106},
      ISSN = {0304-4149,1879-209X},
   MRCLASS = {62G32 (60G70)},
  MRNUMBER = {4405478},
MRREVIEWER = {Wieslaw\ Dziubdziela},
       DOI = {10.1016/j.spa.2022.03.004},
       URL = {https://doi.org/10.1016/j.spa.2022.03.004},
}

@article {Hsing1988,
    AUTHOR = {Hsing, Tailen},
     TITLE = {On the extreme order statistics for a stationary sequence},
   JOURNAL = {Stochastic Process. Appl.},
  FJOURNAL = {Stochastic Processes and their Applications},
    VOLUME = {29},
      YEAR = {1988},
    NUMBER = {1},
     PAGES = {155--169},
      ISSN = {0304-4149,1879-209X},
   MRCLASS = {60F05 (60G10 60G55)},
  MRNUMBER = {952827},
MRREVIEWER = {Sidney\ I.\ Resnick},
       DOI = {10.1016/0304-4149(88)90035-X},
       URL = {https://doi.org/10.1016/0304-4149(88)90035-X},
}

@article {Novak98,
    AUTHOR = {Novak, Serguei Yu and Weissman, Ishay},
     TITLE = {On the joint limiting distribution of the first and the second maxima},
      NOTE = {Special issue in honor of Marcel F. Neuts},
   JOURNAL = {Comm. Statist. Stochastic Models},
  FJOURNAL = {Communications in Statistics. Stochastic Models},
    VOLUME = {14},
      YEAR = {1998},
    NUMBER = {1-2},
     PAGES = {311--318},
      ISSN = {0882-0287},
   MRCLASS = {60G70 (60F05)},
  MRNUMBER = {1617544},
MRREVIEWER = {Toshio\ Mori},
       DOI = {10.1080/15326349808807473},
       URL = {https://doi.org/10.1080/15326349808807473},
}

@article {Lo2017,
    AUTHOR = {Lo, Ambrose},
     TITLE = {Functional generalizations of {H}oeffding's covariance lemma
              and a formula for {K}endall's tau},
   JOURNAL = {Statist. Probab. Lett.},
  FJOURNAL = {Statistics \& Probability Letters},
    VOLUME = {122},
      YEAR = {2017},
     PAGES = {218--226},
      ISSN = {0167-7152,1879-2103},
   MRCLASS = {60E05 (62H20)},
  MRNUMBER = {3584161},
MRREVIEWER = {Wolfgang\ Trutschnig},
       DOI = {10.1016/j.spl.2016.11.016},
       URL = {https://doi.org/10.1016/j.spl.2016.11.016},
}

@article{bucher2025bootstrapping,
    author = {Bücher, Axel and Staud, Torben},
    title = {Bootstrapping estimators based on the block maxima method},
    journal = {Journal of the Royal Statistical Society Series B: Statistical Methodology},
    pages = {qkaf060},
    year = {2025},
    month = {09},
    issn = {1369-7412},
    doi = {10.1093/jrsssb/qkaf060},
    url = {https://doi.org/10.1093/jrsssb/qkaf060},
    eprint = {https://academic.oup.com/jrsssb/advance-article-pdf/doi/10.1093/jrsssb/qkaf060/64301568/qkaf060_supplementary_data.pdf},
}

@article{tradowsky2023attribution,
  title={Attribution of the heavy rainfall events leading to severe flooding in Western Europe during July 2021},
  author={Tradowsky, Jordis S and Philip, Sjoukje Y and Kreienkamp, Frank and Kew, Sarah F and Lorenz, Philip and Arrighi, Julie and Bettmann, Thomas and Caluwaerts, Steven and Chan, Steven C and De Cruz, Lesley and others},
  journal={Climatic Change},
  volume={176},
  number={7},
  pages={90},
  year={2023},
  publisher={Springer}
}

@article{bader2017automated,
    AUTHOR = {Bader, Brian and Yan, Jun and Zhang, Xuebin},
     TITLE = {Automated selection of {$r$} for the {$r$} largest order
              statistics approach with adjustment for sequential testing},
   JOURNAL = {Stat. Comput.},
  FJOURNAL = {Statistics and Computing},
    VOLUME = {27},
      YEAR = {2017},
    NUMBER = {6},
     PAGES = {1435--1451},
      ISSN = {0960-3174,1573-1375},
   MRCLASS = {62G32 (62C25 62F40 62G10)},
  MRNUMBER = {3687319},
       DOI = {10.1007/s11222-016-9697-3},
       URL = {https://doi.org/10.1007/s11222-016-9697-3},
}

@article{White1982,
    AUTHOR = {White, Halbert},
     TITLE = {Maximum likelihood estimation of misspecified models},
   JOURNAL = {Econometrica},
  FJOURNAL = {Econometrica. Journal of the Econometric Society},
    VOLUME = {50},
      YEAR = {1982},
    NUMBER = {1},
     PAGES = {1--25},
      ISSN = {0012-9682,1468-0262},
   MRCLASS = {62P20},
  MRNUMBER = {640163},
MRREVIEWER = {D.\ E. A. Giles},
       DOI = {10.2307/1912526},
       URL = {https://doi.org/10.2307/1912526},
}

@article{Aarnes2012,
  title = {Wave Extremes in the Northeast Atlantic},
  volume = {25},
  ISSN = {1520-0442},
  url = {http://dx.doi.org/10.1175/JCLI-D-11-00132.1},
  DOI = {10.1175/jcli-d-11-00132.1},
  number = {5},
  journal = {Journal of Climate},
  publisher = {American Meteorological Society},
  author = {Aarnes,  Ole Johan and Breivik,  \O yvind and Reistad,  Magnar},
  year = {2012},
  month = mar,
  pages = {1529–1543}
}

@article{GuedesSoares2004,
  title = {Application of the r largest-order statistics for long-term predictions of significant wave height},
  volume = {51},
  ISSN = {0378-3839},
  url = {http://dx.doi.org/10.1016/j.coastaleng.2004.04.003},
  DOI = {10.1016/j.coastaleng.2004.04.003},
  number = {5–6},
  journal = {Coastal Engineering},
  publisher = {Elsevier BV},
  author = {Guedes Soares,  C. and Scotto,  M.G.},
  year = {2004},
  month = aug,
  pages = {387–394}
}

@inproceedings{Said2011,
  title = {R-largest order statistics for the prediction of bursts and serious deteriorations in network traffic},
  url = {http://dx.doi.org/10.1109/ICCSN.2011.6014960},
  DOI = {10.1109/iccsn.2011.6014960},
  booktitle = {2011 IEEE 3rd International Conference on Communication Software and Networks},
  publisher = {IEEE},
  author = {Said,  Abas Md and Hasbullah,  Halabi and Dahab,  Abdelmahamoud Youssouf},
  year = {2011},
  month = may,
  pages = {579–583}
}

@article{Nemukula2018,
  title = {Modelling average maximum daily temperature using $r$ largest order statistics: An application to South African data},
  volume = {10},
  ISSN = {1996-1421},
  url = {http://dx.doi.org/10.4102/jamba.v10i1.467},
  DOI = {10.4102/jamba.v10i1.467},
  number = {1},
  journal = {Jàmbá: Journal of Disaster Risk Studies},
  publisher = {AOSIS},
  author = {Nemukula,  Murendeni M. and Sigauke,  Caston},
  year = {2018},
  month = may 
}

@article{An2007,
  title = {The r largest order statistics model for extreme wind speed estimation},
  volume = {95},
  ISSN = {0167-6105},
  url = {http://dx.doi.org/10.1016/j.jweia.2006.05.008},
  DOI = {10.1016/j.jweia.2006.05.008},
  number = {3},
  journal = {Journal of Wind Engineering and Industrial Aerodynamics},
  publisher = {Elsevier BV},
  author = {An,  Ying and Pandey,  M.D.},
  year = {2007},
  month = mar,
  pages = {165–182}
}

@article{Busababodhin2021,
  title = {Extreme Value Modeling of Daily Maximum Temperature with the r-Largest Order Statistics},
  volume = {20},
  ISSN = {1513-7805},
  url = {http://dx.doi.org/10.14416/j.appsci.2021.01.003},
  DOI = {10.14416/j.appsci.2021.01.003},
  number = {1},
  journal = {The Journal of Applied Science},
  publisher = {King Mongkut’s University of Technology North Bangkok},
  author = {Busababodhin,  Piyapatr and Chiangpradit,  Monchaya and Papukdee,  Nipada and Ruechairam,  Jiraphon and Ruanthaisong,  Kettida and Guayjarernpanishk,  Pannarat},
  year = {2021},
  month = jun,
  pages = {28–38}
}

@article{Wang1995,
  title = {Selection of the
                    k
                    Largest Order Statistics for the Domain of Attraction of the Gumbel Distribution},
  volume = {90},
  ISSN = {1537-274X},
  url = {http://dx.doi.org/10.1080/01621459.1995.10476607},
  DOI = {10.1080/01621459.1995.10476607},
  number = {431},
  journal = {Journal of the American Statistical Association},
  publisher = {Informa UK Limited},
  author = {Wang,  Julian Z.},
  year = {1995},
  month = sep,
  pages = {1055–1061}
}

@article{Chenavier2019,
    AUTHOR = {Chenavier, Nicolas and Nagel, Werner},
     TITLE = {The largest order statistics for the inradius in an isotropic
              {STIT} tessellation},
   JOURNAL = {Extremes},
  FJOURNAL = {Extremes. Statistical Theory and Applications in Science,
              Engineering and Economics},
    VOLUME = {22},
      YEAR = {2019},
    NUMBER = {4},
     PAGES = {571--598},
      ISSN = {1386-1999,1572-915X},
   MRCLASS = {60D05 (60F05 60G70 62G32)},
  MRNUMBER = {4031850},
MRREVIEWER = {Pedro\ Ter\'an},
       DOI = {10.1007/s10687-019-00356-0},
       URL = {https://doi.org/10.1007/s10687-019-00356-0},
}

@article{Silva2020,
  title = {A change-point model for the r-largest order statistics with applications to environmental and financial data},
  volume = {82},
  ISSN = {0307-904X},
  url = {http://dx.doi.org/10.1016/j.apm.2020.01.064},
  DOI = {10.1016/j.apm.2020.01.064},
  journal = {Applied Mathematical Modelling},
  publisher = {Elsevier BV},
  author = {Silva,  Wyara Vanesa Moura e and Nascimento,  Fernando Ferraz do and Bourguignon,  Marcelo},
  year = {2020},
  month = jun,
  pages = {666–679}
}

@conference{osti_194289,
  author       = {Guedes Soares, C and Ferreira, J A},
  title        = {Modeling long-term distributions of significant wave height},
  url          = {https://www.osti.gov/biblio/194289},
  place        = {United States},
  organization = {},
  publisher    = {American Society of Mechanical Engineers, New York, NY (United States)},
  year         = {1995},
  month        = {12}}

\end{document}